\documentclass[10pt]{amsart}
\usepackage{mathtools,appendix,graphicx}
\usepackage{hyperref,cleveref}
\usepackage{tabularx}
\usepackage{appendix}
\usepackage{algorithm, algorithmic}
\usepackage{amsaddr}
\usepackage{multirow}
\usepackage{subcaption}
\usepackage[margin=1.2in]{geometry}
\usepackage[left]{lineno}

\newtheorem{thm}{Theorem}
\newtheorem{lma}[thm]{Lemma}
\newtheorem{cor}[thm]{Corollary}
\newtheorem{asm}{Assumption}
\newtheorem{subasm}{A}
\newtheorem{defi}[thm]{Definition}
\newtheorem{rmk}[thm]{Remark}
\newtheorem{cond}[thm]{Condition}
\newtheorem*{exs}{Example}

\numberwithin{equation}{section}
\numberwithin{subasm}{asm}
\numberwithin{thm}{section}

\crefname{thm}{theorem}{theorems}
\crefname{lma}{lemma}{lemmas}
\crefname{equation}{equation}{equations}
\crefname{cor}{corollary}{corollaries}
\crefname{asm}{assumption}{assumptions}
\crefname{def}{definition}{definitions}

\begin{document}
\title[Multiaffine ADMM]{ADMM for Multiaffine Constrained Optimization}

\date{Oct 04, 2019.}
\author{Wenbo Gao$^{\dagger}$, Donald Goldfarb$^{\dagger}$, and Frank E. Curtis$^{\ddagger}$}
\begin{abstract}
We expand the scope of the alternating direction method of multipliers (ADMM). Specifically, we show that ADMM, when employed to solve problems with multiaffine constraints that satisfy certain verifiable assumptions, converges to the set of constrained stationary points if the penalty parameter in the augmented Lagrangian is sufficiently large. When the Kurdyka-\L{}ojasiewicz (K-\L{}) property holds, this is strengthened to convergence to a single constrained stationary point. Our analysis applies under assumptions that we have endeavored to make as weak as possible. It applies to problems that involve nonconvex and/or nonsmooth objective terms, in addition to the multiaffine constraints that can involve multiple (three or more) blocks of variables. To illustrate the applicability of our results, we describe examples including nonnegative matrix factorization, sparse learning, risk parity portfolio selection, nonconvex formulations of convex problems, and neural network training. In each case, our ADMM approach encounters only subproblems that have closed-form solutions.
\end{abstract}

\thanks{$^\dagger$ Department of Industrial Engineering and Operations Research, Columbia University. Research of this author was supported in part by NSF Grant CCF-1527809.}
\thanks{$^{\ddagger}$ Department of Industrial and Systems Engineering, Lehigh University. Research of this author was supported in part by DOE Grant DE-SC0010615 and NSF Grant CCF-1618717.}
\email{\texttt{wg2279@columbia.edu, goldfarb@columbia.edu, frank.e.curtis@gmail.com}}
\maketitle

\newcommand{\RR}{\mathbb{R}}
\newcommand{\QQ}{\mathbb{Q}}
\newcommand{\BB}{\mathcal{B}}
\newcommand{\MM}{\mathcal{M}}
\newcommand{\NN}{\mathcal{N}}
\renewcommand{\AA}{\mathcal{A}}
\newcommand{\DD}{\mathcal{D}}
\newcommand{\UU}{\mathcal{U}}

\newcommand{\lb}{\lbrack}
\newcommand{\rb}{\rbrack}
\newcommand{\la}{\langle}
\newcommand{\ra}{\rangle}
\newcommand{\Sym}{\Sigma}
\newcommand{\VV}{\mathcal{V}}
\newcommand{\SymG}{\mathcal{S}}
\newcommand{\FF}{\mathcal{F}}
\newcommand{\II}{\mathcal{I}}
\newcommand{\JJ}{\mathcal{J}}
\newcommand{\GG}{\mathcal{G}}
\newcommand{\Lag}{\mathcal{L}}
\newcommand{\XX}{\mathcal{X}}
\newcommand{\YY}{\mathcal{Y}}
\newcommand{\WW}{\mathcal{W}}
\newcommand{\CC}{\mathcal{C}}
\newcommand{\TT}{\mathcal{T}}
\newcommand{\ZZ}{\mathcal{Z}_>}
\newcommand{\ZZZ}{\mathcal{Z}}

\newcommand{\bump}{\hspace{1.5em}}
\newcommand{\ull}{\underline{\lambda}}
\newcommand\argmin{\opn{argmin}}

\newcommand\opn[1]{\operatorname{#1}}
\newcommand\wh[1]{{\widehat{#1}}}
\newcommand\wt[1]{{\widetilde{#1}}}
\newcommand\ov[1]{\overline{#1}}
\newcommand\ud[1]{\underline{#1}}

\newcommand\ux[2]{{#1}^{#2}}
\newcommand\bv[2]{{\begin{pmatrix} #1 \\ #2 \end{pmatrix}}}

\DeclarePairedDelimiter{\floor}{\lfloor}{\rfloor}

\section{Introduction}\label{sec:introduction}
The \emph{alternating direction method of multipliers} (ADMM) is an iterative method which, in its original form, solves linearly-constrained separable optimization problems with the following structure:
$$
(P0) \bump
\left\{
\begin{array}{rll}
\inf\limits_{x,y} &f(x) + g(y) \\
& Ax + By - b= 0.
\end{array}
\right.$$
The \emph{augmented Lagrangian} $\Lag$ of the problem $(P0)$, for some \emph{penalty parameter} $\rho > 0$, is defined to be
$$\Lag(x,y,w) = f(x) + g(y) + \la w, Ax + By - b\ra + \frac{\rho}{2}\|Ax + By - b\|^2.$$
In iteration $k$, with the iterate $(\ux{x}{k},\ux{y}{k},\ux{w}{k})$, ADMM takes the following steps:
\begin{enumerate}
\item Minimize $\Lag(x, \ux{y}{k}, \ux{w}{k})$ with respect to $x$ to obtain $\ux{x}{k+1}$.
\item Minimize $\Lag(\ux{x}{k+1},y,\ux{w}{k})$ with respect to $y$ to obtain $\ux{y}{k+1}$.
\item Set $\ux{w}{k+1} \gets \ux{w}{k} + \rho(A\ux{x}{k+1} + B\ux{y}{k+1} - b)$.
\end{enumerate}

ADMM was first proposed \cite{GM1976,GM1975} for solving variational problems, and was subsequently applied to convex optimization problems with two blocks as in $(P0)$. Several techniques can be used to analyze this case, including an operator-splitting approach \cite{EB1992,LM}. The survey articles \cite{EY2015,BPCPE2011} provide convergence proofs from several viewpoints, and discuss numerous applications of ADMM. More recently, there has been considerable interest in extending ADMM convergence guarantees when solving problems with \emph{multiple blocks} and \emph{nonconvex} objective functions. ADMM directly extends to the problem 
\begin{equation*}\label{multiblocklinearadmm}
(P1) \bump
\left\{
\begin{array}{rll}
\inf\limits_{x_1,x_2,\ldots,x_n} &f_1(x_1) + f_2(x_2) + \ldots + f_n(x_n) \\
& A_1x_1 + A_2x_2 + \ldots + A_nx_n - b= 0
\end{array}
\right.\end{equation*}
by minimizing $\Lag(x_1,\ldots,x_n,w)$ with respect to $x_1,x_2,\ldots,x_n$ successively. The multiblock problem turns out to be significantly different from the classical 2-block problem, even when the objective function is convex; for example, \cite{ADMM_COUNTER} exhibits an example with $n = 3$ blocks and $f_1,f_2,f_3 \equiv 0$ for which ADMM diverges for any value of $\rho$. Under certain conditions, the unmodified 3-block ADMM does converge. In \cite{LMZ2018}, it is shown that if $f_3$ is strongly convex with condition number $\kappa \in \lb 1, 1.0798)$ (among other assumptions), then 3-block ADMM is globally convergent. If $f_1,\ldots,f_n$ are all strongly convex, and $\rho > 0$ is sufficiently \emph{small}, then \cite{HY2012} shows that multiblock ADMM is convergent. Other works along these lines include \cite{LMZ2015SIOPT,LMZ2015,LST2015}.

In the absence of strong convexity, modified versions of ADMM have been proposed that can accommodate multiple blocks. In \cite{DY2017} a new type of 3-operator splitting is introduced that yields a convergent 3-block ADMM (see also \cite{R2018} for a proof that a `lifting-free' 3-operator extension of Douglas-Rachford splitting does not exist). Convergence guarantees for multiblock ADMM can also be achieved through variants such as proximal ADMM, majorized ADMM, linearized ADMM \cite{STY2015,LST2016,DLPY2017,CST2017MP,LLL2015,BN2018arxiv}, and proximal Jacobi ADMM \cite{DLPY2017,WS2017JCAM,SS2018JAMC}.

ADMM has also been extended to problems with \emph{nonconvex} objective functions. In \cite{HLR2016}, it is proved that ADMM converges when the problem $(P1)$ is either a nonconvex \emph{consensus} or \emph{sharing} problem, and \cite{WYZ2019JSC} proves convergence under more general conditions on $f_1,\ldots,f_n$ and $A_1,\ldots,A_n$. Proximal ADMM schemes for nonconvex, nonsmooth problems are considered in \cite{LP2015,ZMZ2017,JLMZ2019,BN2018arxiv}. More references on nonconvex ADMM, and comparisons of the assumptions used, can be found in \cite{WYZ2019JSC}.

In all of the work mentioned above, the system of constraints $C(x_1,\ldots,x_n) = 0$ is assumed to be linear. Consequently, when all variables other than $x_i$ have fixed values, $C(x_1,\ldots,x_n)$ becomes an \emph{affine} function of $x_i$. However, this holds for more general constraints $C(\cdot)$ in the much larger class of \emph{multiaffine} maps (see \Cref{sec:multiaffinedef}). Thus, it seems reasonable to expect that ADMM would behave similarly when the constraints $C(x_1,\ldots,x_n) = 0$ are permitted to be multiaffine. To be precise, consider a more general problem than $(P1)$ of the form
$$
(P2) \bump
\left\{
\begin{array}{rll}
\inf\limits_{x_1,x_2,\ldots,x_n} & f(x_1,\ldots,x_n) \\
& C(x_1,\ldots,x_n) = 0.
\end{array}
\right.$$
The augmented Lagrangian for $(P2)$ is 
$$\Lag(x_1,\ldots,x_n,w) = f(x_1,\ldots,x_n) + \la w, C(x_1,\ldots,x_n)\ra + \frac{\rho}{2}\|C(x_1,\ldots,x_n)\|^2,$$ and ADMM for solving this problem is specified in \Cref{alg:admm}.

\begin{algorithm}[ht]
	\caption{ADMM}
	\label{alg:admm}
	\begin{algorithmic}
		\STATE{\textbf{Input:} $(x_1^0,\ldots,x_n^0), w^0, \rho$}
		\FOR{$k = 0, 1, 2, \ldots$}
		\FOR{$i = 1,\ldots,n$}
		\STATE Compute $\ux{x}{k+1}_i \in \argmin_{x_i} \Lag(\ux{x}{k+1}_1,\ldots, \ux{x}{k+1}_{i-1}, x_i, \ux{x}{k}_{i+1},\ldots,\ux{x}{k}_{n}, \ux{w}{k})$
		\ENDFOR
		\STATE $\ux{w}{k+1} \leftarrow \ux{w}{k} + \rho C(\ux{x}{k+1}_1,\ldots,\ux{x}{k+1}_n)$
		\ENDFOR
	\end{algorithmic}
\end{algorithm}

While many problems can be modeled with multiaffine constraints, existing work on ADMM for solving multiaffine constrained problems appears to be limited. Boyd et al.~\cite{BPCPE2011} propose solving the nonnegative matrix factorization problem formulated as a problem with biaffine constraints, i.e.,
\begin{equation*}\label{originalnmf}
(\text{NMF1}) \bump
\left\{
\begin{array}{rll}
\inf\limits_{Z,X,Y} & \frac{1}{2}\|Z - B\|^2  \\
& Z = XY, X \geq 0, Y \geq 0,
\end{array}
\right.
\end{equation*}
by applying ADMM with alternating minimization on the blocks $Y$ and $(X,Z)$. The convergence of ADMM employed to solve the (NMF1) problem appears to have been an open question until a proof was given in \cite{HCWSH2016}\footnote{\cite{HCWSH2016} shows that every limit point of ADMM for the problem (NMF) is a constrained stationary point, but does not show that such limit points necessarily exist.}. A method derived from ADMM has also been proposed for optimizing a biaffine model for training deep neural networks \cite{TBXSPG2016}. For general nonlinear constraints, a framework for ``monitored'' Lagrangian-based multiplier methods was studied in \cite{BST2018MOR}.

In this paper, we establish the convergence of ADMM for a broad class of problems with multiaffine constraints. Our assumptions are similar to those used in \cite{WYZ2019JSC} for nonconvex ADMM; in particular, we do not make any assumption about the iterates generated by the algorithm. Hence, these results extend the applicability of ADMM to a larger class of problems which naturally have multiaffine constraints. Moreover, we prove several results about ADMM in \Cref{sec:general} that hold in even more generality, and thus may be useful for analyzing ADMM beyond the setting considered here.

\subsection{Organization of this paper}
In \Cref{sec:multiaffinedef}, we define multilinear and multiaffine maps, and specify the precise structure of the problems that we consider. In \Cref{sec:examples}, we provide several examples of problems that can be formulated with multiaffine constraints. In \Cref{sec:mainresults}, we state our assumptions and main results (i.e., \Cref{main,main2,mainkl}). In \Cref{sec:prelim}, we present a collection of necessary technical material. In \Cref{sec:general}, we prove several results about ADMM that hold under weak conditions on the objective function and constraints. Finally, in \Cref{sec:convergence}, we complete the proof of our main convergence theorems (\Cref{main,main2,mainkl}), by applying the general techniques developed in \Cref{sec:general}. \Cref{apdx:technical} contains proofs of technical lemmas. \Cref{apdx:nnet} presents an alternative biaffine formulation for deep neural network training. \Cref{apdx:tables} organizes the major assumptions in tabular form. \Cref{apdx:simple} presents additional formulations of problems where all ADMM subproblems have closed-form solutions. \Cref{apdx:numeric} presents several numerical experiments.

\subsection{Notation and Definitions}
We consider only finite-dimensional real vector spaces. The symbols $\mathbb{E}, \mathbb{E}_1,\ldots,\mathbb{E}_n$ denote finite-dimensional Hilbert spaces, equipped with inner products $\la \cdot, \cdot \ra$. By default, we use the standard inner product on $\RR^n$ and the trace inner product $\la X,Y \ra = \opn{Tr}(Y^TX)$ on the matrix space. Unless otherwise specified, the norm $\|\cdot\|$ is always the induced norm of the inner product. When $A$ is a matrix or linear map, $\|A\|_{op}$ denotes the $L_2$ operator norm, and $\|A\|_*$ denotes the nuclear norm (the sum of the singular values of $A$). Fixed bases are assumed, so we freely use various properties of a linear map $A$ that depend on its representation (such as $\|A\|_{op}$), and view $A$ as a matrix as required.

For $f: \RR^n \rightarrow \RR \cup \{\infty\}$, the \emph{effective domain} $\opn{dom}(f)$ is the set $\{x : f(x) < \infty\}$. The image of a function $f$ is denoted by $\opn{Im}(f)$.  Similarly, when $A$ is a linear map represented by a matrix, $\opn{Im}(A)$ is the column space of $A$. We use $\opn{Null}(A)$ to denote the null space of $A$. The orthogonal complement of a linear subspace $U$ is denoted $U^\perp$.

To distinguish the derivatives of \emph{smooth} (i.e., continuously differentiable) functions from subgradients, we use the notation $\nabla_X$ for partial differentiation with respect to $X$, and reserve the symbol $\partial$ for the set of \emph{general subgradients} (\Cref{sec:gensubgradref}); hence, the use of $\nabla f$ serves as a reminder that $f$ is assumed to be smooth. A function $f$ is \emph{Lipschitz differentiable} if it is differentiable and its gradient is Lipschitz continuous. 

When $\XX$ is a tuple of variables $\XX = (X_0,\ldots,X_n)$, we write $\XX_{\neq \ell}$ for $(X_i : i \neq \ell)$. Similarly, $\XX_{>\ell}$ and $\XX_{<\ell}$ represent $(X_i: i > \ell)$ and $(X_i: i < \ell)$ respectively.

We use the term \emph{constrained stationary point} for a point satisfying necessary first-order optimality conditions; this is a generalization of the Karush-Kuhn-Tucker (KKT) necessary conditions to nonsmooth problems. For the problem $\min_x \{f(x): C(x) = 0\}$, where $C$ is smooth and $f$ possesses general subgradients, $x^\ast$ is a constrained stationary point if $C(x^\ast) = 0$ and there exists $w^\ast$ with $0 \in \partial f(x^\ast) + \nabla C(x^\ast)^Tw^\ast$.

\section{Multiaffine Constrained Problems}\label{sec:multiaffinedef}

The central objects of this paper are multilinear and multiaffine maps, which generalize linear and affine maps.

\begin{defi}
A map $\MM: \mathbb{E}_1 \oplus \ldots \oplus \mathbb{E}_n \rightarrow \mathbb{E}$ is \emph{multilinear} if, for all $i \leq n$ and all points $(\ov{X}_1,\ldots,\ov{X}_{i-1},\ov{X}_{i+1},\ldots,\ov{X}_n) \in \bigoplus_{j \neq i} E_j$, the map $\MM_i: \mathbb{E}_i \rightarrow \mathbb{E}$ given by
$$X_i \mapsto \MM(\ov{X}_1,\ldots, \ov{X}_{i-1},X_i,\ov{X}_{i+1},\ldots,\ov{X}_n)$$ is \emph{linear}. Similarly, $\MM$ is \emph{multiaffine} if the map $\MM_i$ is \emph{affine} for all $i$ and all points of $\bigoplus_{j \neq i} \mathbb{E}_j$. In particular, when $n = 2$, we say that $\MM$ is \emph{bilinear}/\emph{biaffine}.
\end{defi}

We consider the convergence of ADMM for problems of the form:
$$
(P) \bump 
\left\{
\begin{array}{rll}
\displaystyle \inf_{\XX,\ZZZ} &\phi(\XX,\ZZZ) \\
& A(\XX,Z_0) + Q(\ZZ) = 0,
\end{array}
\right.
$$
where $\XX = (X_0,\ldots,X_n)$, $\ZZZ = (Z_0, \ZZ)$, $\ZZ = (Z_1,Z_2)$,
$$\label{problemP}
\begin{aligned}
  \phi(\XX,\ZZZ) &= f(\XX) + \psi(\ZZZ) \\
  \text{and}\ \ 
  A(\XX,Z_0) + Q(\ZZ) &= \begin{bmatrix} A_1(\XX,Z_0) + Q_1(Z_1) \\ A_2(\XX) + Q_2(Z_2) \end{bmatrix}
\end{aligned}
$$
with $A_1$ and $A_2$ being multiaffine maps and $Q_1$ and $Q_2$ being linear maps. 
The augmented Lagrangian $\Lag(\XX,\ZZZ,\WW)$, with penalty parameter $\rho > 0$, is given by $$
\Lag(\XX,\ZZZ,\WW) = \phi(\XX,\ZZZ) + \la \WW, A(\XX,Z_0) + Q(\ZZ) \ra + \frac{\rho}{2}\|A(\XX,Z_0) + Q(\ZZ)\|^2,$$
where $\WW = (W_1,W_2)$ are Lagrange multipliers.

We prove that \Cref{alg:admm} converges to a constrained stationary point under certain assumptions on $\phi,A$, and $Q$, which are described in \Cref{sec:mainresults}. Moreover, since the constraints are nonlinear, there is a question of constraint qualifications, which we address in \Cref{crcqlma}.

We adopt the following notation in the context of ADMM. The variables in the $k$-th iteration are denoted $\ux{\XX}{k}, \ux{\ZZZ}{k}, \ux{\WW}{k}$ (with $\ux{X}{k}_i, \ux{Z}{k}_i, \ux{W}{k}_i$ for the $i$-th variable in each component). When analyzing a single iteration, the index $k$ is omitted, and we write $X = \ux{X}{k}$ and $X^+ = \ux{X}{k+1}$. Similarly, we write $\ux{\Lag}{k} = \Lag(\ux{\XX}{k}, \ux{\ZZZ}{k}, \ux{\WW}{k})$ and will refer to $\Lag = \ux{\Lag}{k}$ and $\Lag^+ = \ux{\Lag}{k+1}$ for values within a single iteration.

\section{Examples of Applications}\label{sec:examples}

In this section, we describe several problems with multiaffine constraints, and show how they can be formulated and solved by ADMM. Many important applications of ADMM involve introducing auxiliary variables so that all subproblems have closed-form solutions; we describe several such reformulations in \cref{apdx:simple} that have this property.

\subsection{Representation Learning}\label{sec:dl}
Given a matrix $B$ of data, it is often desirable to represent $B$ in the form $B = X \ast Y$, where $\ast$ is a bilinear map and the matrices $X,Y$ have some desirable properties. Two important applications follow:
\begin{enumerate}
\item \emph{Nonnegative matrix factorization} (NMF) \cite{LS1999Nature,LS2000NIPS} expresses $B$ as a product of nonnegative matrices $X \geq 0, Y \geq 0$.
\item \emph{Inexact dictionary learning} (DL) \cite{MBPS2010JMLR} expresses every element of $B$ as a sparse combination of \emph{atoms} from a \emph{dictionary} $X$. It is typically formulated as
\begin{align*}
(\text{DL}) \bump
\left\{
\begin{array}{rll}
\inf\limits_{X,Y}& \iota_S(X) + \|Y\|_1 + \frac{\mu}{2}\|XY - B\|^2, \\
\end{array}
\right.
\end{align*}
where $\iota_S$ is the indicator function for the set $S$ of matrices whose columns have unit $L_2$ norm, and here $\|Y\|_1$ is the entrywise 1-norm $\sum_{i,j} |Y_{ij}|$. The parameter $\mu$ is an input that sets the balance between trying to recover $B$ with high fidelity versus finding $Y$ with high sparsity.
\end{enumerate}

Problems of this type can be modeled with bilinear constraints. As already mentioned in \Cref{sec:introduction}, \cite{BPCPE2011,HCWSH2016} propose the bilinear formulation (NMF1) for nonnegative matrix factorization. The inexact dictionary learning problem can similarly be formulated as:
\begin{equation*}
(\text{DL1}) \bump
\left\{
\begin{array}{rll}
\inf\limits_{Z,X,Y} & \iota_S(X) + \|Y\|_1 + \frac{1}{2}\|Z - B\|^2  \\
& Z = XY.
\end{array}
\right.
\end{equation*}

Other variants of dictionary learning such as \emph{convolutional dictionary learning} (CDL), that cannot readily be handled by the method in \cite{MBPS2010JMLR}, have a biaffine formulation which is nearly identical to (DL1), and can be solved using ADMM. For more information on dictionary learning, see \cite{EA2006IEEETRIMP,SQW2017IEEETR1,SQW2017IEEETR2,MBPS2010JMLR,WZLP2012CVPR}.

\subsection{Non-Convex Reformulations of Convex Problems}
Recently, various low-rank matrix and tensor recovery problems have been shown to be efficiently solvable by applying first-order methods to nonconvex reformulations of them.  For example, the convex \emph{Robust Principal Component Analysis} (RPCA) \cite{HE2004BI,CLMW2011JACM} problem
$$(\text{RPCA1}) \bump 
\left\{
\begin{array}{rll}
\inf\limits_{L,S} & \|L\|_* + \lambda \|S\|_1  \\
& L + S = B
\end{array} \right.
$$
can be reformulated as the biaffine problem
$$(\text{RPCA2}) \bump 
\left\{
\begin{array}{rll}
\inf\limits_{U,V,S} & \frac{1}{2}(\|U\|_F^2 + \|V\|_F^2)  + \lambda \|S\|_1  \\
& UV^T + S = B\\
& U \in \RR^{m \times k}, V \in \RR^{n \times n}, S \in \RR^{m \times n}
\end{array} \right.
$$
as long as $k \geq \opn{rank}(L^\ast)$, where $L^\ast$ is an optimal solution of (RPCA1). See \cite{DBA} for a proof of this, and applications of the factorization $UV^T$ to other problems. This is also related to the \emph{Burer-Monteiro approach} \cite{BM2003MP} for semidefinite programming. We remark that (RPCA2) does not satisfy all the assumptions needed for the convergence of ADMM (see \Cref{aobj} and \Cref{tightness}), so slack variables must be added.

\subsection{Max-Cut}
Given a graph $G = (V,E)$ and edge weights $w \in \RR^E$, the (weighted) maximum cut problem is to find a subset $U \subseteq V$ so that $\sum\limits_{u \in U, v \notin U} w_{uv}$ is maximized. This problem is well-known to be NP-hard \cite{K1972CCC}. An approximation algorithm using semidefinite programming can be shown to achieve an approximation ratio of roughly $0.878$ \cite{GW1995JACM}. Applying the Burer-Monteiro approach \cite{BM2003MP} to the max-cut semidefinite program \cite{GW1995JACM} with a rank-one constraint, and introducing a slack variable (see \Cref{afinalblockim}), we obtain the problem
$$(\text{MC1})  \bump 
\left\{
\begin{array}{rll}
\sup\limits_{Z,x,y,s} & \frac{1}{2} \sum\limits_{uv \in E} w_{uv}(1 - Z_{uv}) + \frac{\mu_1}{2} \sum\limits_{u \in V} (Z_{uu} - 1)^2 + \frac{\mu_2}{2} \|s\|^2 \\
& Z = xy^T, \ x - y = s.
\end{array} \right.
$$
It is easy to verify that all subproblems have very simple closed-form solutions.

\subsection{Risk Parity Portfolio Selection}
Given assets indexed by $\{1,\ldots,n\}$, the goal of risk parity portfolio selection is to construct a portfolio weighting $x \in \RR^n$ in which every asset contributes an equal amount of risk. This can be formulated with quadratic constraints; see \cite{BS2015} for details. The feasibility problem in \cite{BS2015} is
$$
(\text{RP}) \bump
\left\{
\begin{array}{l}
x_i(\Sigma x)_i = x_j(\Sigma x)_j \bump \forall i,j\\
a \leq x \leq b, \bump x_1 + \ldots + x_n = 1
\end{array} \right.$$
where $\Sigma$ is the (positive semidefinite) \emph{covariance matrix} of the asset returns, and $a$ and $b$ contain lower and upper bounds on the weights, respectively. The authors in \cite{BS2015} introduce a variable $y = x$ and solve (RP) using ADMM by replacing the quadratic risk-parity constraint by a \emph{fourth-order} penalty function $f(x,y,\theta) = \sum_{i=1}^n (x_i(\Sigma y)_i - \theta)^2$. To rewrite this problem with a bilinear constraint, let $\circ$ denote the Hadamard product $(x \circ y)_i = x_iy_i$ and let $P$ be the matrix $\begin{pmatrix} 0 & 0 \\ e_{n-1} & -I_{n-1} \end{pmatrix}$, where $e_n$ is the all-ones vector of length $n$. Let $X$ be the set of permissible portfolio weights $X = \{ x \in \RR^n: a \leq x \leq b \} \cap \{ x \in \RR^n: e_n^Tx = 1\}$, and let $\iota_X$ be its indicator function. Then we obtain the problem
$$
(\text{RP}1) \bump
\left\{
\begin{array}{rll}
\inf\limits_{x,y,z,s} &  \iota_X(x) + \frac{\mu}{2}(\|z\|^2 + \|s\|^2) \\
& P(x \circ y) = z\\
& y - \Sigma x = s
\end{array}
\right.$$
where we have introduced a slack variable $s$ (see \Cref{afinalblockim}).

\subsection{Training Neural Networks}
An alternating minimization approach is proposed in \cite{TBXSPG2016} for training deep neural networks. By decoupling the linear and nonlinear elements of the network, the backpropagation required to compute the gradient of the network is replaced by a series of subproblems which are easy to solve and readily parallelized. For a network with $L$ layers, let $X_\ell$ be the matrix of edge weights for $1 \leq \ell \leq L$, and let $a_\ell$ be the output of the $\ell$-th layer for $0 \leq \ell \leq L - 1$. Deep neural networks are defined by the structure $a_\ell = h(X_\ell a_{\ell-1})$, where $h(\cdot)$ is an \emph{activation function}, which is often taken to be the rectified linear unit (ReLU) $h(z) = \max\{z,0\}$. The splitting used in \cite{TBXSPG2016} introduces new variables $z_\ell$ for $1 \leq \ell \leq L$ so that the network layers are no longer directly connected, but are instead coupled through the relations $z_\ell = X_\ell a_{\ell-1}$ and $a_\ell = h(z_\ell)$.

Let $E(\cdot,\cdot)$ be an error function, and $R$ a regularization function on the weights. Given a matrix of labeled training data $(a_0,y)$, the learning problem is
$$
(\text{DNN}1) \bump
\left\{
\begin{array}{rll}
\inf\limits_{\{X_\ell\}, \{a_\ell\}, \{z_\ell\}} & E(z_L,y) + R(X_1,\ldots,X_L) \\
&z_\ell - X_\ell a_{\ell - 1} = 0 \text{ for } 1 \leq \ell \leq L \\
& a_\ell - h(z_\ell) = 0 \text{ for } 1 \leq \ell \leq L - 1.
\end{array}
\right.$$

The algorithm proposed in \cite{TBXSPG2016} does not include any regularization $R(\cdot)$, and replaces \emph{both} sets of constraints by quadratic penalty terms in the objective, while maintaining Lagrange multipliers only for the final constraint $z_L = W_La_{L-1}$. However, since all of the equations $z_\ell = X_\ell a_{\ell - 1}$ are biaffine, we can include them in a biaffine formulation of the problem:
$$
(\text{DNN}2) \bump
\left\{
\begin{array}{rll}
\inf\limits_{\{X_\ell\}, \{a_\ell\}, \{z_\ell\}} & E(z_L,y) + R(X_1,\ldots,X_L) + \frac{\mu}{2}\sum\limits_{\ell=1}^{L-1} (a_\ell - h(z_\ell))^2 \\
&z_\ell - X_\ell a_{\ell - 1} = 0  \text{ for } 1 \leq \ell \leq L.
\end{array}
\right.$$

To adhere to our convergence theory, it would be necessary to apply smoothing (such as Nesterov's technique \cite{N2005}) when $h(z)$ is nonsmooth, as is the ReLU. Alternatively, the ReLU can be replaced by an approximation using nonnegativity constraints (see \Cref{apdx:nnet}). In practice \cite[\S7]{TBXSPG2016}, using the ReLU directly yields simple closed-form solutions, and appears to perform well experimentally. However, no proof of the convergence of the algorithm in \cite{TBXSPG2016} is provided.

\section{Main Results}\label{sec:mainresults}
In this section, we state our assumptions and main results. We will show that ADMM (\Cref{alg:admm_mult}) applied to solve a multiaffine constrained problem of the form $(P)$ (refer to page~\pageref{problemP}) produces a bounded sequence $\{(\ux{X}{k}, \ux{\ZZZ}{k})\}_{k=0}^\infty$, and that every limit point $(\XX^\ast, \ZZZ^\ast)$ is a constrained stationary point. While there are fairly general conditions under which $\ZZZ^\ast$ satisfies first-order optimality conditions (see \Cref{baseasm} and the corresponding discussion in \Cref{tightness} of tightness), the situation with $\XX^\ast$ is more complicated because of the many possible structures of multiaffine maps. Accordingly, we divide the convergence proof into two results. Under one broad set of assumptions, we prove that limit points exist, are feasible, and that $\ZZZ^\ast$ is a blockwise constrained stationary point for the problem with $\XX$ fixed at $\XX^\ast$ (\Cref{main}). Then, we present a set of easily-verifiable conditions under which $(\XX^\ast,\ZZZ^\ast)$ is also a constrained stationary point (\Cref{main2}). If the augmented Lagrangian has additional geometric properties (namely, the Kurdyka-\L{}ojasiewicz property (\Cref{sub:klproperty})), then $\{(\ux{X}{k}, \ux{\ZZZ}{k})\}_{k=0}^\infty$ converges to a single limit point $(\XX^\ast, \ZZZ^\ast)$ (\Cref{mainkl}).

\begin{algorithm}
	\caption{ADMM}
	\label{alg:admm_mult}
	\begin{algorithmic}
		\STATE{\textbf{Input:} $(\ux{X}{0}_0,\ldots,\ux{X}{0}_n), (\ux{Z}{0}_0, \ux{Z}{0}_1,\ux{Z}{0}_2),(\ux{W}{0}_1,\ux{W}{0}_2), \rho$}
		\FOR{$k = 0, 1, 2, \ldots$}
		\FOR{$i = 0,\ldots,n$}
		\STATE Compute $\ux{X}{k+1}_i \in \argmin_{X_i} \Lag(\ux{X}{k+1}_0,\ldots, \ux{X}{k+1}_{i-1}, X_i, \ux{X}{k}_{i+1},\ldots,\ux{X}{k}_{n}, \ux{\ZZZ}{k}, \ux{\WW}{k})$
		\ENDFOR
		\STATE Compute $\ux{\ZZZ}{k+1} \in \argmin_\ZZZ \Lag(\ux{\XX}{k+1}, \ZZZ, \ux{\WW}{k})$
		\STATE $\ux{\WW}{k+1} \leftarrow \ux{\WW}{k} + \rho(A(\ux{\XX}{k+1}, \ux{Z_0}{k+1}) + Q(\ux{\ZZ}{k+1}))$
		\ENDFOR
	\end{algorithmic}
\end{algorithm}

\subsection{Assumptions and Main Results}

We consider two sets of assumption for our analysis. We provide intuition and further discussion of them in \Cref{tightness}. (See \Cref{sec:prelim} for definitions related to convexity and differentiability.)


\begin{asm}\label{baseasm}
Solving problem $(P)$ (refer to page~\pageref{problemP}), the following hold.
\begin{subasm}\label{aRegularity}
For sufficiently large $\rho$, every ADMM subproblem attains its optimal value.
\end{subasm}
\begin{subasm}\label{afinalblockim}
$\opn{Im}(Q) \supseteq \opn{Im}(A)$.
\end{subasm}
\begin{subasm}\label{aobj}
The following statements regarding the objective function $\phi$ and $Q_2$ hold:
\begin{enumerate}
\item $\phi$ is coercive on the feasible region $\Omega = \{ (\XX,\ZZZ) : A(\XX,Z_0) + Q(\ZZ) = 0\}$.
\item $\psi(\ZZZ)$ can be written in the form
$$\psi(\ZZZ) = h(Z_0) + g_1(Z_S) + g_2(Z_2)$$
where
\begin{enumerate}
\item $h$ is proper, convex, and lower semicontinuous.
\item $Z_S$ represents either $Z_1$ or $(Z_0, Z_1)$ and $g_1$ is $(m_1,M_1)$-strongly convex. That is, either $g_1(Z_1)$ is a strongly convex function of $Z_1$ or $g_1(Z_0,Z_1)$ is a strongly convex function of $(Z_0,Z_1)$.
\item $g_2$ is $M_2$-Lipschitz differentiable.
\end{enumerate}
\item $Q_2$ is injective.
\end{enumerate}
\end{subasm}
\end{asm}

While \Cref{baseasm} may appear to be complicated, it is no stronger than the conditions used in analyzing nonconvex, \emph{linearly}-constrained ADMM. A detailed comparison is given in \Cref{tightness}.

Under \Cref{baseasm}, 
\Cref{alg:admm_mult} produces a sequence which has limit points, and every limit point $(\ux{\XX}{\ast}, \ux{\ZZZ}{\ast})$ is feasible with $\ux{\ZZZ}{\ast}$ a constrained stationary point for problem $(P)$ with $\XX$ fixed to $\XX^\ast$.
\begin{thm}\label{main}
Suppose that \Cref{baseasm} holds. For sufficiently large $\rho$, the sequence $\{(\ux{\XX}{k}, \ux{\ZZZ}{k}, \ux{\WW}{k})\}_{k=0}^\infty$ produced by ADMM is bounded, and therefore has limit points. Every limit point $(\XX^\ast, \ZZZ^\ast, \WW^\ast)$ satisfies $A(\XX^\ast,Z_0^\ast) + Q(\ZZ^\ast) = 0$. There exists a sequence $\ux{v}{k} \in \partial_\ZZZ \Lag(\ux{\XX}{k}, \ux{\ZZZ}{k}, \ux{\WW}{k})$ such that $\ux{v}{k} \rightarrow 0$, and thus
\begin{equation}\label{eq:zfoc}
0 \in \partial_\ZZZ \psi(\ZZZ^\ast) + C_{\XX^\ast}^T\WW^\ast
\end{equation}
where $C_{\XX^\ast}$ is the linear map $\ZZZ \mapsto A(\XX^\ast,Z_0) + Q(\ZZ)$ and $C_{\XX^\ast}^T$ is its adjoint. That is, $\ZZZ^\ast$ is a constrained stationary point for the problem
$$\min_{\ZZZ}\ \{ \psi(\ZZZ) : A(\XX^\ast,Z_0) + Q(\ZZ) = 0\}.$$
\end{thm}
\begin{rmk}
Let $\sigma := \lambda_{min}(Q_2^TQ_2)$ \footnote{See \Cref{sec:distandtr} for the definition of $\lambda_{min}$ and $\lambda_{++}$.} and $\kappa_1 := \frac{M_1}{m_1}$. One can check that it suffices to choose $\rho$ so that
\begin{equation}\label{eq:rhoexplicitbd}
\frac{\sigma \rho}{2} - \frac{M_2^2}{\sigma \rho} > \frac{M_2}{2} \bump \text{ and } \bump \rho > \max \left\{
\frac{2M_1 \kappa_1}{\lambda_{++}(Q_1^TQ_1)},
\frac{1}{2}(M_1 + M_2) \max
\left\{
\sigma^{-1},
\frac{(1+2\kappa_1)^2}{\lambda_{++}(Q_1^TQ_1)}
\right\}
\right\}.
\end{equation}
\end{rmk}
Note that \Cref{baseasm} makes very few assumptions about $f(\XX)$ and the map $A$ as a function of $\XX$, other than that $A$ is multiaffine. In \Cref{sec:general}, we develop general techniques for proving that $(\XX^\ast,\ZZZ^\ast)$ is a constrained stationary point. We now present an easily checkable set of conditions, that ensure that the requirements for those techniques are satisfied.

\begin{asm}\label{asm2}
Solving problem $(P)$, \Cref{baseasm} and the following hold.
\begin{subasm}\label{anicefx}
The function $f(\XX)$ splits into
$$f(\XX) = F(X_0,\ldots,X_n) + \sum_{i=0}^n f_i(X_i)$$
where $F$ is $M_F$-Lipschitz differentiable, the functions $f_0$, $f_1$, $\ldots$, and $f_n$ are proper and lower semicontinuous, and each $f_i$ is continuous on $\opn{dom}(f_i)$.
\end{subasm}
\begin{subasm}\label{axainj}
For each $1 \leq \ell \leq n$,\footnote{Note that we have deliberately excluded $\ell = 0$. \Cref{axainj} is not required to hold for $X_0$.} at least one of the following two conditions\footnote{That is, either (1a) and (1b) hold, or (2a) and (2b) hold.} holds:
\begin{enumerate}
\item 
\begin{enumerate}
\item $F(X_0,\ldots,X_n)$ is independent of $X_\ell$.
\item $f_\ell(X_\ell)$ satisfies a strengthened convexity condition (\Cref{defscc}).
\end{enumerate}
\item 
\begin{enumerate}
\item Viewing $A(\XX,Z_0) + Q(\ZZ) = 0$ as a system of constraints
\footnote{As an illustrative example, a problem may be formulated with constraints $X_0X_1 + Z_1 = 0, X_0 + P_1(X_1) + Z_2 = 0, X_0X_2 + Z_3 = 0, P_2(X_2) + Z_4 = 0$, where $P_1, P_2$ are injective linear maps. The notation $A(\XX,Z_0) + Q(\ZZ)$ denotes the concatenation of these equations, which can also be seen naturally as a system of four constraints. In this case, the indices $r(\ell) \in \{1,2,3,4\}$, and \Cref{axainj}(2a) is satisfied by the second constraint $X_0 + P_1(X_1) + Z_2 = 0$ for the variables $X_0, X_1$ (i.e. $r(0) = r(1) = 2$ and $R_0 = I, R_1 = P_1$), and by the fourth constraint $P_2(X_2) + Z_4 = 0$ for $X_2$.
}
, there exists an index $r(\ell)$ such that in the $r(\ell)$-th constraint,
$$A_{r(\ell)}(\XX,Z_0) = R_\ell(X_\ell) + A'_\ell (\XX_{\neq \ell}, Z_0)$$
for an injective linear map $R_\ell$ and a multiaffine map $A'_\ell$. In other words, the only term in $A_{r(\ell)}$ that involves $X_\ell$ is an injective linear map $R_\ell(X_\ell)$.
\item $f_\ell$ is either convex or $M_\ell$-Lipschitz differentiable.
\end{enumerate}
\end{enumerate}
\end{subasm}

\begin{subasm}\label{azainj}
At least one of the following holds for $Z_0$:
\begin{enumerate}
\item $h(Z_0)$ satisfies a strengthened convexity condition (\Cref{defscc}).
\item $Z_0 \in Z_S$, so $g_1(Z_S)$ is a strongly convex function of $Z_0$ and $Z_1$.
\item Viewing $A(\XX,Z_0) + Q(\ZZ) = 0$ as a system of constraints, there exists an index $r(0)$ such that $A_{r(0)}(\XX,Z_0) = R_0(Z_0) + A'_0(\XX)$ for an injective linear map $R_0$ and multiaffine map $A_0'$.
\end{enumerate}
\end{subasm}
\end{asm}

With these additional assumptions on $f$ and $A$, we have that every limit point $(\XX^\ast, \ZZZ^\ast)$ is a constrained stationary point of problem $(P)$.

\begin{thm}\label{main2}
Suppose that \Cref{asm2} holds (and hence, \Cref{baseasm} and \Cref{main}). Then for sufficiently large $\rho$, there exists a sequence $\ux{v}{k} \in \partial \Lag(\ux{\XX}{k}, \ux{\ZZZ}{k}, \ux{\WW}{k})$ with $\ux{v}{k} \rightarrow 0$, and thus every limit point $(\XX^\ast, \ZZZ^\ast)$ is a constrained stationary point of problem $(P)$. Thus, in addition to \eqref{eq:zfoc}, $\XX^\ast$ satisfies, for each $0 \leq i \leq n$,
\begin{equation}\label{eq:xfoc}
0 \in \nabla_{X_i}F(\XX^\ast) + \partial_{X_i} f_i(X^\ast_i) + A_{X_i,( \XX^\ast_{\neq i}, Z_0^\ast)}^T\WW^\ast
\end{equation}
where $A_{X_i,( \XX^\ast_{\neq i}, Z_0^\ast)}$ is the $X_i$-linear term of $\XX \mapsto A(\XX,Z_0)$ evaluated at $(\XX_{\neq i}^\ast, Z_0^\ast)$ (see \Cref{linearterm}) and $A_{X_i,( \XX^\ast_{\neq i}, Z_0^\ast)}^T$ is its adjoint.
That is, for each $0 \leq i \leq n$, $X^\ast_i$ is a constrained stationary point for the problem 
$$\min_{X_i}\ \{ F(\XX_{\neq i}^\ast, X_i) + f_i(X_i) : A(\XX_{\neq i}^\ast, X_i,Z_0^\ast) + Q(\ZZ^\ast) = 0\}.$$
\end{thm}

\begin{rmk}
One can check that it suffices to choose $\rho$ so that, in addition to \eqref{eq:rhoexplicitbd}, we have $\rho > \max\{ \lambda_{min}^{-1}(R_\ell^TR_\ell)(\mu_\ell + M_F) \}$, where the maximum is taken over all $\ell$ for which \Cref{axainj}(2)  holds, and
$$\mu_\ell = \left\{ \begin{array}{cl} 0 & \text{if $f_\ell$ convex } \\ M_\ell & \text{if $f_\ell$ nonconvex, Lipschitz differentiable. } \end{array} \right.$$
\end{rmk}

It is well-known that when the augmented Lagrangian has a geometric property known as the \emph{Kurdyka-\L{}ojasiewicz (K-\L{})} property (see \Cref{sub:klproperty}), which is the case for many optimization problems that occur in practice, then results such as \Cref{main2} can typically be strengthened because the limit point is unique.

\begin{thm}\label{mainkl}
Suppose that $\Lag(\XX,\ZZZ,\WW)$ is a K-\L{} function. Suppose that \Cref{asm2} holds, and furthermore, that \Cref{axainj}(2) holds for all $X_0, X_1,\ldots,X_n$\footnote{Note that $X_0$ is included here, unlike in \Cref{asm2}.}, and \Cref{azainj}(2) holds. Then for sufficiently large $\rho$, the sequence $\{(\ux{\XX}{k}, \ux{\ZZZ}{k}, \ux{\WW}{k})\}_{k=0}^\infty$ produced by ADMM converges to a unique constrained stationary point $(\XX^\ast, \ZZZ^\ast,\WW^\ast)$.
\end{thm}

In \Cref{sec:general}, we develop general properties of ADMM that hold without relying on \Cref{baseasm} or \Cref{asm2}. In \Cref{sec:convergence}, the general results are combined with \Cref{baseasm} and then with \Cref{asm2} to prove \Cref{main} and \Cref{main2}, respectively. Finally, we prove \Cref{mainkl} assuming that the augmented Lagrangian is a K-\L{} function. The results of \Cref{sec:general} may also be useful for analyzing ADMM, since the assumptions required are weak.

\subsection{Discussion of Assumptions}\label{tightness}
\Cref{baseasm,asm2} are admittedly long and somewhat involved. In this section, we will discuss them in detail and explore the extent to which they are tight. Again, we wish to emphasize that despite the additional complexity of multiaffine constraints, the basic content of these assumptions is fundamentally the same as in the linear case. There is also a relation between \Cref{asm2} and proximal ADMM, by which \Cref{axainj}(2) can be viewed as introducing a proximal term. This is described in \Cref{subsub:discussprox}.

\subsubsection{Assumption \ref{aRegularity}}

This assumption is necessary for ADMM (\Cref{alg:admm_mult}) to be well-defined. We note that this can fail in surprising ways; for instance, the conditions used in \cite{BPCPE2011} are insufficient to guarantee that the ADMM subproblems have solutions. In \cite{CST2017}, an example is constructed which satisfies the conditions in \cite{BPCPE2011}, and yet the ADMM subproblem fails to attain its (finite) optimal value.

\subsubsection{Assumption \ref{afinalblockim}}\label{subsub:disc:afinalblockim}

The condition that $\opn{Im}(Q) \supseteq \opn{Im}(A)$ plays a crucial role at multiple points in our analysis because $\ZZ$, a subset of the \emph{final} block of variables, has a close relation to the dual variables $\WW$. It would greatly broaden the scope of ADMM, and simplify modeling, if this condition could be relaxed, but unfortunately this condition is tight for general problems. The following example demonstrates that ADMM is not globally convergent when \Cref{afinalblockim} does not hold, even if the objective function is strongly convex.

\begin{thm}\label{im_counterexample}
Consider the problem
$$\min_{x,y} \{ x^2 + y^2: xy = 1\}.$$
If the initial point is $(\ux{x}{0},0, \ux{w}{0})$, or if $\ux{w}{k} = \rho$ for some $k$, then the ADMM sequence satisfies $(\ux{x}{k}, \ux{y}{k}) \rightarrow (0,0)$ and $\ux{w}{k} \rightarrow -\infty$.
\end{thm}

Even for \emph{linearly}-constrained, convex, multiblock problems, this condition\footnote{For linear constraints $A_1x_1 + \ldots + A_nx_n = b$, the equivalent statement is that $\opn{Im}(A_n) \supseteq \bigcup_{i=1}^{n-1} \opn{Im}(A_i)$.} is close to indispensable. When all the other assumptions except \Cref{afinalblockim} are satisfied, ADMM can still diverge if $\opn{Im}(Q) \not\supseteq \opn{Im}(A)$. In fact, \cite[Thm 3.1]{ADMM_COUNTER} exhibits a simple 3-block convex problem with objective function $\phi \equiv 0$ on which ADMM diverges for any $\rho$. This condition is used explicitly \cite{WYZ2019JSC,LP2015,JLMZ2019} and implicitly \cite{HLR2016} in other analyses of multiblock (nonconvex) ADMM.

\subsubsection{Assumption \ref{aobj}}\label{subsub:discussaobj}

This assumption posits that the entire objective function $\phi$ is coercive on the feasible region, and imposes several conditions on the term $\psi(\ZZZ)$ for the final block $\ZZZ$. 

Let us first consider the conditions on $\psi$. The block $\ZZZ$ is composed of three sub-blocks $Z_0, Z_1, Z_2$, and $\psi(\ZZZ)$ decomposes as $h(Z_0) + g_1(Z_S) + g_2(Z_2)$, where $Z_S$ represents either $Z_1$ or $(Z_0,Z_1)$. There is a distinction between $Z_0$ and $\ZZ = (Z_1,Z_2)$: namely, $Z_0$ may be coupled with the other variables $\XX$ in the nonlinear function $A$, whereas $\ZZ$ appears only in the linear function $Q(\ZZ)$ which satisfies $\opn{Im}(Q) \supseteq \opn{Im}(A)$.

To understand the purpose of this assumption, consider the following `abstracted' assumptions, which are implied by \Cref{aobj}:
\begin{description}
\item[M1] The objective is Lipschitz differentiable with respect to a `suitable' subset of $\ZZZ$.
\item[M2] ADMM yields sufficient decrease \cite{ABS2013} when updating $\ZZZ$. That is, for some `suitable' subset $\wt{\ZZZ}$ of $\ZZZ$ and some $\epsilon > 0$, we have $ \Lag(\XX^+, \ZZZ, \WW) - \Lag(\XX^+,\ZZZ^+,\WW) \geq \epsilon \|\wt{\ZZZ} - \wt{\ZZZ}^+\|^2$.
\end{description}
A `suitable' subset of $\ZZZ$ is one whose associated images in the constraints satisfies \Cref{afinalblockim}. By design, our formulation $(P)$ uses the subset $\ZZ = (Z_1,Z_2)$ in this role.
\textbf{M1} follows from the fact that $g_1, g_2$ are Lipschitz differentiable, and the other conditions in \Cref{aobj} are intended to ensure that \textbf{M2} holds. For instance, the strong convexity assumption in \Cref{aobj}(2) ensures that \textbf{M2} holds with respect to $Z_1$ regardless of the properties of $Q_1$. The concept of sufficient decrease for descent methods is discussed in \cite{ABS2013}.

To connect this to the classical linearly-constrained problem, observe that an assumption corresponding to \textbf{M1} is:
\begin{description}
\item[AL] For the problem $(P1)$ (see page~\pageref{multiblocklinearadmm}), $f_n(x_n)$ is Lipschitz differentiable.
\end{description}
Thus, in this sense $\ZZ$ alone corresponds to the final block in the linearly-constrained case. In the multiaffine setting, we can add a sub-block $Z_0$ to the final block $\ZZZ$, a nonsmooth term $h(Z_0)$ to the objective function and a coupled constraint $A_1(\XX,Z_0)$, but only to a limited extent: the interaction of the final block $\ZZZ$ with these elements is limited to the variables $Z_0$.

As with \Cref{afinalblockim}, it would expand the scope of ADMM if \textbf{AL}, or the corresponding \textbf{M1}, could be relaxed. However, we find that for nonconvex problems, \textbf{AL} cannot readily be relaxed even in the linearly-constrained case, where \textbf{AL} is a standard assumption \cite{WYZ2019JSC,LP2015,JLMZ2019}.
Furthermore, an example is given in \cite[36(a)]{WYZ2019JSC} of a 2-block problem in which the function $f_2(x_2) = \|x_2\|_1$ is nonsmooth, and it is shown that ADMM diverges for any $\rho$ when initialized at a given point. Thus, we suspect that \textbf{AL}/\textbf{M1} is tight for general problems, though it may be possible to prove convergence for specific structured problems not satisfying \textbf{M1}.

\textbf{M1} often has implications for modeling. When a constraint $C(\XX) = 0$ fails to have the required structure, one can introduce a new slack variable $Z$, and replace that constraint by $C(\XX) - Z = 0$, and add a term $g(Z)$ to the objective function to penalize $Z$. Because of \textbf{M1}, exact penalty functions such as $\lambda \|Z\|_1$ or the indicator function of $\{0\}$ fail to satisfy \Cref{aobj}, so this reformulation is not exact. Based on the above discussion, this may be a limitation inherent to ADMM (as opposed to merely an artifact of existing proof techniques).

We turn now to \textbf{M2}. Note that \textbf{AL} corresponds only to \textbf{M1}, which is why \Cref{aobj} is more complicated than \textbf{AL}. There are two main sub-assumptions within \Cref{aobj} that ensure \textbf{M2}: that $g_1$ is strongly convex in $Z_1$, and the map $Q_2$ is injective. These assumptions are \emph{not} tight\footnote{in the sense that this \emph{exact} assumption is always necessary and cannot be replaced.} since \textbf{M2} may hold under alternative hypotheses. On the other hand, we are not aware of other assumptions that are as comparably simple \emph{and} apply with the generality of \Cref{aobj}; hence we have chosen to adopt the latter. For example, if we restrict the problem structure by assuming that the sub-block $Z_0$ is not present, then the condition that $g_1$ is strongly convex can be relaxed to the weaker condition that $\nabla^2 g_1(Z_1) + \rho Q_1^TQ_1 \succeq mI$ for $m > 0$. However, even in the absence of \Cref{aobj}, one might show that specific problems, or classes of structured problems, satisfy the sufficient decrease property, using the general principles of ADMM outlined in \Cref{sec:general}.

Property \textbf{M2} often arises implicitly when analyzing ADMM. In some cases, such as \cite{HLR2016,SLYWY2014}, it follows either from strong convexity of the objective function, or because $A_n = I$ (and is thus injective). Proximal and majorized versions of ADMM are considered in \cite{JLMZ2019,LP2015} and add quadratic terms which force \textbf{M2} to be satisfied. The approach in \cite{WYZ2019JSC}, by contrast, takes a different approach and uses an abstract assumption which relates $\|A_k(x_k^+ - x_k)\|^2$ to $\|x_k^+ - x_k\|^2$; in our experience, it is difficult to verify this abstract assumption in general, except when other properties such as strong convexity or injectivity of $A_k$ hold.

Finally, we remark on the coercivity of $\phi$ over the feasible region. It is common to assume coercivity (see, e.g. \cite{LP2015,WYZ2019JSC}) to ensure that the sequence of iterates is bounded, which implies that limit points exist. In many applications, such as (DL) (\Cref{sec:dl}), $\phi$ is independent of some of the variables. However, $\phi$ can still be coercive over the feasible region. For the variable-splitting formulation (DL3), this holds because of the constraints $X = X' + X''$ and $Y = Y' + Y''$. The objective function is coercive in $X'$, $X''$, $Y'$, and $Y''$, and therefore $X$ and $Y$ cannot diverge on the feasible region.

\subsubsection{Assumption \ref{anicefx}}

The key element of this assumption is that $X_0,\ldots,X_n$ may only be coupled by a Lipschitz differentiable function $F(X_0,\ldots,X_n)$, and the (possibly nonsmooth) terms $f_0(X_0), \ldots, f_n(X_n)$ must be separable. This type of assumption is also used in previous works such as \cite{CLST2016,JLMZ2019,WYZ2019JSC}.

\subsubsection{Assumption \ref{axainj}, \ref{azainj}}\label{subsub:discussaxz}

We have grouped \Cref{axainj}, \Cref{azainj} together here because their motivation is the same. Our goal is to obtain conditions under which the convergence of the function differences $\Lag(\XX_{<\ell}^+, X_\ell, \XX_{>\ell}, \ZZZ, \WW) - \Lag(\XX_{<\ell}^+, X_\ell^+,\XX_{>\ell}, \ZZZ,\WW)$ implies that $\|X_\ell - X_\ell^+\| \rightarrow 0$ (and likewise for $Z_0$). This can be viewed as a much weaker analogue of the sufficient decrease property \textbf{M2}. In \Cref{axainj} and \Cref{azainj}, we have presented several alternatives under which this holds. Under \Cref{axainj}(1) and \Cref{azainj}(1), the strengthened convexity condition (\Cref{defscc}), it is straightforward to show that
\begin{equation}\label{eq:lagubounddiff}
\Lag(\XX_{<\ell}^+, X_\ell, \XX_{>\ell}, \ZZZ, \WW) - \Lag(\XX_{<\ell}^+, X_\ell^+,\XX_{>\ell}, \ZZZ,\WW) \geq \Delta(\|X_\ell - X_\ell^+\|)
\end{equation}
(and likewise for $Z_0$), where $\Delta(t)$ is the $0$-forcing function arising from strengthened convexity. For \Cref{axainj}(2) and \Cref{azainj}(2), the inequality \eqref{eq:lagubounddiff} holds with $\Delta(t) = a t^2$, which is the sufficient decrease condition of \cite{ABS2013}. Note that having $\Delta(t) \in O(t^2)$ is important for proving convergence in the K-\L{} setting, hence the additional hypotheses in \Cref{mainkl}.

As with \Cref{aobj}, the assumptions in \Cref{axainj} and \Cref{azainj} are not tight, because \eqref{eq:lagubounddiff} may occur under different conditions. We have chosen to use this particular set of assumptions because they are easily verifiable, and fairly general. The general results of \Cref{sec:general} may be useful in analyzing ADMM for structured problems when the particular conditions of \Cref{axainj} are not satisfied.

\subsubsection{Connection with proximal ADMM}\label{subsub:discussprox}

When modeling, one may always ensure that \Cref{axainj}(2a) is satisfied for $X_\ell$ by introducing a new variable $Z_3$ and a new constraint $X_\ell = Z_3$. This may appear to be a trivial reformulation of the problem, but it in fact promotes regularity of the ADMM subproblem in the same way as introducing a positive semidefinite proximal term.

Generalizing this trick, let $S$ be positive semidefinite, with square root $S^{1/2}$. Consider the constraint $\sqrt{\frac{2}{\rho}}S^{1/2}(X_\ell - Z_3) = 0$. The term of the augmented Lagrangian induced by this constraint is $\|X_\ell - Z_3\|_S^2$, where $\|\cdot\|_S$ is the seminorm $\|X\|_S^2 = \la X, SX\ra$ induced by $S$. To see this, let $W_0$ be the Lagrange multiplier corresponding to this constraint.
\begin{lma}\label{constraint_prox}
If $W_3^0$ is initialized to $0$, then for all $k \geq 1$, $Z_3^k = X_\ell^k$ and $W_3^k = 0$. Consequently, the constraint $\sqrt{\frac{2}{\rho}}S^{1/2}(X_\ell - Z_3) = 0$ is equivalent to adding a proximal term $\|X_\ell - X_\ell^k\|^2_S$ to the minimization problem for $X_\ell$.
\end{lma}

Note that proximal ADMM is often preferable to ADMM in practice \cite{MST2017,FPST2013}. ADMM subproblems, which may have no closed-form solution because of the linear mapping in the quadratic penalty term, can often be transformed into a pure proximal mapping with a closed-form solution, by adding a suitable proximal term. Several applications of this approach are developed in \cite{FPST2013}. Furthermore, for proximal ADMM, the conditions on $f_i$ in \Cref{axainj}(2b) can be slightly weakened, by modifying \Cref{decgeneric} and \Cref{yzdecr} (see \Cref{proxsimpler}) to account for the proximal term as in \cite{JLMZ2019}.

\section{Preliminaries}\label{sec:prelim}
This section is a collection of definitions, terminology, and technical results which are not specific to ADMM. Proofs of the results in this section can be found in \Cref{apdx:technical}, or in the provided references. The reader may wish to proceed directly to \Cref{sec:general} and return here for details as needed.

\subsection{General Subgradients and First-Order Conditions}\label{sec:gensubgradref}

In order to unify our treatment of first-order conditions, we use the notion of \emph{general subgradients}, which generalize gradients and subgradients. When $f$ is smooth or convex, the set of general subgradients consists of the ordinary gradient or subgradients, respectively. Moreover, some useful functions that are neither smooth nor convex such as the indicator function of certain nonconvex sets possess general subgradients.

\begin{defi}
Let $G$ be a closed and convex set. The \emph{tangent cone} $T_G(x)$ of~$G$ at the point $x \in G$ is the set of directions $T_G(x) = \opn{cl}(\{ y - x : y \in G\})$. The \emph{normal cone} $N_G(x)$ is the set $N_G(x) = \{v : \la v, y - x\ra \leq 0~ \forall y \in G\}$.
\end{defi}
\begin{defi}[\cite{RW1997}, 8.3]\label{defgensubgradient}
Let $f: \RR^n \rightarrow \RR \cup \{\infty\}$ and $x \in \opn{dom}(f)$. A vector~$v$ is a \emph{regular subgradient} of $f$ at $x$, indicated by $v \in \wh{\partial}f(x)$, if $f(y) \geq f(x) + \la v, y - x \ra + o(\|y - x\|)$ for all $y \in \RR^n$. A vector $v$ is a \emph{general subgradient}, indicated by $v \in \partial f(x)$, if there exist sequences $x_n \rightarrow x$ and $v_n \rightarrow v$ with $f(x_n) \rightarrow f(x)$ and $v_n \in \wh{\partial}f(x_n)$. A vector $v$ is a \emph{horizon subgradient}, indicated by $v \in \partial^\infty f(x)$, if there exist sequences $x_n \rightarrow x$, $\lambda_n \rightarrow 0$, and $v_n \in \wh{\partial}f(x_n)$ with $f(x_n) \rightarrow f(x)$ and $\lambda_n v_n \rightarrow v$.
\end{defi}

The properties of the general subgradient can be found in \cite[\S8]{RW1997}.

Under the assumption that the objective function is proper and lower semicontinuous, the ADMM subproblems will satisfy a necessary first-order condition.

\begin{lma}[\cite{RW1997}, 8.15]\label{necfoc}
Let $f: \RR^n \rightarrow \RR \cup \{\infty\}$ be proper and lower semicontinuous over a closed set $G \subseteq \RR^n$. Let $x \in G$ be a point at which the following constraint qualification is fulfilled: the set $\partial^\infty f(x)$ of horizon subgradients contains no vector $v \neq 0$ such that $-v \in N_G(x)$. Then, for $x$ to be a local optimum of $f$ over $G$, it is necessary that $0 \in \partial f(x) + N_G(x)$.
\end{lma}

For our purposes, it suffices to note that when $G = \RR^n$, the constraint qualification is trivially satisfied because $N_G(x) = \{0\}$. In the context of ADMM, this implies that the solution of each ADMM subproblem satisfies the first-order condition $0 \in \partial \Lag$.

Problem $(P)$ has nonlinear constraints, and thus it is not guaranteed \emph{a priori} that its minimizers satisfy first-order necessary conditions, unless a constraint qualification holds. However, \Cref{baseasm} implies that the \emph{constant rank constraint qualification} (CRCQ) \cite{J1984MPS,L2011MP} is satisfied by $(P)$, and minimizers of $(P)$ will therefore satisfy first-order necessary conditions as long as the objective function is suitably regular. This follows immediately from \Cref{afinalblockim} and the following lemma.

\begin{lma}\label{crcqlma}
Let $C(x,z) = A(x) + Qz$, where $A(x)$ is smooth and $Q$ is a linear map with $\opn{Im}(Q) \supseteq \opn{Im}(A)$. Then for any points $x,z$, and any vector $w$, $(\nabla C(x,z))^Tw = 0$ if and only if $Q^Tw = 0$.
\end{lma}

\subsection{Multiaffine Maps}
Every multiaffine map can be expressed as a sum of multilinear maps and a constant. This provides a useful concrete representation.

\begin{lma}\label{multiaffinesumform}
Let $\MM(X_1,\ldots,X_n)$ be a multiaffine map. Then, $\MM$ can be written in the form $\MM(X_1,\ldots,X_n) = B + \sum_{j=1}^m \MM_j(\DD_j)$ where $B$ is a constant, and each $\MM_j(\DD_j)$ is a multilinear map of a subset $\DD_j \subseteq (X_1,\ldots,X_n)$.
\end{lma}

Let $\MM(X_1,\ldots,X_n,Y)$ be multiaffine, with $Y$ a particular variable of interest, and $X = (X_1,\ldots,X_n)$ the other variables. By \Cref{multiaffinesumform}, grouping the multilinear terms $\MM_j$ depending on whether $Y$ is one of the arguments of $\MM_j$, we have
\begin{equation}\label{multiaffinestruc}
\MM(X_1,\ldots,X_n,Y) = B + \sum_{j=1}^{m_1} \MM_j(\DD_j, Y) + \sum_{j=m_1+ 1}^{m} \MM_j(\DD_j)
\end{equation}
where each $\DD_j \subseteq (X_1,\ldots,X_n)$.

\begin{defi}\label{linearterm}
Let $\MM(X_1,\ldots,X_n,Y)$ have the structure \eqref{multiaffinestruc}. Let $\FF_Y$ be the space of functions from $Y \rightarrow \opn{Im}(\MM)$. Let $\theta_j: \DD_j \rightarrow \FF_Y$ be the map\footnote{When $j > m_1$, $\theta_j(X)$ is a constant map of $Y$.} given by $(\theta_j(X))(Y) = \MM_j(\DD_j,Y)$. Here, we use the notation $\theta_j(X)$ for $\theta_j(\DD_j)$, with $\DD_j$ taking the values in $X$. Finally, let $\MM_{Y,X} = \sum_{j=1}^{m_1} \theta_j(X)$.

We call $\MM_{Y,X}$ the \emph{$Y$-linear term} of $\MM$ (evaluated at $X$).
\end{defi}

To motivate this definition, observe that when $X$ is fixed, the map $Y \mapsto \MM(X,Y)$ is \emph{affine}, with the linear component given by $\MM_{Y,X}$ and the constant term given by $B_X = B + \sum_{j=m_1+1}^m \MM_j(\DD_j)$. When analyzing the ADMM subproblem in $Y$, a multiaffine constraint $\MM(X,Y) = 0$ becomes the \emph{linear} constraint $\MM_{Y,X}(Y) = -B_X$.

The definition of multilinearity immediately shows the following.
\begin{lma}\label{decomp1}
$\theta_j$ is a multilinear map of $\DD_j$. For every $X$, $\theta_j(X)$ is a linear map of $Y$, and thus $\MM_{Y,X}$ is a linear map of $Y$.
\end{lma}

\begin{exs}
Consider $\MM(X_1,X_2,X_3,X_4) = X_1X_2X_3 + X_2X_3X_4 + X_2 + B = 0$ for square matrices $X_1,X_2,X_3,X_4$. Taking $Y = X_3$ as the variable of focus, and $X = (X_1,X_2,X_4)$, we have $(\theta_1(X))(Y) = X_1X_2Y$, $(\theta_2(X))(Y) = X_2YX_4$, $(\theta_3(X))(Y) = X_2$, $ (\theta_4(X))(Y) = B$, and thus $\MM_{Y,X}$ is the linear map $Y \mapsto X_1X_2Y + X_2YX_4$.
\end{exs}


Our general results in \Cref{sec:general} require smooth constraints, which holds for multiaffine maps.

\begin{lma}\label{multiaffsmooth}
Multiaffine maps are smooth, and in particular, biaffine maps are Lipschitz differentiable.
\end{lma}

\subsection{Smoothness, Convexity, and Coercivity}

\begin{defi}
A function $g$ is \emph{$M$-Lipschitz differentiable} if $g$ is differentiable and its gradient is Lipschitz continuous with modulus $M$, i.e. $\|\nabla g(x) - \nabla g(y)\| \leq M\|x - y\|$ for all $x,y$.

A function $g$ is \emph{$(m,M)$-strongly convex} if $g$ is convex, $M$-Lipschitz differentiable, and satisfies $g(y) \geq g(x) + \la \nabla g(x), y - x\ra + \frac{m}{2}\|y - x\|^2$ for all $x$ and $y$.  The condition number of $g$ is $\kappa := \frac{M}{m}$.
\end{defi}

\begin{lma}\label{lipschdfbound}
If $g$ is $M$-Lipschitz differentiable, then $|g(y) - g(x) - \la \nabla g(x), y-x\ra| \leq \frac{M}{2}\|y - x\|^2$.
\end{lma}

\begin{lma}\label{convexspec}
If $g(\cdot,\cdot)$ is $M$-Lipschitz differentiable, then for any fixed $y$, the function $h_y(\cdot) = g(\cdot,y)$ is $M$-Lipschitz differentiable. If $g(\cdot,\cdot)$ is $(m,M)$-strongly convex, then $h_y(\cdot)$ is $(m,M)$-strongly convex.
\end{lma}

\begin{defi}
A function $\phi$ is said to be \emph{coercive on the set $\Omega$} if for every sequence $\{x_k\}_{k=1}^\infty \subseteq \Omega$ with $\|x_k\| \rightarrow \infty$, then $\phi(x_k) \rightarrow \infty$.
\end{defi}

\begin{defi}\label{defscc}
A function $\Delta: \RR \rightarrow \RR$ is \emph{0-forcing} if $\Delta \geq 0$, and any sequence $\{t_k\}$ has $\Delta(t_k) \rightarrow 0$ only if $t_k \rightarrow 0$. A function $f$ is said to satisfy a \emph{strengthened convexity condition} if there exists a $0$-forcing function $\Delta$ such that for any $x, y$, and any $v \in \partial f(x)$, $f$ satisfies
\begin{equation}\label{eq:scceq}
f(y) - f(x) - \la v, y - x \ra \geq \Delta(\|y - x\|).
\end{equation}
\end{defi}

\begin{rmk}The strengthened convexity condition is stronger than convexity, but weaker than strong convexity. An example is given by higher-order polynomials of degree $d$ composed with the Euclidean norm, which are not strongly convex for $d > 2$. An example is the function $f(x) = \|x\|_2^3$, which is not strongly convex but does satisfy the strengthened convexity condition. This function appears in the context of the \emph{cubic-regularized Newton method} \cite{NP2006MP}. We note that ADMM can be applied to solve the nonconvex cubic-regularized Newton subproblem
$$\min_x \frac{1}{2}x^TAx + b^Tx + \frac{\mu}{3}\|x\|_2^3$$
by performing the splitting $\min\limits_{x,y} q(x) + h(y)$ for $q(x) = \frac{1}{2}x^TAx + b^Tx$ and $h(y) = \frac{\mu}{3}\|x\|_2^3$.
	
Since we will subsequently show (\Cref{main}) that the sequence of ADMM iterates is bounded, the strengthened convexity condition can be relaxed. It would be sufficient to assume that for every compact set $G$, a 0-forcing function $\Delta_G$ exists so that \eqref{eq:scceq} holds with $\Delta_G$ whenever $x,y \in G$.
\end{rmk}

\subsection{Distances and Translations}\label{sec:distandtr}
\begin{defi}
For a symmetric matrix $S$, let $\lambda_{min}(S)$ be the minimum eigenvalue of $S$, and let $\lambda_{++}(S)$ be the minimum \emph{positive} eigenvalue of $S$.
\end{defi}

\begin{lma}\label{leastim}
Let $R$ be a matrix and $y \in \opn{Im}(R)$. Then $\|y\|^2 \leq \lambda_{++}^{-1}(R^TR)\|R^Ty\|^2$. 
\end{lma}

\begin{lma}\label{subdistgeneric}
Let $A$ be a matrix, and $b,c \in \opn{Im}(A)$. There exists a constant $\alpha_A$ with $\opn{dist}(\{x: Ax = b\},\{x: Ax = c\}) \leq \alpha_A \|b - c\|$. Furthermore, we may take $\alpha_A \leq \sqrt{\lambda_{++}^{-1}(AA^T)}$.
\end{lma}

\begin{lma}\label{cvxargmin}
Let $g$ be a $(m,M)$-strongly convex function with condition number $\kappa = \frac{M}{m}$, let $\CC$ be a closed and convex set with $\CC_1 = a + \CC$ and $\CC_2 = b + \CC$ being two translations of $\CC$, let $\delta = \|b - a\|$, and let $x^\ast = \opn{argmin} \{g(x) : x \in \CC_1\}$ and $y^\ast = \opn{argmin} \{g(y) : y \in \CC_2\}$. Then, $\|x^\ast - y^\ast\| \leq (1 + 2\kappa)\delta$.
\end{lma}

\begin{lma}\label{bicvxargmin}
Let $h$ be a $(m,M)$-strongly convex function, $A$ a linear map of $x$, and $\CC$ a closed and convex set. Let $b_1,b_2 \in \opn{Im}(A)$, and consider the sets $\UU_1 = \{ x : Ax + b_1 \in \CC\}$ and $\UU_2 = \{x : Ax + b_2 \in \CC \}$, which we assume to be nonempty. Let $x^\ast = \argmin\{ h(x) : x \in \UU_1\}$ and $y^\ast = \argmin \{ h(y) : y \in \UU_2\}$. Then, there exists a constant $\gamma$, depending on $\kappa$ and $A$ but independent of $\CC$, such that $\|x^\ast - y^\ast\| \leq \gamma \|b_2 - b_1\|$.
\end{lma}

\subsection{K-\L{} Functions}\label{sub:klproperty}
\begin{defi}
Let $f$ be proper and lower semicontinuous. The domain $\opn{dom}(\partial f)$ of the general subgradient mapping is the set $\{x: \partial f(x) \neq \emptyset\}$.
\end{defi}

\begin{defi}[\cite{ABS2013}, 2.4]
A function $f: \RR^n \rightarrow \RR \cup \{\infty\}$ is said to have the \emph{Kurdyka-\L{}ojasiewicz (K-\L{})} property at $x \in \opn{dom}(\partial f)$ if there exist $\eta \in (0, \infty\rb$, a neighborhood $U$ of $x$, and a continuous concave function $\varphi: \lb 0, \eta) \rightarrow \RR$ such that:
\begin{enumerate}
\item $\varphi(0) = 0$
\item $\varphi$ is smooth on $(0,\eta)$
\item For all $s \in (0,\eta)$, $\varphi'(s) > 0$
\item For all $y \in U \cap \{w: f(x) < f(w) < f(x) + \eta\}$, the Kurdyka-\L{}ojasiewicz inequality holds:
$$\varphi'(f(y) - f(x))\opn{dist}(0,\partial f(x)) \geq 1$$
\end{enumerate}
A proper, lower semicontinuous function $f$ that satisfies the K-\L{} property at every point of $\opn{dom}(\partial f)$ is called a \emph{K-\L{}} function.
\end{defi}

A large class of K-L{} functions is provided by the \emph{semialgebraic functions}, which include many functions of importance in optimization.

\begin{defi}[\cite{ABS2013}, 2.1]
A subset $S$ of $\RR^n$ is \emph{(real) semialgebraic} if there exists a finite number of real polynomial functions $P_{ij}, Q_{ij}: \RR^n \rightarrow \RR$ such that
$$S = \bigcup_{j=1}^p \bigcap_{i=1}^q \{x \in \RR^n: P_{ij}(x) = 0, Q_{ij}(x) < 0\}.$$

A function $f: \RR^n \rightarrow \RR^m$ is \emph{semialgebraic} if its graph $\{ (x,y) \in \RR^{n+m}: f(x) = y\}$ is a real semialgebraic subset of $\RR^{n+m}$.
\end{defi}

The set of semialgebraic functions is closed under taking finite sums and products, scalar products, and composition. The indicator function of a semialgebraic set is a semialgebraic function, as is the generalized inverse of a semialgebraic function. More examples can be found in \cite{ABRS2010}.

The key property of K-\L{} functions is that if a sequence $\{x^k\}_{k=0}^\infty$ is a `descent sequence' with respect to a K-\L{} function, then limit points of $\{x^k\}$ are necessarily unique. This is formalized by the following;

\begin{thm}[\cite{ABS2013}, 2.9]\label{absconvergence}
Let $f: \RR^n \rightarrow \RR$ be a proper and lower semicontinuous function. Consider a sequence $\{x^k\}_{k=0}^\infty$ satisfying the properties:
\begin{description}
\item[H1] There exists $a > 0$ such that for each $k$, $f(x^{k+1})  - f(x^k) \leq -a \|x^{k+1} - x^k\|^2$.
\item[H2] There exists $b > 0$ such that for each $k$, there exists $w^{k+1} \in \partial f(x^{k+1})$ with $\|w^{k+1}\| \leq b\|x^{k+1} - x^k\|$.
\end{description}
If $f$ is a K-\L{} function, and $x^\ast$ is a limit point of $\{x^k\}$ with $f(x^k) \rightarrow f(x^\ast)$, then $x^k \rightarrow x^\ast$.
\end{thm}

\section{General Properties of ADMM}\label{sec:general}

In this section, we will derive results that are inherent properties of ADMM, and require minimal conditions on the structure of the problem. We first work in the most general setting where $C$ in the constraint $C(U_0,\ldots,U_n) = 0$ may be any smooth function, the objective function $f(U_0,\ldots,U_n)$ is proper and lower semicontinuous, and the variables $\{U_0,\ldots,U_n\}$ may be coupled. We then specialize to the case where the constraint $C(U_0,\ldots,U_n)$ is multiaffine, which allows us to quantify the changes in the augmented Lagrangian using the subgradients of $f$. Finally, we specialize to the case where the objective function splits into $F(U_0,\ldots,U_n) + \sum_{i=0}^n g_i(U_i)$ for a smooth coupling function $F$, which allows finer quantification using the subgradients of the augmented Lagrangian.

The results given in this section hold under very weak conditions; hence, these results may be of independent interest, as tools for analyzing ADMM in other settings.

\subsection{General Objective and Constraints}\label{sec:genobjcs}

In this section, we consider 
$$\left\{ \begin{array}{rl} \displaystyle \inf_{U_0,\dots,U_n} & f(U_0,\ldots,U_n) \\
&C(U_0,\ldots,U_n) = 0.\end{array} \right.$$

The augmented Lagrangian is given by $$\Lag(U_0,\ldots,U_n,W) = f(U_0,\ldots,U_n) + \la W, C(U_0,\ldots,U_n)\ra + \frac{\rho}{2}\|C(U_0,\ldots,U_n)\|^2$$ and ADMM performs the updates as in \Cref{alg:admm}. We assume only the following.
\begin{asm}\label{generalasm}
The following hold.
\begin{subasm}
For sufficiently large $\rho$, every ADMM subproblem attains its optimal value.
\end{subasm}
\begin{subasm}
$C(U_0,\ldots,U_n)$ is smooth.
\end{subasm}	
\begin{subasm}
$f(U_0,\ldots,U_n)$ is proper and lower semicontinuous.
\end{subasm}
\end{asm}
This assumption ensures that the $\opn{argmin}$ in \Cref{alg:admm} is well-defined, and that the first-order condition in \Cref{necfoc} holds at the optimal point. The results in this section are extensions of similar results for ADMM in the classical setting (linear constraints, separable objective function), so it is interesting that the ADMM algorithm retains many of the same properties under the generality of \Cref{generalasm}. We defer the proofs to the appendix.

\begin{lma}\label{dualgeneric}
Let $\UU = (U_0,\ldots,U_n)$ denote the set of all variables.  The ADMM update of the dual variable $W$ increases the augmented Lagrangian such that $\Lag(\UU^+,W^+) - \Lag(\UU^+,W) = \rho \|C(\UU^+)\|^2 = \frac{1}{\rho}\|W - W^+\|^2$. If $\|W - W^+\| \rightarrow 0$, then $\nabla_W \Lag(\ux{\UU}{k},\ux{W}{k}) \rightarrow 0$ and every limit point $\UU^\ast$ of $\{\ux{\UU}{k}\}_{k=0}^\infty$ satisfies $C(\UU^\ast) = 0$.
\end{lma}

Consider the ADMM update of the primal variables. ADMM minimizes $\Lag(U_0,\ldots,U_n,W)$ with respect to each of the variables $U_0,\ldots,U_n$ in succession. Let $Y = U_j$ be a particular variable of focus, and let $U = \UU_{\neq j} = (U_i : i \neq j)$ denote the other variables. For fixed $U$, let $f_U(Y) = f(U,Y)$. When $Y$ is given, we let $U_<$ denote the variables that are updated before $Y$, and $U_>$ the variables that are updated after $Y$. The ADMM subproblem for $Y$ is
$$\min_Y \Lag(U,Y,W) = \min_Y f_U(Y) + \la W, C(U,Y)\ra + \frac{\rho}{2}\|C(U,Y)\|^2.$$

\begin{lma}\label{subgradgeneric2}
The general subgradient of $\Lag(U,Y,W)$ with respect to $Y$ is given by
$$\partial_Y \Lag(U,Y,W) = \partial f_U(Y) + (\nabla_Y C(U,Y))^TW + \rho (\nabla_Y C(U,Y))^TC(U,Y)$$
where $\nabla_Y C(U,Y)$ is the Jacobian of $Y \mapsto C(U,Y)$ and $(\nabla_Y C(U,Y))^T$ is its adjoint.

Defining $V(U,Y,W) = (\nabla_Y C(U,Y))^TW + \rho (\nabla_Y C(U,Y))^TC(U,Y)$, the function $V(U,Y,W)$ is continuous, and $\partial_Y \Lag(U,Y,W) = \partial f_U(Y) + V(U,Y,W)$. The first-order condition satisfied by $Y^+$ is therefore
\begin{align*}
0 &\in \partial f_{U_<^+,U_>}(Y^+) + (\nabla_Y C(U_<^+,Y^+,U_>))^TW + \rho (\nabla_Y C(U_<^+,Y^+,U_>))^TC(U_<^+,Y^+,U_>) \\
&=\partial f_{U_<^+,U_>}(Y^+) + V(U_<^+,Y^+,U_>,W).
\end{align*}
\end{lma}

For the next results, we add the following assumption.

\begin{asm}\label{weakcoupling}
The function $f$ has the form $f(U_0,\ldots,U_n) = F(U_0,\ldots,U_n) + \sum_{i=0}^n g_i(U_i)$, where $F$ is smooth and each $g_i$ is continuous on $\opn{dom}(g_i)$.
\end{asm}

\begin{lma}\label{generalinnerblock}
Suppose that \Cref{generalasm,weakcoupling} hold. The general subgradient $\partial_Y \Lag(\ux{U}{k+1},\ux{Y}{k+1},\ux{W}{k+1})$ contains
\begin{align*}
&V(\ux{U}{k+1}_<,\ux{Y}{k+1},\ux{U}{k+1}_>,\ux{W}{k+1}) - V(\ux{U}{k+1}_<,\ux{Y}{k+1},\ux{U}{k}_>,\ux{W}{k}) \\
&\bump + \nabla_Y F(\ux{U}{k+1}_<,\ux{Y}{k+1},\ux{U}{k+1}_>) - \nabla_Y F(\ux{U}{k+1}_<,\ux{Y}{k+1},\ux{U}{k}_>).
\end{align*}

Consider any limit point $(U^\ast, Y^\ast, W^\ast)$ of ADMM. If $\|W^+ - W\| \rightarrow 0$ and $\|U^+_> - U_>\| \rightarrow 0$, then for any subsequence $\{(\ux{U}{k(s)}, \ux{Y}{k(s)}, \ux{W}{k(s)})\}_{s=0}^\infty$ converging to $(U^\ast,Y^\ast,W^\ast)$, there exists a sequence $\ux{v}{s} \in \partial_Y \Lag(\ux{U}{k(s)}, \ux{Y}{k(s)}, \ux{W}{k(s)})$ with $\ux{v}{s} \rightarrow 0$.
\end{lma}

\begin{lma}\label{limitfocgeneric}
Suppose that \Cref{generalasm,weakcoupling} hold. Let $(U^\ast, Y^\ast,W^\ast)$ be a feasible limit point. By passing to a subsequence converging to the limit point, let $\{(\ux{U}{s},\ux{Y}{s},\ux{W}{s})\}$ be a subsequence of the ADMM iterates with $(\ux{U}{s},\ux{Y}{s},\ux{W}{s}) \rightarrow (U^\ast,Y^\ast,W^\ast)$. Suppose that there exists a sequence $\{v^s\}$ such that $\ux{v}{s} \in \partial_Y \Lag (\ux{U}{s},\ux{Y}{s},\ux{W}{s})$ for all $s$ and $\ux{v}{s} \rightarrow 0$. Then $0 \in \partial g_y(Y^\ast) + \nabla_Y F(U^\ast,Y^\ast) + (\nabla_Y C(U^\ast,Y^\ast))^TW^\ast$, so $(U^\ast, Y^\ast, W^\ast)$ is a constrained stationary point.
\end{lma}

\begin{cor}\label{iteratesconvergeimpliesstationary}
If \Cref{generalasm,weakcoupling} hold, and $\|\ux{U}{k+1}_\ell - \ux{U}{k}_\ell\| \rightarrow 0$ for $\ell \geq 1$, and $\|\ux{W}{k+1} - \ux{W}{k}\| \rightarrow 0$, then every limit point is a constrained stationary point.
\end{cor}

\begin{rmk}
The assumption that the successive differences $U_i - U_i^+$ converge to 0 is used in analyses of nonconvex ADMM such as \cite{JMZ2014,YWZY}. \Cref{iteratesconvergeimpliesstationary} shows that this is a very strong assumption: it alone implies that every limit point of ADMM is a constrained stationary point, even when $f$ and $C$ only satisfy \Cref{generalasm,weakcoupling}.
\end{rmk}

\subsection{General Objective and Multiaffine Constraints}
In this section, we assume that $f$ satisfies \Cref{generalasm} and that $C$ is multiaffine. Note that we do \emph{not} use \Cref{weakcoupling} in this section.

As in \Cref{sec:genobjcs}, let $Y$ be a particular variable of focus, and $U$ the remaining variables. We let $f_U(Y) = f(U,Y)$. Since $C(U,Y)$ is multiaffine, the resulting function of $Y$ when $U$ is fixed is an \emph{affine} function of $Y$. Therefore, we have $C(U,Y) = C_U(Y) - b_U$ for a \emph{linear} map $C_U$ and a constant $b_U$. The Jacobian of the constraints is then $\nabla_Y C(U,Y) = C_U$ with adjoint $(\nabla_Y C(U,Y))^T = C_U^T$ such that the relation $\la W, C_U(Y)\ra = \la C_U^TW, Y\ra$ holds.

\begin{cor}\label{subgradgeneric}
Taking $\nabla_Y C(U,Y) = C_U$ in \Cref{subgradgeneric2}, the general subgradient of $Y \mapsto \Lag(U,Y,W)$ is given by $\partial_Y \Lag(U,Y,W) = \partial f_U(Y) + C_U^TW + \rho C_U^T(C_U(Y) - b_U)$. Thus, the first-order condition for $Y \mapsto \Lag(U,Y,W)$ at $Y^+$ is given by $0 \in \partial f_U(Y^+) + C_U^TW + \rho C_U^T(C_U(Y^+) - b_U)$.
\end{cor}

Using this corollary, we can prove the following.

\begin{lma}\label{decgeneric}
The change in the augmented Lagrangian when the primal variable $Y$ is updated to $Y^+$ is given by
$$\Lag(U,Y,W) - \Lag(U,Y^+,W) = f_U(Y) - f_U(Y^+) - \la v, Y - Y^+ \ra + \frac{\rho}{2}\|C_U(Y) - C_U(Y^+)\|^2$$
for some $v \in \partial f_U(Y^+)$.
\end{lma}
\begin{proof}
Expanding $\Lag(U,Y,W)-\Lag(U,Y^+,W)$, the change is equal to
\begin{align}
\nonumber &f_U(Y) - f_U(Y^+) + \la W, C_U(Y) - C_U(Y^+)\ra + \frac{\rho}{2}(\|C_U(Y) - b_U\|^2 - \|C_U(Y^+) - b_U\|^2) \\
\label{eq:linnecessary} &= f_U(Y) - f_U(Y^+) + \la W, C_U(Y) - C_U(Y^+)\ra \\
\nonumber &\bump + \rho \la C_U(Y) - C_U(Y^+), C_U(Y^+) - b_U\ra +  \frac{\rho}{2}\|C_U(Y) - C_U(Y^+)\|^2.
\end{align}
To derive \eqref{eq:linnecessary}, we use the identity $\|Q - P\|^2 - \|R - P\|^2 = \|Q - R\|^2 + 2\la Q - R, R - P\ra$ which holds for any elements $P,Q,R$ of an inner product space. Next, observe that 
\begin{align*}
&\la W, C_U(Y) - C_U(Y^+)\ra +  \rho \la C_U(Y) - C_U(Y^+), C_U(Y^+) - b_U\ra \\
& = \la C_U(Y) - C_U(Y^+), W + \rho(C_U(Y^+) - b_U) \ra \\
&= \la Y - Y^+, C_U^T(W + \rho(C_U(Y^+) - b_U)) \ra
\end{align*}
From \Cref{subgradgeneric}, $v = C_U^TW + \rho C_U^T(C_U(Y^+) - b_U)) \in -\partial f_U(Y^+)$. Hence
$$\Lag(U,Y,W) - \Lag(U,Y^+,W) = f(Y) - f(Y^+) - \la v, Y - Y^+ \ra + \frac{\rho}{2}\|C_U(Y) - C_U(Y^+)\|^2.$$
\end{proof}

\begin{rmk}\label{whymulti}
The proof of \Cref{decgeneric} provides a hint as to why ADMM can be extended naturally to multiaffine constraints, but not to arbitrary nonlinear constraints. When $C(U,Y) = 0$ is a general nonlinear system, we cannot manipulate the difference of squares \eqref{eq:linnecessary} to arrive at the first-order condition for $Y^+$, which uses the crucial fact $\nabla_Y C(U,Y) = C_U$.
\end{rmk}

\begin{rmk}\label{proxsimpler}
If we introduce a proximal term $\|Y - Y^k\|_S^2$, the change in the augmented Lagrangian satisfies $\Lag(U,Y,W) - \Lag(U,Y^+,W) \geq \|Y - Y^+\|_S^2$, regardless of the properties of $f$ and $C$\footnote{To see this, define the prox-Lagrangian $\Lag^P(U,Y,W,O) = \Lag(U,Y,W) + \|Y-O\|_S^2$. By definition, $Y^+$ decreases the prox-Lagrangian, so $\Lag^P(U,Y^+,W,Y^k) \leq \Lag^P(U,Y^k,W,Y^k) = \Lag(U,Y,W)$ and the desired result follows.}. This is usually stronger than \Cref{decgeneric}. Hence, one can generally obtain convergence of proximal ADMM under weaker assumptions than ADMM.
\end{rmk}


Our next lemma shows a useful characterization of $Y^+$.

\begin{lma}\label{ymingeneric}
It holds that $Y^+ = \opn{argmin}_Y \{ f_U(Y) : C_U(Y) = C_U(Y^+)\}$. 
\end{lma}
\begin{proof}
For any two points $Y_1$ and $Y_2$ with $C_U(Y_1) = C_U(Y_2)$, it follows that $\Lag(U,Y_1,W) - \Lag(U,Y_2,W) = f_U(Y_1) - f_U(Y_2)$. Hence $Y^+$, the minimizer of $Y \mapsto \Lag(U,Y,W)$ with $U$ and $W$ fixed, must satisfy $f_U(Y^+) \leq f_U(Y)$ for all $Y$ with $C_U(Y) = C_U(Y^+)$. That is, $Y^+ = \opn{argmin}_Y \{ f_U(Y) : C_U(Y) = C_U(Y^+)\}$.
\end{proof}

We now show conditions under which the sequence of computed augmented Lagrangian values is bounded below.

\begin{lma}\label{lagbdgeneric}
Suppose that $Y$ represents the final block of primal variables updated in an ADMM iteration and that $f$ is bounded below on the feasible region. Consider the following condition:
\begin{cond}\label{ideal}
The following two statements hold true.
\begin{enumerate}
\item $Y$ can be partitioned\footnote{To motivate the sub-blocks $(Y_0, Y_1)$ in \Cref{ideal}, one should look to the decomposition of $\psi(\ZZZ)$ in \Cref{baseasm}, where we can take $Y_0 = \{Z_0\}$ and $Y_1 = \ZZ$. Intuitively, $Y_1$ is a sub-block such that $\psi$ is a smooth function of $Y_1$, and which is `absorbing' in the sense that for any $U^+$ and $Y_0^+$, there exists $Y_1$ making the solution feasible.} into sub-blocks $Y = (Y_0, Y_1)$ such that there exists a constant $M_Y$ such that, for any $U$, $Y_0$, $Y_1$, $Y_1'$, and $v \in \partial f_U(Y_0, Y_1)$,
$$f_U(Y_0, Y_1') - f_U(Y_0, Y_1) - \la v, (Y_0, Y_1') - (Y_0, Y_1)\ra \leq \frac{M_Y}{2} \|Y_1' - Y_1\|^2.$$
			
\item\label{hatbound} There exists a constant $\zeta$ such that for every $U^+$ and $Y^+$ produced by ADMM\footnote{\ref{hatbound} is assumed to hold for the iterates $U^+$ and $Y^+$ generated by ADMM as the minimal required condition, but one should not, in general, think of this property as being specifically related to the iterates of the algorithm. In the cases we consider, it will be a property of the function $f$ and the constraint $C$ that for \emph{any} point $(\wt{U}, \wt{Y})$, there exists $\wh{Y}_1 \in \opn{argmin}_{Y_1} \{ f_{\wt{U}}(\wt{Y}_0,Y_1) : C_{\wt{U}}(\wt{Y}_0, Y_1) = b_{\wt{U}}\}$ such that $\|\wh{Y}_1 - \wt{Y}_1\|^2 \leq \zeta \|C_{\wt{U}}(Y^+) - b_{\wt{U}}\|^2$.}, we can find a solution
$$\wh{Y}_1 \in \opn{argmin}_{Y_1} \{ f_{U^+}(Y_0^+,Y_1) : C_{U^+}(Y_0^+, Y_1) = b_{U^+}\} \footnote{To clarify the definition of $\wh{Y}_1$, the sub-block for $Y_0$ is fixed to the value of $Y_0^+$ on the given iteration, and then $\wh{Y}_1$ is obtained by minimizing $f_{U^+}(Y_0^+,Y_1)$ for the $Y_1$ sub-block over the feasible region $C_{U^+}(Y_0^+, Y_1) = b_{U^+}$.}$$
satisfying $\|\wh{Y}_1 - Y^+_1\|^2 \leq \zeta \|C_{U^+}(Y^+) - b_{U^+}\|^2$.
\end{enumerate}
\end{cond}
If \Cref{ideal} holds, then there exists $\rho$ sufficiently large such that the sequence $\{\ux{\Lag}{k}\}_{k=0}^\infty$ is bounded below.
\end{lma}
\begin{proof}
Suppose that \Cref{ideal} holds. We proceed to bound the value of $\Lag^+$ by relating $Y^+$ to the solution $(Y_0^+, \wh{Y}_1)$. Since $f$ is bounded below on the feasible region and $(U^+, Y_0^+, \wh{Y}_1)$ is feasible by construction, it follows that $f(U^+,Y_0^+,\wh{Y}_1) \geq \nu$ for some $\nu > -\infty$. Subtracting $0 = \la W^+, C_{U^+}(Y_0^+, \wh{Y}_1) - b_{U^+} \ra$ from $\Lag^+$ yields
\begin{align}\label{eq:lagpluszero}
\Lag^+ &= f_{U^+}(Y^+) + \la W^+, C_{U^+}(Y^+ - (Y_0^+,\wh{Y}_1))\ra + \frac{\rho}{2}\|C_{U^+}(Y^+) - b_{U^+}\|^2.
\end{align}
Since $Y$ is the \emph{final} block before updating $W$, all other variables have been updated to $U^+$, and \Cref{subgradgeneric} implies that the first-order condition satisfied by $Y^+$ is
$$0 \in \partial f_{U^+}(Y^+) + C_{U^+}^TW + \rho C_{U^+}^T(C_{U^+}(Y^+) - b_{U^+}) = \partial f_{U^+}(Y^+) + C_{U^+}^TW^+.$$
Hence $v = C_{U^+}^TW^+ \in -\partial f_{U^+}(Y^+)$. Substituting this into \eqref{eq:lagpluszero}, we have
$$\Lag^+ = f_{U^+}(Y^+) + \la v, Y^+ - (Y_0^+, \wh{Y}_1) \ra + \frac{\rho}{2}\|C_{U^+}(Y^+) - b_{U^+}\|^2.$$
Adding and subtracting $f_{U^+}(Y_0^+,\wh{Y}_1)$ yields
\begin{align*}
\Lag^+ &= f_{U^+}(Y_0^+,\wh{Y}_1) + \frac{\rho}{2}\|C_{U^+}(Y^+) - b_{U^+}\|^2 \\
&\bump - (f_{U^+}(Y_0^+,\wh{Y}_1) - f_{U^+}(Y^+) - \la -v, (Y_0^+, \wh{Y}_1) - Y^+\ra).
\end{align*}
Since $Y^+ = (Y_0^+, Y_1^+)$ and $-v \in \partial f_{U^+}(Y^+)$, \Cref{ideal} implies that
$$f_{U^+}(Y_0^+, \wh{Y}_1) - f_{U^+}(Y^+) - \la -v, (Y_0^+, \wh{Y}_1) - Y^+\ra \leq \frac{M_Y}{2} \|\wh{Y}_1\ - Y_1^+\|^2.$$
Hence, we have
\begin{align}
\nonumber \ux{\Lag}{+} &\geq f_{U^+}(Y_0^+,\wh{Y}_1)  + \frac{\rho}{2}\|C_{U^+}(Y^+) - b_{U^+}\|^2 - \frac{M_Y}{2} \|\wh{Y}_1 - Y_1^+\|^2 \\
\label{eq:lag_above} &\geq f_{U^+}(Y_0^+,\wh{Y}_1) + \left( \frac{\rho - M_Y\zeta}{2}\right) \|C_{U^+}(Y^+) - b_{U^+}\|^2.\
\end{align}
It follows that if $\rho \geq M_Y\zeta$, then $\ux{\Lag}{k} \geq \nu$ for all $k \geq 1$.
\end{proof}

The following useful corollary is an immediate consequence of the final inequalities in the proof of the previous lemma.


\begin{cor}\label{varbdgeneric}
Recall the notation from \Cref{lagbdgeneric}. Suppose that $f(U,Y)$ is coercive on the feasible region, \Cref{ideal} holds, and $\rho$ is chosen sufficiently large so that $\{\ux{\Lag}{k}\}$ is bounded above and below. Then $\{\ux{U}{k}\}$ and $\{\ux{Y}{k}\}$ are bounded. 
\end{cor}
\begin{proof}
Under the given conditions, $\{\ux{\Lag}{k}\}$ is monotonically decreasing and it can be seen from \eqref{eq:lag_above} that $\{f(\ux{U}{k}, \ux{Y}{k}_0, \ux{\wh{Y}}{k}_1)\}$ and $\{\|C_{\ux{U}{k}}(\ux{Y}{k}) - b_{\ux{U}{k}}\|^2\}$ are bounded above. Since $f$ is coercive on the feasible region, and $(\ux{U}{k}, \ux{Y}{k}_0, \ux{\wh{Y}}{k}_1)$ is feasible by construction, this implies that $\{\ux{U}{k}\}, \{\ux{Y}{k}_0\}$, and $\{\ux{\wh{Y}}{k}_1\}$ are bounded. It only remains to show that the `true' sub-block $\{ \ux{Y}{k}_1\}$ is bounded. From \Cref{ideal}, there exists $\zeta$ with $\|\ux{\wh{Y}}{k}_1 - \ux{Y}{k}_1\|^2 \leq \zeta \|C_{\ux{U}{k}}(\ux{Y}{k}) - b_{\ux{U}{k}}\|^2$. \eqref{eq:lag_above} also implies that $\{ \|C_{\ux{U}{k}}(\ux{Y}{k}) - b_{\ux{U}{k}}\|^2\}$ is bounded. Hence $\{\ux{Y}{k}_1\}$ is also bounded.
\end{proof}

\subsection{Separable Objective and Multiaffine Constraints}
Now, in addition to \Cref{generalasm}, we require that $C(U_0,\ldots,U_n)$ is multiaffine, and that \Cref{weakcoupling} holds. Most of the results in this section can be obtained from the corresponding results in \Cref{sec:genobjcs}; however, since we will extensively use these results in \Cref{sec:convergence}, it is useful to see their specific form when $C$ is multiaffine.

Again, let $Y = U_j$ be a particular variable of focus, and $U$ the remaining variables. Since $f$ is separable, minimizing $f_U(Y)$ is equivalent to minimizing $f_j(Y)$. Hence, writing $f_y$ for $f_j$, we have
$$\partial_Y \Lag(U,Y,W) = \partial f_y(Y) + \nabla_Y F(U,Y) + C_U^TW + \rho C_U^T(C_U(Y) - b_U)$$
and $Y^+$ satisfies the first-order condition $0 \in \partial f_y(Y^+) + \nabla_Y F(U,Y^+)+ C_U^TW + \rho C_U^T(C_U(Y^+) - b_U)$. The crucial property is that $\partial f_y(Y)$ depends only on $Y$.

\begin{cor}\label{innerblockgeneric}
Suppose that $Y$ is a block of variables in ADMM, and let $U_<, U_>$ be the variables that are updated before and after $Y$, respectively. During an iteration of ADMM, let $C_<(Y) = b_<$ denote the constraint $C(U_<^+,Y,U_>) = b_<$ as a linear function of $Y$, after updating the variables $U_<$, and let $C_>(Y) = b_>$ denote the constraint $C(U_<^+,Y,U_>^+) = b_>$. Then the general subgradient $\partial_Y \Lag(U_<^+,Y^+,U_>^+,W^+)$ at the final point contains
\begin{align*}
&(C_>^T - C_<^T)W^+ + C_<^T(W^+ - W) + \rho (C_>^T - C_<^T)(C_>(Y^+) - b_>) \\
&\bump + \rho C_<^T(C_>(Y^+) - b_> - (C_<(Y^+) - b_<)) \\
&\bump + \nabla_Y F(U_<^+,Y^+,U_>^+) - \nabla_Y F(U_<^+,Y^+,U_>)
\end{align*}
In particular, if $Y$ is the final block, then $C_<^T(W^+ - W) \in \partial_Y \Lag(U_<^+,Y^+,W^+)$.
\end{cor}
\begin{proof}
This is an application of \Cref{generalinnerblock}. Since we will use this special case extensively in \Cref{sec:convergence}, we also show the calculation. By \Cref{subgradgeneric}
$$\partial_Y \Lag(U_<^+,Y^+,U_>^+,W^+) = \partial f_y(Y^+) + \nabla_Y F(U_<^+,Y^+,U_>^+) + C_>^TW^+ + \rho C_>^T(C_>(Y^+) - b_>)$$
By \Cref{subgradgeneric}, $-(\nabla_Y F(U_<^+,Y^+,U_>) + C_<^TW + \rho C_<^T(C_<(Y^+) - b_<)) \in \partial f_y(Y^+)$. 
To obtain the result, write $C_>^TW^+ - C_<^TW = (C_>^T - C_<^T)W^+ + C_<^T(W^+ - W)$ and
\begin{align*}
C_>^T(C_>(Y^+) - b_>) - C_<^T(C_<(Y^+) - b_<) &= (C_>^T - C_<^T)(C_>(Y^+) - b_>) \\
&\bump + C_<^T(C_>(Y^+) - b_> - (C_<(Y^+) - b_<)).
\end{align*}
\end{proof}

\begin{lma}\label{limitfocconditionsgeneric}
Recall the notation from \Cref{innerblockgeneric}. Suppose that
\begin{enumerate}
\item\label{item:lastvar} $\|W - W^+\| \rightarrow 0$,
\item $\|C_> - C_<\| \rightarrow 0$,
\item $\|b_> - b_<\| \rightarrow 0$, and
\item $\{\ux{W}{k}\}$, $\{ \ux{Y}{k}\}$, $\{\ux{C}{k}_<\}$, $\{\ux{C}{k}_>(Y^+) - b_>\}$ are bounded, and
\item $\|U_>^+ - U_>\| \rightarrow 0$.
\end{enumerate}
Then there exists a sequence $\ux{v}{k} \in \partial_Y \ux{\Lag}{k}$ with $\ux{v}{k} \rightarrow 0$. In particular, if $Y$ is the \emph{final} block, then only condition~\ref{item:lastvar} and the boundedness of $\{\ux{C}{k}_<\}$ are needed.
\end{lma}
\begin{proof}
If the given conditions hold, then the triangle inequality and the continuity of $\nabla_YF$ show that the subgradients identified in \Cref{innerblockgeneric} converge to 0.
\end{proof}

The previous results have focused on a single block $Y$, and the resulting equations $C_U(Y) = b_U$. Let us now relate $C_U, b_U$ to the full constraints. Suppose that we have variables $U_0,\ldots,U_n, Y$ (not necessarily listed in update order), and the constraint $C(U_0,\ldots,U_n,Y) = 0$ is multiaffine. Using the decomposition \eqref{multiaffinestruc} and the notation $\theta_j(U)$ from \Cref{linearterm}, we express $C_U$ and $b_U$ as
\begin{equation}\label{decomp3}
C_U = \sum_{j=1}^{m_1} \theta_j(U), \qquad b_U = -( B + \sum_{j=m_1+1}^m \theta_j(U)).
\end{equation}

This allows us to verify the conditions of \Cref{limitfocconditionsgeneric} when certain variables are known to converge.

\begin{lma}\label{constraintconvergegeneric}
Adopting the notation from \Cref{innerblockgeneric}, assume that $\{ \ux{U}{k}_<\}$, $\{\ux{Y}{k}\}$, $\{\ux{U}{k}_>\}$ are bounded, and that $\|U_>^+ - U_>\| \rightarrow 0$. Then $\|C_> - C_<\| \rightarrow 0$ and $\|b_> - b_<\| \rightarrow 0$.
\end{lma}
\begin{proof}
Unpacking our definitions, $C_<$ corresponds to the system of constraints $C(U_<^+,Y,U_>) = b$, and $C_>$ corresponds to $C(U_<^+,Y,U_>^+) = b$. Let $U = (U_<^+, U_>)$ and $U' = (U_<^+,U_>^+)$. By \eqref{decomp3}, we have $C_> - C_< = \sum_{j=1}^{m_1} \theta_j(U') - \theta_j(U)$. From \Cref{multiaffsmooth}, each $\theta_j$ is smooth, and therefore uniformly continuous over a compact set containing $\{\ux{U}{k}_<, \ux{Y}{k}, \ux{U}{k}_>\}_{k=0}^\infty$. Thus, $\|U_>^+ - U_>\| \rightarrow 0$ implies that $\|C_> - C_<\| \rightarrow 0$. The same applies to $b_> - b_<$.
\end{proof}

\section{Convergence Analysis of Multiaffine ADMM}\label{sec:convergence}
We now apply the results from \Cref{sec:general} to multiaffine problems of the form $(P)$ that satisfy \Cref{baseasm,asm2}.

\subsection{Proof of \Cref{main}}\label{sec:mainproof}
Under \Cref{baseasm}, we prove \Cref{main}. The proof appears at the end of this subsection after we prove a few intermediate results.

\begin{cor}\label{subgradients}
The general subgradients $\partial_\ZZZ \Lag(\XX,\ZZZ,\WW)$ are given by
\begin{align*}
\partial_{Z_0} \Lag(\XX,Z_0,Z_1,Z_2,\WW) &= \partial_{Z_0}\psi(\ZZZ) + A_{Z_0,\XX}^T\WW + \rho A_{Z_0,\XX}^T(A(\XX,Z_0) + Q(\ZZ)) \ {and} \\
\nabla_{Z_i} \Lag(\XX,Z_0,Z_1,Z_2,\WW) &= \nabla_{Z_i}\psi(\ZZZ) + Q_i^TW_i + \rho Q_i^T(A_i(\XX,Z_0) + Q_i(Z_i)) \text{ for } i \in \{1,2\}.
\end{align*}
\end{cor}
\begin{proof}
This follows from \Cref{subgradgeneric}. Recall that $A_{Z_0,\XX}$ is the $Z_0$-linear term of $Z_0 \mapsto A(\XX,Z_0)$ (see \Cref{linearterm}).
\end{proof}

\begin{cor}\label{foc}For all $k \geq 1$, 
\begin{equation}\label{eq:qtw}
-\nabla_{Z_i} \psi(\ux{\ZZZ}{k}) = Q_i^T\ux{W}{k}_i \bump \text{ for } i \in \{1,2\}.
\end{equation}
\end{cor}
\begin{proof}
This follows from \Cref{subgradients,subgradgeneric}, and the updating formula for $\WW^+$ in \Cref{alg:admm_mult}. Note that $g_1$ and $g_2$ are smooth, so the first-order conditions for each variable simplifies to $$-\nabla_{Z_i} \psi(\ZZZ^+) = Q_i^T(W_i + \rho (A_i(\XX^+, Z_0^+) + Q_i(Z_i^+))) = Q_i^TW_i^+.$$
Hence, \eqref{eq:qtw} immediately follows.
\end{proof} 



Next, we quantify the decrease in the augmented Lagrangian using properties of $h$, $g_1$, $g_2$, and $Q_2$.

\begin{cor}\label{yzdecr}
The change in the augmented Lagrangian after updating the final block $\ZZZ$ is bounded below by
\begin{equation}\label{postdecbd}\frac{m_1}{2}\|Z_S - Z_S^+\|^2 + \left( \frac{\rho \sigma - M_2}{2}\right) \|Z_2 - Z_2^+\|^2,\end{equation}
where $\sigma = \lambda_{min}(Q_2^TQ_2) > 0$.
\end{cor}
\begin{proof}
We apply \Cref{decgeneric} to $\ZZZ$. Recall that $\psi = h(Z_0) + g_1(Z_1) + g_2(Z_2)$. The decrease in the augmented Lagrangian is given, for some $v \in \partial h(Z_0^+)$, by
\begin{align}
&\label{eq:dec} h(Z_0) - h(Z_0^+) - \la v, Z_0 - Z_0^+\ra + g_1(Z_S) - g_1(Z_S^+) - \la \nabla g_1(Z_S^+), Z_S - Z_S^+\ra\\
\nonumber &+ g_2(Z_2) - g_2(Z_2^+) - \la \nabla g_2(Z_2^+), Z_2 - Z_2^+\ra \\
\nonumber &+ \frac{\rho}{2}\|A_1(\XX^+,Z_0 - Z_0^+) + Q_1(Z_1 - Z_1^+)\|^2 + \frac{\rho}{2}\|Q_2(Z_2 - Z_2^+)\|^2.
\end{align}
By \Cref{aobj}, we can show the following bounds for the components of \eqref{eq:dec}:
\begin{enumerate}
\item $h$ is convex, so $h(Z_0) - h(Z_0^+) - \la v, Z_0 - Z_0^+\ra \geq 0$.
\item $g_1$ is $(m_1,M_1)$-strongly convex, so $g_1(Z_S) - g_1(Z_S^+) - \la \nabla g_1(Z_S^+), Z_S - Z_S^+\ra \geq \frac{m_1}{2}\|Z_S - Z_S^+\|^2$.
\item $g_2$ is $M_2$-Lipschitz differentiable, so $g_2(Z_2) - g_2(Z_2^+) - \la \nabla g_2(Z_2^+), Z_2 - Z_2^+\ra \geq -\frac{M_2}{2}\|Z_2 - Z_2^+\|^2$.
\end{enumerate}
Since $Q_2$ is injective, $Q_2^TQ_2$ is positive definite. It follows that with $\sigma = \lambda_{min}(Q_2^TQ_2) > 0$, $\frac{\rho}{2}\|Q_2(Z_2 - Z_2^+)\|^2 \geq \frac{\rho}{2}\sigma\|Z_2 - Z_2^+\|^2$. Since $\frac{\rho}{2}\|A_1(\XX^+,Z_0 - Z_0^+) + Q_1(Z_1 - Z_1^+)\|^2 \geq 0$, summing the inequalities establishes the lower bound \eqref{postdecbd} on the decrease in $\Lag$.
\end{proof}

We now bound the change in the Lagrange multipliers by the changes in the variables in $\ZZZ$.

\begin{lma}\label{winc}
We have $\|\WW - \WW^+\|^2 \leq \beta_1 \|Z_S - Z_S^+\|^2 + \beta_2 \|Z_2 - Z_2^+\|^2$, where $\beta_1 = M_1^2 \lambda^{-1}_{++}(Q_1^TQ_1)$ and $\beta_2 = M_2^2 \lambda^{-1}_{++}(Q_2^TQ_2) = M_2^2\sigma^{-1}$.
\end{lma}
\begin{proof}
From \Cref{foc}, we have $Q_i^TW_i = -\nabla_{Z_i} \psi(\ZZZ)$ and $Q_i^TW_i^+ = -\nabla_{Z_i} \psi(\ZZZ^+)$ for $i \in \{1,2\}$. By definition of the dual update, $W_i^+ - W_i = \rho (A_i(\XX^+,Z_0^+) + Q_i(Z_i^+))$. Since $\opn{Im}(Q_i)$ contains the image of $A_i$, we have $W_i^+ - W_i \in \opn{Im}(Q_i)$. \Cref{leastim} applied to $R = Q_i$ and $y = W_i^+ - W_i$ then implies that
\begin{align*}
\|W_i - W_i^+\|^2 &\leq \lambda_{++}^{-1}(Q_i^TQ_i)\|Q_i^TW_i - Q_i^TW_i^+\|^2 = \lambda_{++}^{-1}(Q_i^TQ_i) \|\nabla_{Z_i} \psi(\ZZZ) - \nabla_{Z_i} \psi(\ZZZ^+)\|^2.
\end{align*}
Since $\psi(\ZZZ) = h(Z_0) + g_1(Z_S) + g_2(Z_2)$, we have, for $Z_1$, the bound
\begin{align*}
\|\nabla_{Z_1} \psi(\ZZZ) - \nabla_{Z_1} \psi(\ZZZ^+)\|^2 &= \| \nabla_{Z_1} g_1(Z_S) - \nabla_{Z_i} g_1(Z_S^+)\|^2 \\
&\leq \|\nabla g_1(Z_S) - \nabla g_1(Z_S^+)\|^2 \leq M_1^2\|Z_S - Z_S^+\|^2
\end{align*}
and thus $\|W_1 - W_1^+\|^2 \leq M_1^2\lambda_{++}^{-1}(Q_1^TQ_1)\|Z_S - Z_S^+\|^2 = \beta_1\|Z_S - Z_S^+\|^2$. A similar calculation applies to $W_2$. Summing over $i \in \{1,2\}$, we have the desired result.
\end{proof}

\begin{lma}\label{lagdec}
For sufficiently large $\rho$, $\Lag(\XX^+,\ZZZ,\WW) - \Lag(\XX^+,\ZZZ^+,\WW^+) \geq 0$, and therefore $\{\ux{\Lag}{k}\}_{k=1}^\infty$ is monotonically decreasing. Moreover, for sufficiently small $\epsilon > 0$, we may choose $\rho$ so that $\Lag - \Lag^+ \geq \epsilon( \|Z_S - Z_S^+\|^2 + \|Z_2 - Z_2^+\|^2)$.
\end{lma}
\begin{proof}
Since the ADMM algorithm involves successively minimizing the augmented Lagrangian over sets of primal variables, it follows that the augmented Lagrangian does not increase after each block of primal variables is updated.  In particular, since it does not increase after the update from $\XX$ to $\XX^+$, one finds
\begin{align*}
\Lag - \Lag^+ &= \Lag(\XX,\ZZZ,\WW) - \Lag(\XX^+, \ZZZ,\WW) + \Lag(\XX^+,\ZZZ,\WW) - \Lag(\XX^+,\ZZZ^+,\WW) \\
&\bump + \Lag(\XX^+,\ZZZ^+,\WW) - \Lag(\XX^+,\ZZZ^+,\WW^+) \\
&\geq  \Lag(\XX^+,\ZZZ,\WW) - \Lag(\XX^+,\ZZZ^+,\WW) + \Lag(\XX^+,\ZZZ^+,\WW) - \Lag(\XX^+,\ZZZ^+,\WW^+).
\end{align*}
The only step which increases the augmented Lagrangian is updating $\WW$. It suffices to show that the size of $\Lag(\XX^+,\ZZZ,\WW) - \Lag(\XX^+,\ZZZ^+,\WW)$ exceeds the size of $\Lag(\XX^+,\ZZZ^+,\WW) - \Lag(\XX^+,\ZZZ^+,\WW^+)$ by at least $\epsilon( \|Z_S - Z_S^+\|^2 + \|Z_2 - Z_2^+\|^2)$.

By \Cref{dualgeneric}, $\Lag(\XX^+,\ZZZ^+,\WW) - \Lag(\XX^+,\ZZZ^+,\WW^+) = -\frac{1}{\rho}\|\WW - \WW^+\|^2$. Using \Cref{winc}, this is bounded by $-\frac{1}{\rho}(\beta_1 \|Z_S - Z_S^+\|^2 + \beta_2 \|Z_2 - Z_2^+\|^2)$. On the other hand, eq. \cref{postdecbd} of \Cref{yzdecr} implies that $\Lag(\XX^+,\ZZZ,\WW) - \Lag(\XX^+,\ZZZ^+,\WW) \geq \frac{m_1}{2}\|Z_S - Z_S^+\|^2 + \left( \frac{\rho \sigma - M_2}{2}\right) \|Z_2 - Z_2^+\|^2$. Hence, for any $0 < \epsilon < \frac{m_1}{2}$, we may choose $\rho$ sufficiently large so that $\frac{m_1}{2} \geq \frac{\beta_1}{\rho} + \epsilon$ and $\frac{\rho\sigma - M_2}{2} \geq \frac{\beta_2}{\rho} + \epsilon$.
\end{proof}

We next show that $\ux{\Lag}{k}$ is bounded below.

\begin{lma}\label{lagbd}
For sufficiently large $\rho$, the sequence $\{\ux{\Lag}{k}\}$ is bounded below, and thus with \Cref{lagdec}, the sequence $\{\ux{\Lag}{k}\}$ is convergent.
\end{lma}
\begin{proof}
We will apply \Cref{lagbdgeneric}. By \Cref{aobj}, $\phi$ is coercive on the feasible region. Thus, it suffices to show that \Cref{ideal} holds for the objective function $\phi$ and constraint $A(\XX,Z_0) + Q(\ZZ) = 0$, with final block $\ZZZ$.

In the notation of \Cref{lagbdgeneric}, we take $Y_0 = \{Z_0\}$, $Y_1 = \ZZ = (Z_1,Z_2)$. We first verify that \Cref{ideal}(1) holds. Recall that $\psi = h(Z_0) + g_1(Z_S) + g_2(Z_2)$ with $g_1$ and $g_2$ Lipschitz differentiable.  Fix any $Z_0$. For any $v \in \partial \psi(Z_0, \ZZ)$, we have
\begin{align*}
&\psi(Z_0,\ZZ') - \psi(Z_0,\ZZ) - \la v, (Z_0,\ZZ') - (Z_0,\ZZ)\ra \\
&\bump = g_1(Z_0,Z_1') - g_1(Z_0,Z_1) - \la \nabla g_1(Z_0,Z_1), (Z_0,Z'_1) - (Z_0,Z_1)\ra \\
&\bump\bump + g_2(Z'_2) - g_2(Z_2) - \la \nabla g_2(Z_2), Z'_2 - Z_2\ra \\
&\bump \leq \frac{M_1}{2}\|Z'_1 - Z_1\|^2 + \frac{M_2}{2}\|Z'_2 - Z_2\|^2
\end{align*}
Thus, \Cref{ideal}(1) is satisfied with $M_\psi = \frac{1}{2}(M_1 + M_2)$.

Next, we construct $\wh{\ZZ}$, a minimizer of $\psi(Z_0^+, \ZZ)$ over the feasible region with $\XX^+$ and $Z_0^+$ fixed, and find a value of $\zeta$ satisfying \Cref{ideal}\eqref{hatbound}. There is a unique solution $\wh{Z}_2$ which is feasible for $A_2(\XX^+) + Q_2(Z_2) = 0$, so we take $\wh{Z}_2 = -Q_2^{-1}A_2(\XX^+)$. We find that $\|\wh{Z}_2 - Z_2^+\|^2 \leq \lambda^{-1}_{min}(Q_2^TQ_2)\|Q_2(Z_2^+ - \wh{Z}_2)\|^2 = \lambda^{-1}_{min}(Q_2^TQ_2)\|A_2(\XX^+) + Q_2(Z_2^+)\|^2$. Thus, if $\zeta \geq \lambda^{-1}_{min}(Q_2^TQ_2)$, then $\|\wh{Z}_2 - Z_2^+\|^2 \leq \zeta \|A_2(\XX^+) + Q_2(Z_2^+)\|^2$.

To construct $\wh{Z}_1$, consider the spaces $
\UU_1 = \{ Z_1 : Q_1(Z_1) = - A_1(\XX^+,Z_0^+)\}$ and  $\UU_2 = \{ Z_1 : Q_1(Z_1) = Q_1(Z_1^+)\}$. From \Cref{ymingeneric}, $(Z_0^+,Z_1^+)$ is the minimizer of $h(Z_0) + g_1(Z_0,Z_1)$ over the subspace $$\UU_3 = \{ (Z_0,Z_1) : A_1(\XX^+,Z_0) + Q_1(Z_1) = A_1(\XX^+,Z_0^+) + Q_1(Z_1^+)\}.$$ Consider the function $g_0$ given by $g_0(Z_1) = g_1(Z_0^+,Z_1)$. It must be the case that $Z_1^+$ is the minimizer of $g_0$ over $\UU_2$, as any other $Z_1'$ with $Q_1(Z_1') = Q_1(Z_1^+)$ also satisfies $(Z_0^+, Z_1') \in \UU_3$.  By \Cref{convexspec}, $g_0$ inherits the $(m_1,M_1)$-strong convexity of $g_1$. Let
$$\wh{Z}_1 = \argmin_{Z_1} \{ g_0(Z_1): Z_1 \in \UU_1\}.$$
Notice that we can express the subspaces $\UU_1, \UU_2$ as $\UU_1 = \{ Z_1| Q_1(Z_1) + A_1(\XX^+,Z_0^+) \in \CC\}$ and $\UU_2 = \{ Z_1| Q_1(Z_1) - Q_1(Z_1^+) \in \CC\}$ for the closed convex set $\CC = \{0\}$. Since $Z_1^+$ is the minimizer of $g_0$ over $\UU_2$, \Cref{bicvxargmin} with $h = g_0$, and the subspaces $\UU_1$ and $\UU_2$, implies that
$$\|\wh{Z}_1 - Z_1^+\| \leq \gamma \|A_1(\XX^+,Z_0^+) + Q_1(Z_1^+)\|$$
where $\gamma$ is dependent only on $\kappa = \frac{M_1}{m_1}$ and $Q_1$.  Hence, taking $\zeta = \max\{\gamma^2, \lambda_{min}^{-1}(Q_2^TQ_2)\}$,
\begin{align*}
&\|\wh{\ZZ} - \ZZ^+\|^2 = \|\wh{Z}_2 - Z_2^+\|^2 + \|\wh{Z}_1 - Z_1^+\|^2 \\
&\bump \leq \zeta(\|A_2(\XX^+) + Q_2(Z_2^+)\|^2 + \|A_1(\XX^+,Z_0^+) + Q_1(Z_1^+)\|^2).
\end{align*}
Overall, we have shown that \Cref{ideal} is satisfied. Having verified the conditions of \Cref{lagbdgeneric}, we conclude that for sufficiently large $\rho$, $\{\ux{\Lag}{k}\}$ is bounded below.
\end{proof}

\begin{cor}\label{allbded}
For sufficiently large $\rho$, the sequence $\{(\ux{\XX}{k}, \ux{\ZZZ}{k}, \ux{\WW}{k})\}_{k=0}^\infty$ is bounded.
\end{cor}
\begin{proof}
In \Cref{lagbd}, we showed that \Cref{ideal} holds. By assumption, $\phi$ is coercive on the feasible region. Thus, the conditions for \Cref{varbdgeneric} are satisfied, so $\{\ux{\XX}{k}\}$ and $\{\ux{\ZZZ}{k}\}$ are bounded.

To show that $\{ \ux{\WW}{k}\}$ is bounded, recall that $\WW^+ - \WW \in \opn{Im}(Q)$ by \Cref{afinalblockim}, and that $Q^T\WW^+= -\nabla_{(Z_1,Z_2)} \psi(\ZZZ^+)$ by \Cref{foc}. Taking an orthogonal decomposition of $\ux{\WW}{0}$ for the subspaces $\opn{Im}(Q)$ and $\opn{Im}(Q)^\perp$, we express $\ux{\WW}{0} = \ux{\WW}{0}_Q + \ux{\WW}{0}_P$, where $\ux{\WW}{0}_Q \in \opn{Im}(Q)$ and $\ux{\WW}{0}_P \in \opn{Im}(Q)^\perp$. Since $\WW^+ - \WW\in \opn{Im}(Q)$, it follows that if we decompose $\ux{\WW}{k} = \ux{\WW}{k}_Q + \ux{\WW}{k}_P$ with $\ux{\WW}{k}_P \in \opn{Im}(Q)^\perp$, then we have $\ux{\WW}{k}_P = \ux{\WW}{0}_P$ for every $k$. Thus, $\|\ux{\WW}{k}\|^2 = \|\ux{\WW}{k}_Q\|^2 + \|\ux{\WW}{0}_P\|^2$ for every $k$. Hence, it suffices to bound $\|\ux{\WW}{k}_Q\|$. Observe that $Q^T\ux{\WW}{k} = Q^T\ux{\WW}{0}_P + Q^T\ux{\WW}{k}_Q = Q^T\ux{\WW}{k}_Q$, because $\ux{\WW}{0}_P \in \opn{Im}(Q)^\perp = \opn{Null}(Q^T)$. Thus, by \Cref{foc}, $Q^T\ux{\WW}{k}_Q = -\nabla_{(Z_1,Z_2)} \psi(\ux{\ZZZ}{k})$. Since $\{\ux{\ZZZ}{k}\}$ is bounded and $g_1$ and $g_2$ are Lipschitz differentiable, we deduce that $\{ \|Q^T\ux{\WW}{k}_Q\|\}$ is bounded. By \Cref{leastim}, $\|\ux{\WW}{k}_Q\|^2 \leq \lambda^{-1}_{++}(Q^TQ) \|Q^T\ux{\WW}{k}_Q\|^2$, and so $\{ \|\ux{\WW}{k}_Q\|\}$ is bounded. Hence $\{\ux{\WW}{k}\}$ is bounded, completing the proof.
\end{proof}

\begin{cor}\label{lastconverges}
For sufficiently large $\rho$, we have $\|Z_S - Z_S^+\| \to 0$ and $\|Z_2 - Z_2^+\| \rightarrow 0$. Consequently, $\|\WW - \WW^+\| \rightarrow 0$ and every limit point is feasible.
\end{cor}
\begin{proof}
From \Cref{lagdec}, we may choose $\rho$ so that the augmented Lagrangian decreases by at least $\epsilon(\|Z_S - Z_S^+\|^2 + \|Z_2 - Z_2^+\|^2)$ for some $\epsilon > 0$ in each iteration. Summing over $k$, $\epsilon \sum_{k=0}^\infty \|\ux{Z}{k}_S - \ux{Z}{k+1}_S\|^2 + \|\ux{Z}{k}_2 - \ux{Z}{k+1}_2\|^2 \leq \ux{\Lag}{0} - \lim_{k} \ux{\Lag}{k}$, which is finite by \Cref{lagbd}; hence, $\|Z_S - Z_S^+\| \rightarrow 0$ and $\|Z_2 - Z_2^+\| \rightarrow 0$.

Using \Cref{winc}, $\|\WW - \WW^+\|^2 \leq \beta_1 \|Z_S - Z_S^+\|^2 + \beta_2\|Z_2 - Z_2^+\|^2$, so $\|\WW - \WW^+\| \rightarrow 0$ as well. \Cref{dualgeneric} then implies that every limit point is feasible.
\end{proof}

Finally, we are prepared to prove the main theorems.

\begin{proof}[Proof (of \Cref{main})] \Cref{allbded} implies that limit points of $\{(\ux{\XX}{k}, \ux{\ZZZ}{k}, \ux{\WW}{k})\}$ exist. From \Cref{lastconverges}, every limit point is feasible.

We check the conditions of \Cref{limitfocconditionsgeneric}. Since $\ZZZ$ is the final block, it suffices to verify that $\|\WW - \WW^+\| \rightarrow 0$, and that the maps $\{\ux{C}{k}_<\}$ are uniformly bounded. That $\|\WW - \WW^+\| \rightarrow 0$ follows from \Cref{lastconverges}. Recall from \Cref{innerblockgeneric} that $\ux{C}{k}_<$ is the $\ZZZ$-linear term of $\ZZZ \mapsto A(\ux{\XX}{k}, Z_0) + Q(\ZZ)$; since $A$ is multiaffine, \Cref{multiaffsmooth} and the boundedness of $\{\ux{\XX}{k}\}_{k=0}^\infty$ (\Cref{allbded}) imply that indeed, $\{\ux{C}{k}_<\}$ is uniformly bounded in operator norm. Thus, the conditions of \Cref{limitfocconditionsgeneric} are satisfied. This exhibits the desired sequence $\ux{v}{k} \in \partial_\ZZZ \Lag (\ux{\XX}{k}, \ux{\ZZZ}{k}, \ux{\WW}{k})$ with $\ux{v}{k} \rightarrow 0$ of \Cref{main}. \Cref{limitfocgeneric} then completes the proof.
\end{proof}

\subsection{Proof of \Cref{main2}}\label{sec:main2proof}

Under \Cref{asm2}, we proceed to prove \Cref{main2}. For brevity, we introduce the notation $\XX_{<i}$ for the variables $(X_0,\ldots,X_{i-1})$ and $\XX_{>i}$ for $(X_{i+1},\ldots,X_n)$.

\begin{lma}\label{xconverges}
For sufficiently large $\rho$, we have $\|X_\ell - X_\ell^+\| \rightarrow 0$ for each $1 \leq \ell \leq n$, and $\|Z_0 - Z_0^+\| \rightarrow 0$.
\end{lma}
\begin{proof}
First, we consider $X_\ell$ for $1 \leq \ell \leq n$. Let $A_X(X_\ell) = b_X$ denote the linear system of constraints when updating $X_\ell$. Recall that under \Cref{asm2}, $f(\XX) = F(\XX) + \sum_{i=0}^n f_i(X_i)$, where $F$ is a smooth function. By \Cref{decgeneric}, the change in the augmented Lagrangian after updating $X_\ell$ is given (for some $v \in \partial f_\ell(X_\ell)$) by
\begin{align}\label{eq:xelldec} 
&f_\ell(X_\ell) - f_\ell(X_\ell^+) - \la v, X_\ell - X_\ell^+  \ra + \frac{\rho}{2}\|A_X(X_\ell) - A_X(X_\ell^+)\|^2 \\
\nonumber &+ F(\XX_{<\ell}^+,X_\ell,\XX_{>\ell}) - F(\XX_{<\ell}^+,X_\ell^+,\XX_{>\ell})  - \la \nabla_{X_\ell}F(\XX_{<\ell}^+,X_\ell^+,\XX_{>\ell}) , X_\ell - X_\ell^+\ra.
\end{align}

By \Cref{lagdec}, the change in the augmented Lagrangian from updating $\WW$ is less than the change from updating $\ZZZ$. Since \eqref{eq:xelldec} is nonnegative for every $\ell$, it follows that the change in the augmented Lagrangian in each iteration is greater than the sum of the change from updating each $X_\ell$, and therefore greater than \eqref{eq:xelldec} for each $\ell$. By \Cref{lagbd}, the augmented Lagrangian converges, so the expression \eqref{eq:xelldec} must converge to 0. We will show that this implies the desired result for both cases of \Cref{axainj}.

\begin{description}
\item[1] $F(X_0,\ldots,X_n)$ is independent of $X_\ell$ and there exists a 0-forcing function $\Delta_\ell$ such that for any $v \in \partial f_\ell(X_\ell^+)$, $f_\ell(X_\ell) - f_\ell(X_\ell^+) - \la v, X_\ell - X_\ell^+ \ra \geq \Delta_\ell(\|X_\ell^+ - X_\ell\|)$. In this case, \eqref{eq:xelldec} is bounded below by $\Delta_\ell(\|X_\ell^+ - X_\ell\|)$. Since \eqref{eq:xelldec} converges to 0, $\Delta_\ell(\|X_\ell^+ - X_\ell\|) \rightarrow 0$, which implies that $\|X_\ell - X_\ell^+\| \rightarrow 0$.

\item[2] There exists an index $r(\ell)$ such that $A_{r(\ell)}(\XX,Z_0)$ can be decomposed into the sum of a multiaffine map of $\XX_{\neq \ell}, Z_0$, and an injective linear map $R_\ell(X_\ell)$. Since $A_X = \nabla_{X_\ell}A(\XX,Z_0)$, the $r(\ell)$-th component of $A_X$ is equal to $R_\ell$. Thus, the $r(\ell)$-th component of $A_X(X_\ell) - A_X(X_\ell^+)$ is $R_\ell(X_\ell - X_\ell^+)$.

Let $\mu_\ell = M_\ell$ if $f_\ell$ is $M_\ell$-Lipschitz differentiable, and $\mu_\ell = 0$ if $f_\ell$ is convex and nonsmooth. We then have
\begin{align}
\nonumber &\Lag(\XX_{<\ell}^+, X_\ell, \XX_{>\ell}, \ZZZ, \WW) - \Lag(\XX_{<\ell}^+, X_\ell^+, \XX_{>\ell}, \ZZZ, \WW) \\
\nonumber & = f_\ell(X_\ell) - f_\ell(X_\ell^+) - \la v, X_\ell - X_\ell^+\ra + \frac{\rho}{2}\|A_X(X_\ell) - A_X(X_\ell^+)\|^2 \\
\nonumber & \bump + F(\XX_{<\ell}^+,X_\ell,\XX_{>\ell}) - F(\XX_{<\ell}^+,X_\ell^+,\XX_{>\ell})  - \la \nabla_{X_\ell}F(\XX_{<\ell}^+,X_\ell^+,\XX_{>\ell}) , X_\ell - X_\ell^+\ra \\
\nonumber &\geq -\frac{(\mu_\ell + M_F)}{2}\|X_\ell - X_\ell^+\|^2 + \frac{\rho}{2}\|R_\ell(X_\ell - X_\ell^+)\|^2 \\
\label{eq:xsuffdec} &\geq \frac{1}{2} \left(\rho\lambda_{min}(R_\ell^TR_\ell) - \mu_\ell - M_F\right) \|X_\ell - X_\ell^+\|^2.
\end{align}
Taking $\rho \geq \lambda_{min}^{-1}(R_\ell^TR_\ell)(\mu_\ell + M_F)$, we see that $\|X_\ell - X_\ell^+\| \rightarrow 0$.
\end{description}

It remains to show that $\|Z_0 - Z_0^+\| \rightarrow 0$ in all three cases of \Cref{azainj}. Two cases are immediate. If $Z_0 \in Z_S$, then $\|Z_0 - Z_0^+\| \rightarrow 0$ is implied by \Cref{lastconverges}, because $\|Z_S - Z_S^+\| \rightarrow 0$. If $h(Z_0)$ satisfies a strengthened convexity condition, then by inspecting the terms of \cref{eq:dec}, we see that the same argument for $X_\ell$ applies to $Z_0$. Thus, we assume that \Cref{azainj}(3) holds. Let $A_X(\ZZZ) = b_X$ denote the system of constraints when updating $\ZZZ$. The third condition of \Cref{azainj} implies that for $r = r(0)$, the $r$-th component of the system of constraints $A_1(\XX,Z_0) + Q_1(Z_1) = 0$ is equal to $A_0'(\XX) + R_0(Z_0) + Q_r(Z_1) = 0$ for the corresponding submatrix $Q_r$ of $Q_1$. Hence, the $r$-th component of $A_X(\ZZZ)$ is equal to $R_0(Z_0) + Q_r(Z_1)$. Inspecting the terms of \cref{eq:dec}, we see that 
$$\Lag(\XX^+,\ZZZ, \WW) - \Lag(\XX^+,\ZZZ^+,\WW) \geq \frac{\rho}{2}\|R_0(Z_0) + Q_r(Z_1) - (R_0(Z_0^+) + Q_r(Z_1^+))\|^2$$
Since $\ux{\Lag}{k}$ converges, and the increases of $\ux{\Lag}{k}$ are bounded by $\frac{1}{\rho}\|\WW - \WW^+\| \rightarrow 0$, we must also have $\Lag(\XX^+,\ZZZ, \WW) - \Lag(\XX^+,\ZZZ^+,\WW) \rightarrow 0$, or else the updates of $\ZZZ$ would decrease $\ux{\Lag}{k}$ to $-\infty$. By \Cref{lastconverges}, $\|Z_1 - Z_1^+\| \rightarrow 0$, since $Z_1$ is always part of $Z_S$. Hence $\|R_0(Z_0 - Z_0^+)\| \rightarrow 0$, and the injectivity of $R_0$ implies that $\|Z_0 - Z_0^+\| \rightarrow 0$. Combined with \Cref{lastconverges}, we conclude that $\|\ZZZ - \ZZZ^+\| \rightarrow 0$.
\end{proof}

\begin{proof}[Proof (of \Cref{main2})] We first confirm that the conditions of \Cref{limitfocconditionsgeneric} hold for $\{X_0,\ldots,X_n\}$. \Cref{allbded,lastconverges} together show that all variables and constraints are bounded, and that $\|\WW - \WW^+\| \rightarrow 0$. Since $\|X_\ell - X_\ell^+\| \rightarrow 0$ for all $\ell \geq 1$, and $\|\ZZZ - \ZZZ^+\| \rightarrow 0$, we have $\|U_>^+ - U_>\| \rightarrow 0$, and the conditions $\|C_> - C_<\| \rightarrow 0$ and $\|b_> - b_<\| \rightarrow 0$  follow from \Cref{constraintconvergegeneric}. Note that $X_0$ is not part of $X_>$ for any $\ell$, which is why we need only that $\{\ux{X}{k}_0\}$ is bounded, and $\|X_\ell - X_\ell^+\| \rightarrow 0$ for $\ell \geq 1$. Thus, \Cref{limitfocconditionsgeneric} implies that we can find $\ux{v}{k}_x \in \partial_\XX \Lag(\ux{\XX}{k}, \ux{\ZZZ}{k}, \ux{\WW}{k})$ with $\ux{v}{k}_x \rightarrow 0$; combined with the subgradients in $\partial_{\ZZZ}\Lag(\ux{\XX}{k}, \ux{\ZZZ}{k}, \ux{\WW}{k})$ converging to 0 (\Cref{main}) and the fact that $\nabla_W \Lag(\ux{\XX}{k}, \ux{\ZZZ}{k}, \ux{\WW}{k}) \rightarrow 0$ (\Cref{dualgeneric}), we obtain a sequence $\ux{v}{k} \in \partial \Lag(\ux{\XX}{k}, \ux{\ZZZ}{k}, \ux{\WW}{k})$ with $\ux{v}{k} \rightarrow 0$.

Having verified the conditions for \Cref{limitfocconditionsgeneric}, \Cref{limitfocgeneric} then shows that all limit points are constrained stationary points. Part of this theorem (that every limit point is a constrained stationary point) can also be deduced directly from \Cref{iteratesconvergeimpliesstationary} and \Cref{xconverges}.
\end{proof}

\subsection{Proof of \Cref{mainkl}}\label{sec:mainklproof}
\begin{proof}
We will apply \Cref{absconvergence} to $\Lag(\XX,\ZZZ,\WW)$. First, for \textbf{H2}, observe that the desired subgradient $w^{k+1}$ is provided by \Cref{generalinnerblock}. Since the functions $V$ and $\nabla F$ are continuous, and all variables are bounded by \Cref{allbded}, $V$ and $\nabla F$ are \emph{uniformly} continuous on a compact set containing $\{(\ux{\XX}{k}, \ux{\ZZZ}{k}, \ux{\WW}{k})\}_{k=0}^\infty$. Hence, we can find $b$ for which \textbf{H2} is satisfied. 

Together, \Cref{winc} and \Cref{lagdec} imply that \textbf{H1} holds for $\WW$ and $\ZZ$. Using the hypothesis that \Cref{axainj}(2) holds for $X_0, X_1,\ldots,X_n$, the inequality \eqref{eq:xsuffdec} implies that property \textbf{H1} in \Cref{absconvergence} holds for $X_0,X_1,\ldots,X_n$. Lastly, \Cref{azainj}(2) holds, so $Z_S = (Z_0,Z_1)$ and thus $g_1$ is a strongly convex function of $Z_0$, so \Cref{yzdecr} implies that \textbf{H1} also holds for $Z_0$. Thus, we see that \textbf{H1} is satisfied for all variables. Finally, \Cref{asm2} implies that $\phi$, and therefore $\Lag$, is continuous on its domain, so \Cref{absconvergence} applies and completes the proof.
\end{proof}

\section*{Acknowledgements}
We thank Qing Qu, Yuqian Zhang, and John Wright for helpful discussions about applications of multiaffine ADMM. We thank Wotao Yin for his feedback, and for bringing the paper \cite{HCWSH2016} to our attention. We thank Yenson Lau for discussions about numerical experiments.

\bibliographystyle{siam}
\bibliography{admm}

\appendix
\section{Proofs of Technical Lemmas}\label{apdx:technical}
We provide proofs of the technical results in \Cref{sec:mainresults,sec:prelim,sec:general}. \Cref{lipschdfbound,convexspec,leastim,subdistgeneric} are standard results, so we omit their proofs for space considerations.

\subsection{Proof of \Cref{im_counterexample}}
\begin{proof}
The augmented Lagrangian of this problem is $\Lag(x,y,w) = x^2 + y^2 + w(xy - 1) + \frac{\rho}{2}(xy - 1)^2$, and thus $\frac{\partial}{\partial x} \Lag(x,y,w) = x(2 + \rho y^2) + y(w-\rho)$. If $y = 0$ or $w = \rho$, the minimizer of the $x$-subproblem is $x^+ = 0$. Likewise, if $x = 0$, then $y^+ = 0$. Hence, if either $\ux{y}{k} = 0$ or $\ux{w}{k} = \rho$, we have $(\ux{x}{j}, \ux{y}{j}) = (0,0)$ for all $j > k$. The multiplier update is then $w^+ = w - \rho$, so $\ux{w}{k} \rightarrow -\infty$.
\end{proof}

\subsection{Proof of \Cref{constraint_prox}}
\begin{proof}
We proceed by induction. Since $Z_3$ is part of the final block and $W_3^k = 0$, the minimization problem for $Z_3^{k+1}$ is $\min_{Z_3} \|S^{1/2}(Z_3 - X_\ell^{k+1})\|^2$, for which $Z_3^{k+1} = X_\ell^{k+1}$ is an optimal solution. The update for $W_3^{k+1}$ is then $W_3^{k+1} = \rho(X_\ell^{k+1} - Z_3^{k+1}) = 0$.
\end{proof}

\subsection{Proof of \Cref{crcqlma}}
\begin{proof}
Observe that $\nabla C(x,z) = \begin{pmatrix} \nabla A(x) & Q \end{pmatrix}$. The condition $\opn{Im}(Q) \supseteq \opn{Im}(A)$ implies that for every $x$, $\opn{Im}(Q) \supseteq \opn{Im}(\nabla A(x))$, and thus $\opn{Null}(Q^T) \subseteq \opn{Null}((\nabla A(x))^T)$. The result follows immediately.
\end{proof}

\subsection{Proof of \Cref{multiaffinesumform}}
\begin{proof}
We proceed by induction on $n$. When $n = 1$, a multiaffine map is an affine map, so $\MM(X_1) = A(X_1) + B$ as desired. Suppose now that the desired result holds for any multiaffine map of $n-1$ variables. Given a subset $S \subseteq \{1,\ldots,n\}$, let $X_S$ denote the point with $(X_S)_j = X_j$ for $j \in S$, and $(X_S)_j = 0$ for $j \notin S$. That is, the variables not in $S$ are set to 0 in $X_S$. Consider the multiaffine map $\NN$ given by
$$\NN(X_1,\ldots,X_n) = \MM(X_1,\ldots,X_n) + \sum_{|S| \leq n- 1} (-1)^{n-|S|} \MM(X_S)$$
where the sum runs over all subsets $S \subseteq \{1,\ldots,n\}$ with $|S| \leq n -1$. Since $X_S \mapsto \MM(X_S)$ is a multiaffine map of $|S|$ variables, the induction hypothesis implies that $\MM(X_S)$ can be written as a sum of multilinear maps. Hence, it suffices to show that $\NN(X_1,\ldots,X_n)$ is multilinear, in which case $\MM(X_1,\ldots,X_n) = \NN(X_1,\ldots,X_n) - \sum_S (-1)^{n-|S|}\MM(X_S)$ is a sum of multilinear maps.

We verify the condition of multilinearity. Take $k \in \{1,\ldots,n\}$, and write $U = (X_j:j \neq k)$. Since $\MM$ is multiaffine, there exists a linear map $A_U(X_k)$ such that $\MM(U,X_k) = A_U(X_k) + \MM(U,0)$. Hence, we can write
\begin{equation}\label{eq:masub}
\MM(U,X_k + \lambda Y_k) = \MM(U,X_k) + \lambda \MM(U,Y_k) - \lambda \MM(U,0).
\end{equation}
By the definition of $\NN$, $\NN(U,X_k + \lambda Y_k)$ is equal to
\begin{align*}
\MM(U,X_k + \lambda Y_k) + \sum_{k \in S} (-1)^{n - |S|} \MM(U_{S \setminus k}, X_k + \lambda Y_k) + \sum_{k \notin S} (-1)^{n - |S|} \MM(U_S,0)
\end{align*}
where the sum runs over $S$ with $|S| \leq n-1$. Making the substitution \eqref{eq:masub} for every $S$ with $k \in S$, we find that $\NN(U,X_k + \lambda Y_k)$ is equal to
\begin{align}
\label{eq:forN} &\MM(U,X_k) + \lambda \MM(U,Y_k) - \lambda \MM(U,0) \\
\nonumber &\bump + \sum_{k \in S} (-1)^{n - |S|} (\MM(U_{S \setminus k}, X_k) + \lambda \MM(U_{S \setminus k},Y_k) - \lambda \MM(U_{S \setminus k},0)) \\
&\nonumber \bump + \sum_{k \notin S} (-1)^{n - |S|} \MM(U_S,0)
\end{align}
Our goal is to show that $\NN(U,X_k + \lambda Y_k) = \NN(U,X_k) + \lambda \NN(U,Y_k)$. Since
$$\NN(U,Y_k) = \MM(U,Y_k) + \sum_{k \in S} (-1)^{n - |S|}\MM(U_{S \setminus k},Y_k) + \sum_{k \notin S}(-1)^{n - |S|}\MM(U_S,0)$$
we add and subtract $\lambda \sum_{k \notin S} (-1)^{n-|S|}\MM(U_S,0)$ in \eqref{eq:forN} to obtain the desired expression $\NN(U,X_k) + \lambda \NN(U,Y_k)$, minus a residual term $$\lambda \left( \MM(U,0) + \sum_{k \in S} (-1)^{n - |S|} \MM(U_{S \setminus k},0) + \sum_{k \notin S} (-1)^{n-|S|}\MM(U_S,0) \right)$$ It suffices to show the term in parentheses is 0.

There is exactly one set $S$ with $k \notin S$ with $|S| = n - 1$, and for this set, $U_S = U$. For this $S$, the terms $\MM(U,0)$ and $(-1)^{n - (n-1)}\MM(U_S,0)$ cancel out. The remaining terms are
\begin{equation}\label{eq:invol}
\sum_{k \in S, |S| \leq n-1} (-1)^{n-|S|} \MM(U_{S \setminus k},0) + \sum_{k \notin S, |S| \leq n-2} (-1)^{n-|S|} \MM(U_S,0)\end{equation}
There is a bijective correspondence between $\{S:k \notin S, |S| \leq n-2\}$ and $\{S:k \in S, |S| \leq n-1\}$ given by $S \leftrightarrow S \cup \{k\}$. Since $|S \cup \{k\}| = |S| + 1$, \eqref{eq:invol} becomes
$$\sum_{k \notin S, |S| \leq n-2} ((-1)^{n-|S|-1} + (-1)^{n-|S|}) \MM(U_S,0) = 0$$
which completes the proof.
\end{proof}

\subsection{Proof of \Cref{multiaffsmooth}}
We prove two auxiliary lemmas, from which \Cref{multiaffsmooth} follows as a corollary.

\begin{lma}\label{bilinearuniformbd}
Let $\MM$ be a multilinear map. There exists a constant $\sigma_M$ such that $\|\MM(X_1,\ldots,X_n)\| \leq \sigma_M\prod \|X_i\|$.
\end{lma}
\begin{proof}
We proceed by induction on $n$. When $n = 1$, $\MM$ is linear. Suppose it holds for any multilinear map of up to $n - 1$ blocks. Given $U = (X_1,\ldots,X_{n-1})$, let $\MM_U$ be the linear map $\MM_U(X_n) = \MM(U,X_n)$, and let $\FF$ be the family of linear maps $\FF = \{ \MM_U : \|X_1\| = 1,\ldots,\|X_{n-1}\| = 1\}$. Now, given $X_n$, let $\MM_{X_n}$ be the \emph{multilinear} map $\MM_{X_n}(U) = \MM(U,X_n)$. By induction, there exists some $\sigma_{X_n}$ for $\MM_{X_n}$. For every $X_n$, we see that
\begin{align*}
\sup_\FF \|\MM_U(X_n)\| &= \sup \{ \|\MM(X_1,\ldots,X_n)\|: \|X_1\| = 1,\ldots,\|X_{n-1}\| = 1\} \\
&= \sup_{\|X_1\| = 1,\ldots,\|X_{n-1}\| = 1} \|\MM_{X_n}(U)\| \leq \sigma_{X_n} < \infty
\end{align*}
Thus, the uniform boundedness principle \cite{S2011} implies that 
$$\sigma_M := \sup_\FF \|\MM_U\|_{op} =  \sup\limits_{\|X_1\| = 1,\ldots,\|X_n\| = 1} \| \MM(X_1,\ldots,X_n)\| < \infty$$
Given any $\{X_1,\ldots,X_n\}$, we then have $\|\MM(X_1,\ldots,X_n)\| \leq \sigma_M \prod \|X_i\|$.
\end{proof}

\begin{lma}\label{multilindiffbd}
Let $\MM(X_1,\ldots,X_n)$ be a multilinear map with $X = (X_1,\ldots,X_n)$ and $X' = (X_1', \ldots,X_n')$ being two points with the property that, for all $i$, $\|X_i\| \leq d$, $\|X'_i\| \leq d$, and $\|X_i - X'_i\| \leq \epsilon$. Then $\|\MM(X) - \MM(X')\| \leq n\sigma_M d^{n-1} \epsilon$, where $\sigma_M$ is from \Cref{bilinearuniformbd}.
\end{lma}
\begin{proof}
For each $0 \leq k \leq n$, let $X_k'' = (X_1,\ldots,X_k, X_{k+1}',\ldots, X_n')$. By \Cref{bilinearuniformbd}, $\|\MM(X''_k) - \MM(X''_{k-1})\| = \|\MM(X_1,\ldots,X_{k-1}, X_k - X_k',X_{k+1}',\ldots,X_n') \|$ is bounded by $\sigma_M d^{n-1} \|X_k - X_k'\| \leq \sigma_M d^{n-1}\epsilon$. Observe that $\MM(X) - \MM(X') = \sum_{k=1}^{n} \MM(X''_k) - \MM(X''_{k-1})$, and thus we obtain $\|\MM(X) - \MM(X')\| \leq n\sigma_M d^{n-1}\epsilon$.
\end{proof}



\subsection{Proof of \Cref{cvxargmin}}
\begin{proof}
Let $d = b - a$, and define $\delta = \|d\|$. Define $x' = y^\ast - d \in \CC_1$, $y' = x^\ast + d \in \CC_2$, and $s = x' - x^\ast \in T_{\CC_1}(x^\ast)$. Let $\sigma = g(y') - g(x^\ast)$. We can express $\sigma$ as $\sigma = \int_0^1 \nabla g(x^\ast + td)^Td~dt$. Since $\nabla g$ is Lipschitz continuous with constant $M$, we have
\begin{align*}
g(y^\ast) - g(x') &= \int_0^1 \nabla g(x' + td)^Td~dt \\
&= \sigma + \int_0^1 (\nabla g(x' + td) - \nabla g(x^\ast + td))^Td~dt
\end{align*}
and thus $|g(y^\ast) - g(x') - \sigma| \leq \int_0^1 \|\nabla g(x'+td) - \nabla g(x^\ast+td)\|\|d\|~dt \leq M\|s\|\delta$, by Lipschitz continuity of $\nabla g$. Therefore $g(y^\ast) \geq g(x') + \sigma - M\|s\|\delta$. Since $g$ is differentiable and $\CC_1$ is closed and convex, $x^\ast$ satisfies the first-order condition $\nabla g(x^\ast) \in -N_{\CC_1}(x^\ast)$. Hence, since $s \in T_{\CC_1}(x^\ast) = N_{\CC_1}(x^\ast)^\circ$, we have $g(x') \geq g(x^\ast) + \la \nabla g(x^\ast), s\ra + \frac{m}{2}\|s\|^2 \geq g(x^\ast) + \frac{m}{2}\|s\|^2$. Combining these inequalities, we have $g(y^\ast) \geq g(x^\ast) + \sigma + \frac{m}{2}\|s\|^2 - M\|s\|\delta$. Since $y^\ast$ attains the minimum of $g$ over $\CC_2$, $g(y') \geq g(y^\ast)$. Thus
\begin{align*}
g(y') = g(x^\ast) + \sigma &\geq g(y^\ast) \geq g(x^\ast) + \sigma + \frac{m}{2}\|s\|^2 - M\|s\|\delta
\end{align*}
We deduce that $\frac{m}{2}\|s\|^2 - M\|s\|\delta \leq 0$, so $\|s\| \leq 2\kappa \delta$. Since $y^\ast  - x^\ast = s + d$, we have $\|x^\ast - y^\ast\| \leq \|s\| + \|d\| \leq \delta + 2\kappa \delta = (1+2\kappa)\delta$.
\end{proof}

\subsection{Proof of \Cref{bicvxargmin}}
\begin{proof}
Note that $x \in \UU_1$ is equivalent to $Ax \in -b_1 + \CC$, and thus $\UU_1 = A^{-1}(-b_1 + \CC)$, where $A^{-1}(S) = \{x: Ax \in S\}$ is the preimage of a set $S$ under $A$. Since $\UU_1$ is the preimage of the closed, convex set $-b_1 + \CC$ under a linear map, $\UU_1$ is closed and convex. Similarly, $\UU_2 = A^{-1}(-b_2 + \CC)$ is closed and convex.

We claim that $\UU_1, \UU_2$ are translates. Since $b_1, b_2 \in \opn{Col}(A)$, we can find $d$ such that $Ad = b_1 - b_2$. Given $x \in \UU_1$, $A(x + d) \in -b_2 + \CC$, so $x + d \in \UU_2$, and thus $\UU_1 + d \subseteq \UU_2$. Conversely, given $y \in \UU_2$, $A(y - d) \in -b_1 + \CC$, so $y - d \in \UU_1$ and $\UU_1 + d \supseteq \UU_2$. Hence $\UU_2 = \UU_1 + d$. Applying \Cref{cvxargmin} to $\UU_1, \UU_2$, we find that $\|x^\ast - y^\ast\| \leq (1 + 2\kappa)\|d\|$. We may choose $d$ to be a solution of minimum norm satisfying $Ad = b_1 - b_2$; applying \Cref{subdistgeneric} to the spaces $\{x:Ax = 0\}$ and $\{x: Ax = b_1 - b_2\}$, we see that $\|d\| \leq \alpha\|b_1 - b_2\|$, where $\alpha$ depends only on $A$. Hence $\|x^\ast - y^\ast\| \leq (1+2\kappa)\alpha \|b_2 - b_1\|$.
\end{proof}

\subsection{Proof of \Cref{dualgeneric}}
\begin{proof}
The dual update is given by $W^+ = W + \rho C(\UU^+)$. Thus, we have $$\Lag(\UU^+,W^+) - \Lag(\UU^+,W) = \la W^+ - W, C(\UU^+)\ra = \rho \|C(\UU^+)\|^2 = \frac{1}{\rho}\|W - W^+\|^2.$$

For the second statement, observe that $\nabla_W \Lag(\UU,W) = C(\UU)$. From the dual update, we have $W^+ - W = \rho C(U^+)$. Hence $\|C(\UU^+)\| = \frac{1}{\rho}\|W - W^+\| \rightarrow 0$. It follows that $\nabla_W \Lag(\ux{U}{k}, \ux{W}{k}) \rightarrow 0$ and, by continuity of $C$, any limit point $\UU^\ast$ of $\{\ux{U}{k}\}_{k=0}^\infty$ satisfies $C(\UU^\ast) = 0$.
\end{proof}

\subsection{Proof of \Cref{subgradgeneric2}}
\begin{proof}
Since $\la W, C(U,Y)\ra + \frac{\rho}{2}\|C(U,Y)\|^2$ is smooth, \cite[8.8(c)]{RW1997} implies that
\begin{align*}
\partial_Y \Lag(U,Y,W) &= \partial f_U(Y) + \nabla_Y \la W, C(U,Y) \ra + \nabla_Y \left( \frac{\rho}{2}\|C(U,Y)\|^2 \right) \\
&= \partial f_U(Y) + (\nabla_Y C(U,Y))^TW + \rho (\nabla_Y C(U,Y))^TC(U,Y).
\end{align*}
\end{proof}

\subsection{Proof of \Cref{generalinnerblock}}
\begin{proof}
Let $g_y$ denote the separable term in $Y$ (that is, if $Y = U_j$, then $g_y = g_j$). By \Cref{subgradgeneric2},
\begin{align*}
0 &\in \partial f_{U_<^{k+1},U_>^k}(Y^{k+1}) + V(U_<^{k+1},Y^{k+1},U_>^k,W^k) \\
&= \nabla_Y F(\ux{U}{k+1}_<, \ux{Y}{k+1},\ux{U}{k}_>) + \partial g_y(\ux{Y}{k+1}) + V(U_<^{k+1},Y^{k+1},U_>^k,W^k).
\end{align*}
Hence, 
\begin{equation}\label{eq:diffobj}
-(\nabla_Y F(\ux{U}{k+1}_<, \ux{Y}{k+1},\ux{U}{k}_>) +  V(U_<^{k+1},Y^{k+1},U_>^k,W^k)) \in \partial g_y(\ux{Y}{k+1}).
\end{equation}
In addition, by \Cref{subgradgeneric2},
$$\begin{aligned} &\ \partial_Y \Lag(\ux{U}{k+1},\ux{Y}{k+1},\ux{W}{k+1}) \\ =&\ \partial g_y(\ux{Y}{k+1}) + \nabla_Y F(\ux{U}{k+1}_<, \ux{Y}{k+1},\ux{U}{k+1}_>) + V(\ux{U}{k+1}_<,\ux{Y}{k+1},\ux{U}{k+1},\ux{W}{k+1}).\end{aligned}$$
Combining this with \eqref{eq:diffobj} implies the desired result.

Applying this to $\partial_Y \Lag(\ux{U}{k(s)}, \ux{Y}{k(s)}, \ux{W}{k(s)})$, we obtain the subgradient
\begin{align*}
\ux{v}{s} &:= V(\ux{U}{k(s)}_<,\ux{Y}{k(s)},\ux{U}{k(s)}_>,\ux{W}{k(s)}) - V(\ux{U}{k(s)}_<,\ux{Y}{k(s)},\ux{U}{k(s)-1}_>,\ux{W}{k(s)-1}) \\
&\bump + \nabla_Y F(\ux{U}{k(s)}_<,\ux{Y}{k(s)},\ux{U}{k(s)}_>) - \nabla_Y F(\ux{U}{k(s)}_<,\ux{Y}{k(s)},\ux{U}{k(s)-1}_>).
\end{align*}
Since $\{(\ux{U}{k(s)}, \ux{Y}{k(s)}, \ux{W}{k(s)})\}_{s=0}^\infty$ converges, and $\|\ux{U}{k+1}_> - \ux{U}{k}_>\| \rightarrow 0$ and $\|\ux{W}{k+1} - \ux{W}{k}\| \rightarrow 0$ by assumption, there exists a compact set $\mathcal{B}$ containing the points $\{\ux{U}{k(s)}_<, \ux{U}{k(s)-1}_>, \ux{Y}{k(s),}, \ux{W}{k(s)}, \ux{W}{k(s)-1}\}_{s=0}^\infty$. $V$ and $\nabla_Y F$ are continuous, so it follows that $V$ and $\nabla_Y F$ are \emph{uniformly} continuous over $\mathcal{B}$. It follows that when $s$ is sufficiently large,
$$V(\ux{U}{k(s)}_<,\ux{Y}{k(s)},\ux{U}{k(s)}_>,\ux{W}{k(s)}) - V(\ux{U}{k(s)}_<, \ux{Y}{k(s)}, \ux{U}{k(s)-1}_>,\ux{W}{k(s)-1})$$
and
$$\nabla_Y F(\ux{U}{k(s)}_<,\ux{Y}{k(s)},\ux{U}{k(s)}_>) - \nabla_Y F(\ux{U}{k(s)}_<, \ux{Y}{k(s)}, \ux{U}{k(s)-1}_>)$$
can be made arbitrarily small. This completes the proof.
\end{proof}

\subsection{Proof of \Cref{limitfocgeneric}}
We require the following simple fact.
\begin{lma}\label{gentoreg}
Let $f: \RR^n \rightarrow \RR \cup \{\infty\}$. Suppose that we have sequences $x_k \rightarrow x$ and $v_k \in \partial f(x_k)$ such that $f(x_k) \rightarrow f(x)$ and $v_k \rightarrow v$. Then $v \in \partial f(x)$.
\end{lma}
\begin{proof}
This result would follow by definition if $v_k \in \wh{\partial}f(x_k)$, but instead we have $v_k \in \partial f(x_k)$. However, for each $k$, there exists sequences $x_{j,k} \rightarrow x_k$ and $v_{j,k} \in \wh{\partial}f(x_{j,k})$ with $f(x_{j,k}) \rightarrow f(x_k)$ and $v_{j,k} \rightarrow v_k$. By a simple approximation, we can select subsequences $y_s \rightarrow x, z_s \in \wh{\partial}f(y_s)$ with $f(y_s) \rightarrow f(x), z_s \rightarrow v$.
\end{proof}

\begin{proof}[Proof (of \Cref{limitfocgeneric})]
By \Cref{subgradgeneric2}, $\partial_Y \Lag(\ux{U}{s},\ux{Y}{s},\ux{W}{s}) = \partial g_y(\ux{Y}{s}) + \nabla_Y F(\ux{U}{s},\ux{Y}{s}) + V(\ux{U}{s},\ux{Y}{s},\ux{W}{s})$. Since $V$ is continuous, the sequence $\{V(\ux{U}{s},\ux{Y}{s},\ux{W}{s})\}$ converges to $V(U^\ast,Y^\ast,W^\ast)$, which is equal to $(\nabla_Y C(U^\ast,Y^\ast))^TW^\ast$ because $(U^\ast, Y^\ast,W^\ast)$ is feasible. Likewise, $\{\nabla_Y F(\ux{U}{s},\ux{Y}{s})\}$ converges to $\nabla_Y F(U^\ast, Y^\ast)$.
	
Since $\ux{v}{s} \in \partial_Y \Lag(\ux{U}{s},\ux{Y}{s},\ux{W}{s})$ for all $s$ and $\ux{v}{s} \rightarrow 0$, we deduce that there exists a sequence $\{\ux{v}{s}_y\}$ such that $\ux{v}{s}_y \in \partial g_y(\ux{Y}{s})$ for all $s$ and $\ux{v}{s}_y \rightarrow -(\nabla_Y F(U^\ast, Y^\ast) + (\nabla_Y C(U^\ast,Y^\ast))^TW^\ast)$.
Hence, by \Cref{gentoreg}\footnote{The assumption that each $g_i$ is continuous on $\opn{dom}(g_i)$ was introduced in \Cref{weakcoupling} to ensure that $g_y(Y^s) \rightarrow g_y(Y^\ast)$, which is required to obtain the general subgradient $\partial g_y(Y^\ast)$.} applied to $g_y$ and the sequences $\{\ux{Y}{s}\}$ and $\{\ux{v}{s}_y\}$, we find $-(\nabla_Y F(U^\ast, Y^\ast) + (\nabla_Y C(U^\ast,Y^\ast))^TW^\ast) \in \partial g_y(Y^\ast)$, as desired.
\end{proof}

\subsection{Proof of \Cref{iteratesconvergeimpliesstationary}}
\begin{proof}
If $\|\ux{W}{k+1} - \ux{W}{k}\| \rightarrow 0$ and $\|\ux{U}{k+1}_\ell - \ux{U}{k}_\ell\| \rightarrow 0$ for all $\ell \geq 1$, then the conditions of \Cref{generalinnerblock} are satisfied for all blocks $U_0,\ldots,U_n$. Thus, \Cref{dualgeneric} implies that $\UU^\ast$ is feasible, and by \Cref{limitfocgeneric}, $(\UU^\ast, W^\ast)$ satisfies the first-order conditions. Note that we do not need to assume $\|\ux{U}{k+1}_0 - \ux{U}{k}_0\| \rightarrow 0$ because $U_0$ is not part of $U_>$ for any block.
\end{proof}

\section{Alternate Deep Neural Net Formulation}\label{apdx:nnet}

When $h(z) = \max\{z,0\}$, we can approximate the constraint $a_\ell - h(z_\ell) = 0$ by introducing a variable $a'_\ell \geq 0$, and minimizing a combination of $\|a'_\ell - z_\ell\|^2, \|a'_\ell - a_\ell\|^2$. This leads to the following biaffine formulation, which satisfies \Cref{baseasm,asm2}, for the deep learning problem:
$$
\left\{
\begin{array}{rll}
\inf & E(z_L,y)
+ \sum_{\ell=1}^{L-1} \iota(a'_\ell)
+  \frac{\mu}{2} \sum_{\ell=1}^{L-1}[ \|\hat{a_\ell}\|^2 + \|s_\ell\|^2 ] + R(X_1,\ldots,X_L) \\
& X_La_{L-1} - z_L = 0 \\
& \begin{bmatrix} X_\ell a_{\ell-1} \\ a'_\ell \\ a'_\ell - a_\ell \end{bmatrix} - \begin{bmatrix} I & 0 & 0 \\ I & I & 0 \\ 0 & 0 & I \end{bmatrix} \begin{bmatrix} z_\ell \\ s_\ell \\ \hat{a}_\ell \end{bmatrix} = 0 \quad \text{ for } 1 \leq \ell \leq L - 1. \\
\end{array}
\right.
.$$

\section{Table of Assumptions}\label{apdx:tables}

We organize \Cref{baseasm,asm2} in tabular form and present them in \Cref{tab:assumption1_tabular,tab:assumption2_tabular}.

\begin{table}[h]
	\caption{Contents of \Cref{baseasm}}
	\begin{center}
		\begin{tabular}{
				|p{0.3\textwidth}
				|p{0.5\textwidth}|
			}
			\hline
			& \Cref{baseasm} \\ \hline
			
			$\ZZZ$-block operator $Q$ & $\opn{Im}(Q) \supseteq \opn{Im}(A)$\\ \hline
			
			\multirow{2}{*}{Total objective function $\phi$} 
			& $\phi$ coercive on feasible region \\
			\cline{2-2}
			& $\phi$ splits as $f(\XX) + \psi(\ZZZ)$ \\ \hline
			
			$\ZZZ$-block objective function $\psi$ & $\psi$ splits as $h(Z_0) + g_1(Z_2) + g_2(Z_2)$ \\ \hline
			$h(Z_0)$ & $h(Z_0)$ is proper, convex, and lower semicontinuous \\ \hline
			$g_1(Z_S)$ & $g_1$ is strongly convex, and either
			\begin{itemize}
				\item $Z_S = Z_1$, or
				\item $Z_S = (Z_0, Z_1)$
			\end{itemize} \\ \hline
			$g_2(Z_2)$ & $g_2(Z_2)$ is Lipschitz differentiable \\ \hline
			$Z_2$-operator $Q_2$ & $Q_2$ is injective\\ \hline
		\end{tabular}
	\end{center}
	\label{tab:assumption1_tabular}
\end{table}

Note that \Cref{asm2} is a superset of \Cref{baseasm}, and introduces additional assumptions on the function $f(\XX)$.

\begin{table}[h]
	\caption{Contents of \Cref{asm2}}
	\begin{center}
		\begin{tabular}{
				|p{0.2\textwidth}
				|p{0.6\textwidth}|
			}
			\hline
			& \Cref{asm2} \\ \hline
			$\XX$-block function $f$ & $f(\XX)$ splits into $F(X_0,\ldots,X_n) + \sum_{i=0}^n f_i(X_i)$ \\ \hline
			
			$F(X_0,\ldots,X_n)$ & $F(X_0,\ldots,X_n)$ is Lipschitz differentiable \\ \hline
			
			$f_0(X_0),\ldots, f_n(X_n)$ & $f_i(X_i)$ is proper and lower semicontinuous, and continuous on $\opn{dom}(f_i)$ \\ \hline
			
			$X_\ell$, for $1 \leq \ell \leq n$
			& One of the following cases holds:
			\begin{description}
				\item[Case 1]  $F(X_0,\ldots,X_n)$ is independent of $X_\ell$, and $f_\ell(X_\ell)$ satisfies a strengthened convexity condition (see \Cref{defscc}).
				\item [Case 2] $f_\ell$ is either convex or Lipschitz differentiable. Viewing $A(\XX,Z_0) + Q(\ZZ) = 0$ as a system of constraints, there exists an constraint in the system, with index $r(\ell)$, such that in the $r(\ell)$-th constraint, we have $A_{r(\ell)}(\XX,Z_0) = R_\ell(X_\ell) + A'_\ell (\XX_{\neq \ell}, Z_0)$ for an injective linear map $R_\ell$ and a multiaffine map $A'_\ell$.
			\end{description}
			\\ \hline
			$Z_0$ & One of the following cases holds:
			\begin{description}
				\item[Case 1] $h(Z_0)$ satisfies a strengthened convexity condition (\Cref{defscc}).
				\item [Case 2] $Z_0 \in Z_S$, so $g_1(Z_S)$ is a strongly convex function of $Z_0$ and $Z_1$.
				\item [Case 3] Viewing $A(\XX,Z_0) + Q(\ZZ) = 0$ as a system of constraints, there exists an index $r(0)$ such that $A_{r(0)}(\XX,Z_0) = R_0(Z_0) + A'_0(\XX)$ for an injective linear map $R_0$ and multiaffine map $A_0'$.
			\end{description}
			\\ \hline
		\end{tabular}
	\end{center}
	\label{tab:assumption2_tabular}
\end{table}

\section{Formulations with Closed-Form Subproblems}\label{apdx:simple}

\subsection{Representation Learning}\label{apdx:replearning}

Observe that in (NMF1), the ADMM subproblems for $X$ and $Y$, which have quadratic objective functions and nonnegativity constraints, do not have closed-form solutions. To update $X$ and $Y$, \cite{HCWSH2016} proposes using ADMM to approximately solve the subproblems. This difficulty can be removed through variable splitting. Specifically, by introducing auxiliary variables $X'$ and $Y'$, one obtains the equivalent problem:
$$
(\text{NMF}2) \bump
\left\{
\begin{array}{rll}
\inf_{X,X',Y,Y',Z}\limits & \iota(X') + \iota(Y') + \frac{1}{2}\|Z - B\|^2 \\
& Z = XY,\ X = X',\ Y = Y',
\end{array}
\right.$$
where $\iota$ is the indicator function for the nonnegative orthant; i.e., $\iota(X) = 0$ if $X \geq 0$ and $\iota(X) = \infty$ otherwise. One can now apply ADMM, updating the variables in the order $Y$, $Y'$, $X'$, then $(Z,X)$. Notice that the subproblems for $Y$ and $(Z,X)$ now merely involve minimizing quadratic functions (with no constraints). The solution to the subproblem for $Y'$,
\begin{equation}\label{eq:yp}
\inf_{Y' \geq 0} \la W, -Y'\ra + \frac{\rho}{2}\|Y - Y'\|^2 = \inf_{Y' \geq 0} \left\|Y' - (Y + \frac{1}{\rho}W)\right\|^2,
\end{equation}
is obtained by setting the negative entries of $Y + \frac{1}{\rho}W$ to 0. An analogous statement holds for $X'$.

Unfortunately, while this splitting and order of variable updates yields easy subproblems, it does not satisfy all the assumptions we require in \Cref{aobj} (see also \Cref{tightness}). One reformulation which keeps all the subproblems easy \emph{and} satisfies our assumptions involves introducing slacks $X''$ and $Y''$ and penalizing them by a smooth function, as in
$$
(\text{NMF}3) \bump
\left\{
\begin{array}{rll}
\inf_{X,X',X'',Y,Y',Y'',Z}\limits & \iota(X') + \iota(Y') + \frac{1}{2}\|Z - B\|^2 + \frac{\mu}{2}\|X''\|^2 + \frac{\mu}{2}\|Y''\|^2\\
& Z = XY,\ X = X' + X'',\ Y = Y' + Y''. \end{array}
\right.$$
The variables can be updated in the order $Y$, $Y'$, $X$, $X'$, then $(Z,X'',Y'')$. It is straightforward to verify that the ADMM subproblems either involve minimizing a quadratic (with no constraints) or projecting onto the nonnegative orthant, as in \eqref{eq:yp}.

Next, we consider (DL). In \cite{MBPS2010JMLR}, a block coordinate descent (BCD) method is proposed for solving (DL), which requires an iterative subroutine for the Lasso \cite{T1996JRSS} problem  ($L_1$-regularized least squares regression). To obtain easy subproblems, we can formulate (DL) as
\begin{align*}
(\text{DL}2) \bump
\left\{
\begin{array}{rll}
\inf\limits_{X,Y,Z,X',Y'} &  \iota_S(X') + \|Y'\|_1 + \frac{\mu}{2}\|Z - B\|_2^2 \\
& Z = XY,\ Y = Y',\ X = X'.
\end{array}
\right.
\end{align*}
Notice that the Lasso has been replaced by soft thresholding, which has a closed-form solution. As with (NMF2), not all assumptions in \Cref{baseasm} are satisfied, so to retain easy subproblems and satisfy all assumptions, we introduce slack variables to obtain the problem
\begin{align*}
(\text{DL}3) \bump
\left\{
\begin{array}{rll}
\inf\limits_{X,X',X'',Y,Y',Y'',Z} &  \iota_S(X') + \|Y'\|_1 + \frac{\mu_Z}{2}\|Z - B\|_2^2  + \frac{\mu_X}{2}\|X''\|^2 + \frac{\mu_Y}{2}\|Y''\|_2^2\\
& Z = XY,\ Y = Y' + Y'',\ X = X' + X''.
\end{array}
\right.
\end{align*}

\subsection{Risk Parity Portfolio Selection}

As before, we can split the variables in a biaffine model to make each subproblem easy to solve. The projection onto the set of permissible weights $X$ has no closed-form solution, so let $X_B$ be the box $\{x \in \RR^n : a \leq x \leq b\}$, and $\iota_{X_B}$ its indicator function.  One can then solve:
$$
(\text{RP}2) \bump
\left\{
\begin{array}{rll}
\inf\limits_{x,x',y,z,z',z'',z'''} &  \iota_{X_B}(x') + \frac{\mu}{2}(\|z\|^2 + \|z'\|^2 + \|z''\|^2 + \|z'''\|^2) \\
& P(x \circ y) = z, \bump y = \Sigma x + z'\\
& x = x' + z'', \bump e_n^Tx = 1 + z'''.
\end{array}
\right.$$
The variables can be updated in the order $x$, $x'$, $y$, $(z,z',z'',z''')$. It is easy to see that every subproblem involves minimizing a quadratic function with no constraints, except for the update of $x'$, which consists of projection onto the box $X_B$ and can be evaluated in closed-form.

\section{Numerical Demonstration}\label{apdx:numeric}
In this section, we provide several numerical illustrations of ADMM applied to multiaffine problems. We consider an instance of the representation learning problem (\Cref{sec:dl,apdx:replearning}) known as \emph{spherical blind deconvolution}. In this setting, data $Y$ is generated by the \emph{circular convolution} $A \ast X$ of a \emph{kernel} $A$ and a sparse matrix $X$. An example of the convolution operation is shown in \Cref{fig:conv_example}.

\begin{figure}
\begin{minipage}{0.15\textwidth}
$\underbrace{\includegraphics[trim={5.3cm 9.6cm 4cm 9cm}, clip=true, scale=0.2]{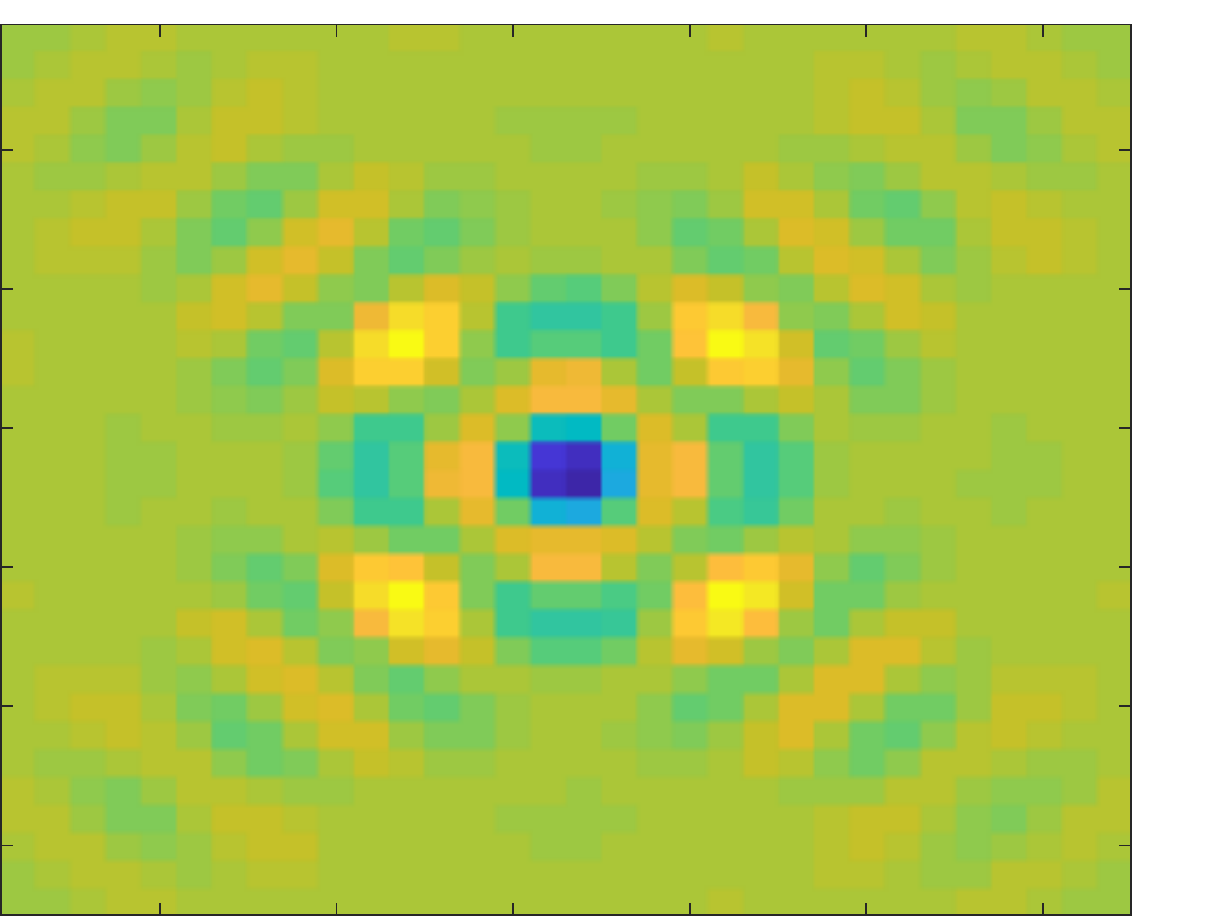}}_{A ~ (32 \times 32)}$
\end{minipage}
\begin{minipage}{0.02\textwidth}
$\ast$
\end{minipage}
\begin{minipage}{0.2\textwidth}
$\underbrace{\includegraphics[trim={5.3cm 9.6cm 4cm 9cm}, clip=true, scale=0.25]{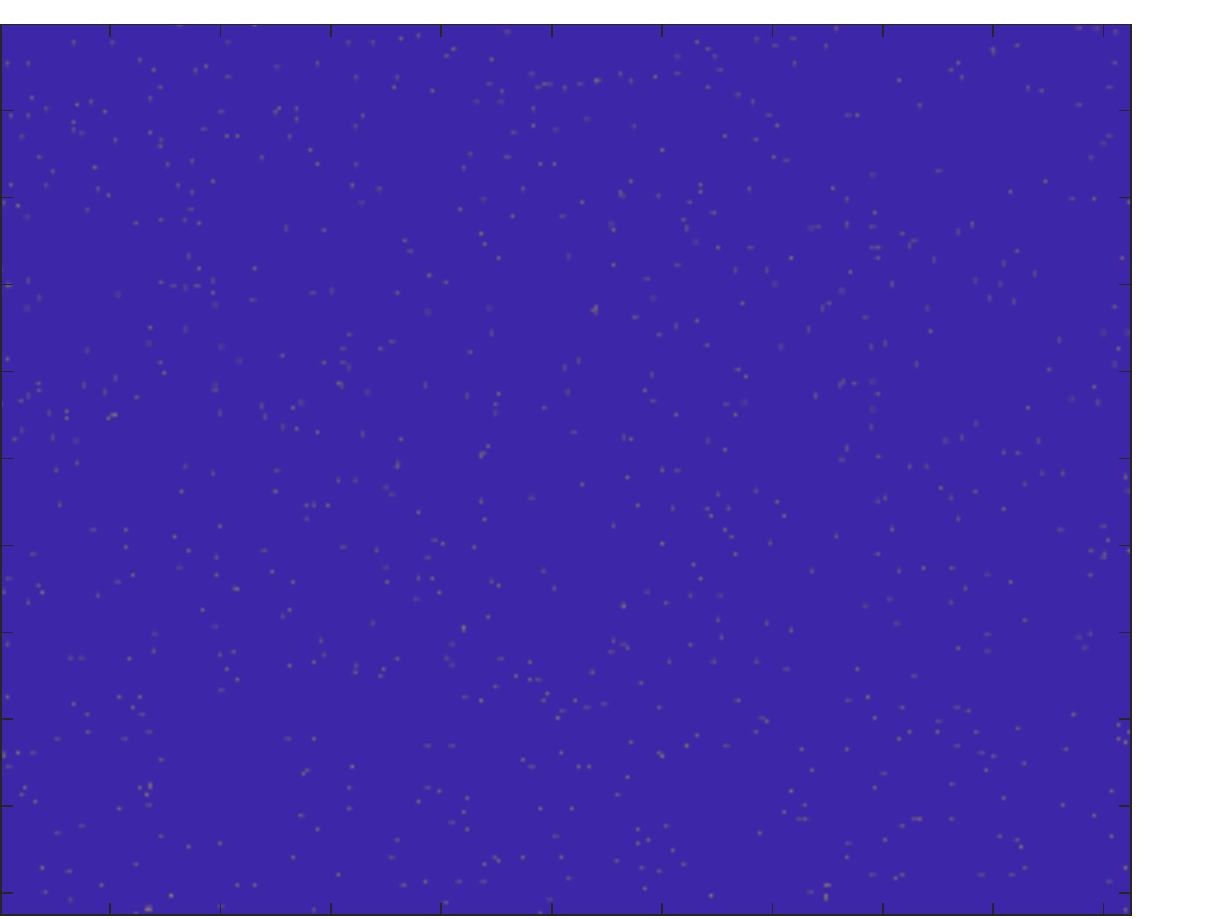}}_{X ~ (512 \times 512)}$
\end{minipage}
\begin{minipage}{0.02\textwidth}
$=$
\end{minipage}
\begin{minipage}{0.2\textwidth}
$\underbrace{\includegraphics[trim={5.3cm 9.6cm 4cm 9cm}, clip=true, scale=0.25]{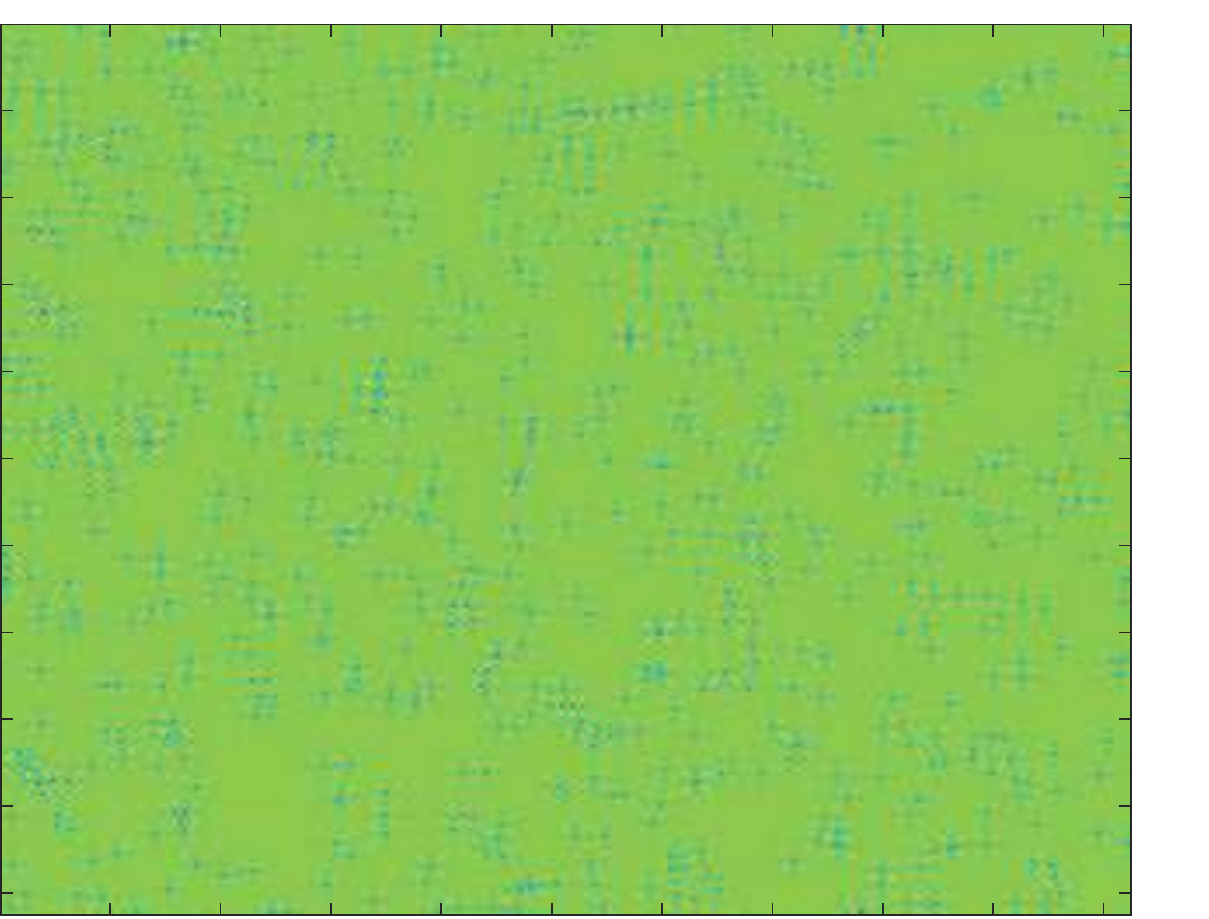}}_{Y ~ (512 \times 512)}$
\end{minipage}
\caption{An example of 2D circular convolution.}
\label{fig:conv_example}
\end{figure}

The data generating process for blind deconvolution can be modeled as
\begin{equation}\label{eq:sbd_generating}
Y = A \ast X + b\mathbf{1} + \xi,
\end{equation}
where $b\mathbf{1}$ is a bias term added to each entry, and $\xi$ denotes random noise. Recovering $A$ and $X$ in the ideal (noiseless) case can be formulated as the multiaffine optimization problem:
\begin{align*}
(\text{SBD0}) ~ 
\left\{
\begin{array}{ll}
\min\limits_{A,X,b} & \|X\|_1\\
&A \ast X + b\mathbf{1} = Y
\end{array}
\right.
\end{align*}
As discussed in \Cref{apdx:replearning}, this formulation does not satisfy \Cref{aobj}. We can introduce a slack variable $Z$ to ensure that \Cref{baseasm} is satisfied.
\begin{align*}
(\text{SBD1}) ~ 
\left\{
\begin{array}{ll}
\min\limits_{A,X,b,Z} &\|X\|_1 + \frac{\mu}{2}\|Z\|_F^2\\
&A \ast X + b\mathbf{1} - Z = Y
\end{array}
\right.
\end{align*}
The algorithm obtained by applying ADMM to (SBD0) will be referred to as \textsc{ADMM-exact}, and the one based on (SBD1) will be referred to as \textsc{ADMM-slack}.

\begin{figure}
\centering
\begin{subfigure}[t]{0.2\textwidth}
\includegraphics[trim=135pt 250pt 130pt 250pt, clip,scale=0.25]{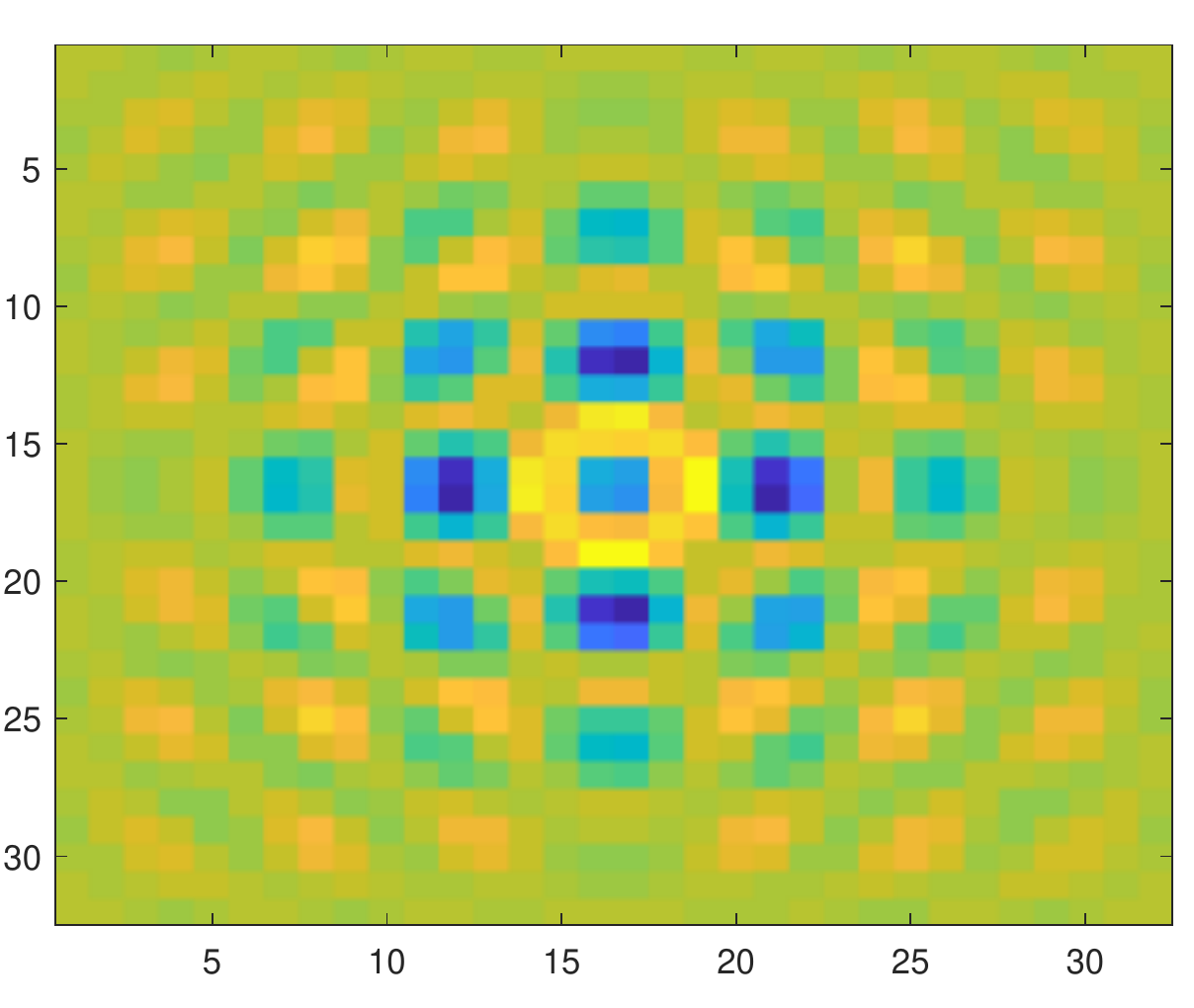}
\caption{True Kernel $A^\ast$}
\end{subfigure}
\begin{subfigure}[t]{0.2\textwidth}
\includegraphics[trim=125pt 250pt 130pt 250pt, clip,scale=0.25]{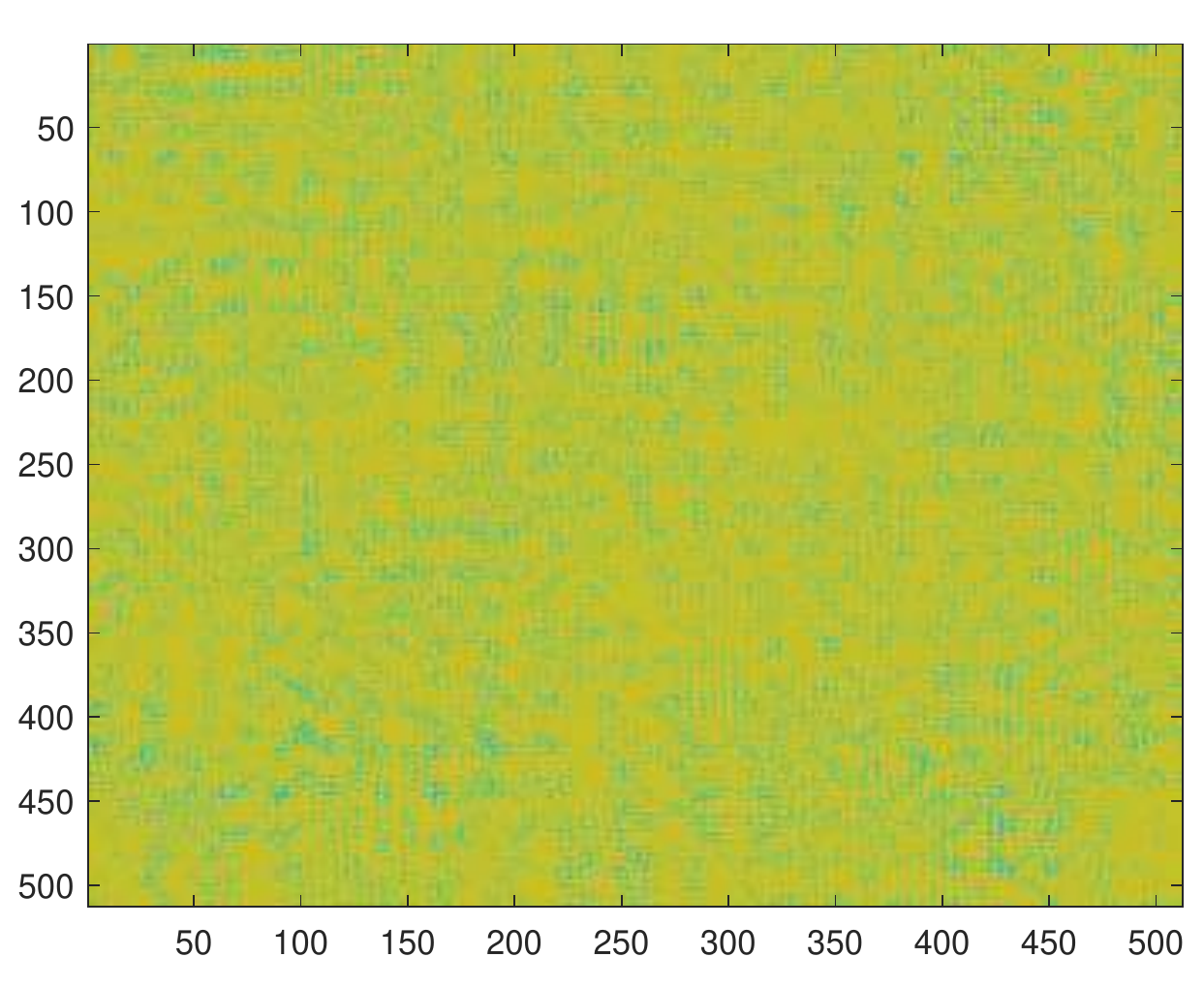}
\caption{Observations $Y$}
\end{subfigure}
\begin{subfigure}[t]{0.5\textwidth}
\includegraphics[trim=125pt 250pt 130pt 250pt, clip,scale=0.3]{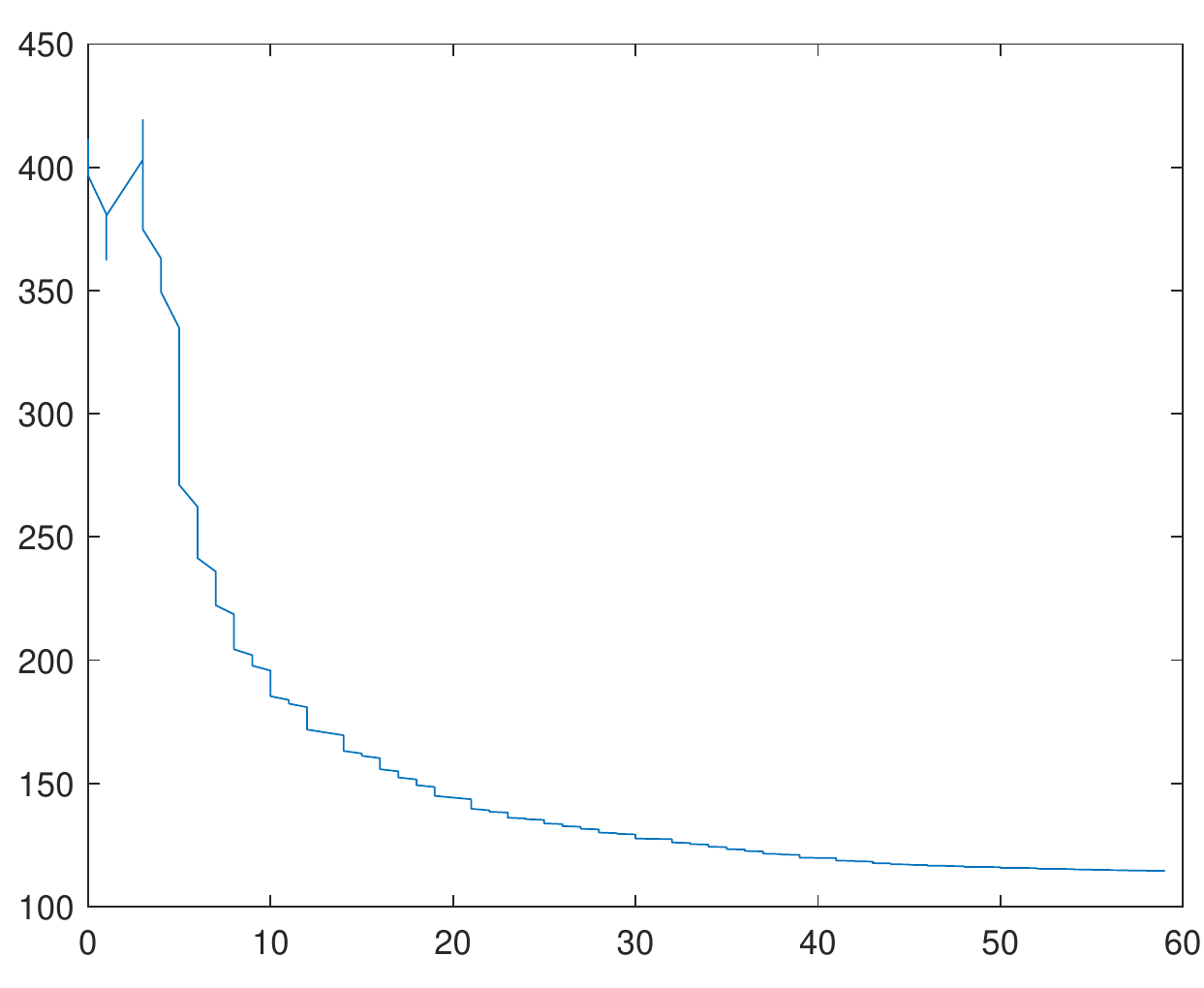}
\hfill
\includegraphics[trim=125pt 250pt 130pt 250pt, clip,scale=0.3]{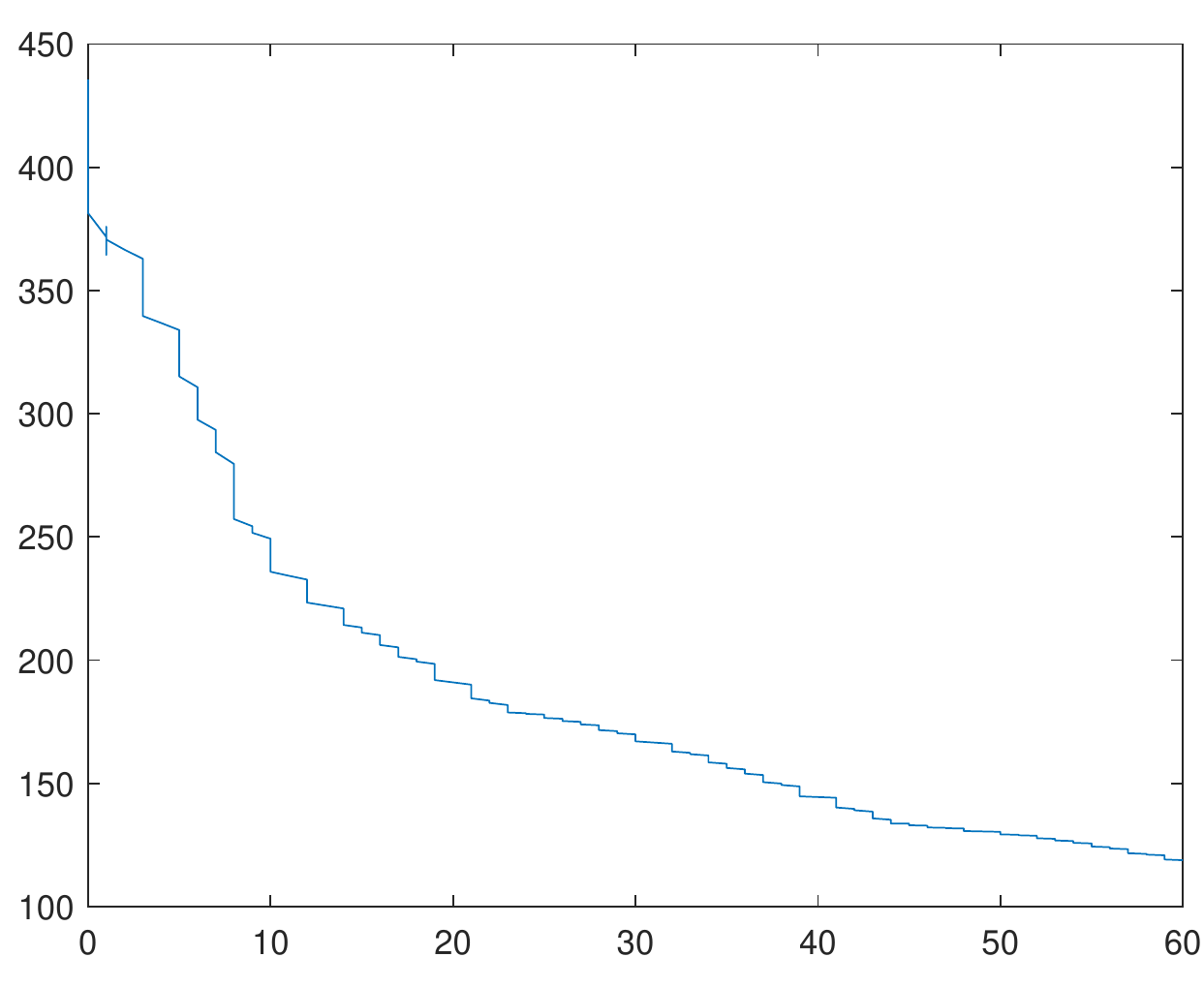}
\caption{Objective value over time (s). The left plot shows \textsc{ADMM-slack}, and the right shows \textsc{ADMM-exact}.}
\end{subfigure}

\bigskip

\centering
\begin{subfigure}[t]{1.0\textwidth}
\includegraphics[trim=125pt 250pt 130pt 250pt, clip,scale=0.23]{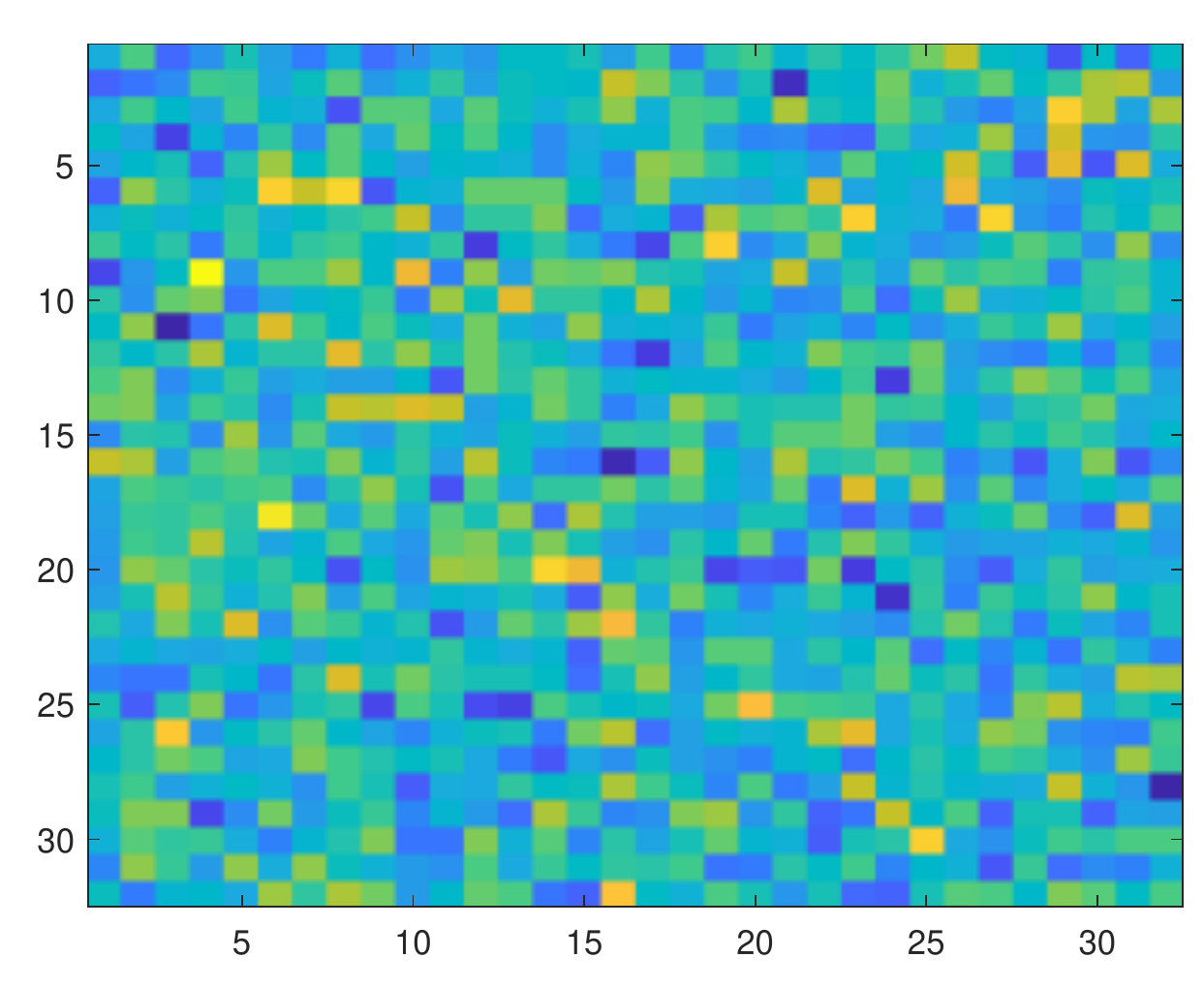}
\hfill
\includegraphics[trim=125pt 250pt 130pt 250pt, clip,scale=0.23]{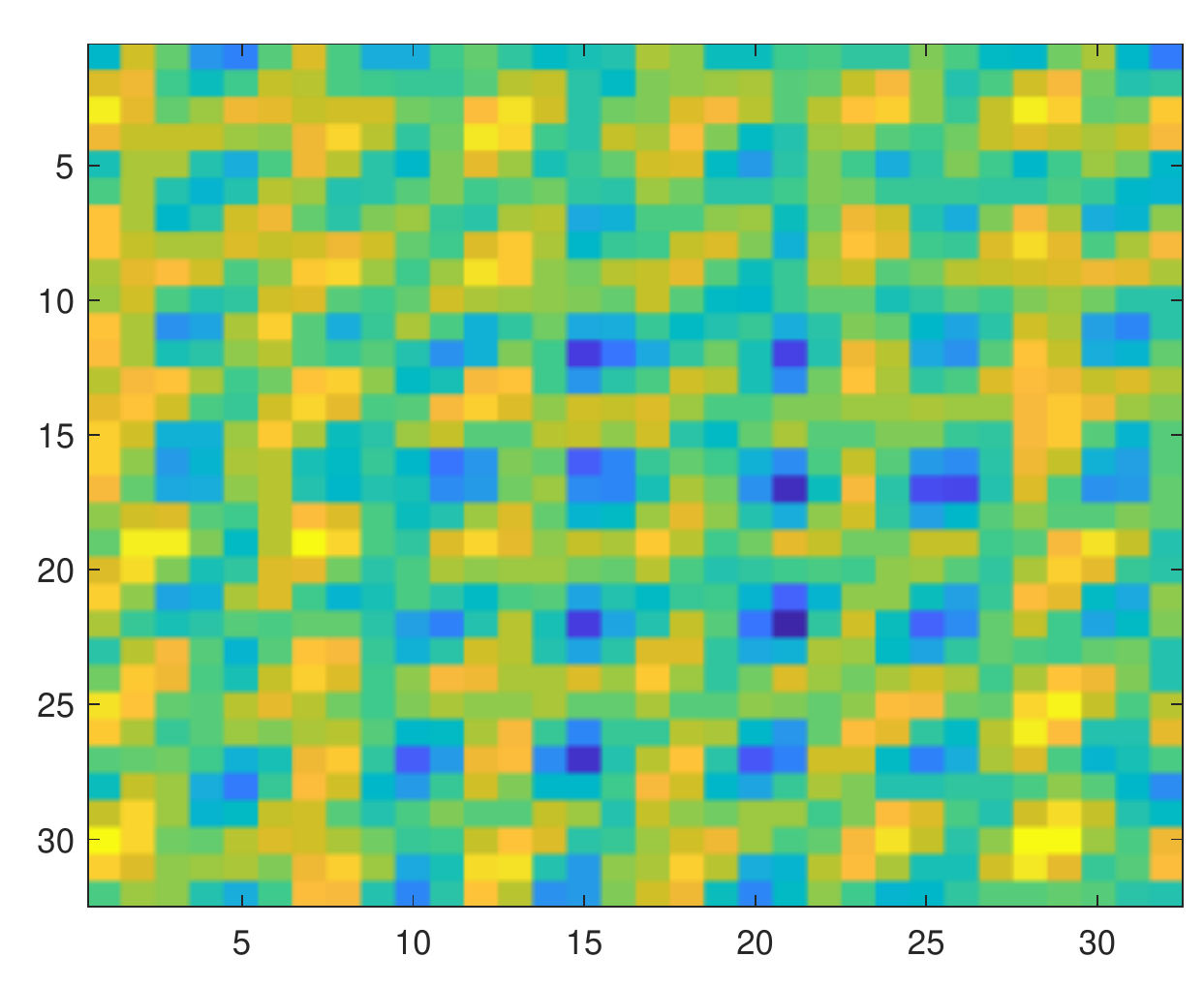}
\hfill
\includegraphics[trim=125pt 250pt 130pt 250pt, clip,scale=0.23]{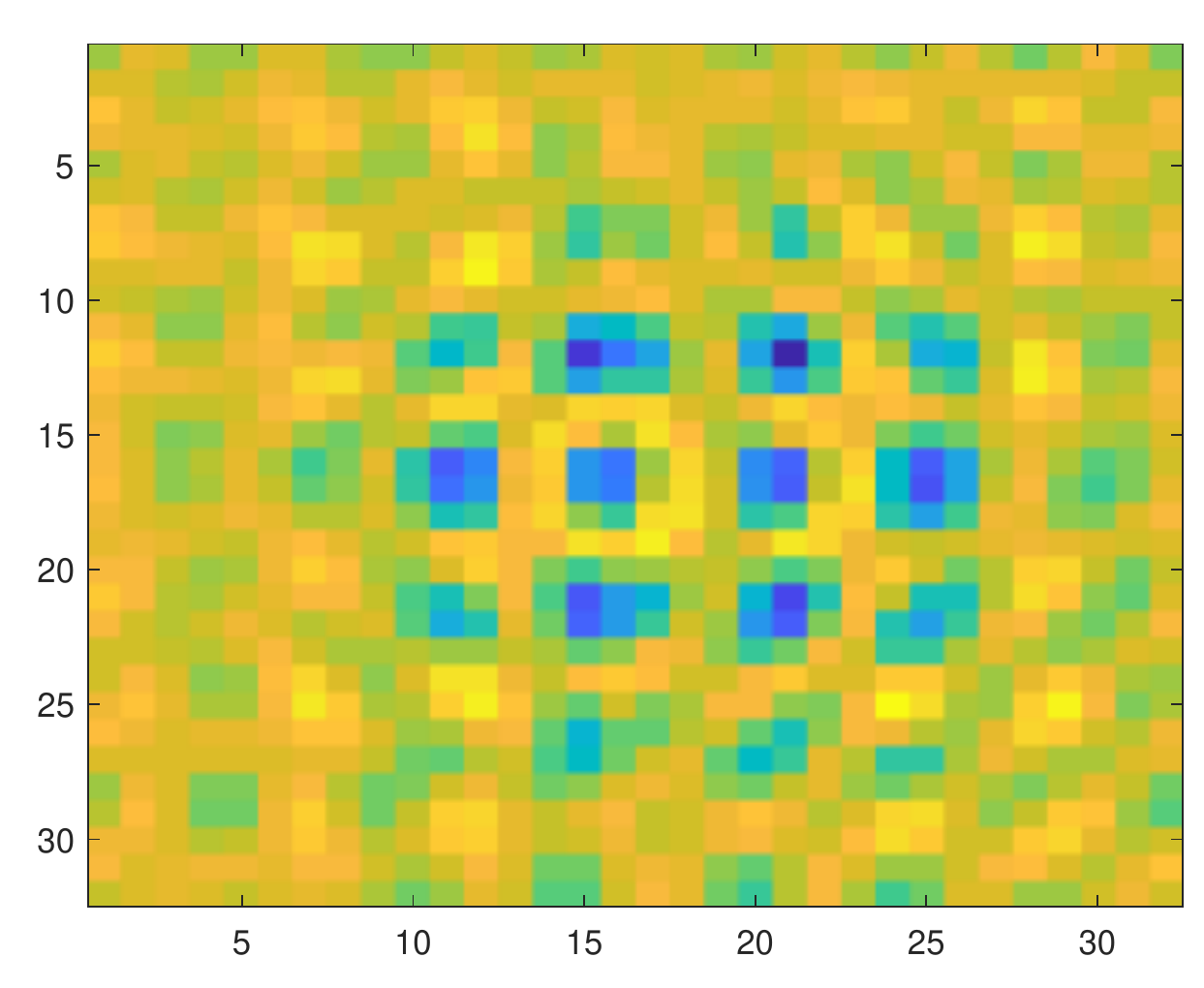}
\hfill
\includegraphics[trim=125pt 250pt 130pt 250pt, clip,scale=0.23]{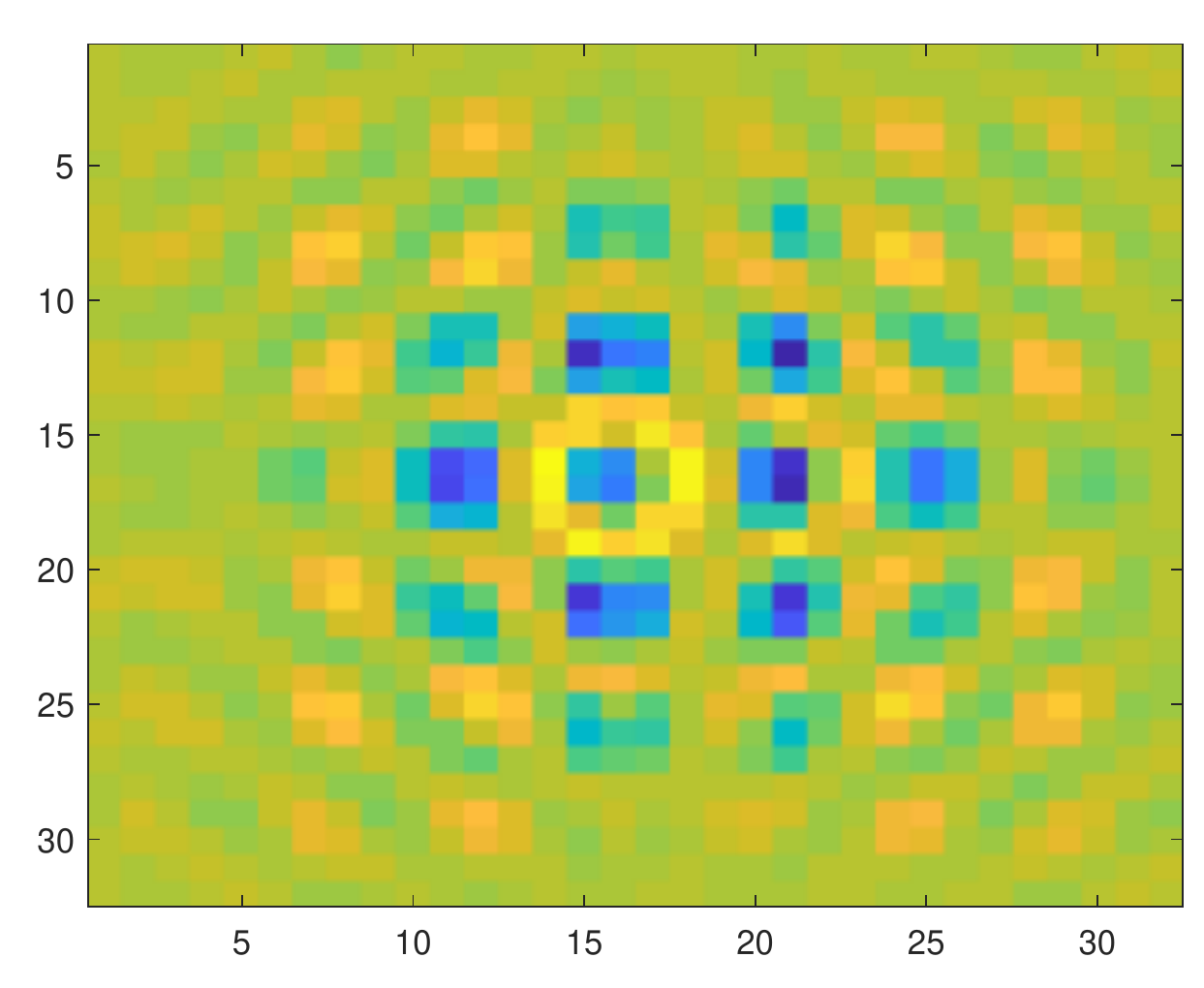}
\hfill
\includegraphics[trim=125pt 250pt 130pt 250pt, clip,scale=0.23]{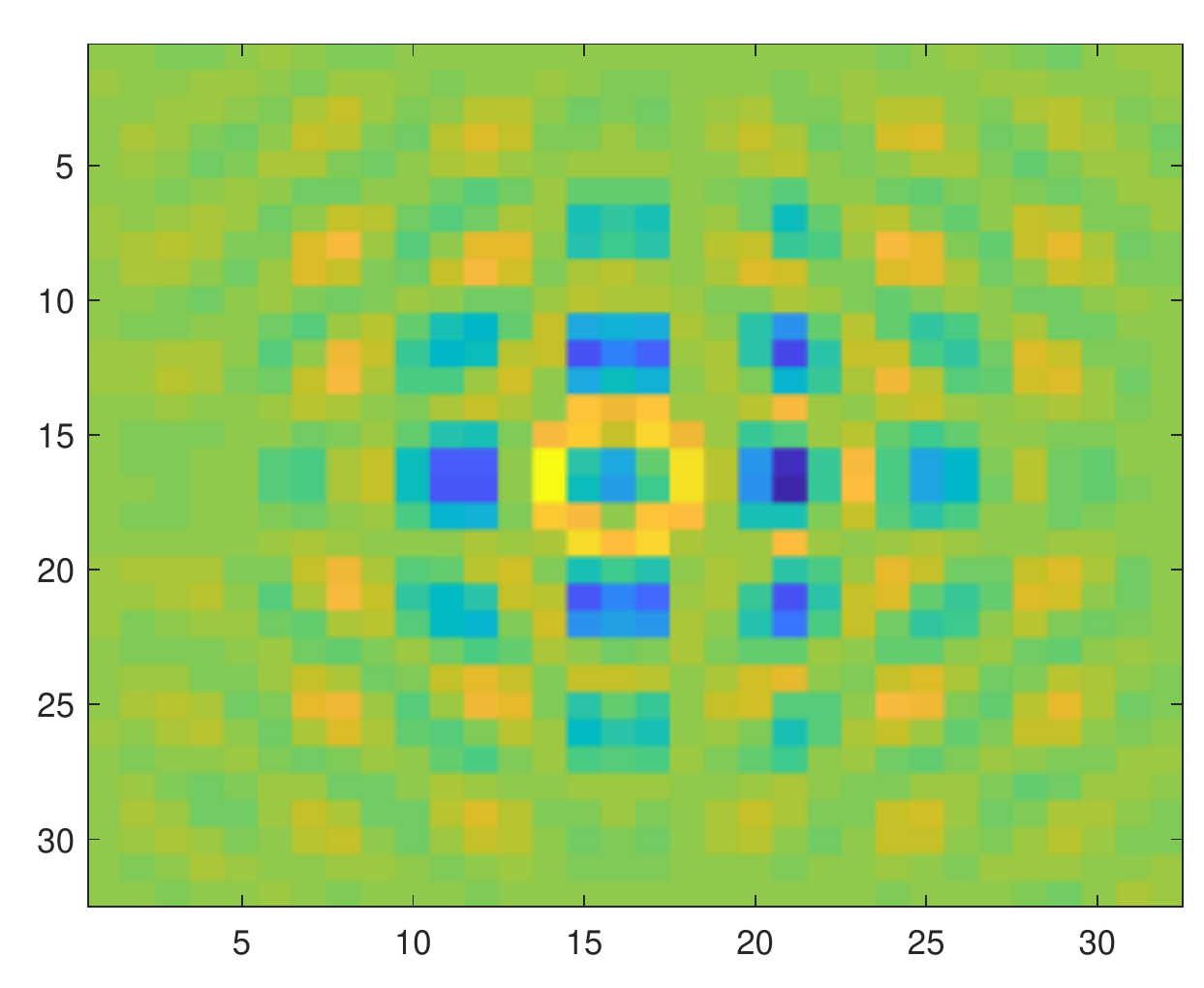}
\caption{Kernel matrices $A$ for \textsc{ADMM-slack} at iterations $k = 0, 10, 20, 50, 100$.}
\end{subfigure}

\bigskip
\centering
\begin{subfigure}[t]{1.0\textwidth}
\includegraphics[trim=125pt 250pt 130pt 250pt, clip,scale=0.23]{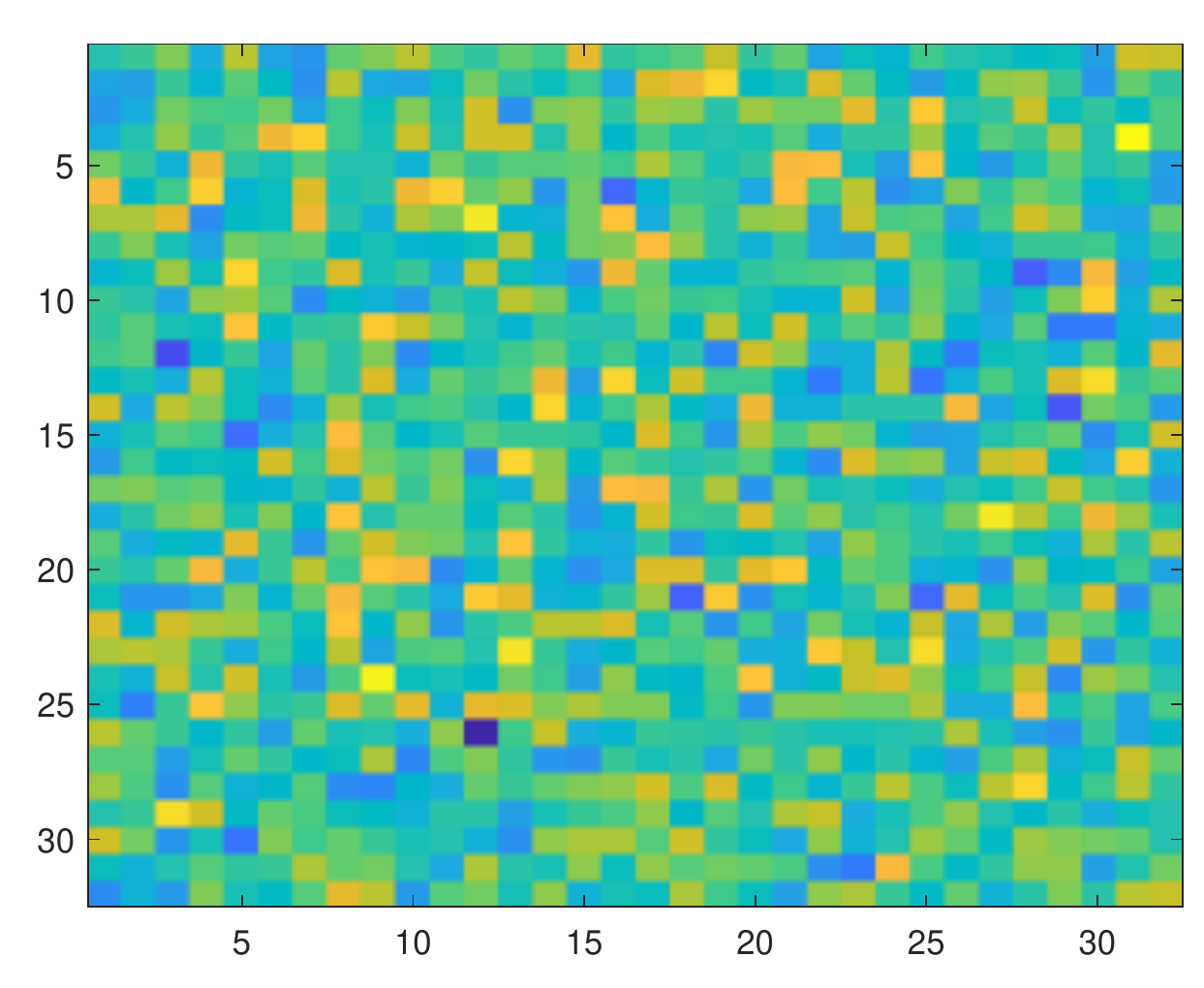}
\hfill
\includegraphics[trim=125pt 250pt 130pt 250pt, clip,scale=0.23]{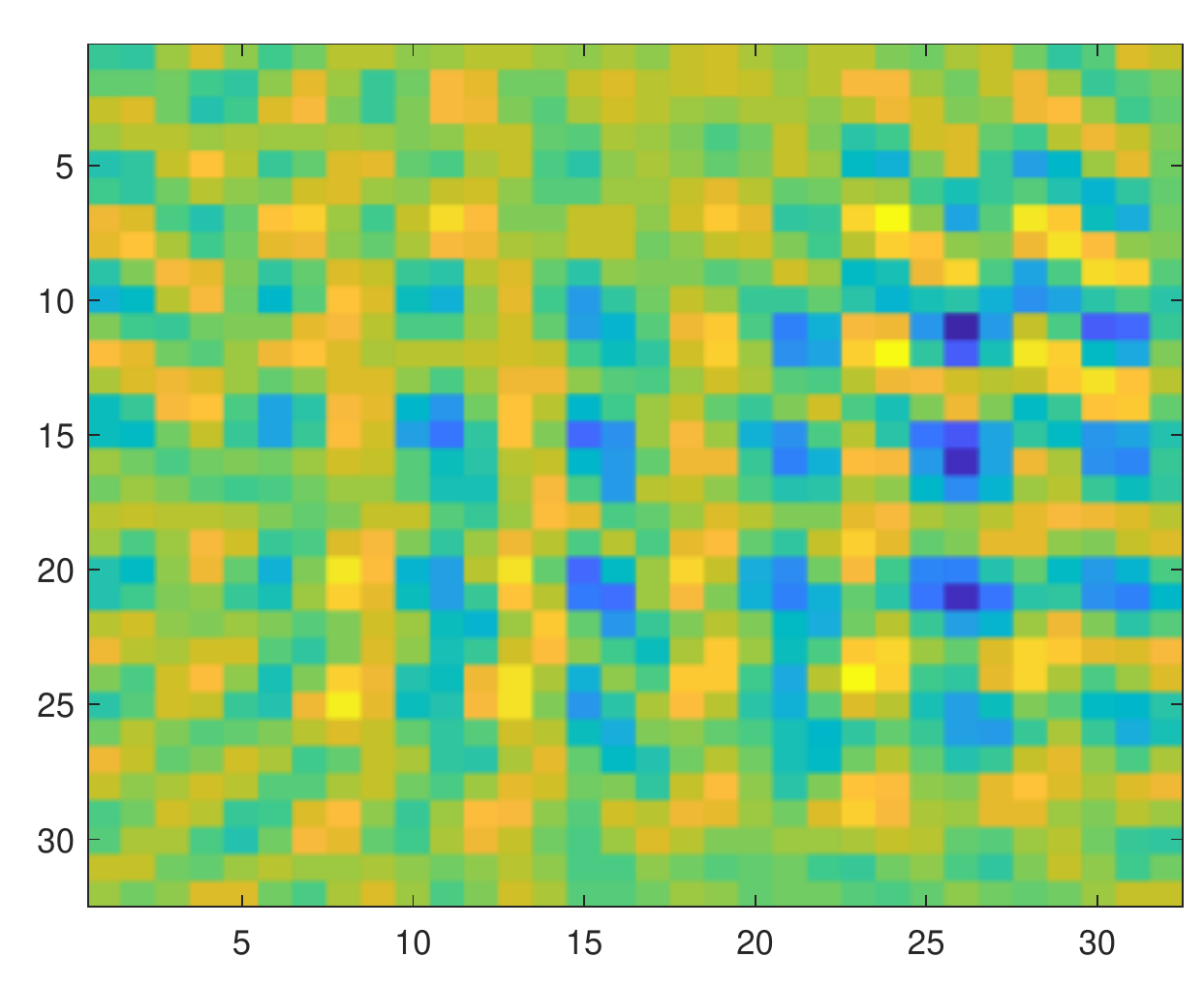}
\hfill
\includegraphics[trim=125pt 250pt 130pt 250pt, clip,scale=0.23]{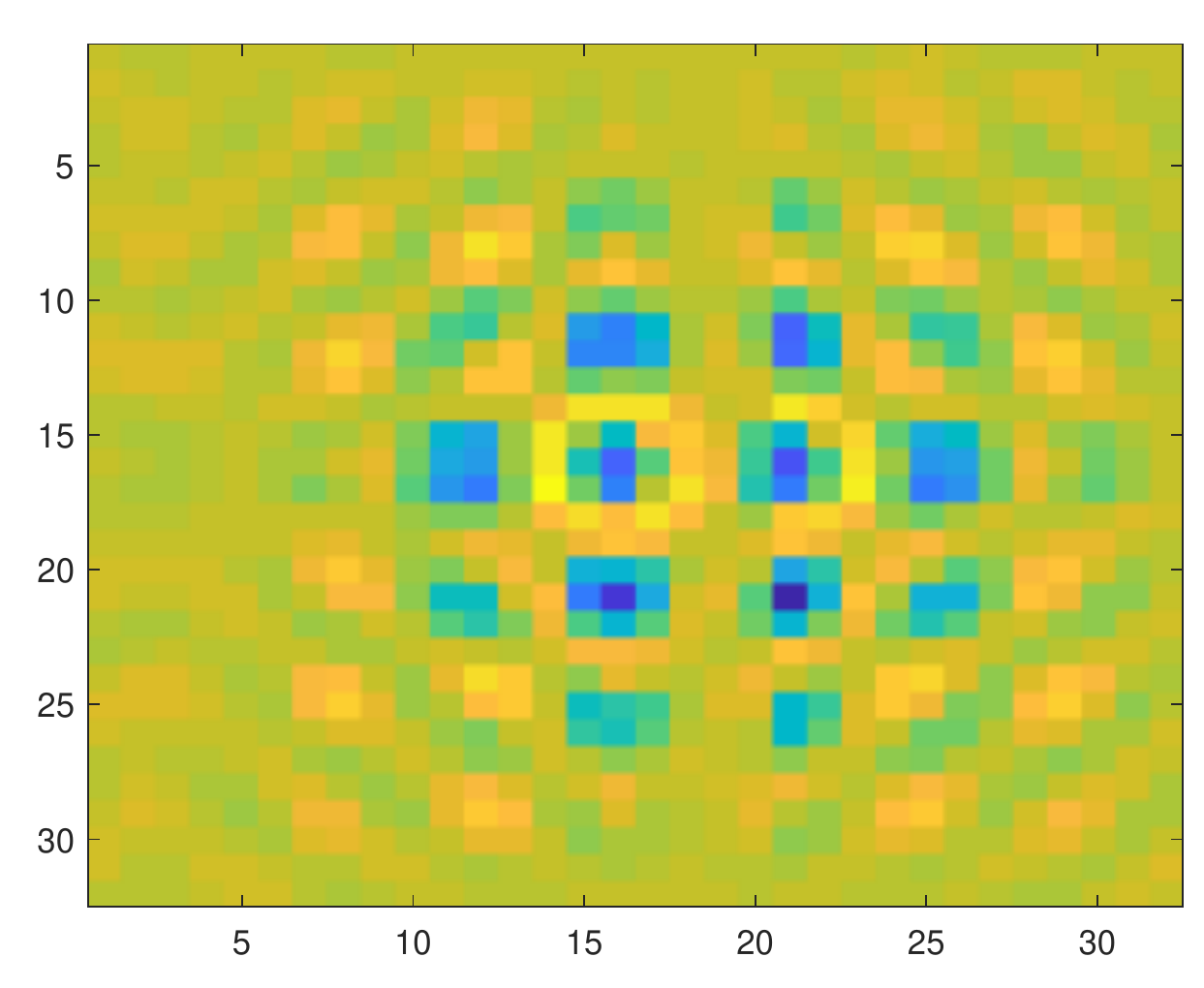}
\hfill
\includegraphics[trim=125pt 250pt 130pt 250pt, clip,scale=0.23]{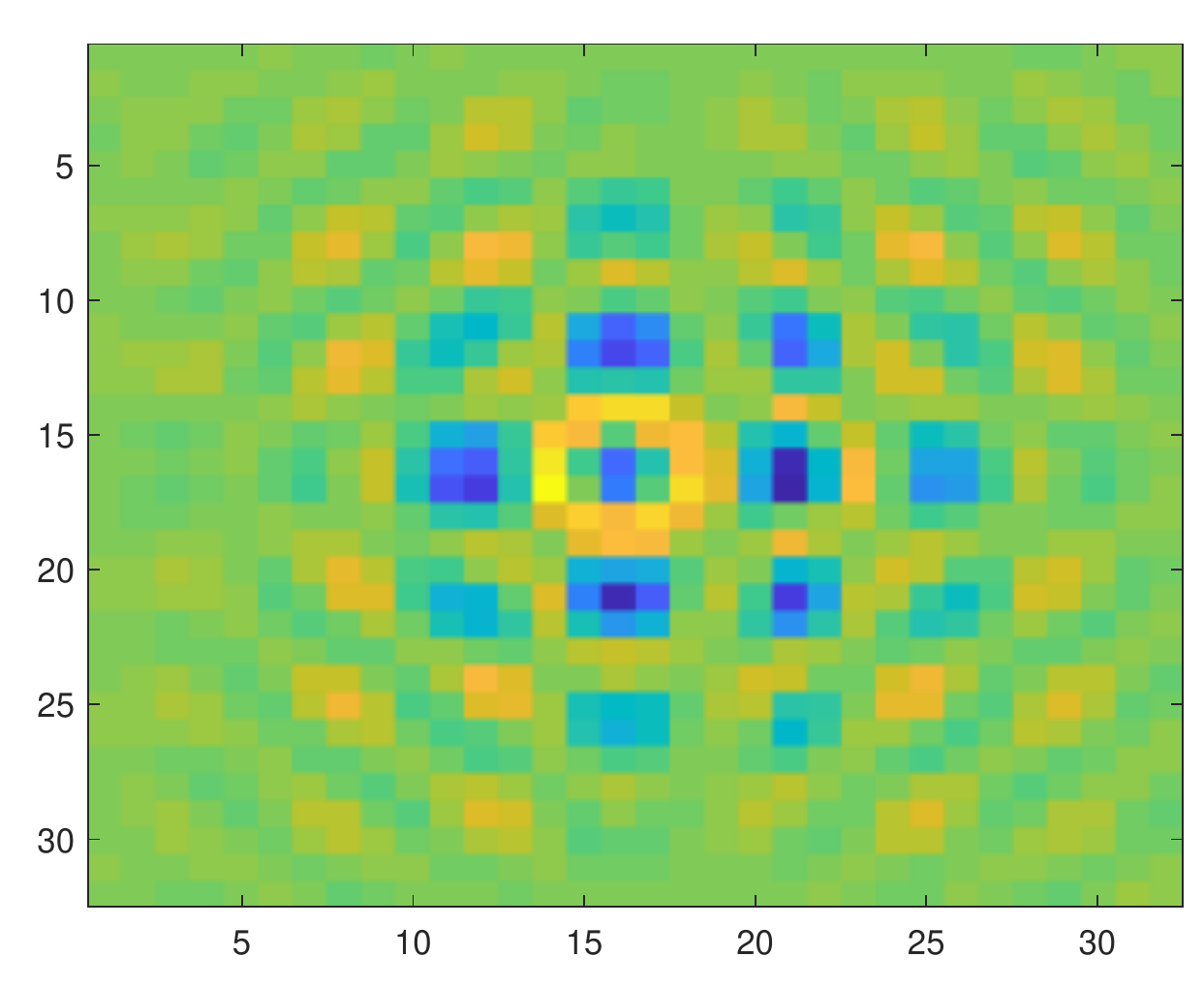}
\hfill
\includegraphics[trim=125pt 250pt 130pt 250pt, clip,scale=0.23]{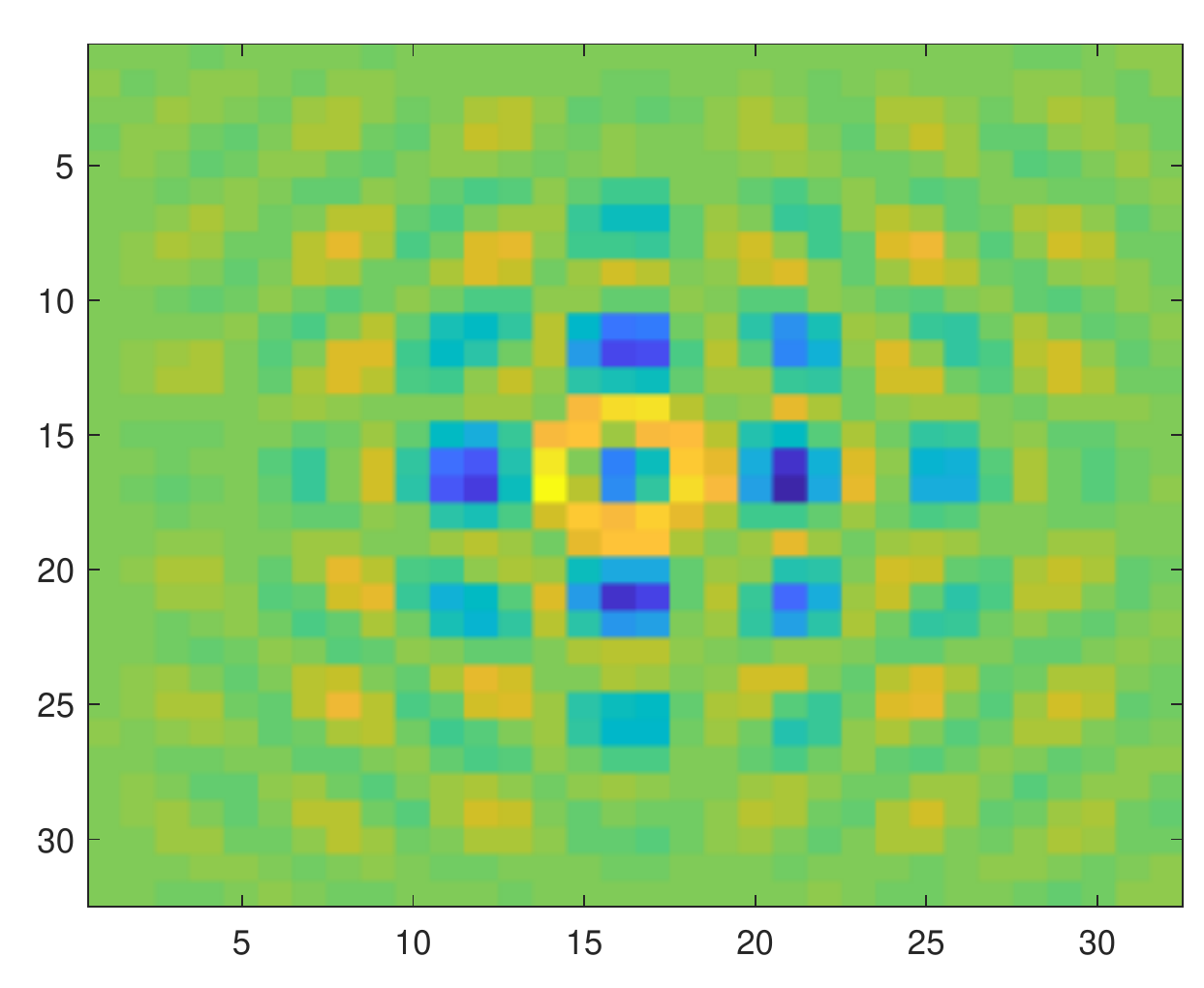}
\caption{Kernel matrices $A$ for \textsc{ADMM-exact} at iterations $k = 0, 10, 50, 100, 150$.}
\end{subfigure}
\caption{\textsc{ADMM-slack} and \textsc{ADMM-exact}, noiseless.}
\label{fig:exp_ker27_noiseless}
\end{figure}

We first compare \textsc{ADMM-exact} and \textsc{ADMM-slack} on noiseless data to assess the impact of the slack variable in (SBD1), which makes the formulation inexact. \Cref{fig:exp_ker27_noiseless,fig:exp_ker5_noiseless,fig:exp_ker23_noiseless} show the performance of \textsc{ADMM-exact} and \textsc{ADMM-slack} on synthetic problems generated according to \eqref{eq:sbd_generating} with $\xi = 0$. The top row of each figure shows the true kernel, the observation data $Y$, and the objective values over 60 seconds for \textsc{ADMM-slack} and \textsc{ADMM-exact}. To measure both the sparsity and accuracy, the objective function reported is given by $\frac{1}{10}\|X\|_1 + \frac{1}{2}\|Y - A \ast X - b\mathbf{1}\|_F^2$. The next two rows shows the kernel matrix $A$ recovered by each method at various stages of progress. In the noiseless case, the performance of \textsc{ADMM-exact} and \textsc{ADMM-slack} appears to be similar. We also observe empirically that the convergence is sublinear, which is consistent with the general rate $O(\frac{1}{k})$ for ADMM.

\begin{figure}
\centering
\begin{subfigure}[t]{0.2\textwidth}
\includegraphics[trim=135pt 250pt 130pt 250pt, clip,scale=0.25]{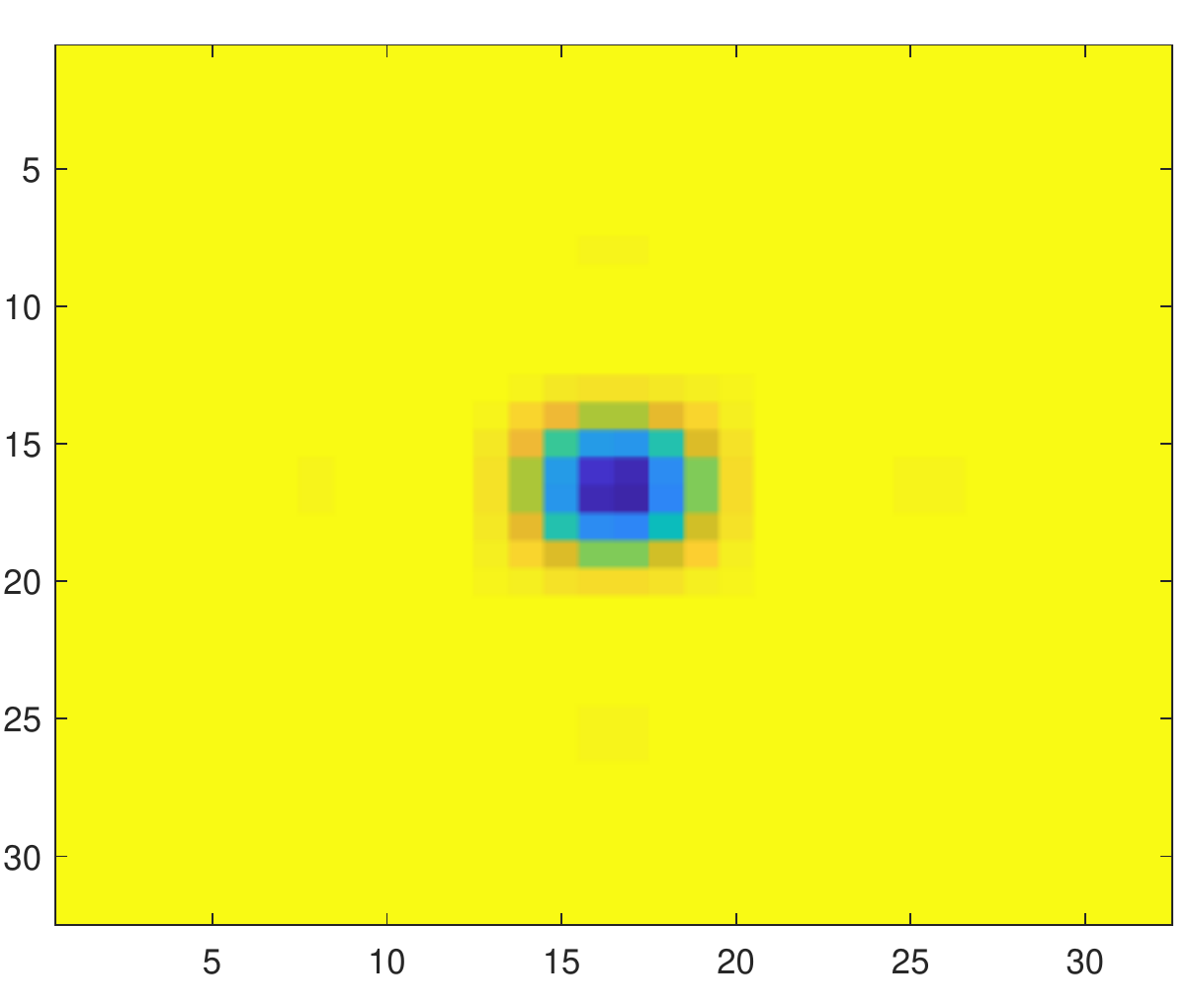}
\caption{True Kernel $A^\ast$}
\end{subfigure}
\begin{subfigure}[t]{0.2\textwidth}
\includegraphics[trim=125pt 250pt 130pt 250pt, clip,scale=0.25]{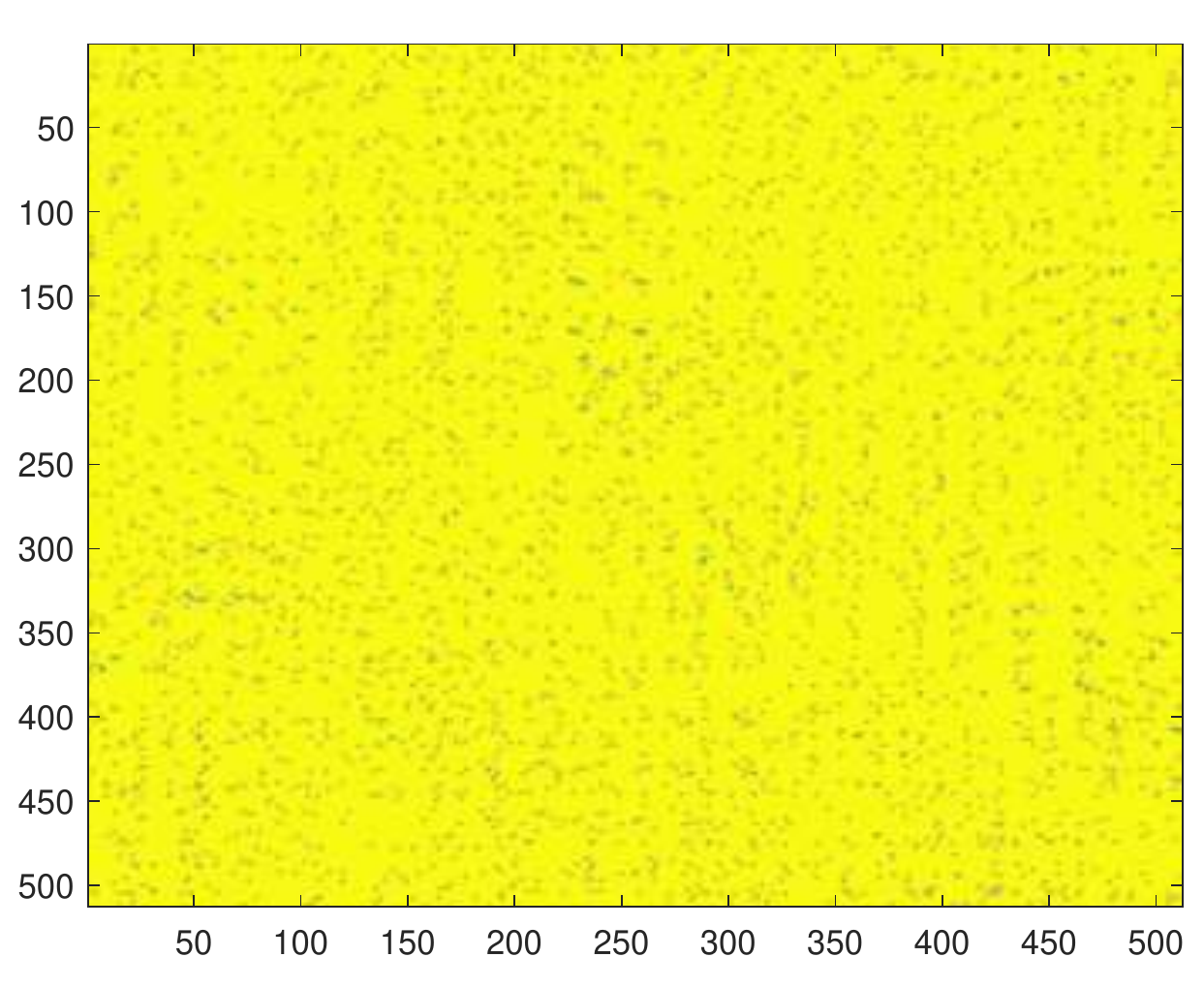}
\caption{Observations $Y$}
\end{subfigure}
\begin{subfigure}[t]{0.5\textwidth}
\includegraphics[trim=125pt 250pt 130pt 250pt, clip,scale=0.3]{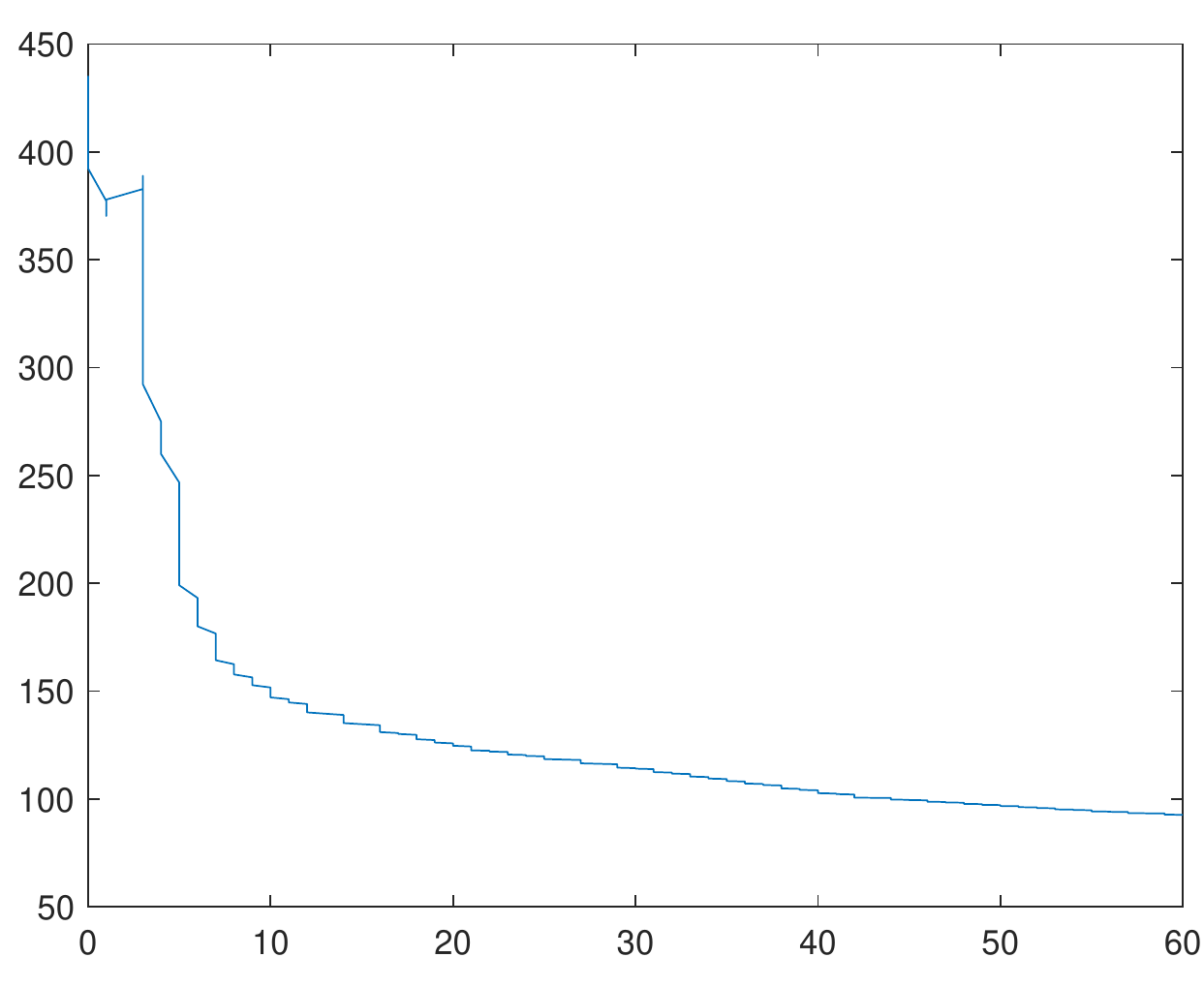}
\hfill
\includegraphics[trim=125pt 250pt 130pt 250pt, clip,scale=0.3]{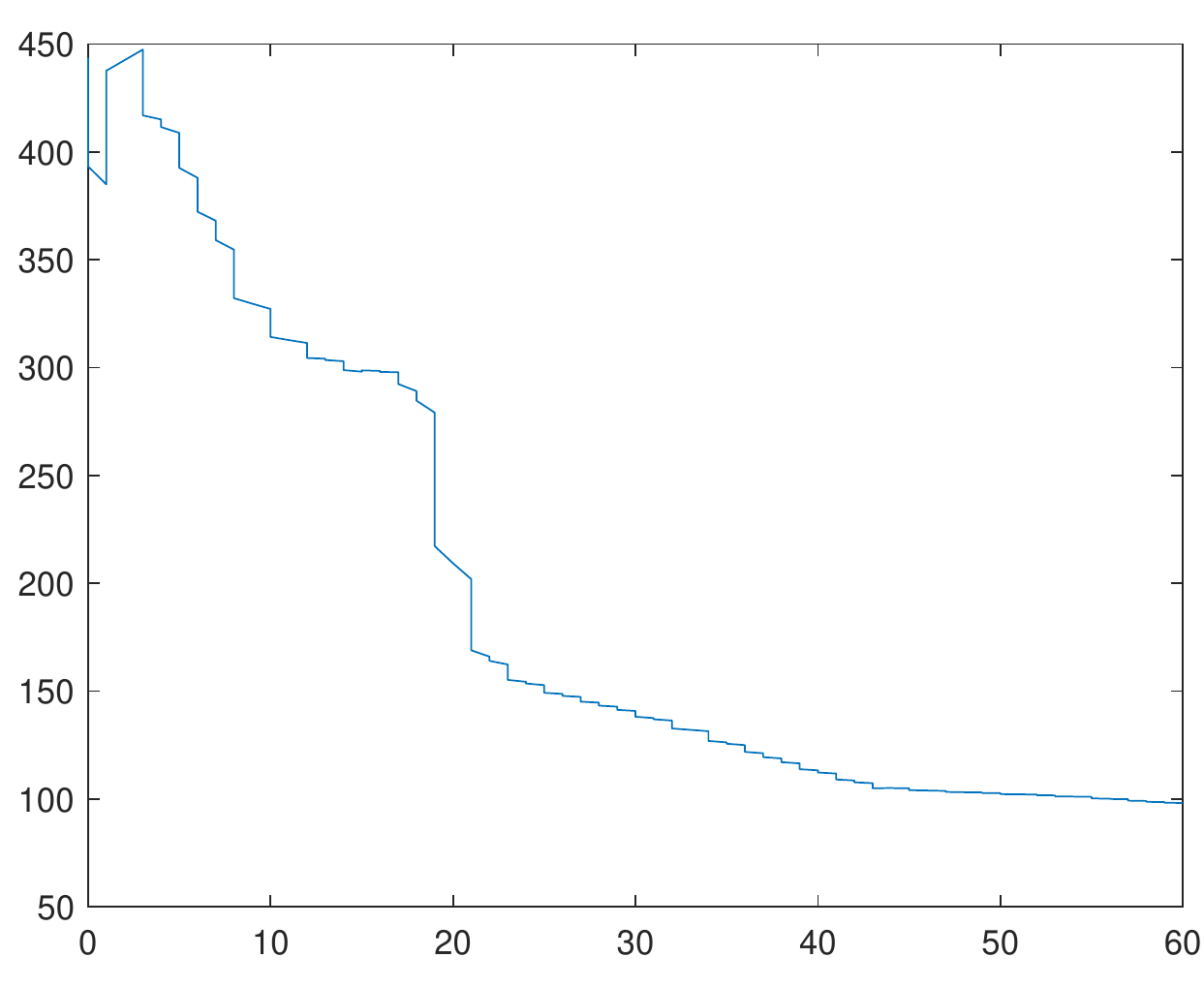}
\caption{Objective value over time (s). The left plot shows \textsc{ADMM-slack}, and the right shows \textsc{ADMM-exact}.}
\end{subfigure}

\bigskip

\centering
\begin{subfigure}[t]{1.0\textwidth}
\includegraphics[trim=125pt 250pt 130pt 250pt, clip,scale=0.23]{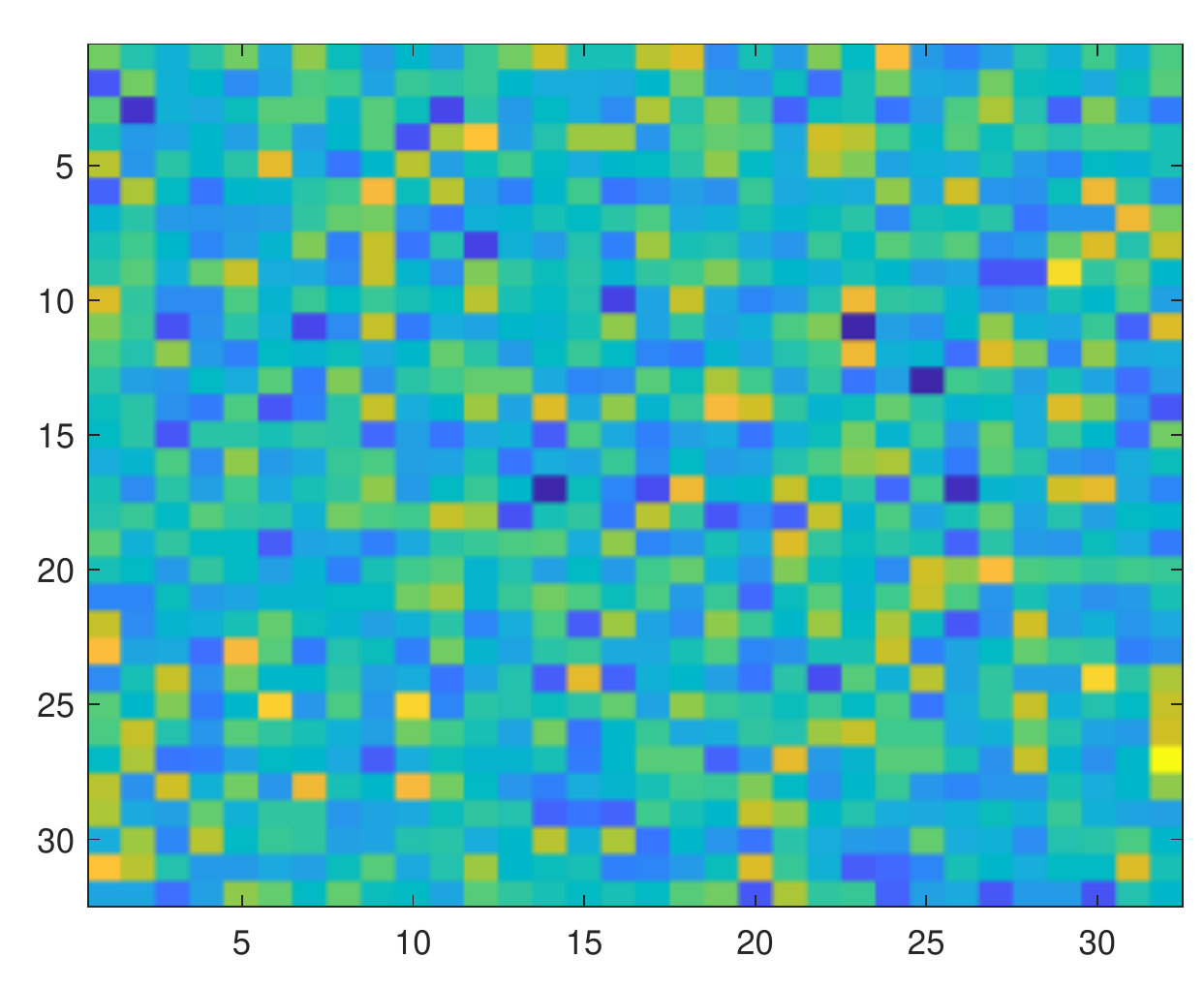}
\hfill
\includegraphics[trim=125pt 250pt 130pt 250pt, clip,scale=0.23]{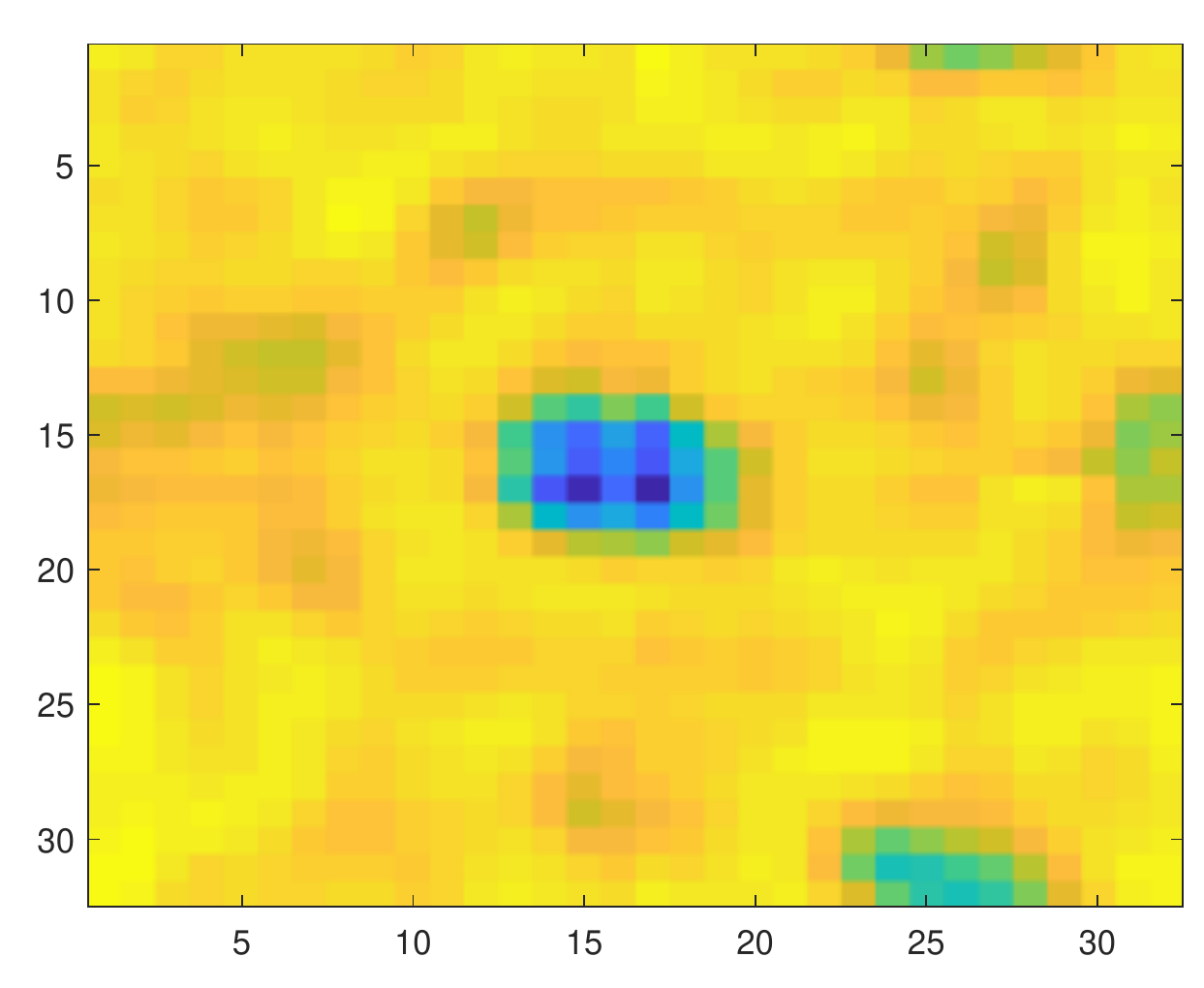}
\hfill
\includegraphics[trim=125pt 250pt 130pt 250pt, clip,scale=0.23]{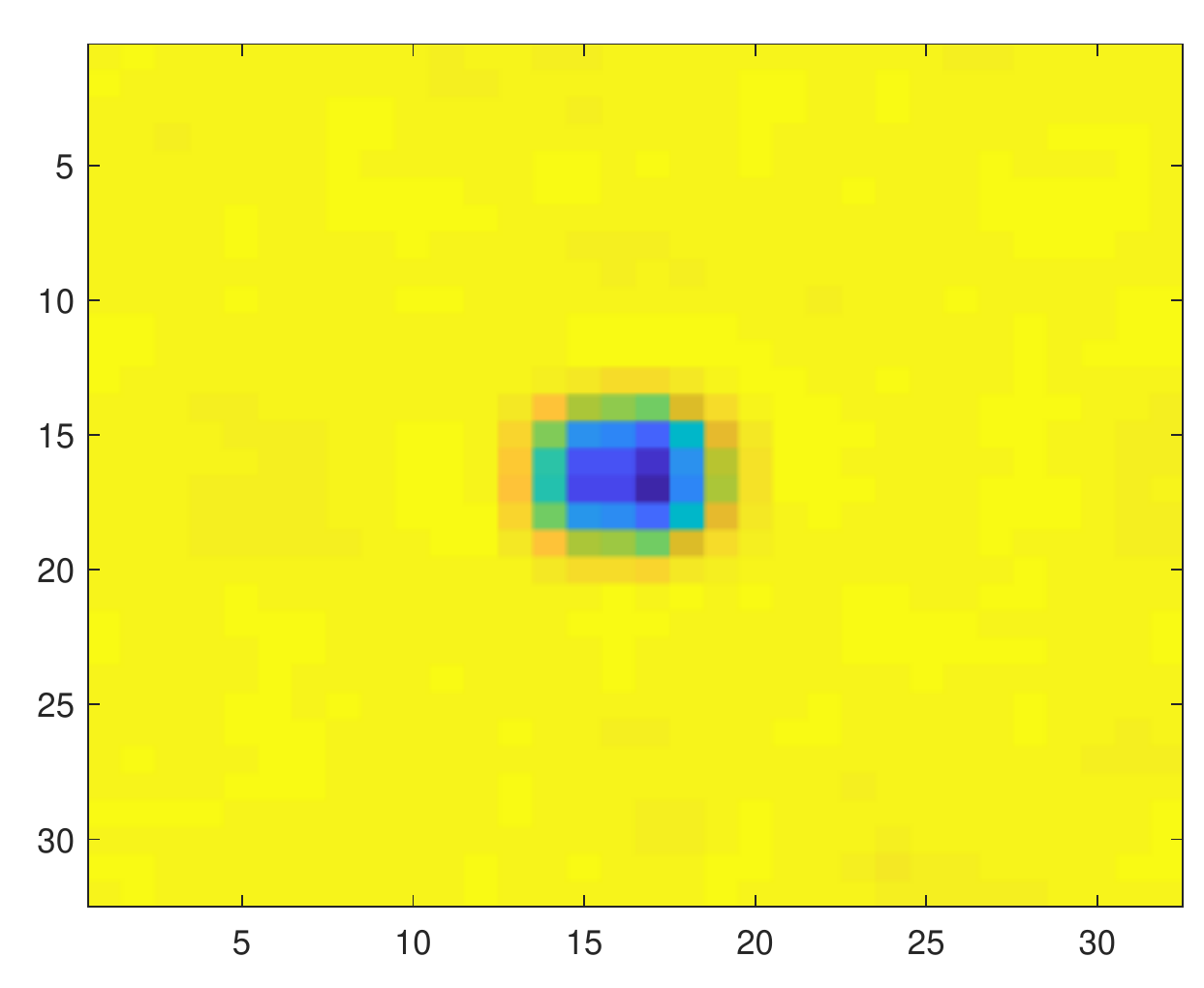}
\hfill
\includegraphics[trim=125pt 250pt 130pt 250pt, clip,scale=0.23]{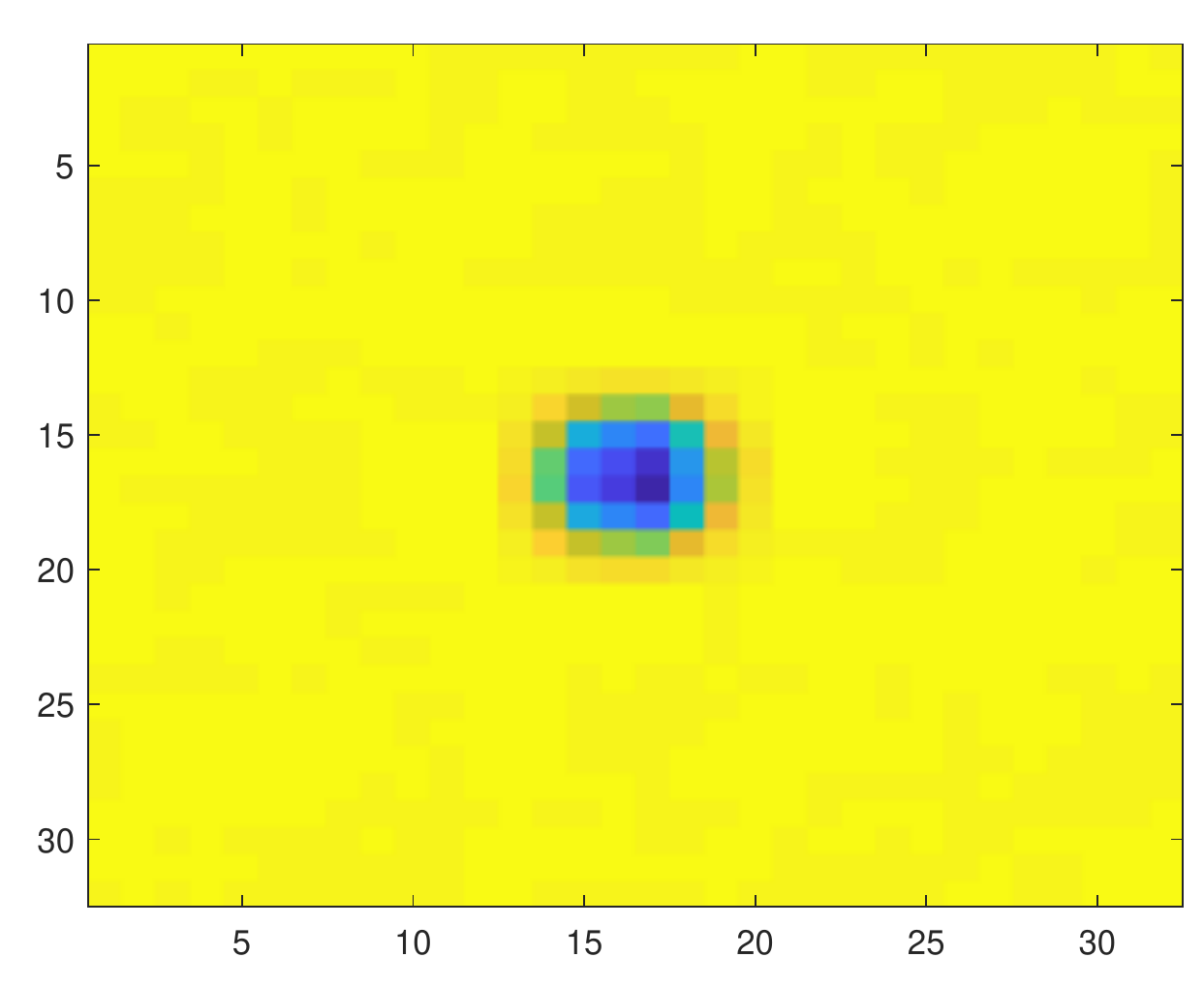}
\hfill
\includegraphics[trim=125pt 250pt 130pt 250pt, clip,scale=0.23]{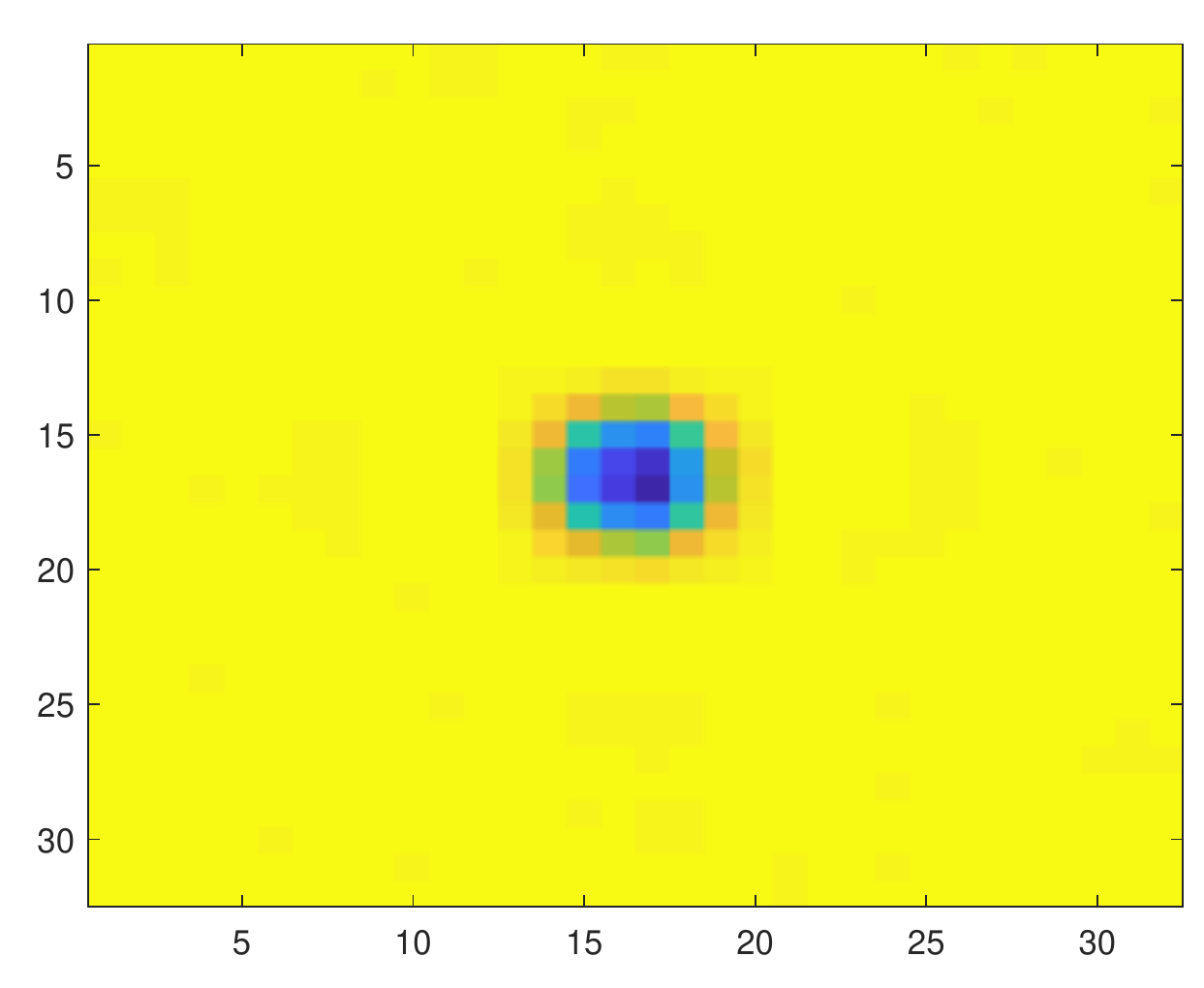}
\caption{Kernel matrices $A$ for \textsc{ADMM-slack} at iterations $k = 0, 10, 20, 30, 50$.}
\end{subfigure}

\bigskip
\centering
\begin{subfigure}[t]{1.0\textwidth}
\includegraphics[trim=125pt 250pt 130pt 250pt, clip,scale=0.23]{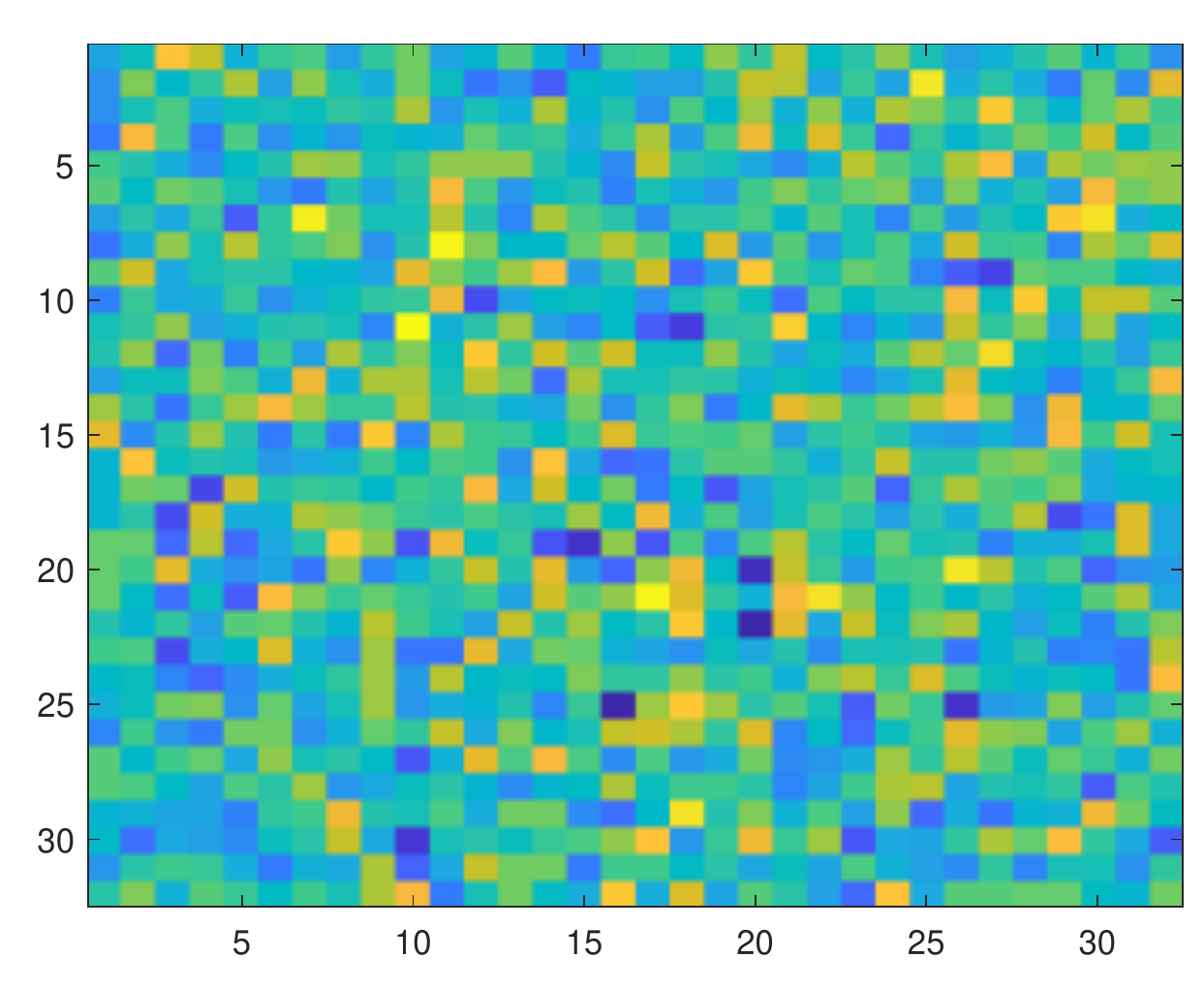}
\hfill
\includegraphics[trim=125pt 250pt 130pt 250pt, clip,scale=0.23]{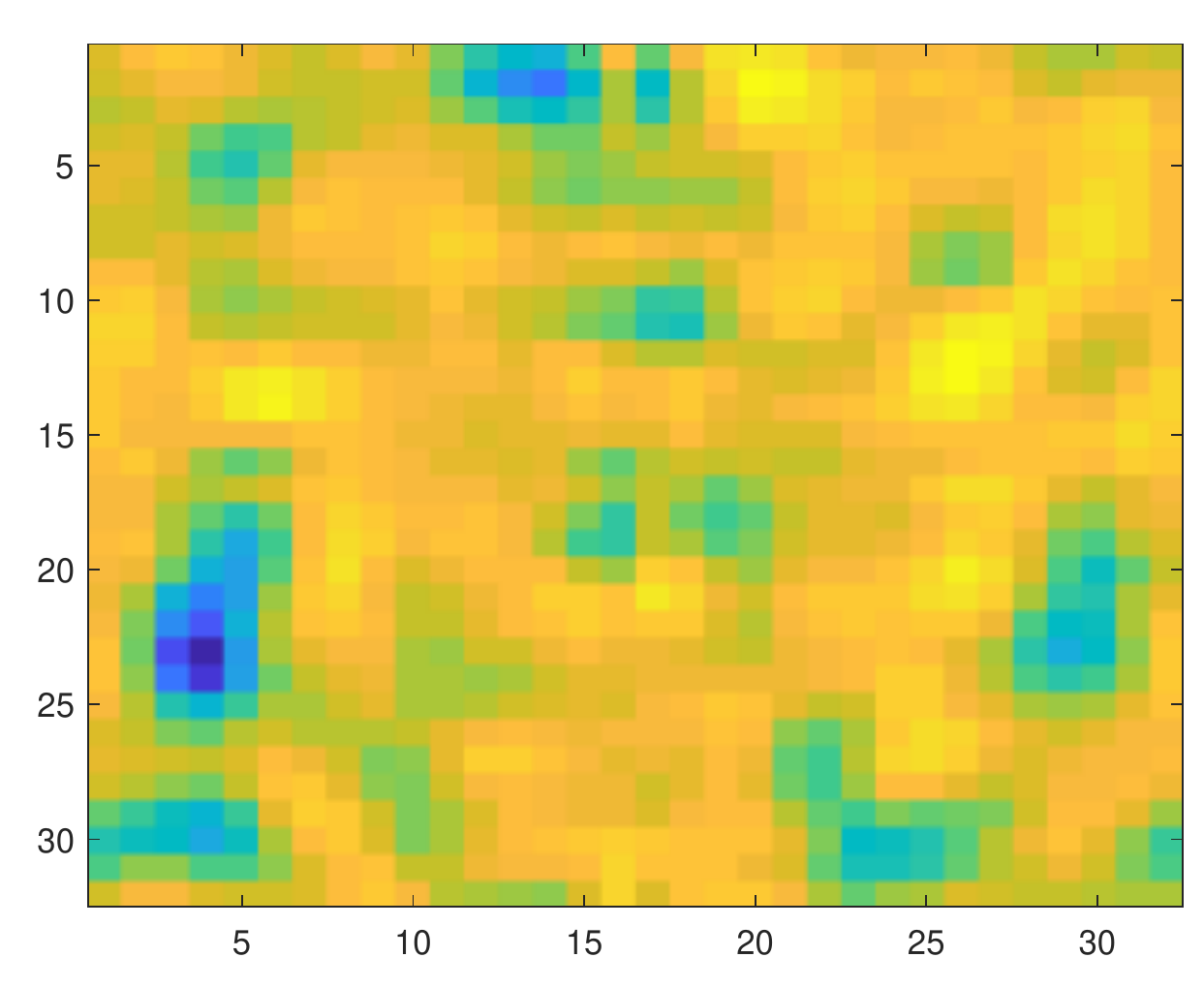}
\hfill
\includegraphics[trim=125pt 250pt 130pt 250pt, clip,scale=0.23]{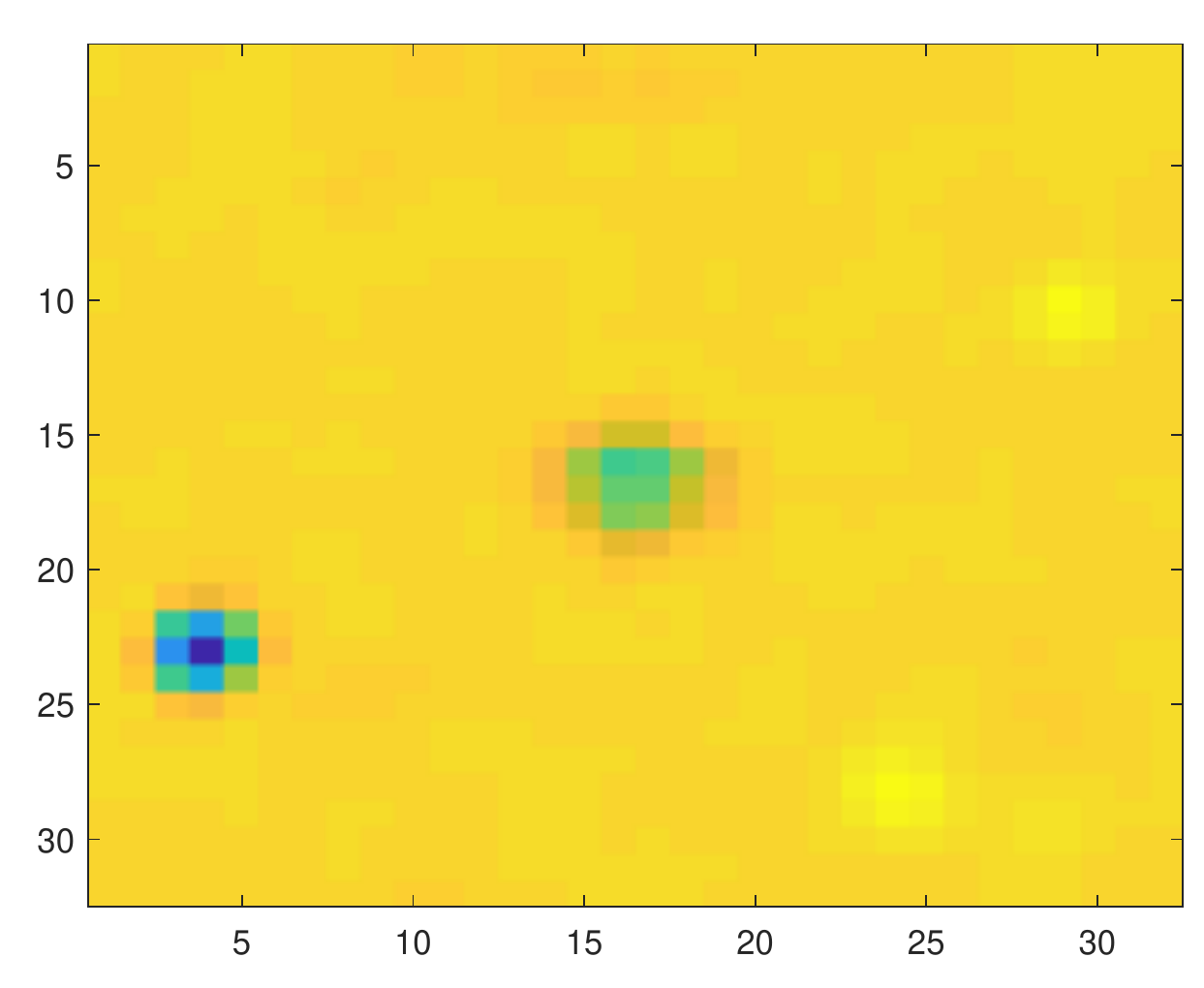}
\hfill
\includegraphics[trim=125pt 250pt 130pt 250pt, clip,scale=0.23]{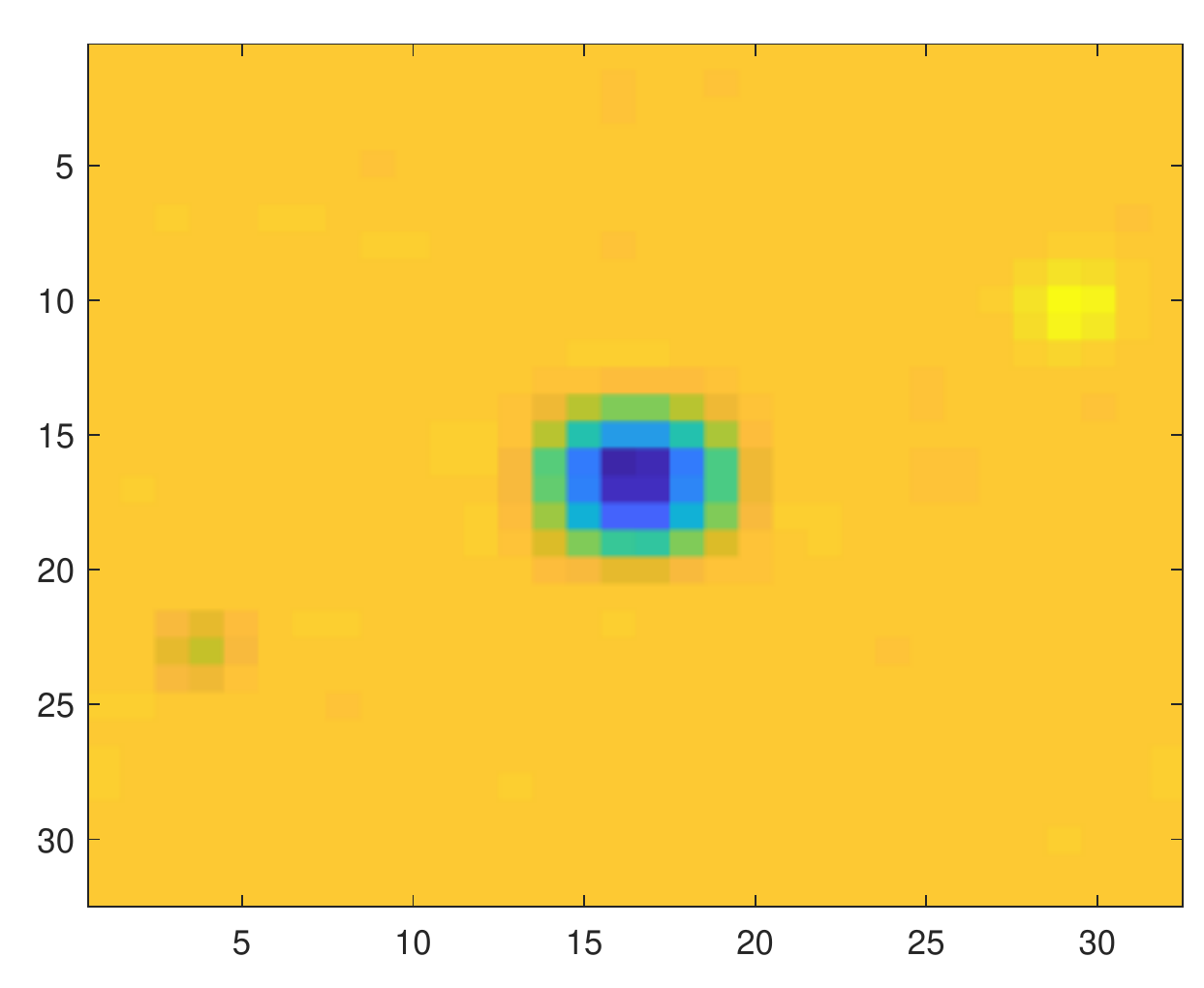}
\hfill
\includegraphics[trim=125pt 250pt 130pt 250pt, clip,scale=0.23]{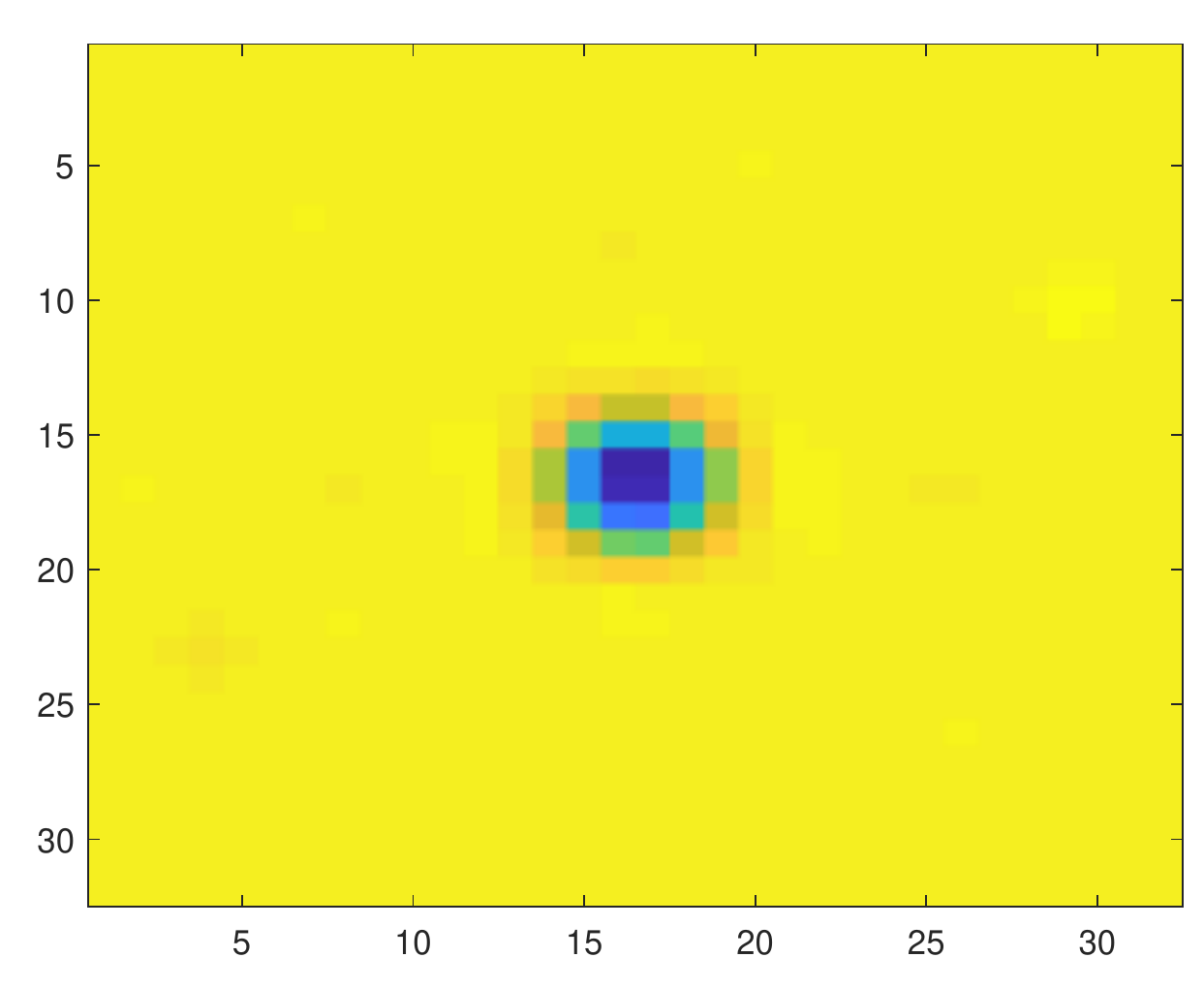}
\caption{Kernel matrices $A$ for \textsc{ADMM-exact} at iterations $k = 0, 10, 50, 100, 120$.}
\end{subfigure}
\caption{\textsc{ADMM-slack} and \textsc{ADMM-exact}, noiseless.}
\label{fig:exp_ker5_noiseless}
\end{figure}

\begin{figure}
\centering
\begin{subfigure}[t]{0.2\textwidth}
\includegraphics[trim=135pt 250pt 130pt 250pt, clip,scale=0.25]{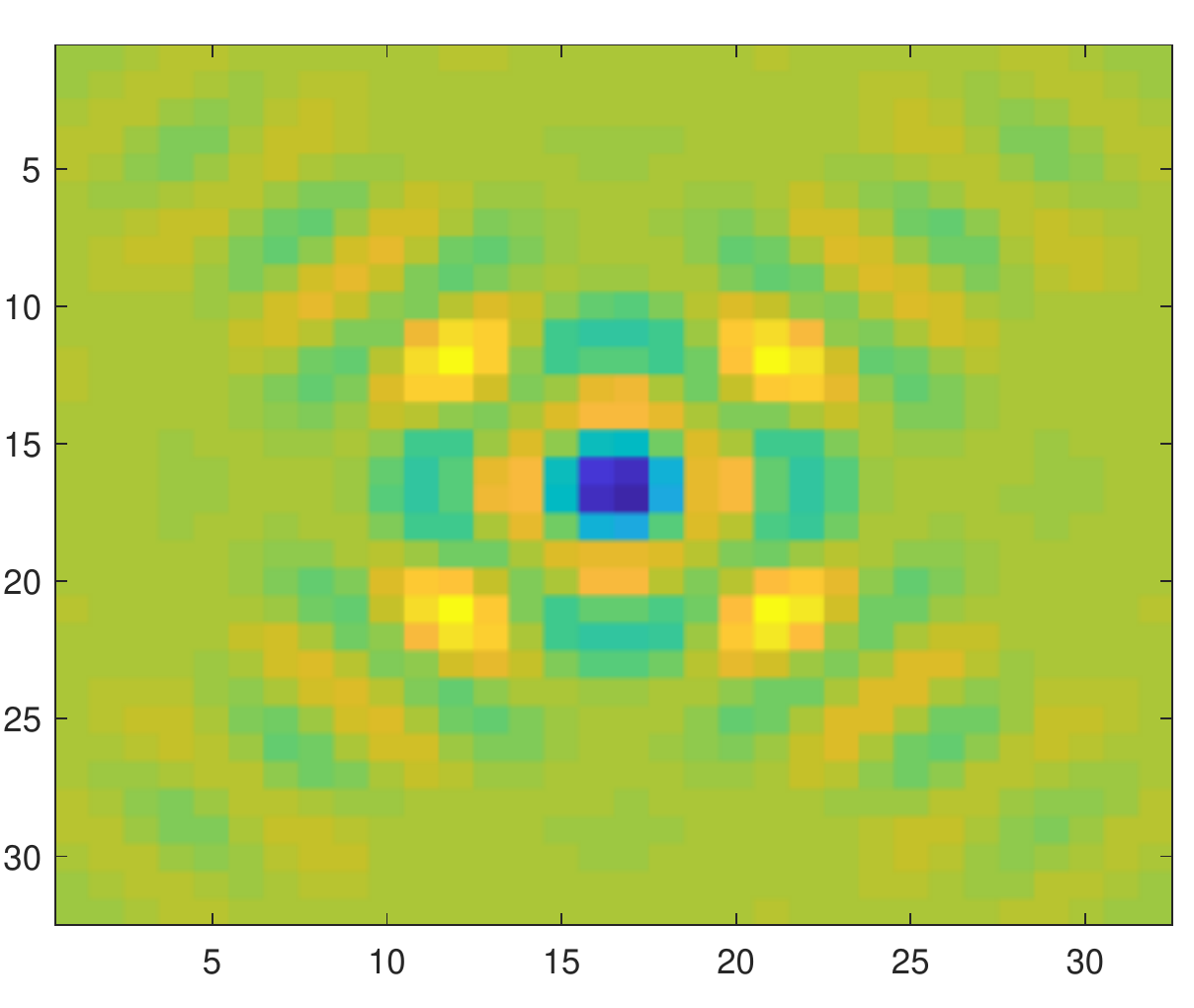}
\caption{True Kernel $A^\ast$}
\end{subfigure}
\begin{subfigure}[t]{0.2\textwidth}
\includegraphics[trim=125pt 250pt 130pt 250pt, clip,scale=0.25]{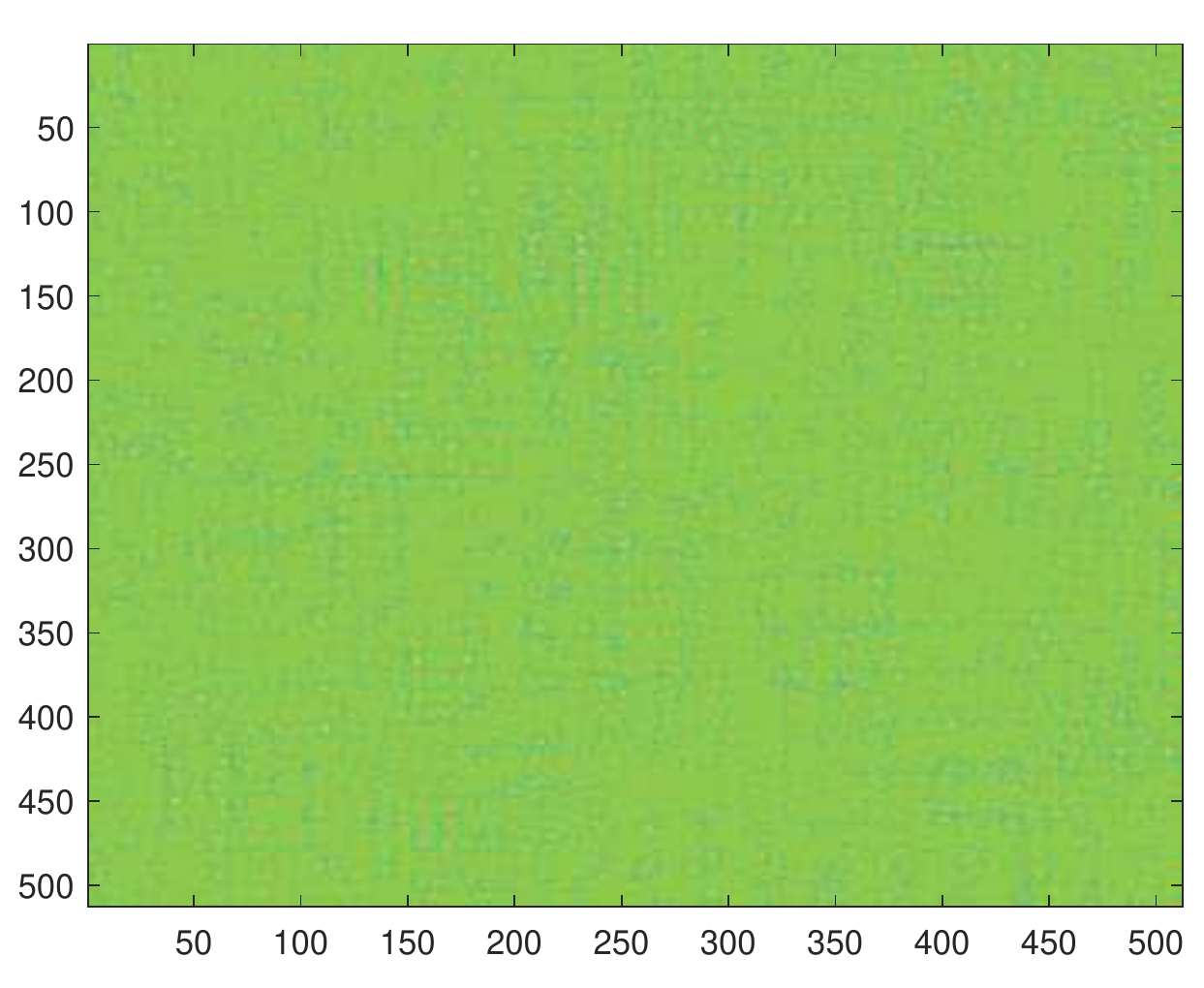}
\caption{Observations $Y$}
\end{subfigure}
\begin{subfigure}[t]{0.5\textwidth}
\includegraphics[trim=125pt 250pt 130pt 250pt, clip,scale=0.3]{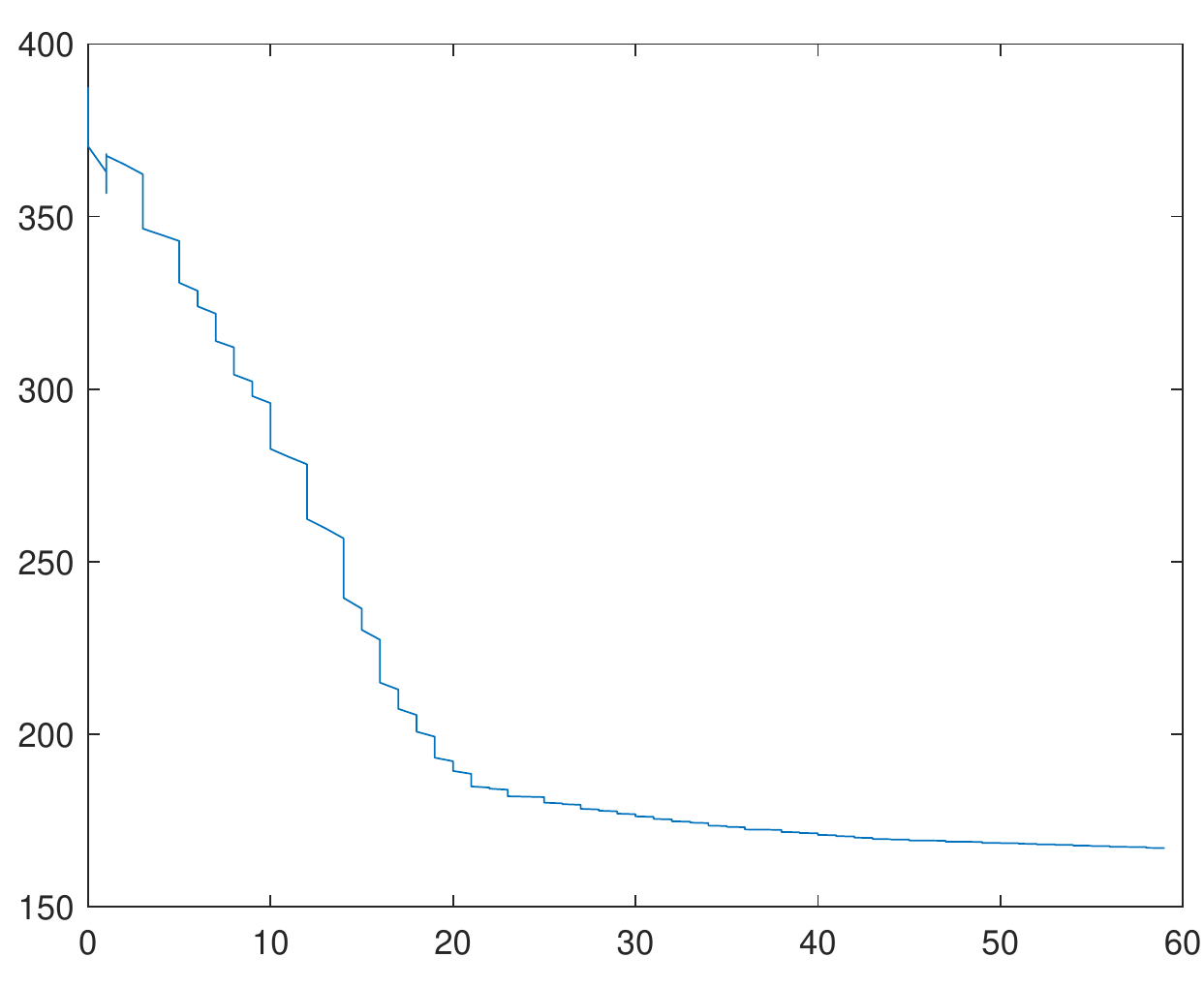}
\hfill
\includegraphics[trim=125pt 250pt 130pt 250pt, clip,scale=0.3]{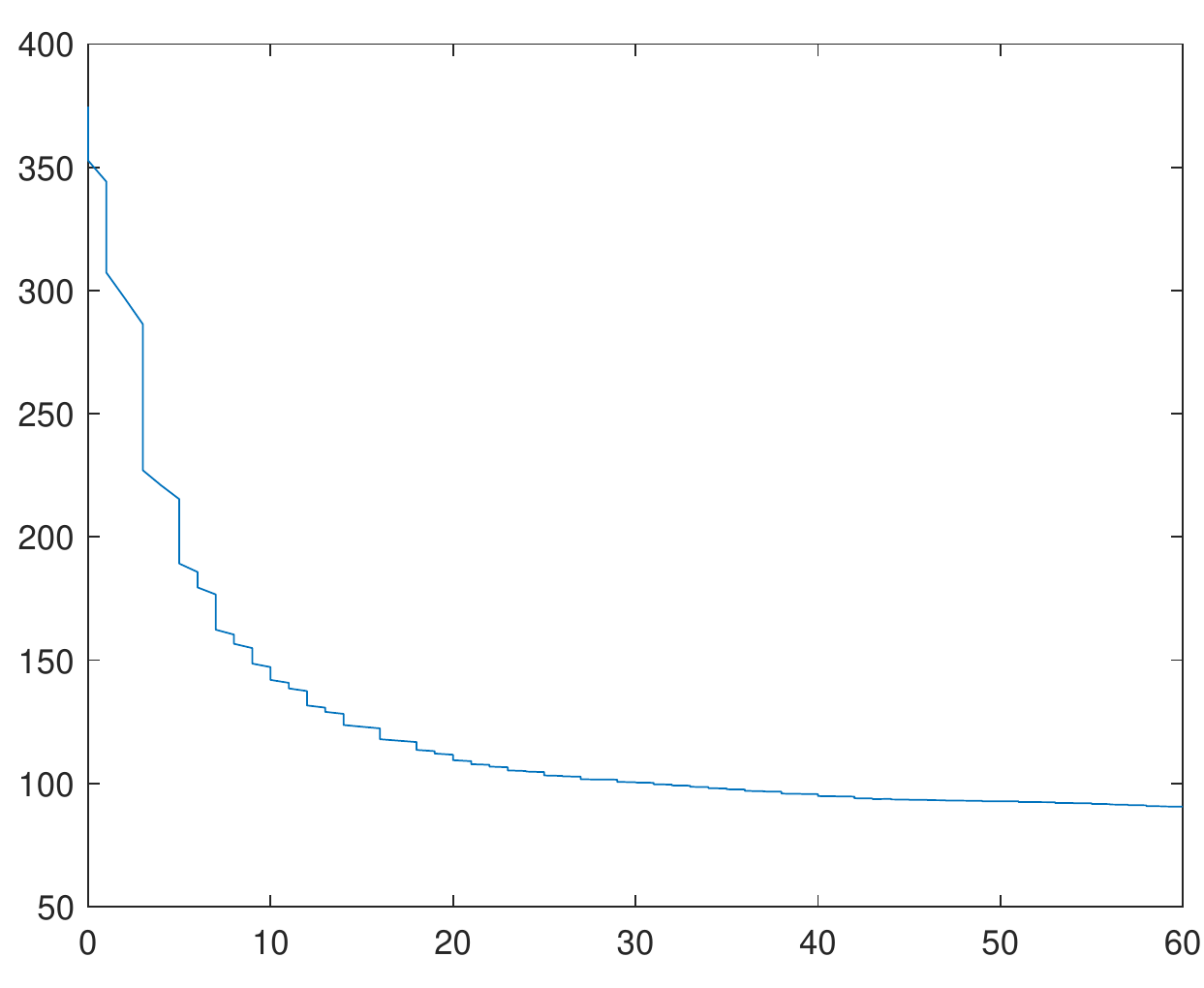}
\caption{Objective value over time (s). The left plot shows \textsc{ADMM-slack}, and the right shows \textsc{ADMM-exact}.}
\end{subfigure}

\bigskip

\centering
\begin{subfigure}[t]{1.0\textwidth}
\includegraphics[trim=125pt 250pt 130pt 250pt, clip,scale=0.23]{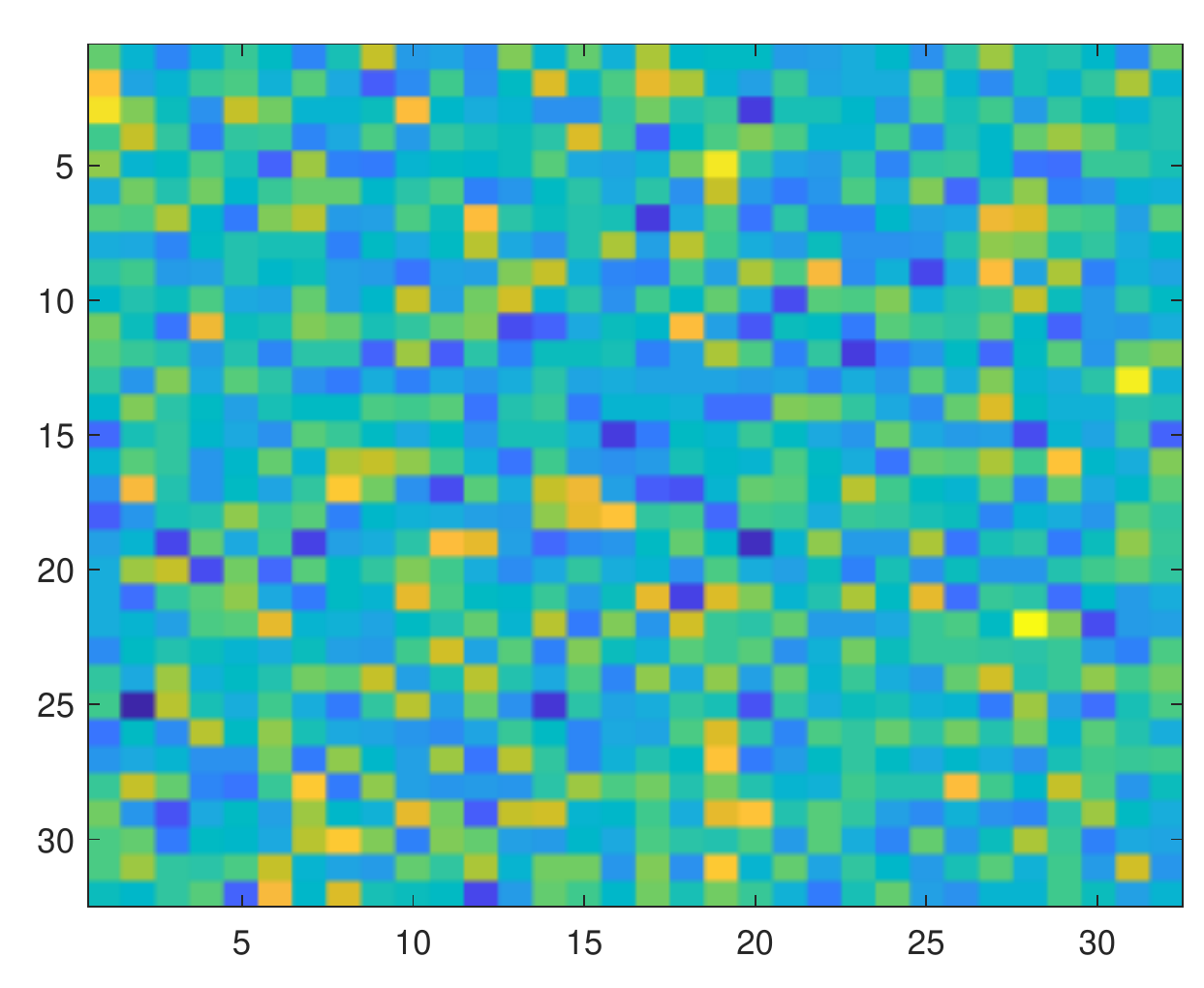}
\hfill
\includegraphics[trim=125pt 250pt 130pt 250pt, clip,scale=0.23]{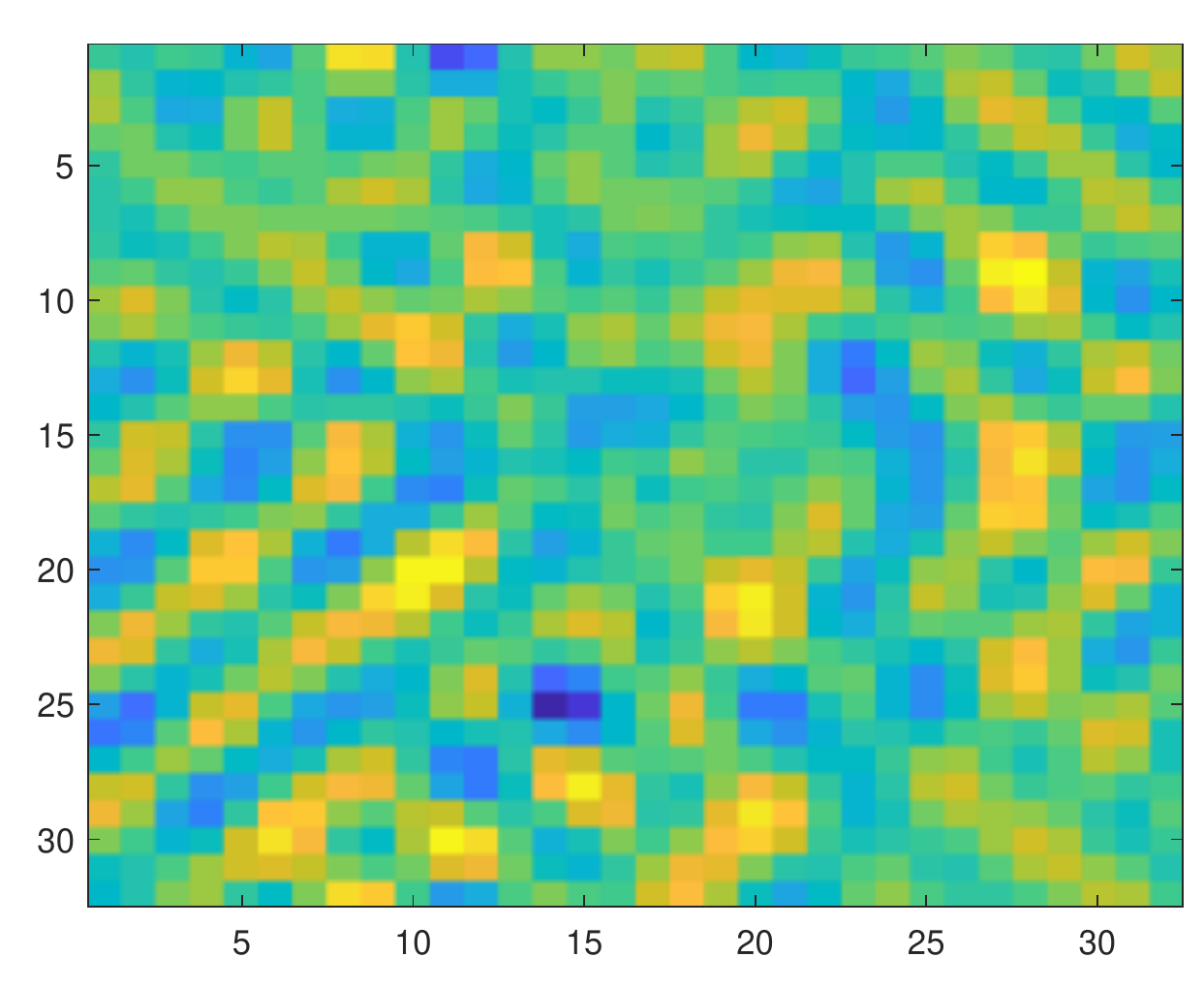}
\hfill
\includegraphics[trim=125pt 250pt 130pt 250pt, clip,scale=0.23]{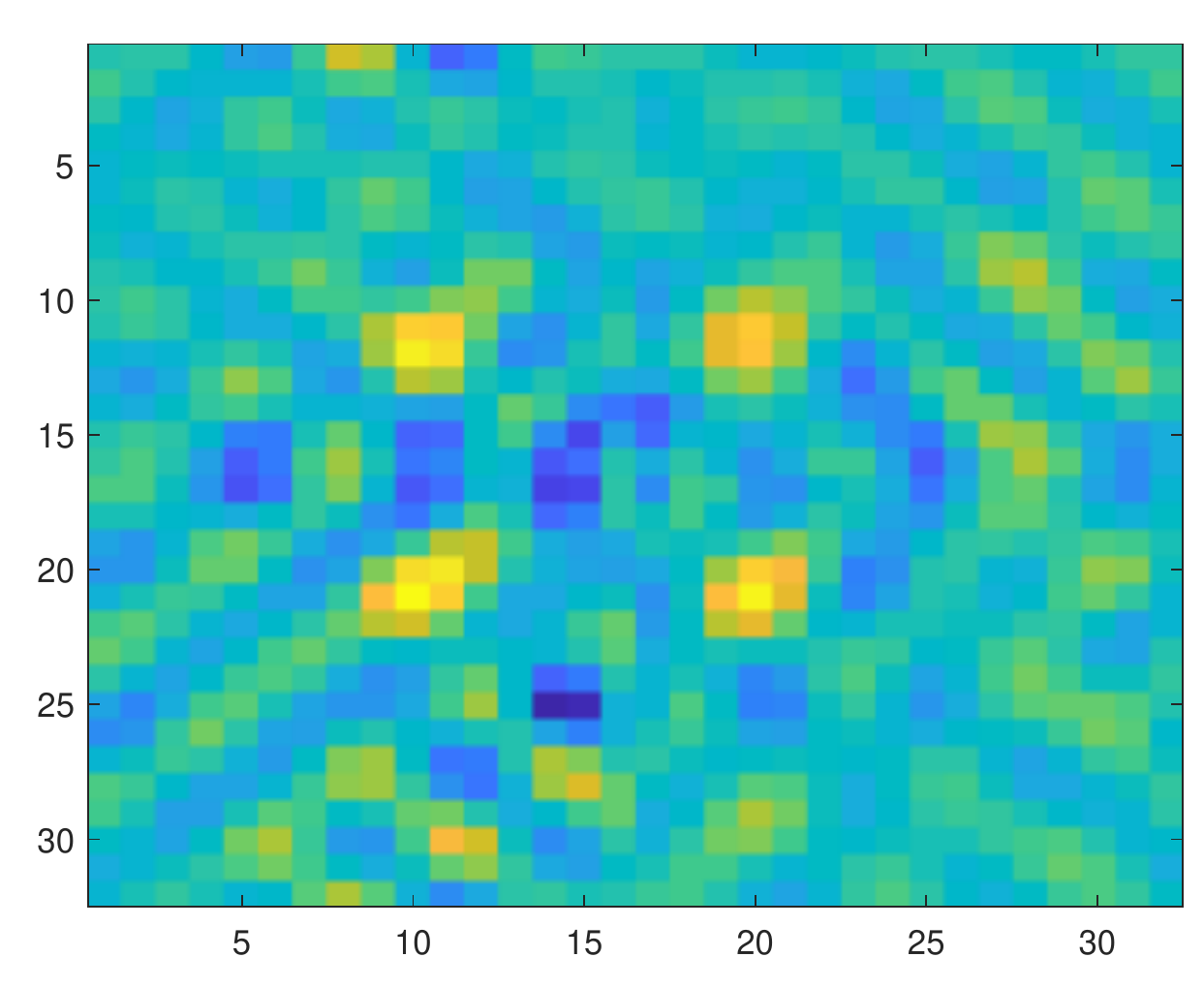}
\hfill
\includegraphics[trim=125pt 250pt 130pt 250pt, clip,scale=0.23]{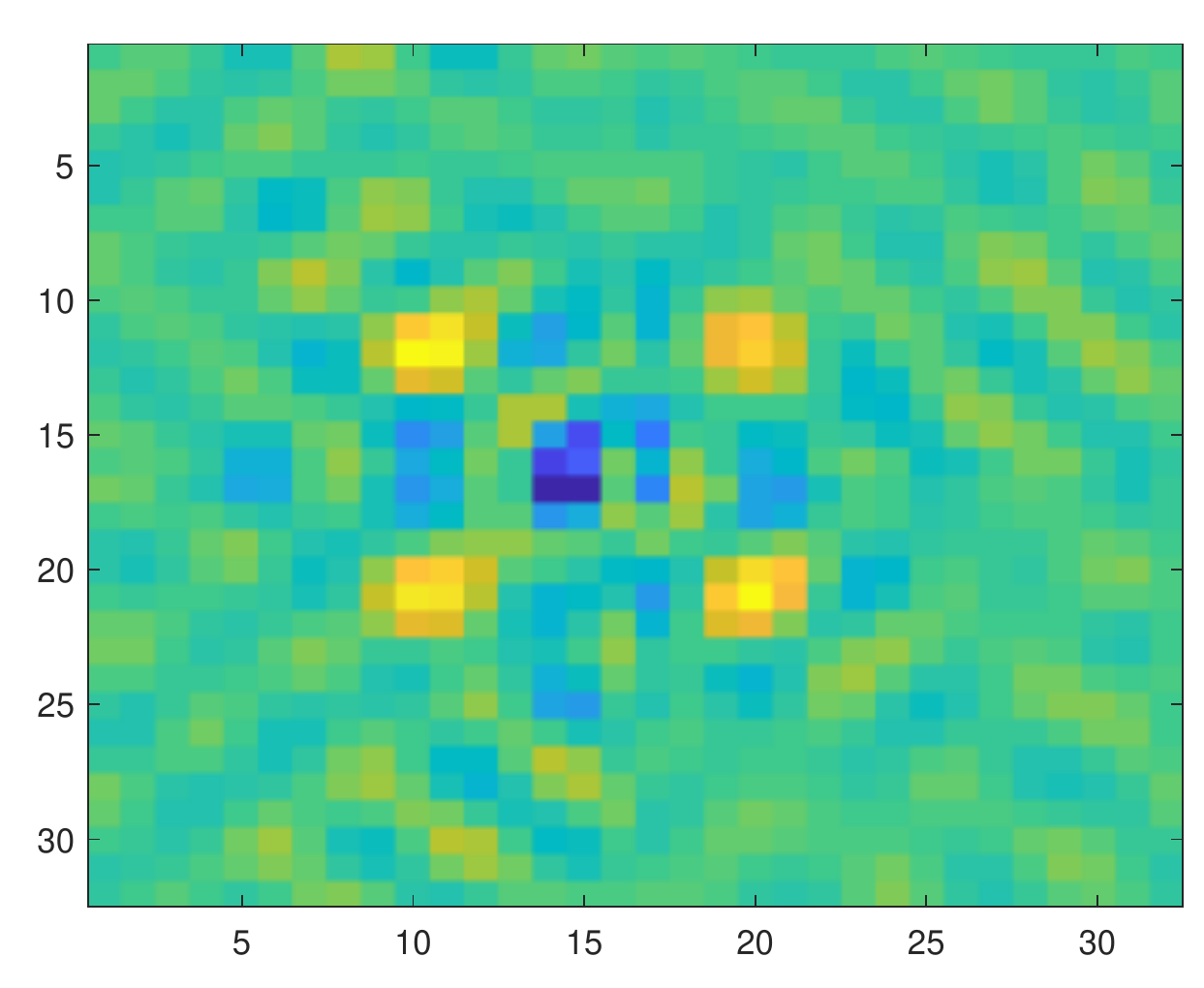}
\hfill
\includegraphics[trim=125pt 250pt 130pt 250pt, clip,scale=0.23]{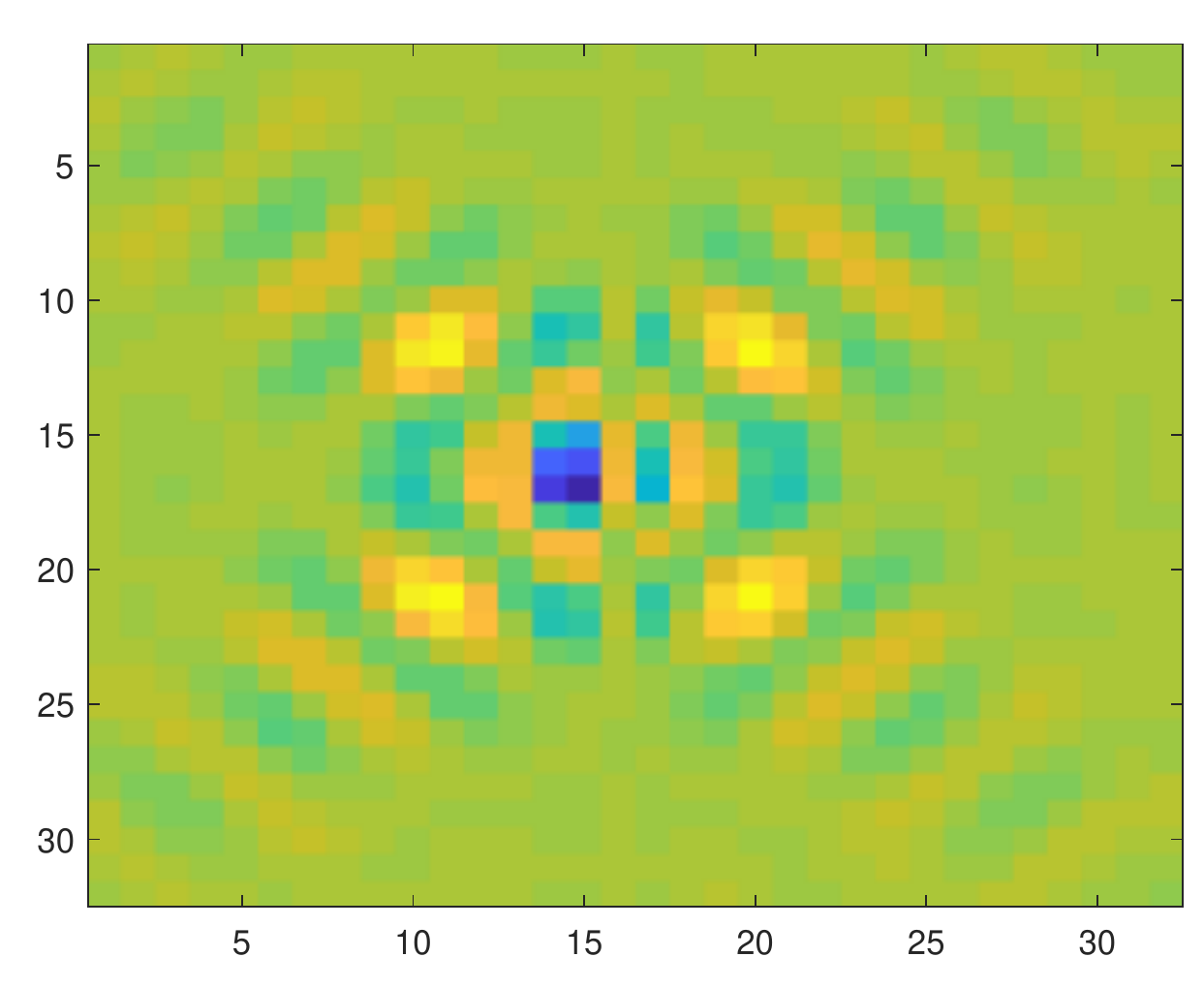}
\caption{Kernel matrices $A$ for \textsc{ADMM-slack} at iterations $k = 0, 10, 30, 50, 100$.}
\end{subfigure}

\bigskip
\centering
\begin{subfigure}[t]{1.0\textwidth}
\includegraphics[trim=125pt 250pt 130pt 250pt, clip,scale=0.23]{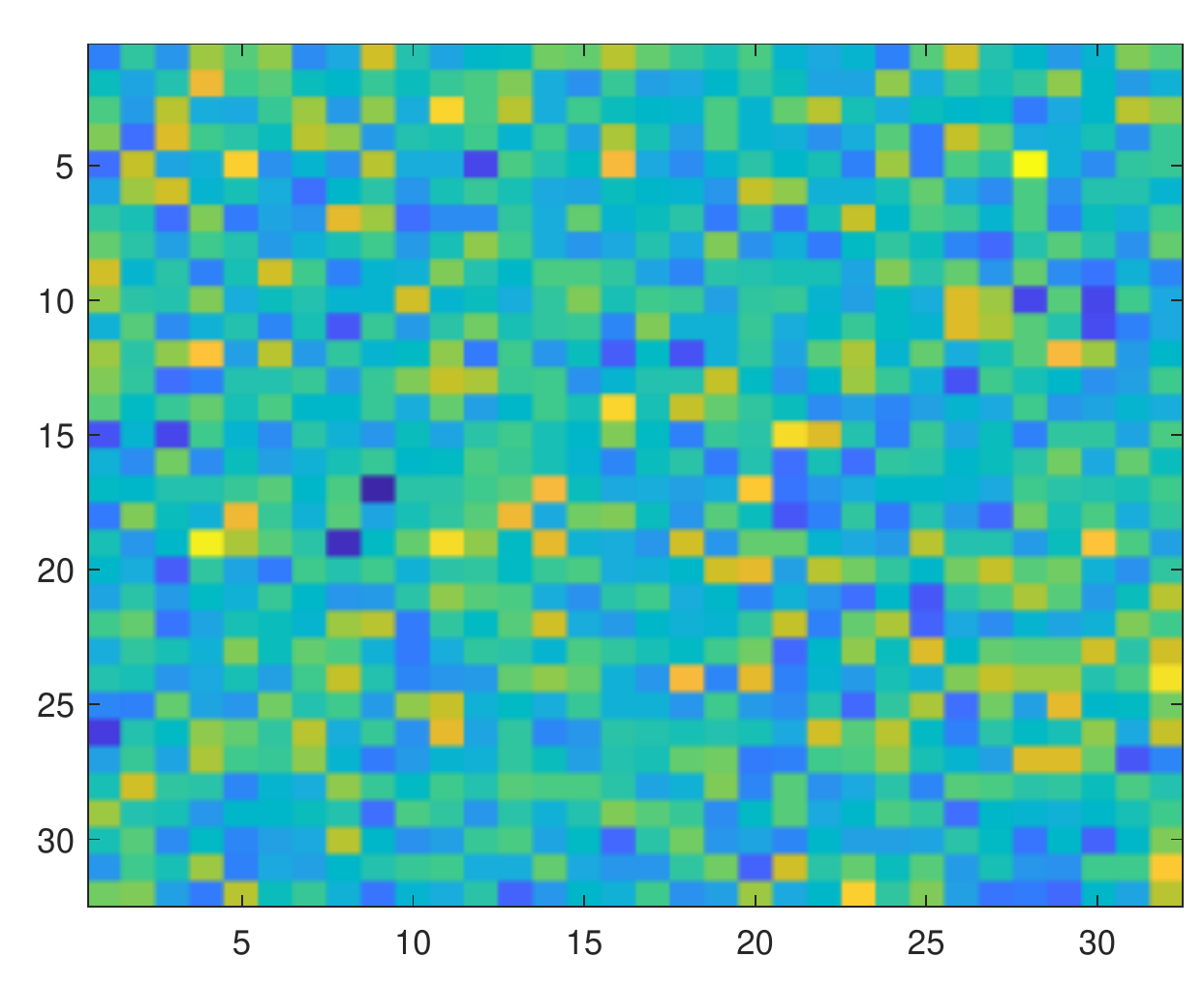}
\hfill
\includegraphics[trim=125pt 250pt 130pt 250pt, clip,scale=0.23]{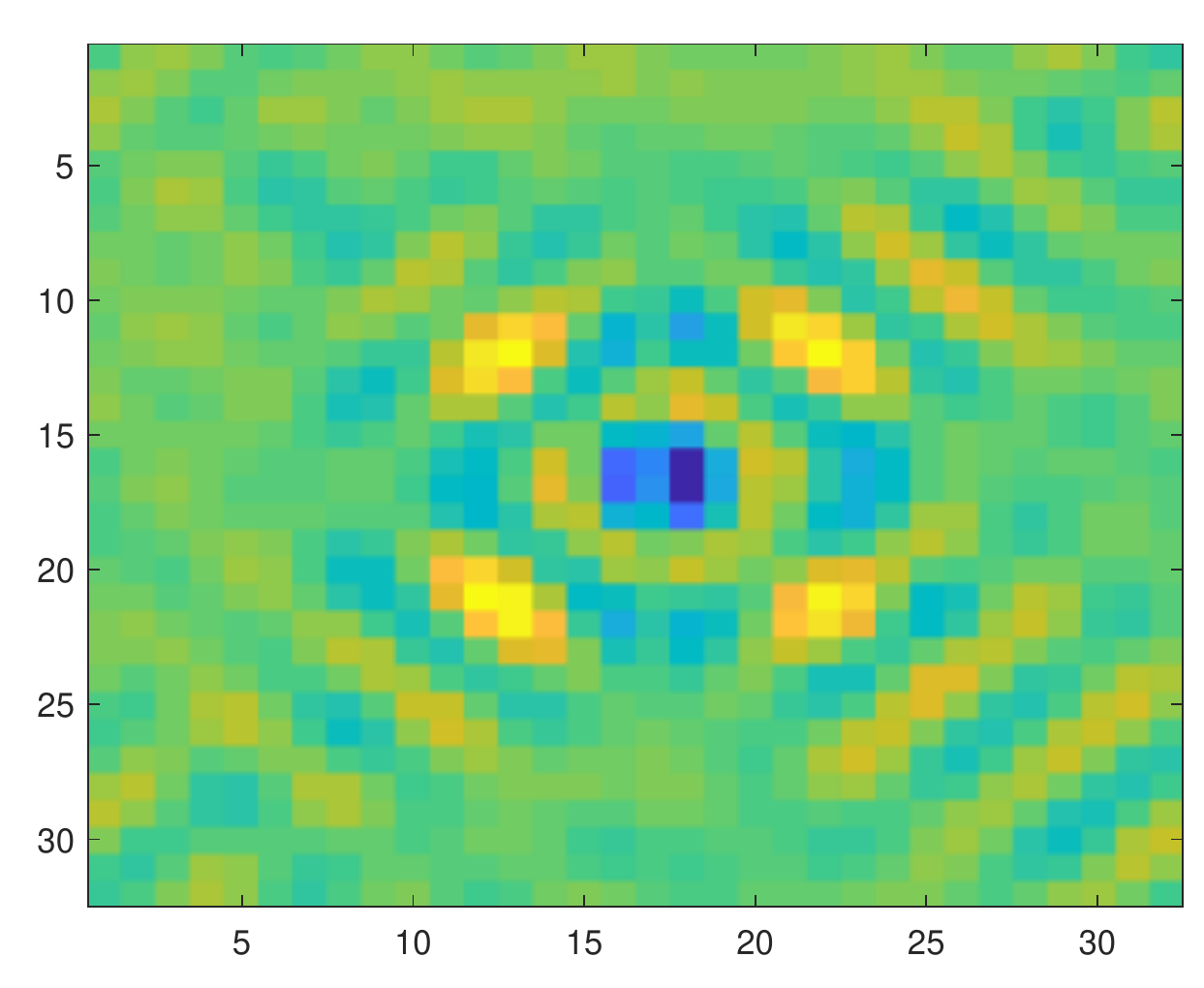}
\hfill
\includegraphics[trim=125pt 250pt 130pt 250pt, clip,scale=0.23]{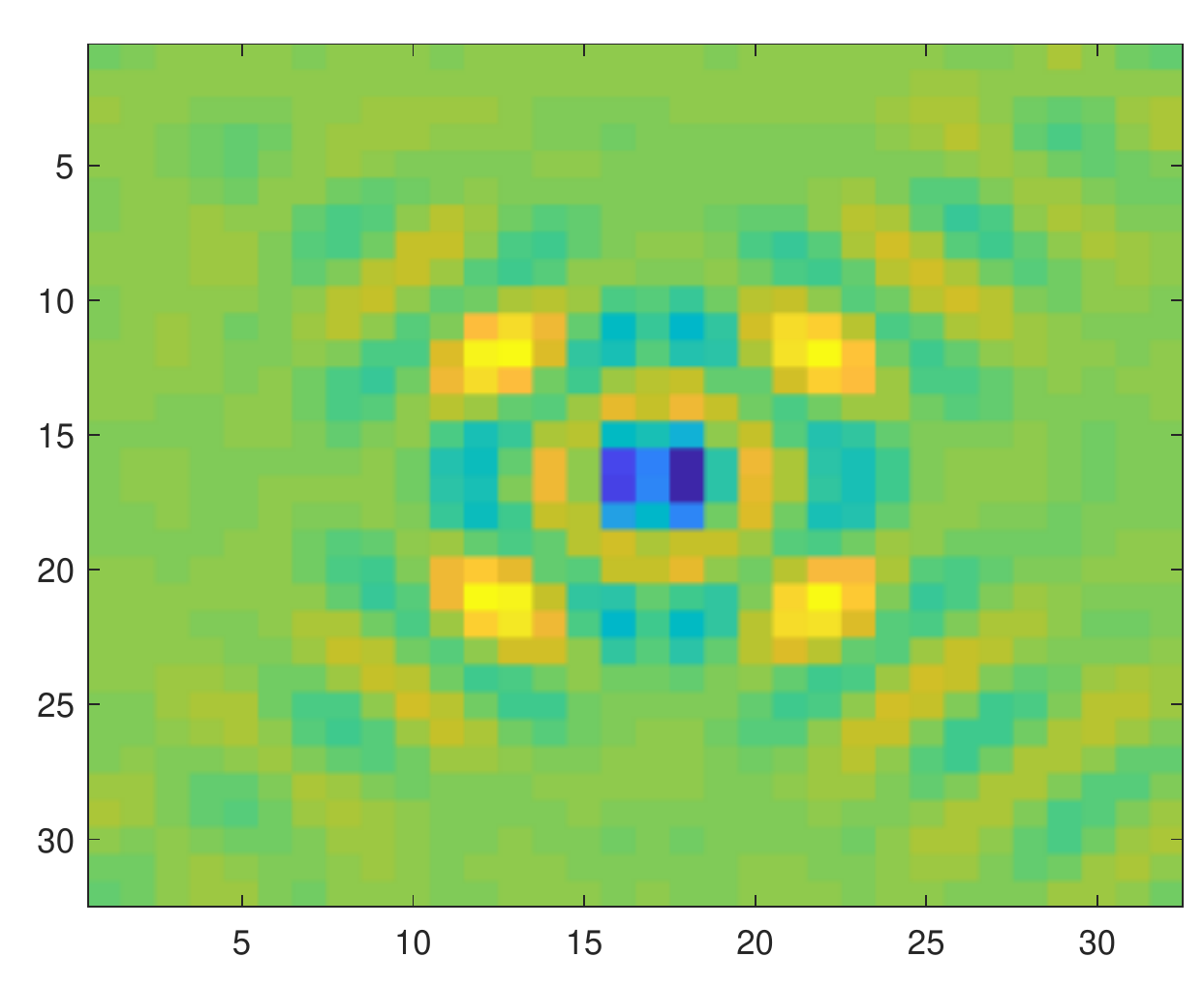}
\hfill
\includegraphics[trim=125pt 250pt 130pt 250pt, clip,scale=0.23]{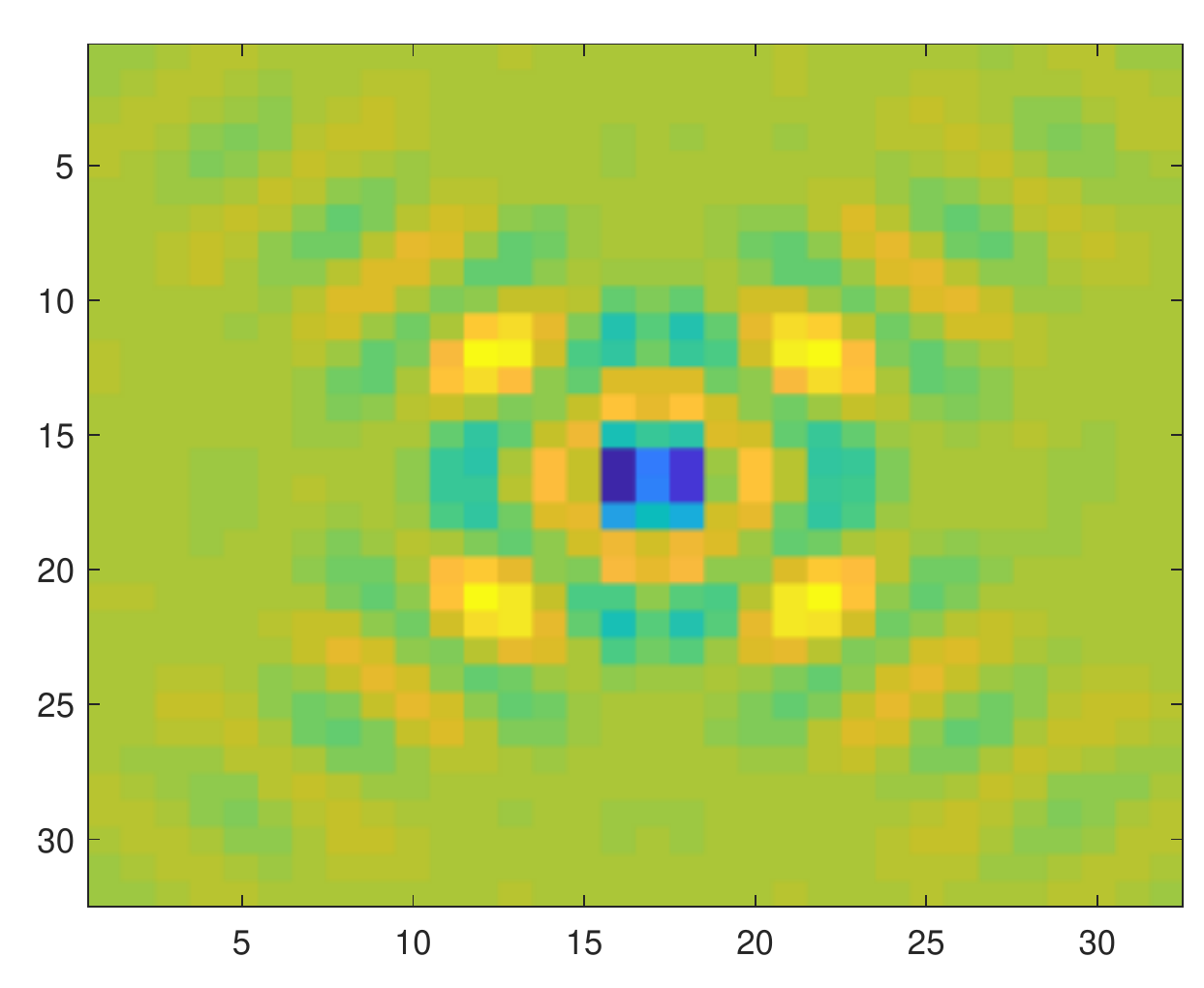}
\hfill
\includegraphics[trim=125pt 250pt 130pt 250pt, clip,scale=0.23]{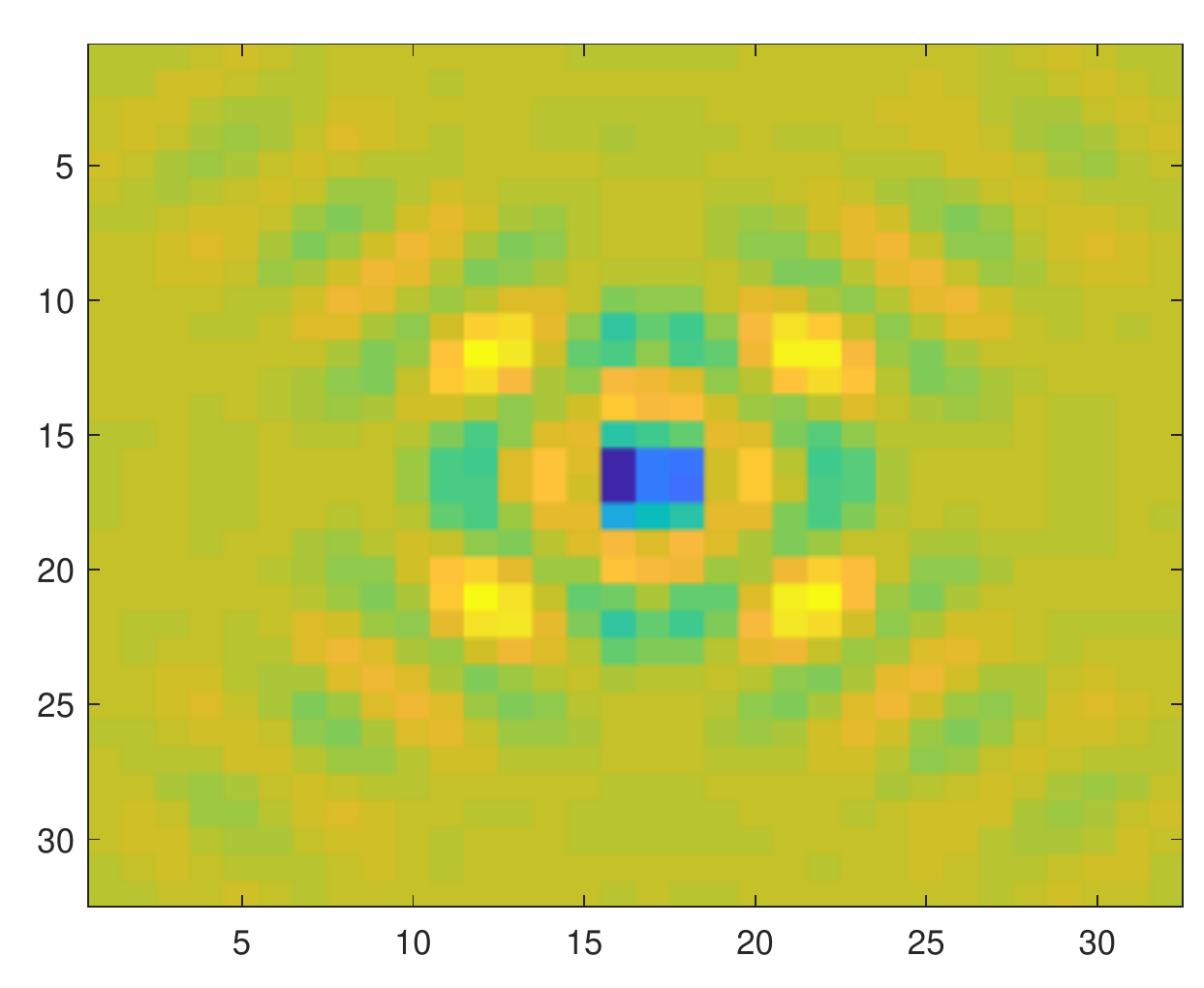}
\caption{Kernel matrices $A$ for \textsc{ADMM-exact} at iterations $k = 0, 10, 20, 50, 100$.}
\end{subfigure}
\caption{\textsc{ADMM-slack} and \textsc{ADMM-exact}, noiseless.}
\label{fig:exp_ker23_noiseless}
\vspace{-50pt}
\end{figure}

\subsection{Noisy Observations}

We next consider the case where the observations $Y$ are corrupted by noise. In this case, the formulation (SBD1) is more natural because $A \ast X + b\mathbf{1} = Y$ cannot be satisfied exactly, even by the true $A,X,b$. The slack variable $Z$ then serves to absorb the noise term $\xi$.

\begin{figure}[h]
\centering
\begin{subfigure}[t]{0.2\textwidth}
\includegraphics[trim=135pt 250pt 130pt 250pt, clip,scale=0.25]{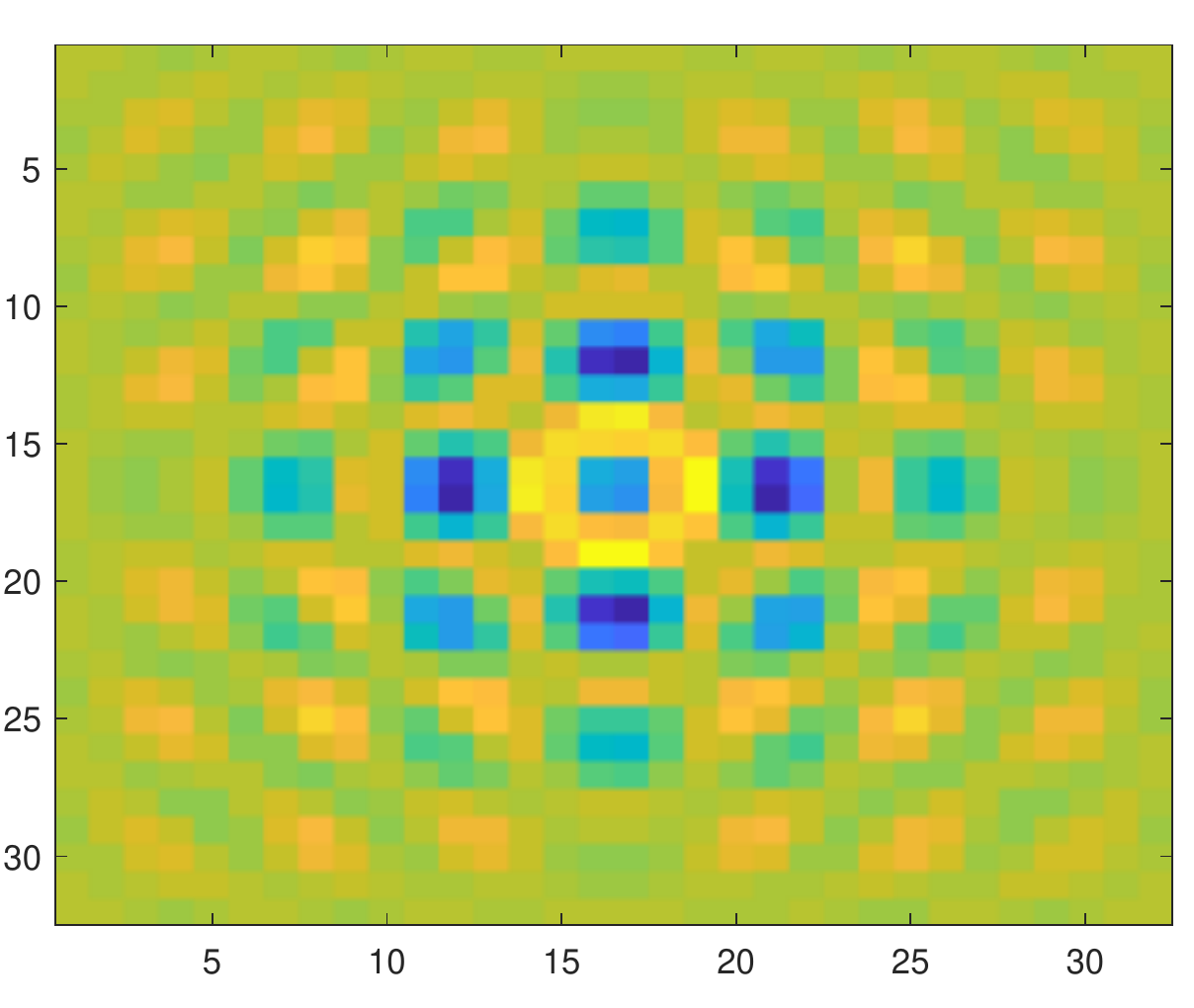}
\caption{True Kernel $A^\ast$}
\end{subfigure}
\begin{subfigure}[t]{0.2\textwidth}
\includegraphics[trim=125pt 250pt 130pt 250pt, clip,scale=0.25]{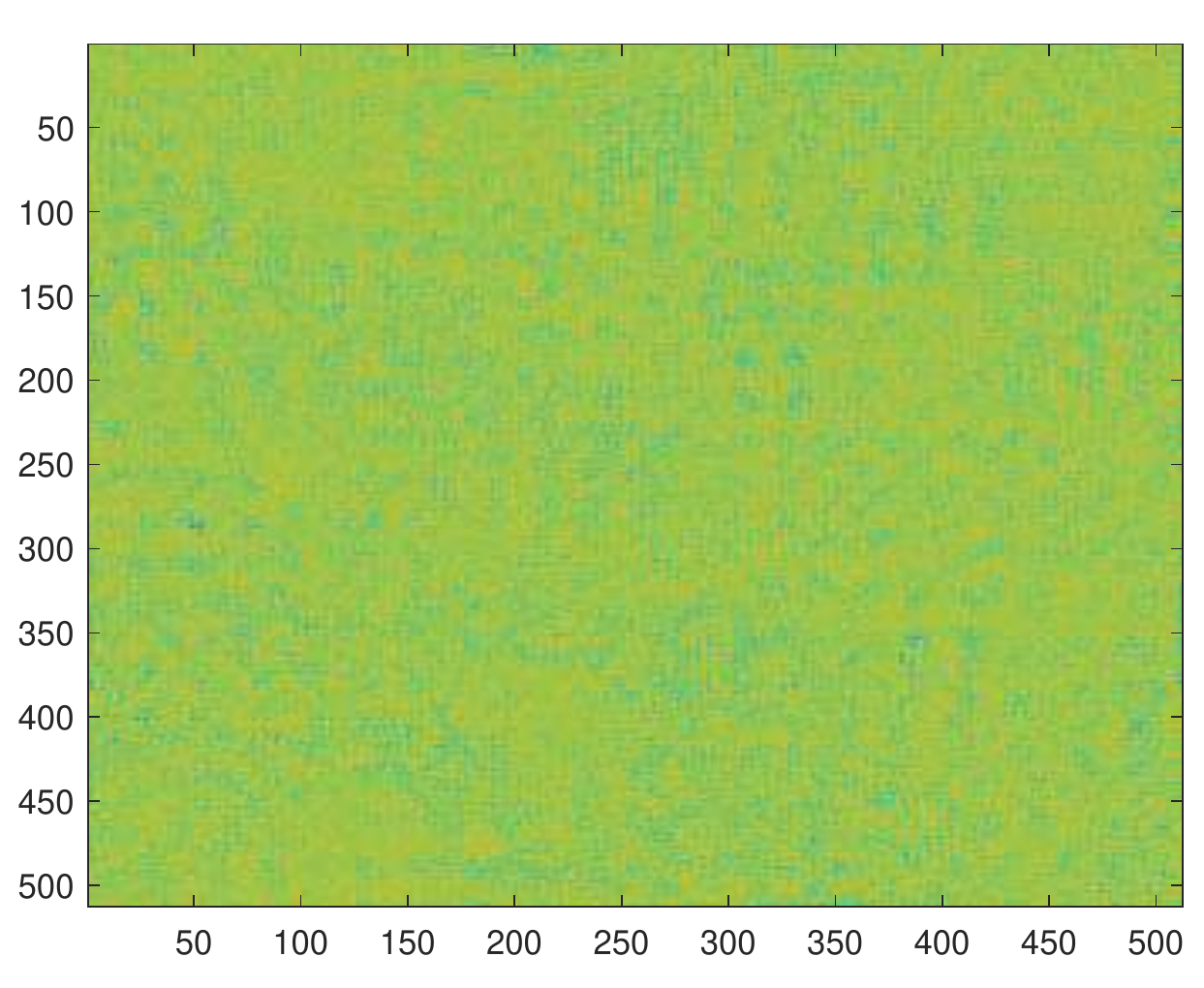}
\caption{Observations $Y$ (with noise)}
\end{subfigure}
\begin{subfigure}[t]{0.5\textwidth}
\includegraphics[trim=125pt 250pt 130pt 250pt, clip,scale=0.3]{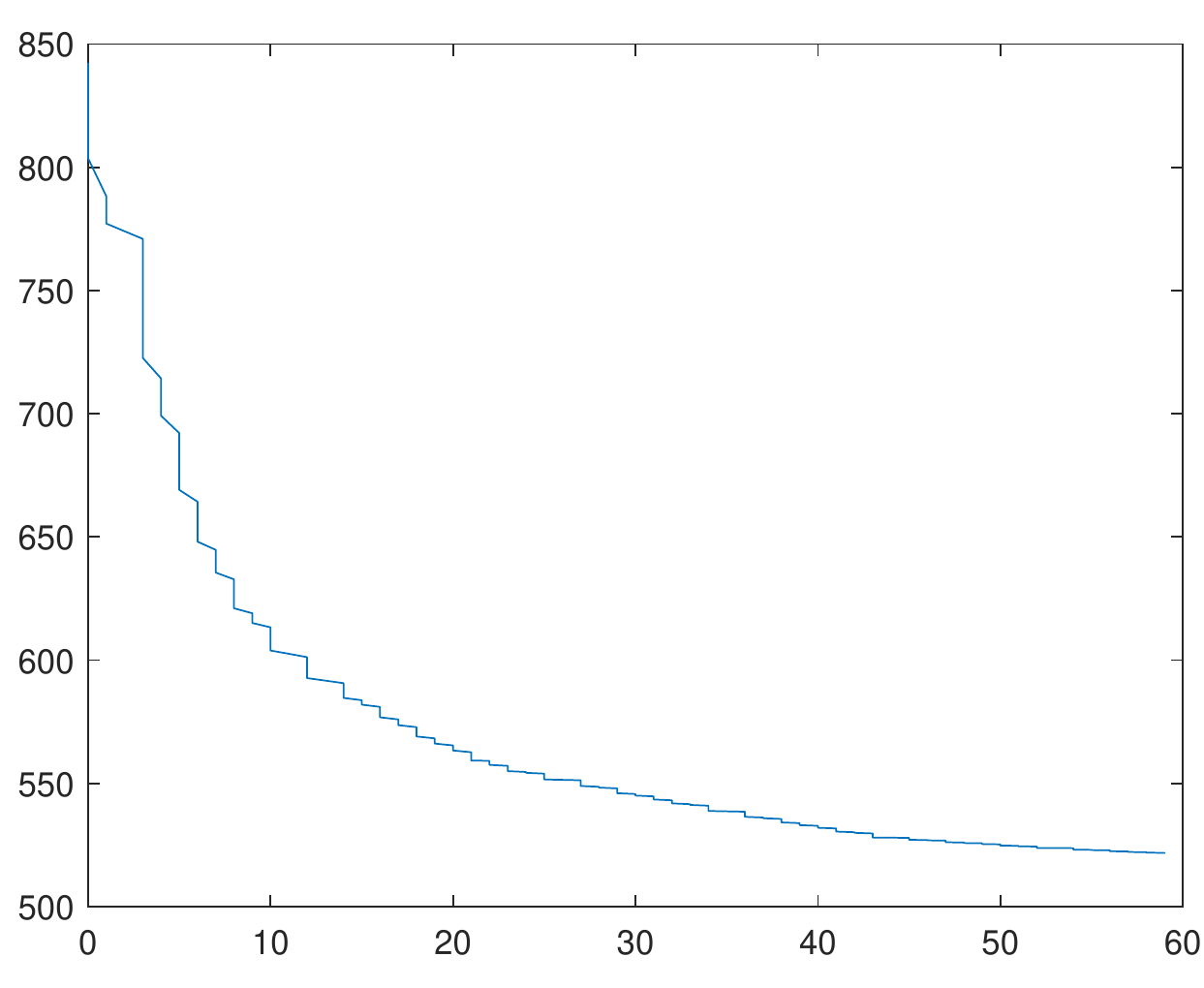}
\hfill
\includegraphics[trim=125pt 250pt 130pt 250pt, clip,scale=0.3]{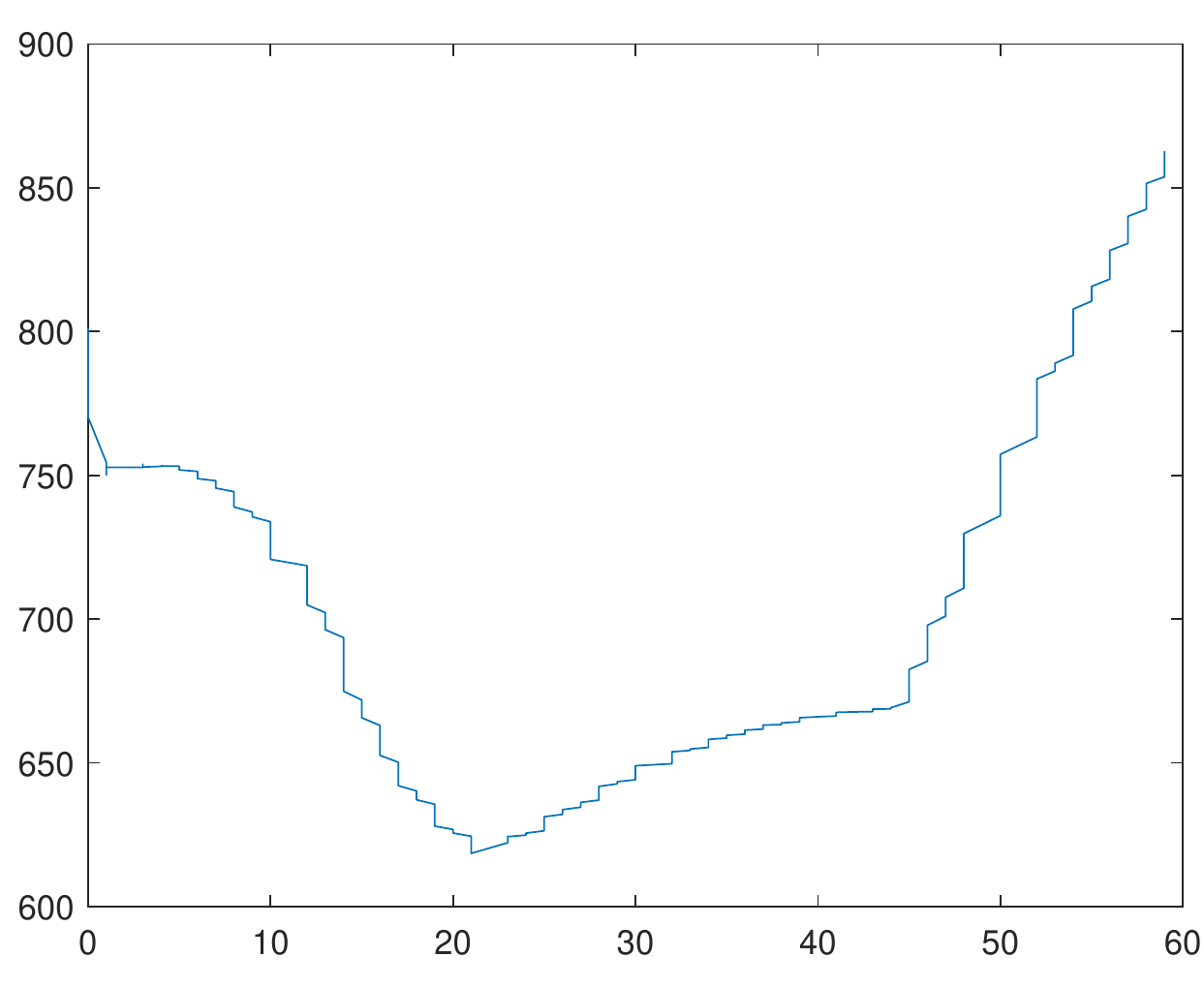}
\caption{Objective value over time (s). The left plot shows \textsc{ADMM-slack}, and the right shows \textsc{ADMM-exact}.}
\end{subfigure}

\bigskip

\centering
\begin{subfigure}[t]{1.0\textwidth}
\includegraphics[trim=125pt 250pt 130pt 250pt, clip,scale=0.23]{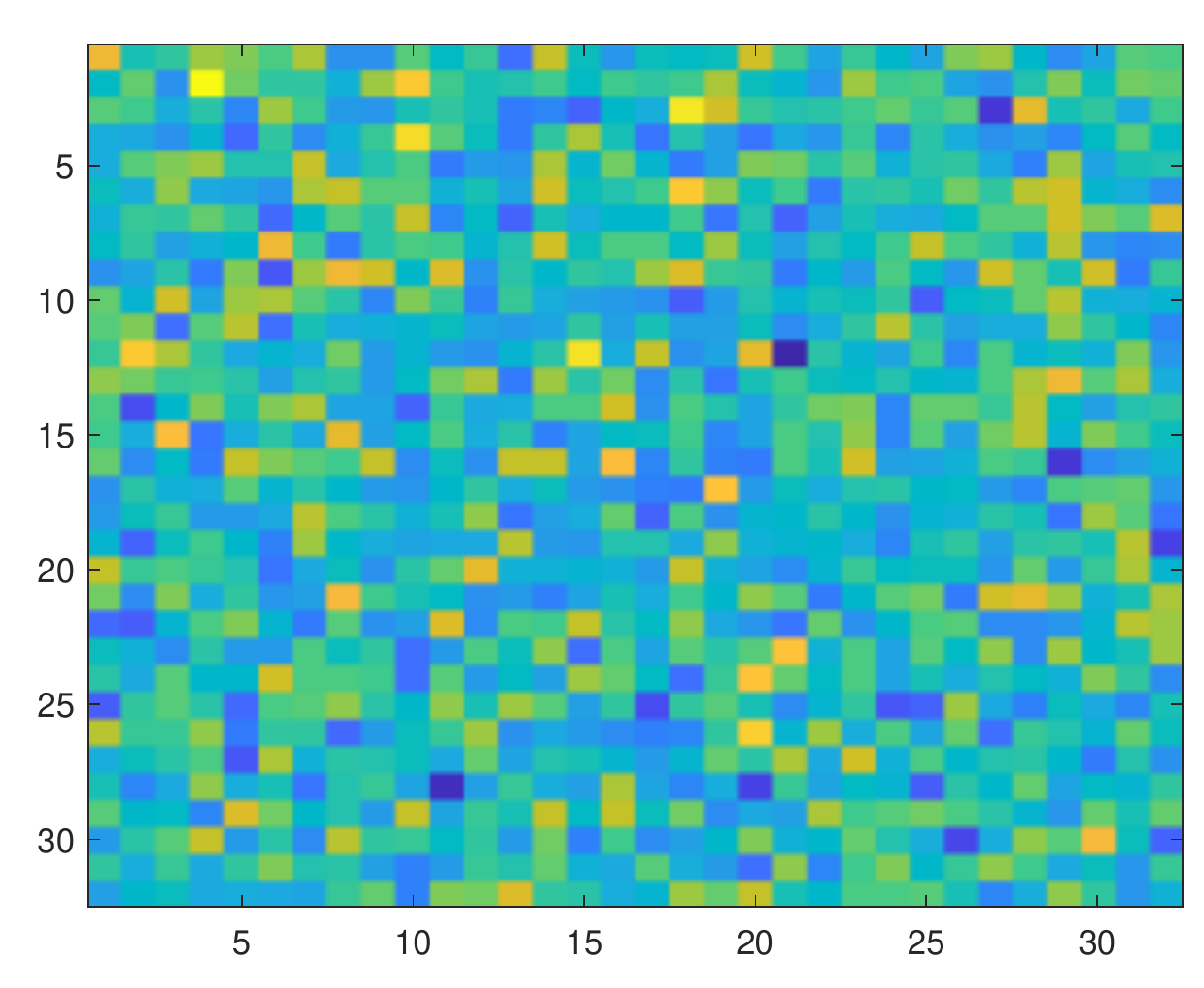}
\hfill
\includegraphics[trim=125pt 250pt 130pt 250pt, clip,scale=0.23]{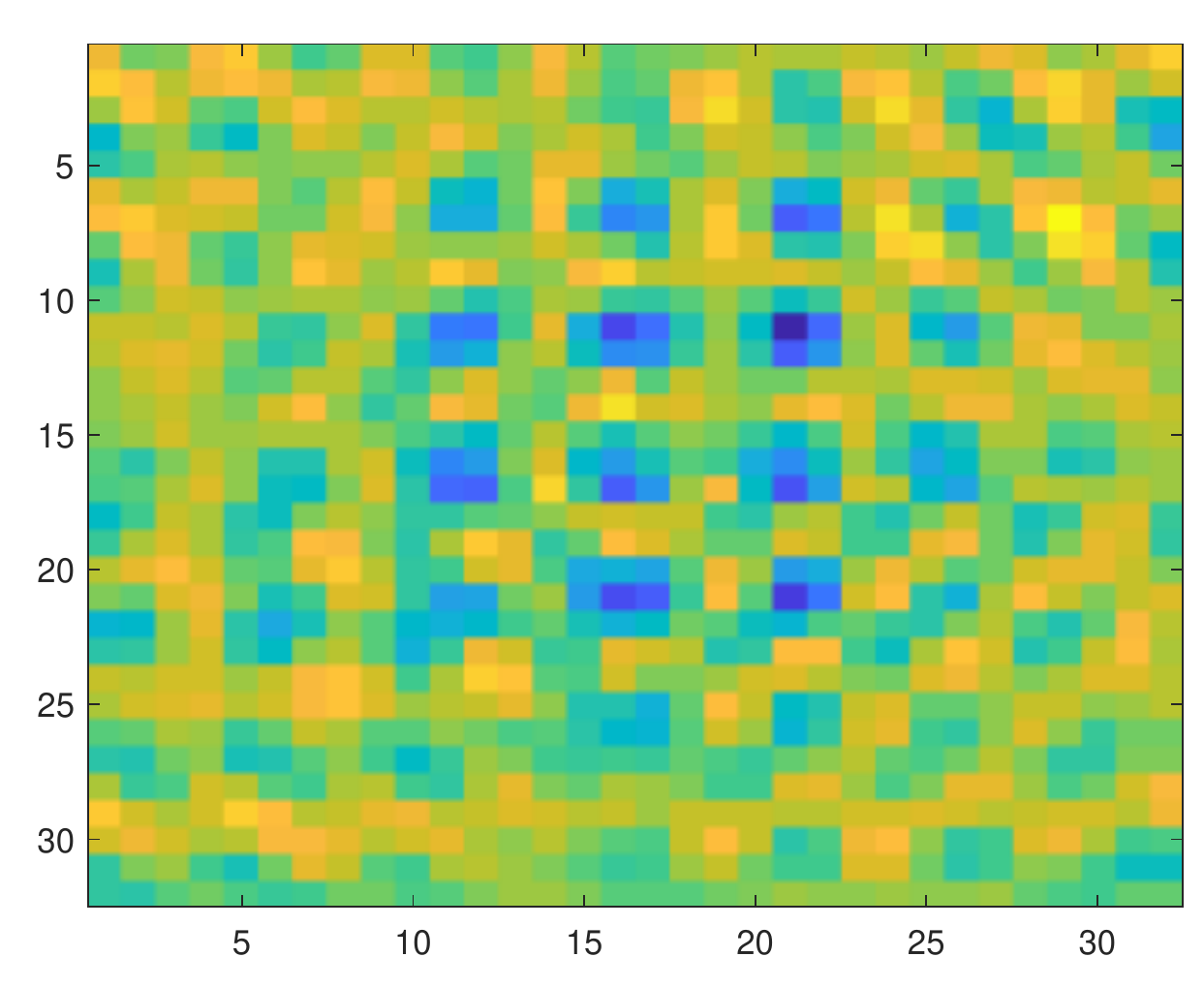}
\hfill
\includegraphics[trim=125pt 250pt 130pt 250pt, clip,scale=0.23]{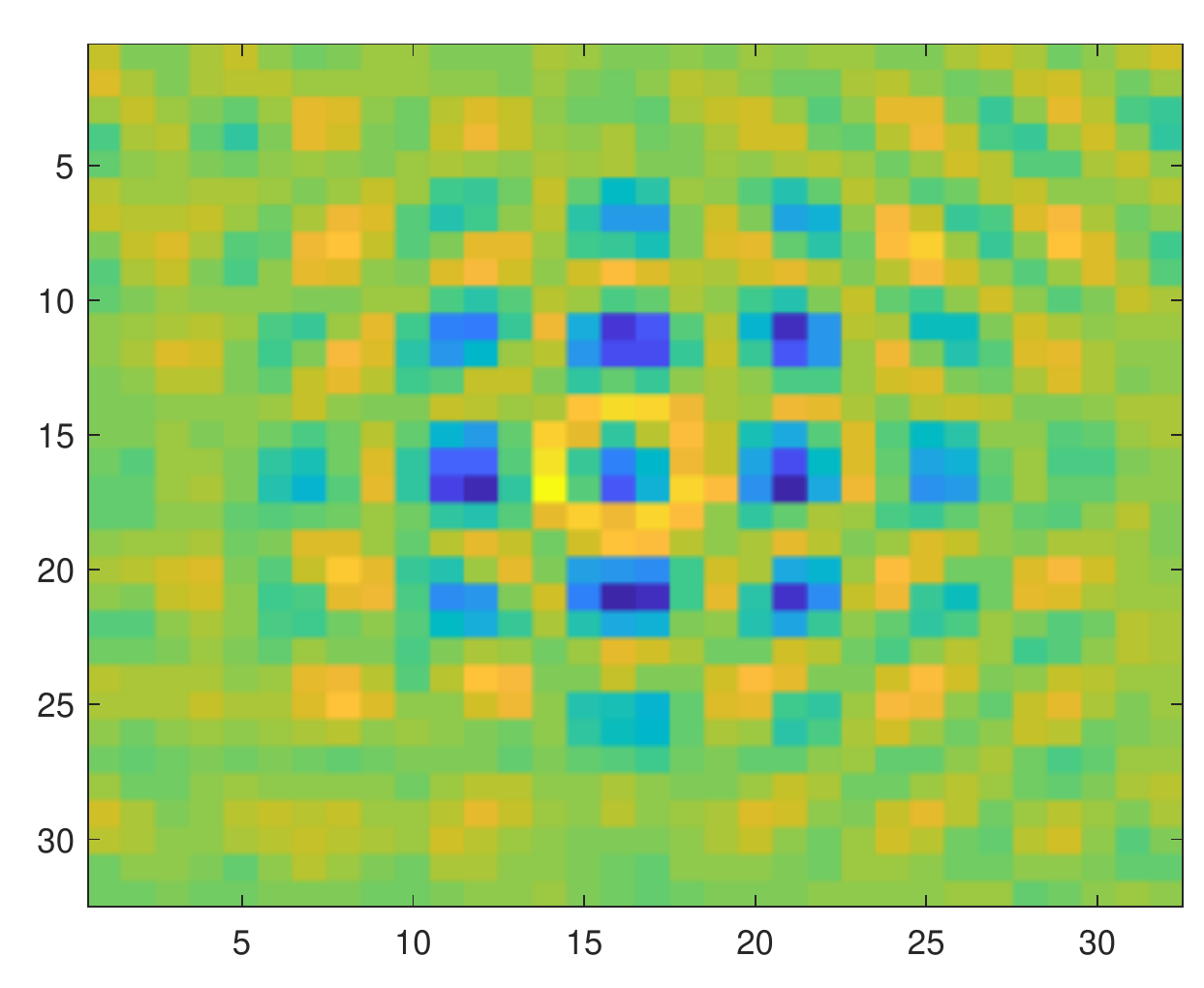}
\hfill
\includegraphics[trim=125pt 250pt 130pt 250pt, clip,scale=0.23]{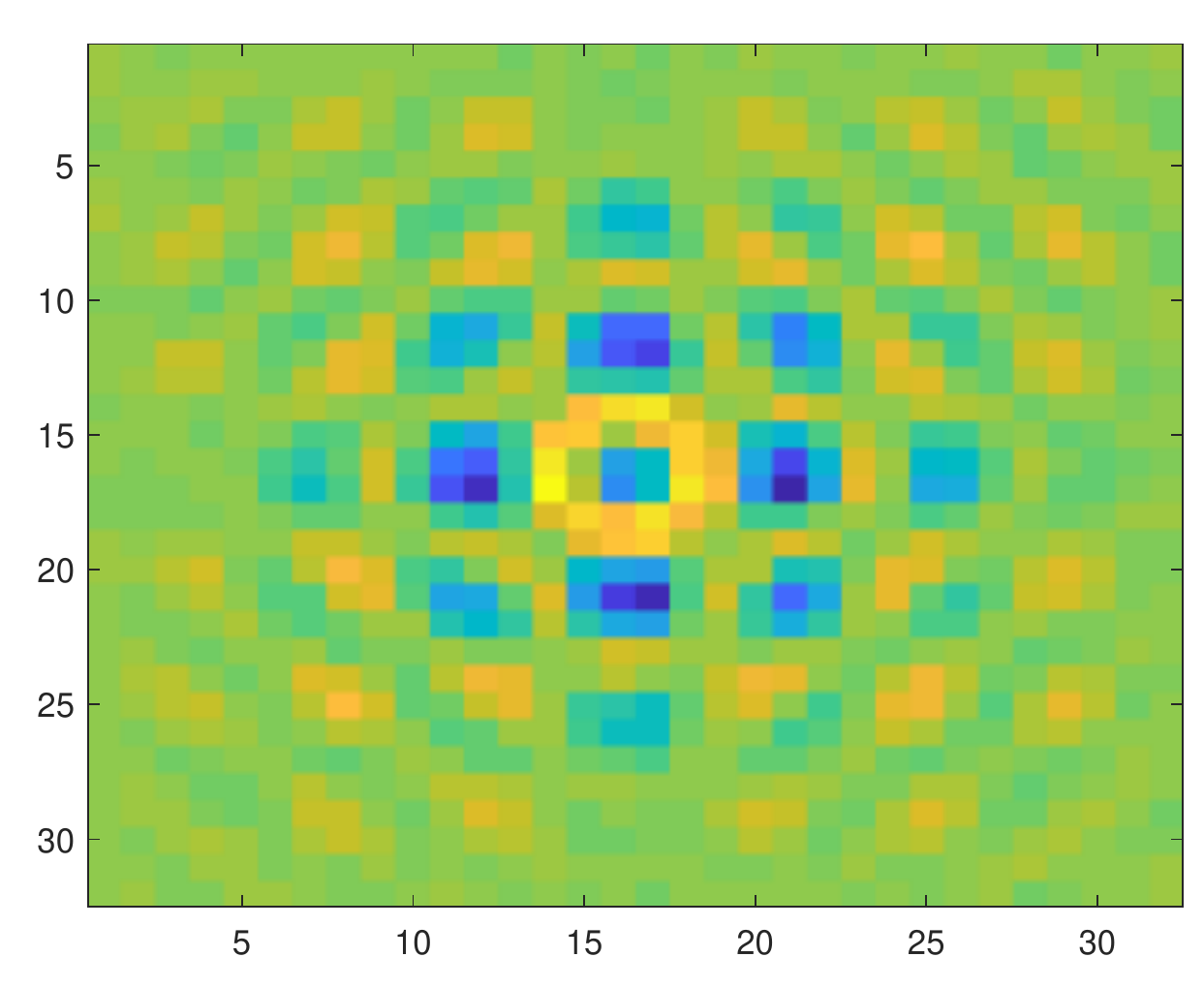}
\hfill
\includegraphics[trim=125pt 250pt 130pt 250pt, clip,scale=0.23]{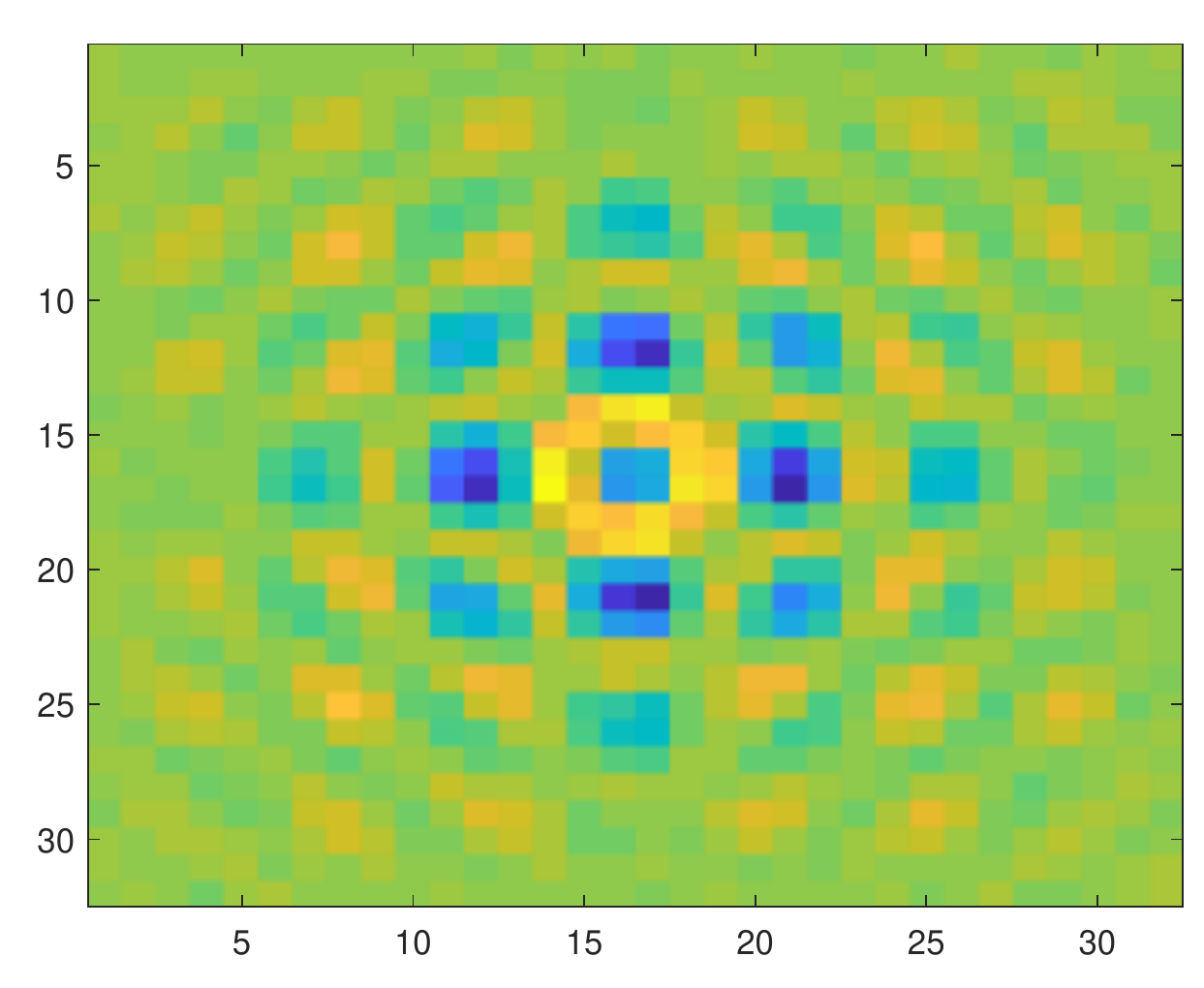}
\caption{Kernel matrices $A$ for \textsc{ADMM-slack} at iterations $k = 0, 10, 20, 50, 100$.}
\end{subfigure}

\bigskip
\centering
\begin{subfigure}[t]{1.0\textwidth}
\includegraphics[trim=125pt 250pt 130pt 250pt, clip,scale=0.23]{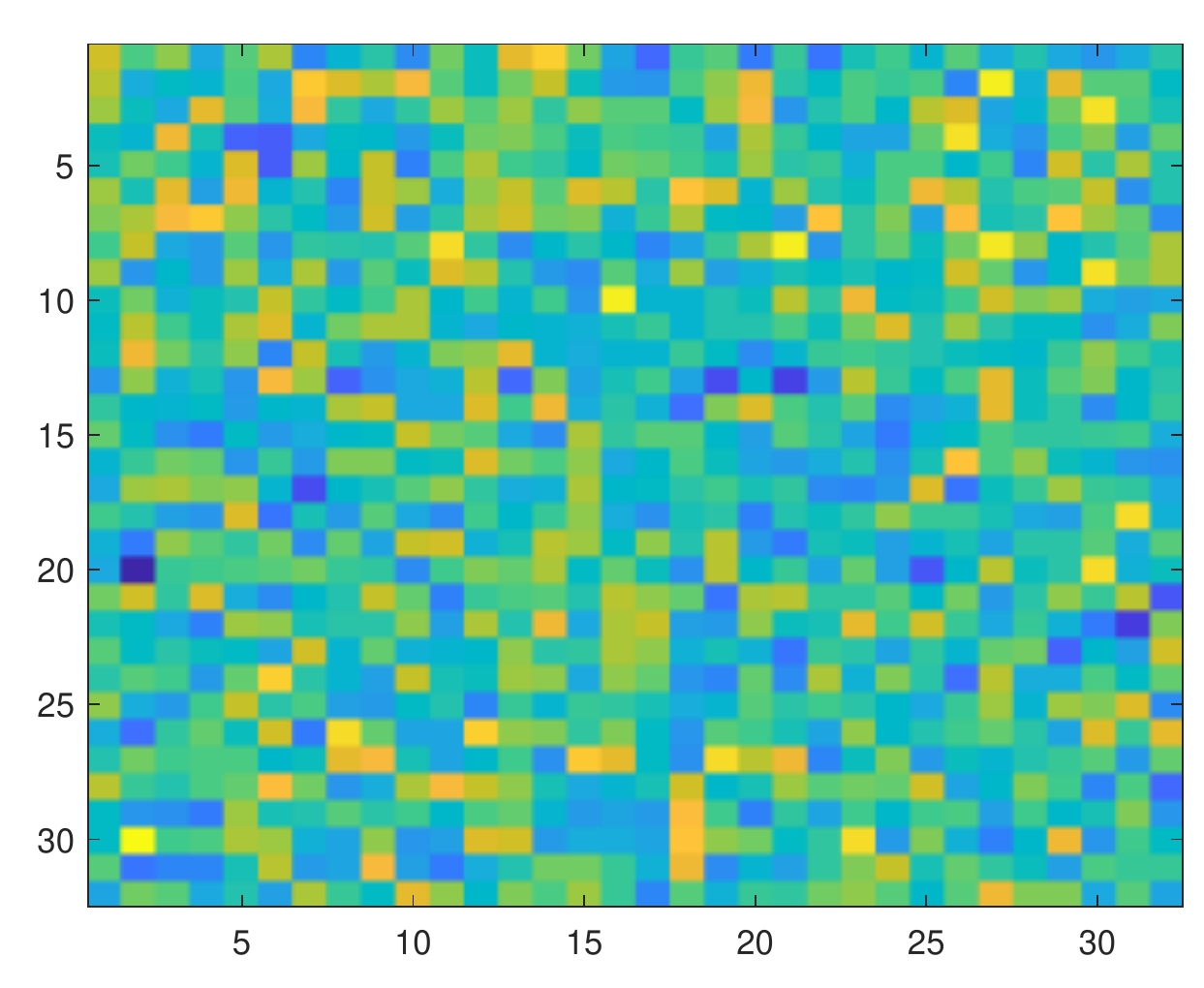}
\hfill
\includegraphics[trim=125pt 250pt 130pt 250pt, clip,scale=0.23]{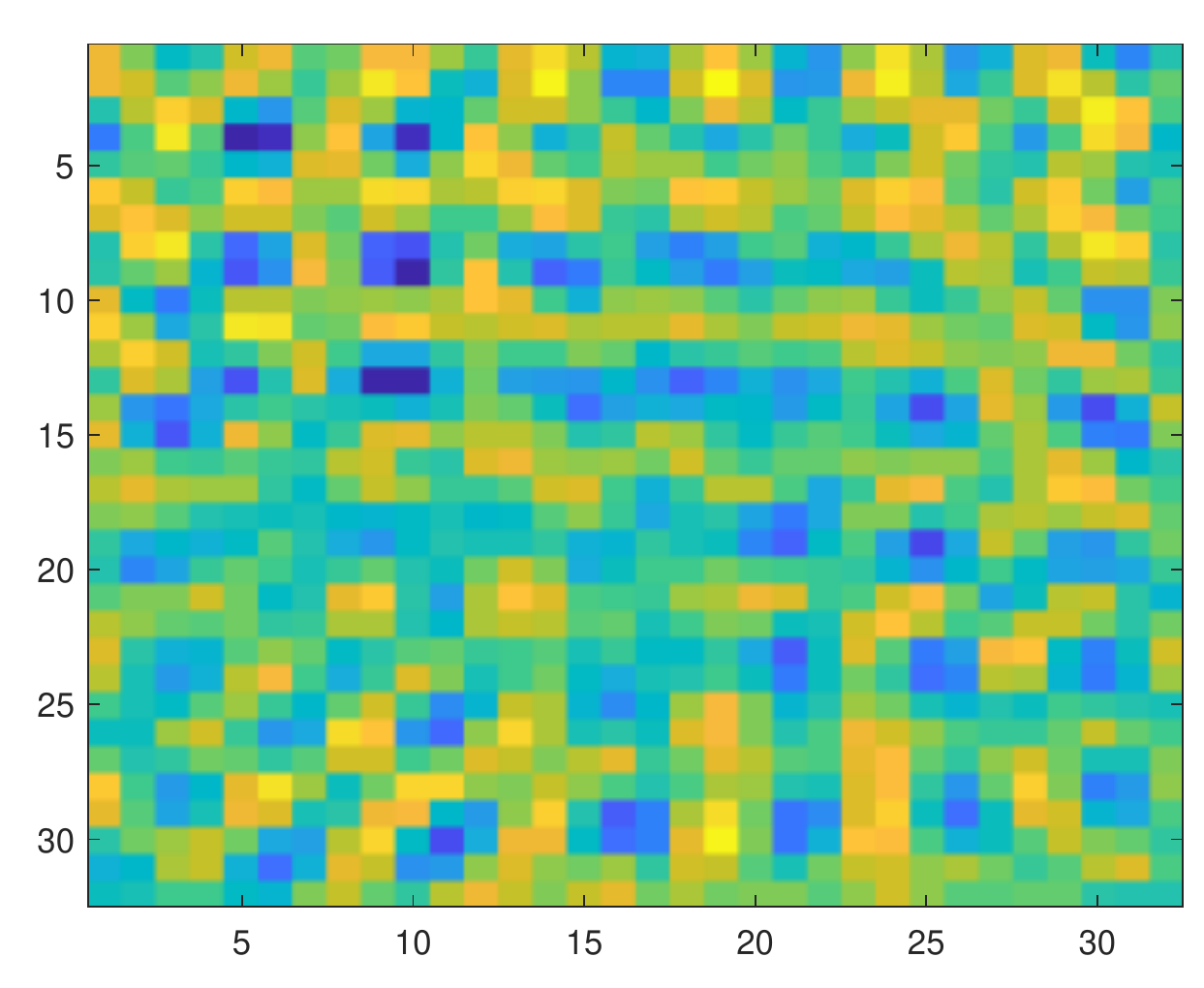}
\hfill
\includegraphics[trim=125pt 250pt 130pt 250pt, clip,scale=0.23]{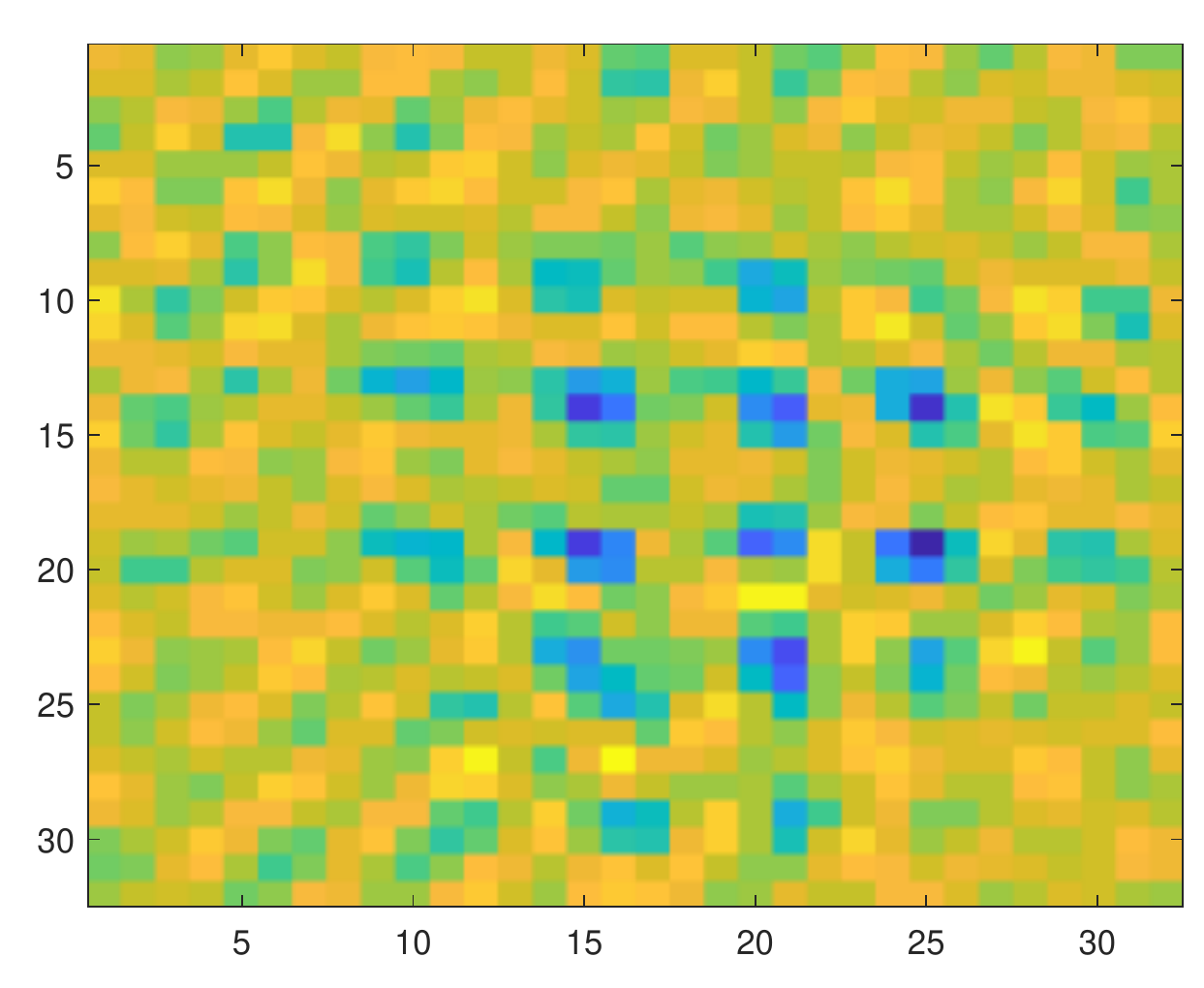}
\hfill
\includegraphics[trim=125pt 250pt 130pt 250pt, clip,scale=0.23]{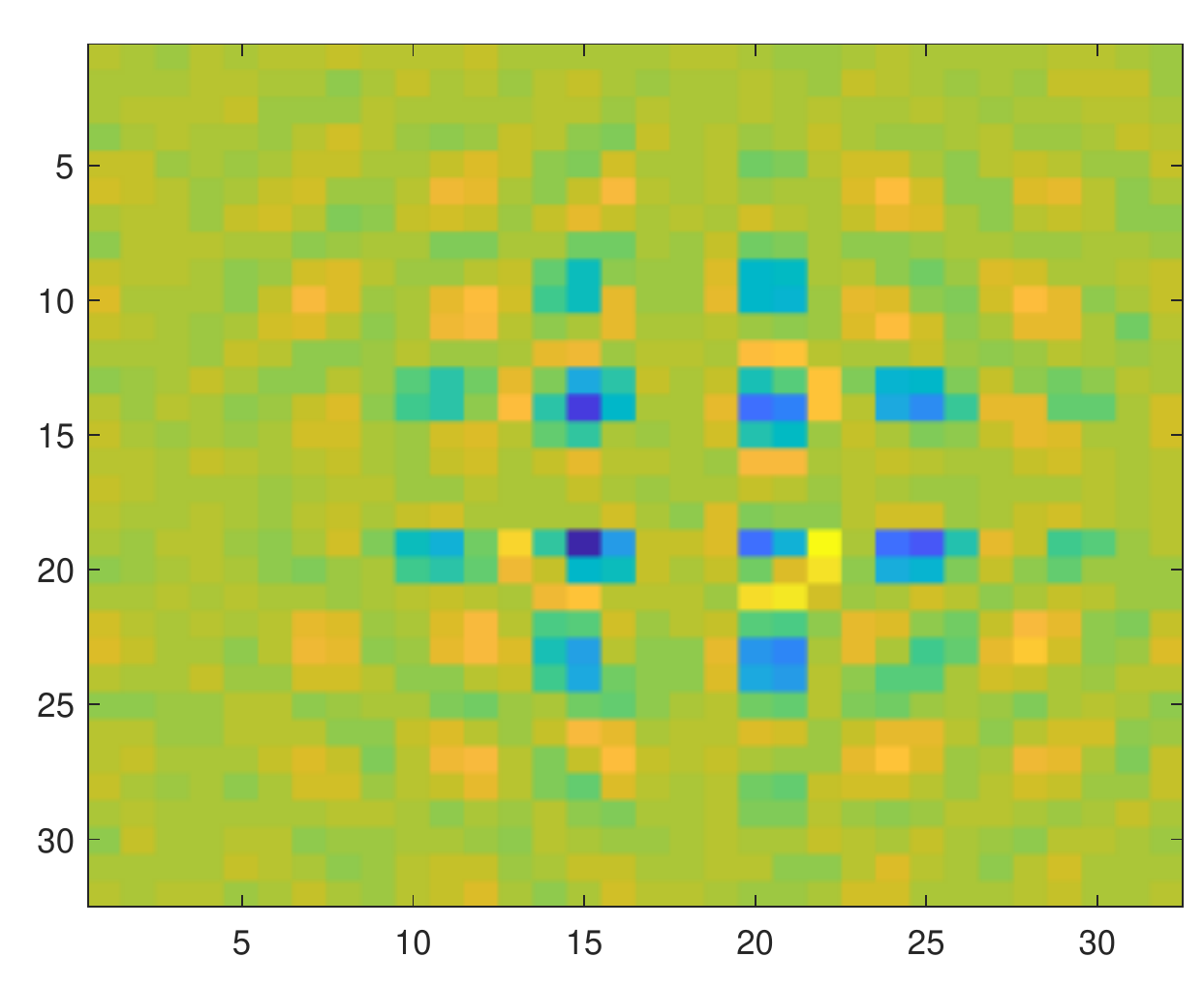}
\hfill
\includegraphics[trim=125pt 250pt 130pt 250pt, clip,scale=0.23]{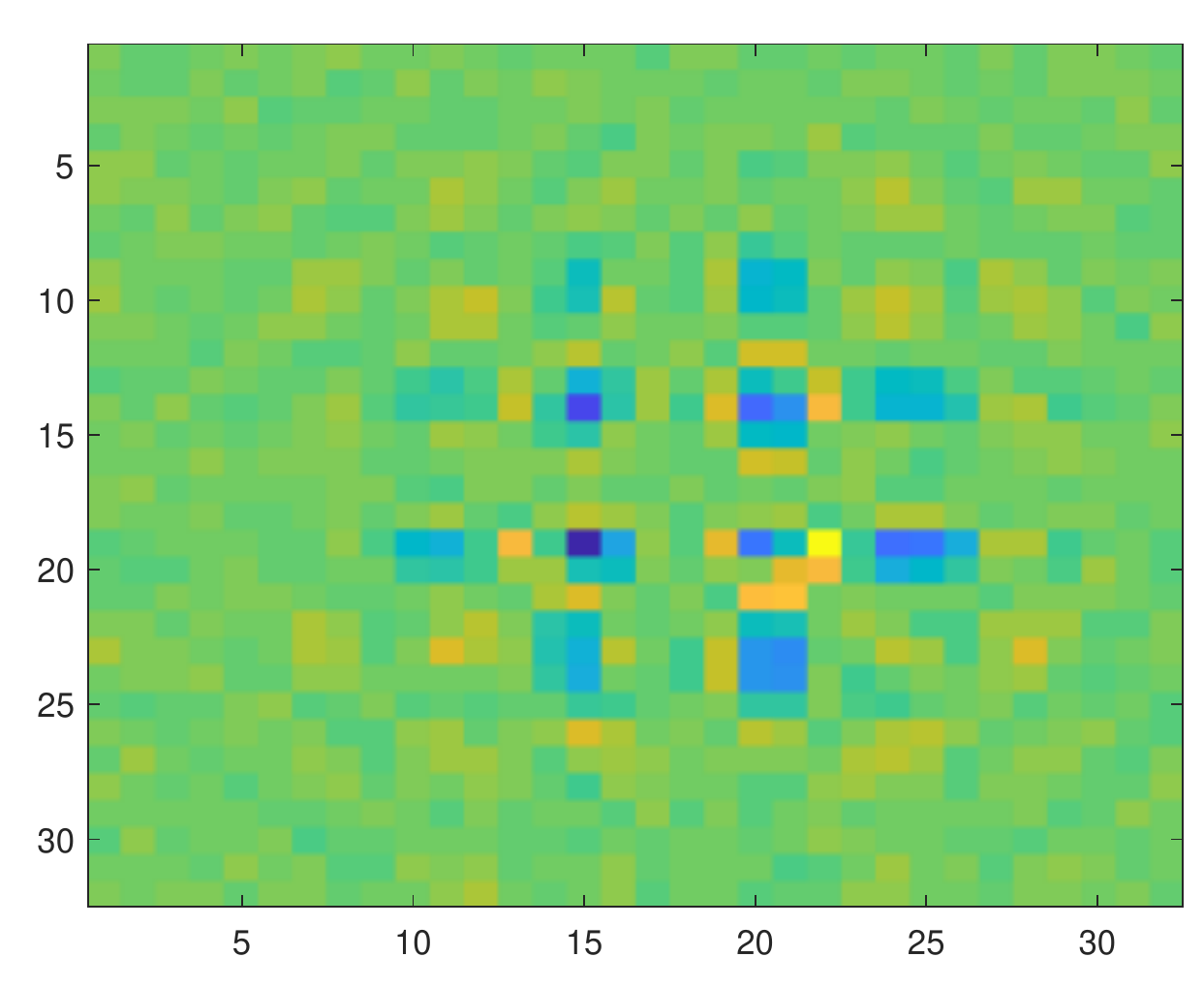}
\caption{Kernel matrices $A$ for \textsc{ADMM-exact} at iterations $k = 0, 10, 50, 100, 200$.}
\end{subfigure}
\caption{\textsc{ADMM-slack} and \textsc{ADMM-exact} on noisy observations.}
\label{fig:exp_ker27_noise}
\end{figure}

We observe in \Cref{fig:exp_ker27_noise} that while \textsc{ADMM-slack} converges even in the presence of noise, \textsc{ADMM-exact} eventually begins to diverge. Since the constraint cannot be satisfied exactly, the Lagrange multipliers grow over time, causing $A, X$ to behave erratically to overfit to the noise and decreasing the relative importance of the sparsity of $X$. This suggests that for practical applications, where a certain amount of noise is unavoidable, the formulation (SBD1) with the slack variable $Z$ is both more practical, as well as being provably convergent within our framework.

\Cref{fig:exp_ker5_noise,fig:exp_ker23_noise} show additional examples in the noisy setting.
 
\begin{figure}
\centering
\begin{subfigure}[t]{0.2\textwidth}
	\includegraphics[trim=135pt 250pt 130pt 250pt, clip,scale=0.25]{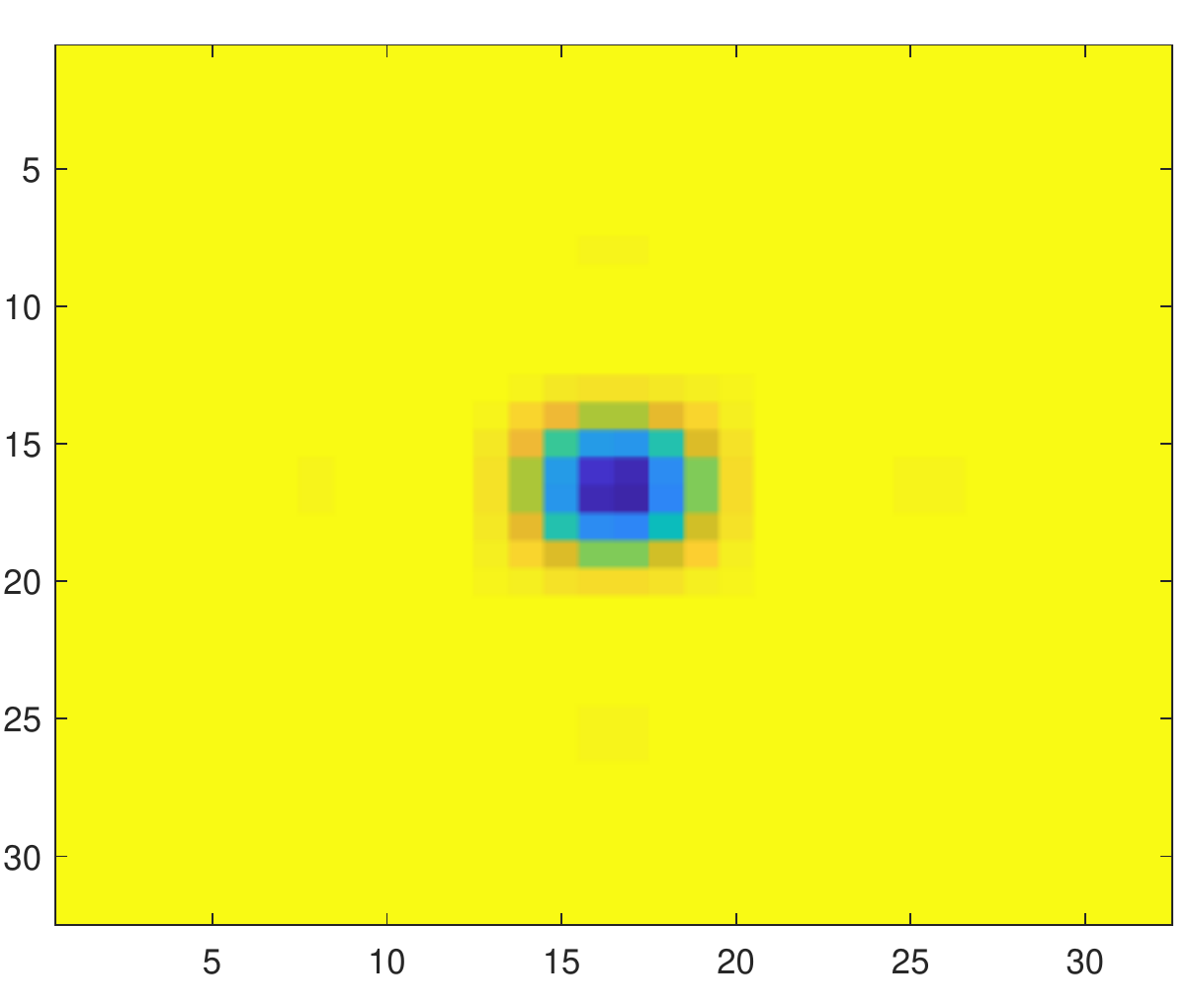}
	\caption{True Kernel $A^\ast$}
\end{subfigure}
\begin{subfigure}[t]{0.2\textwidth}
	\includegraphics[trim=125pt 250pt 130pt 250pt, clip,scale=0.25]{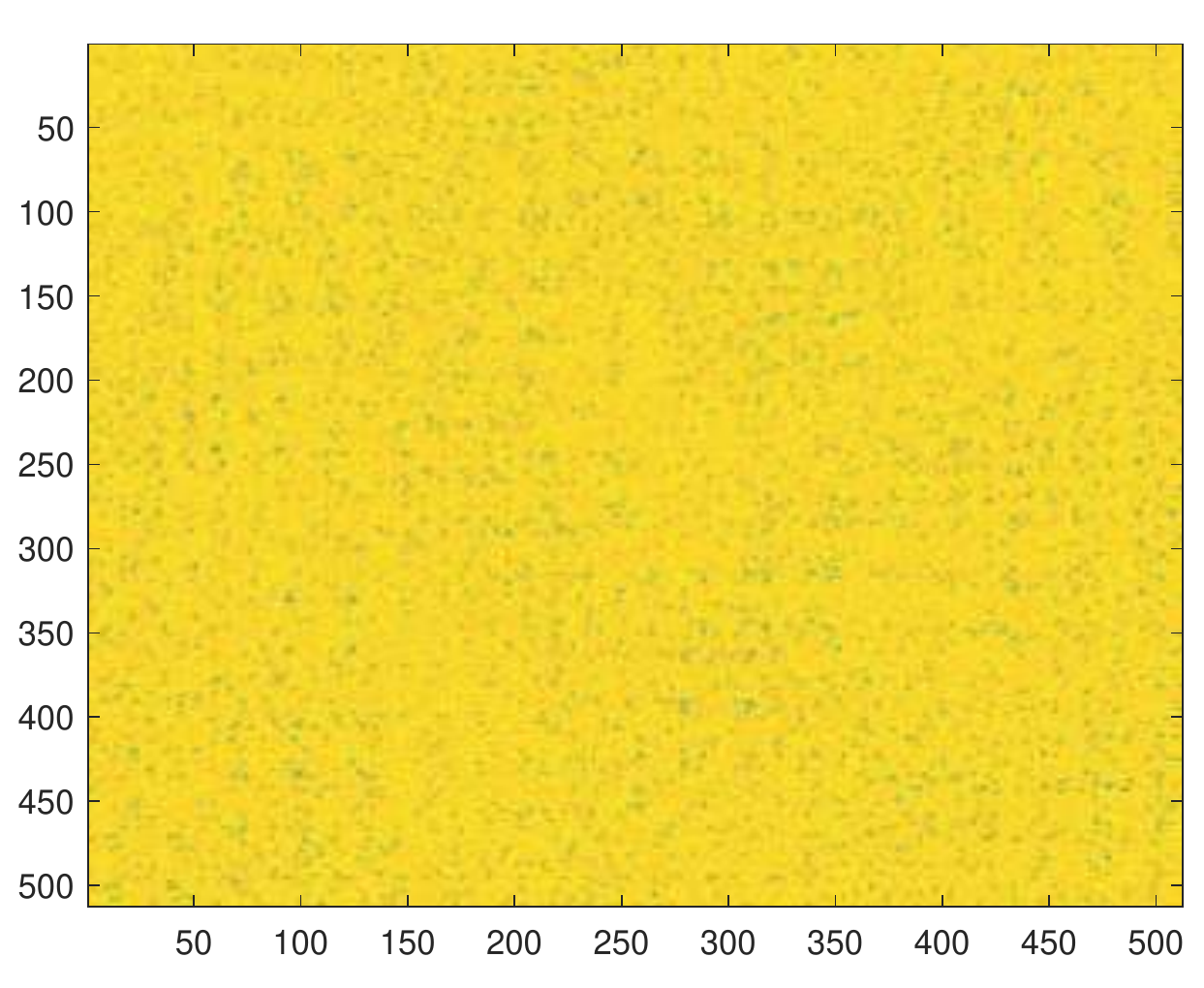}
	\caption{Observations $Y$ (with noise)}
\end{subfigure}
\begin{subfigure}[t]{0.5\textwidth}
	\includegraphics[trim=125pt 250pt 130pt 250pt, clip,scale=0.3]{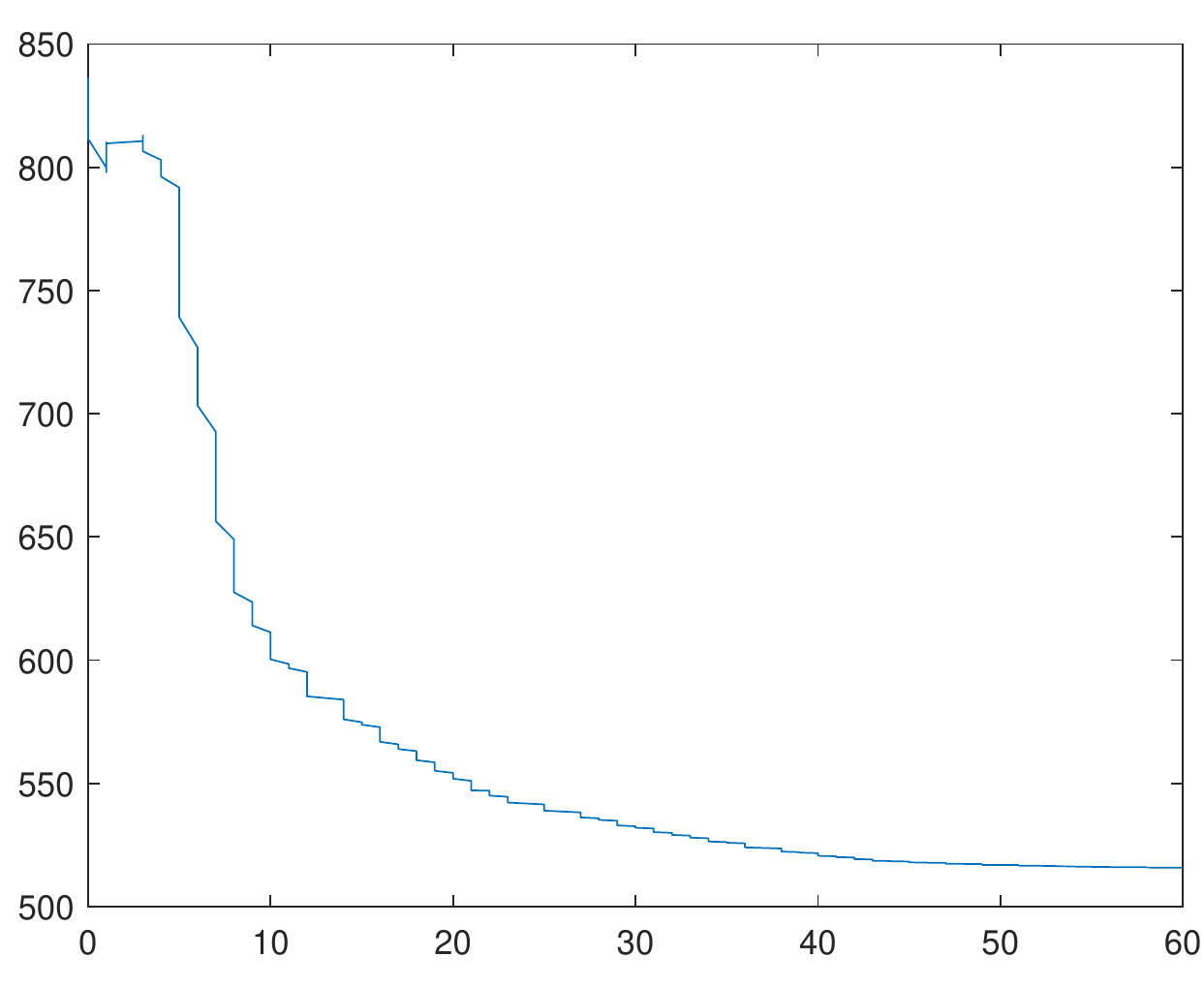}
	\hfill
	\includegraphics[trim=125pt 250pt 130pt 250pt, clip,scale=0.3]{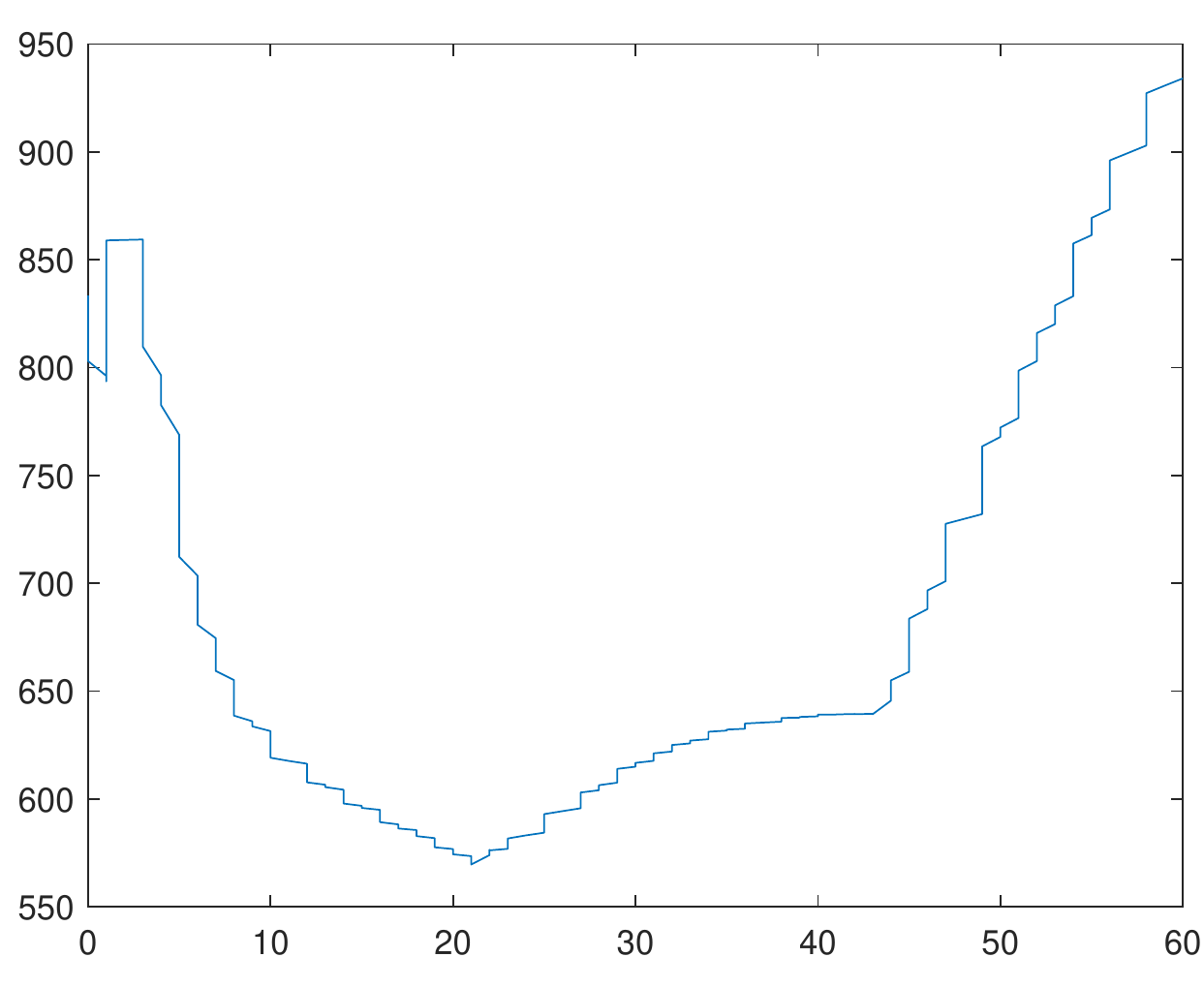}
	\caption{Objective value over time (s). The left plot shows \textsc{ADMM-slack}, and the right shows \textsc{ADMM-exact}.}
\end{subfigure}

\bigskip

\centering
\begin{subfigure}[t]{1.0\textwidth}
	\includegraphics[trim=125pt 250pt 130pt 250pt, clip,scale=0.23]{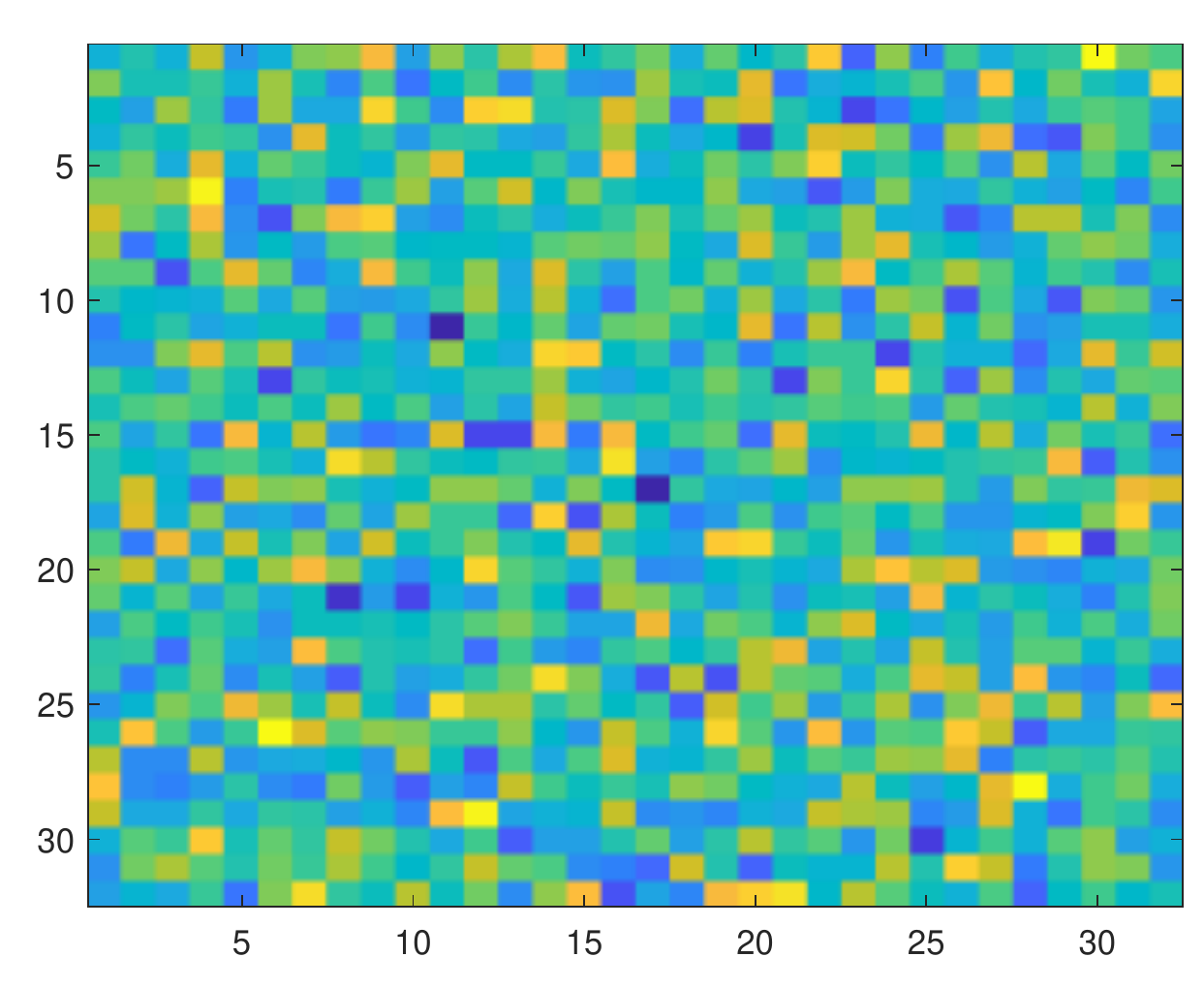}
	\hfill
	\includegraphics[trim=125pt 250pt 130pt 250pt, clip,scale=0.23]{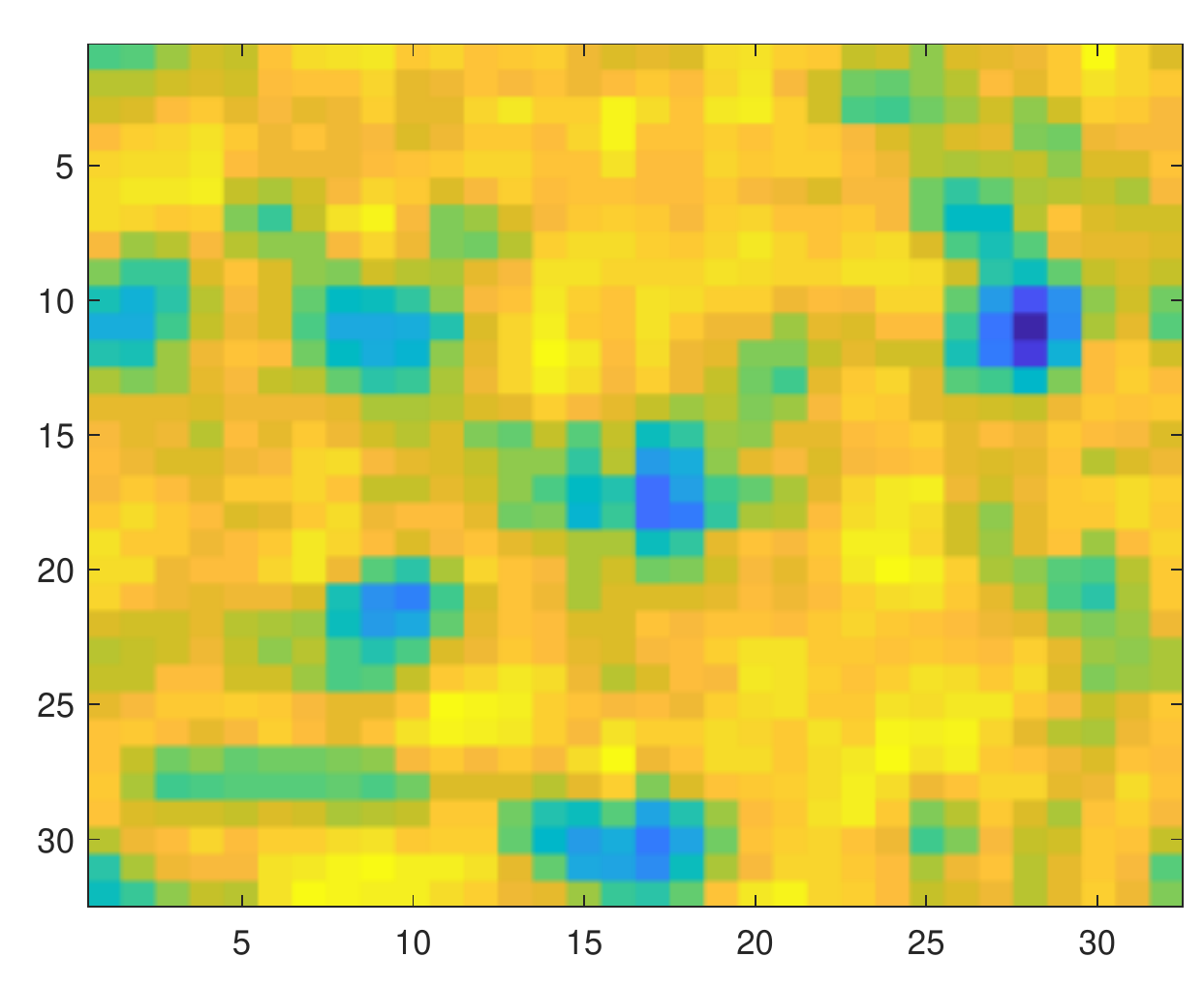}
	\hfill
	\includegraphics[trim=125pt 250pt 130pt 250pt, clip,scale=0.23]{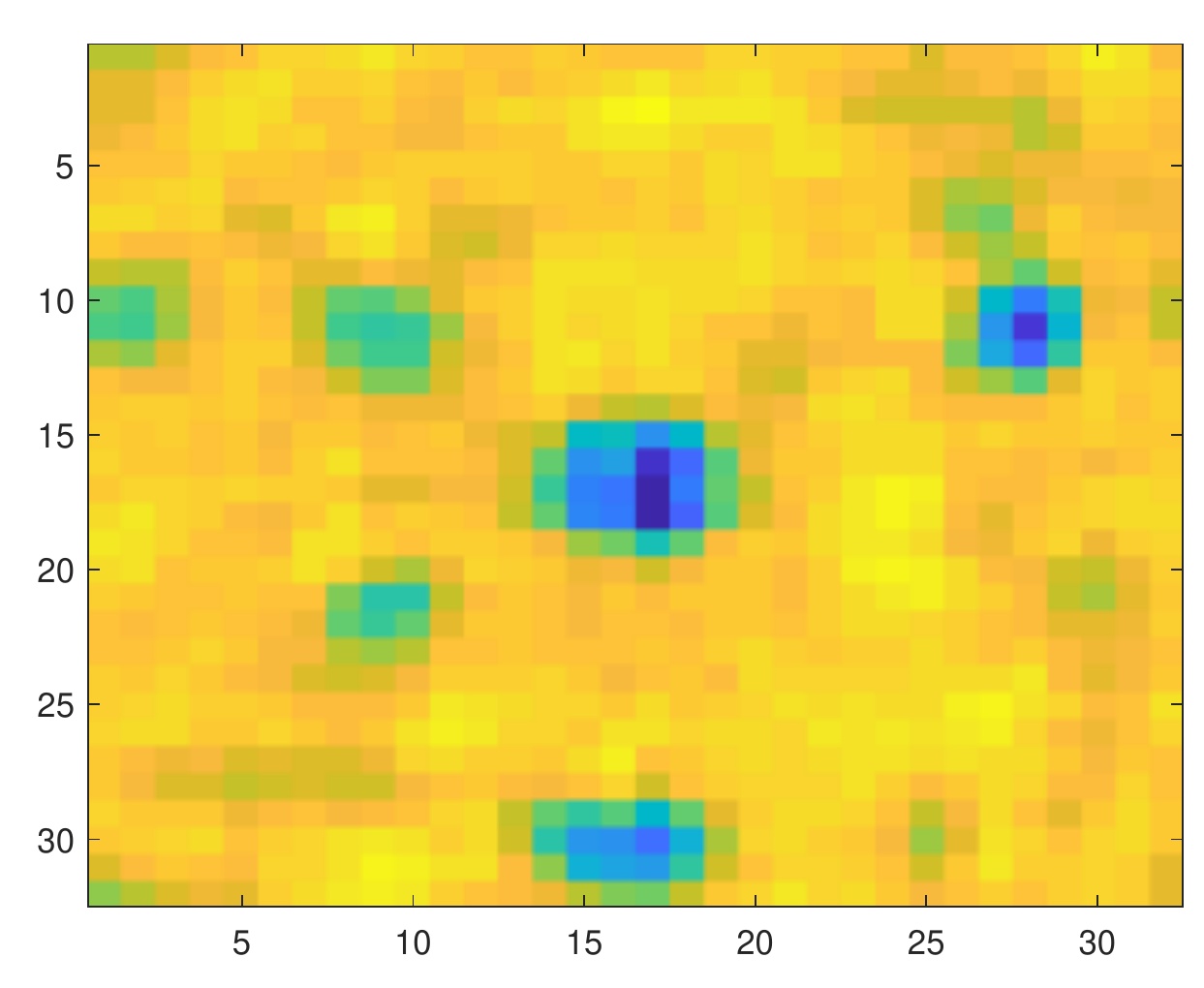}
	\hfill
	\includegraphics[trim=125pt 250pt 130pt 250pt, clip,scale=0.23]{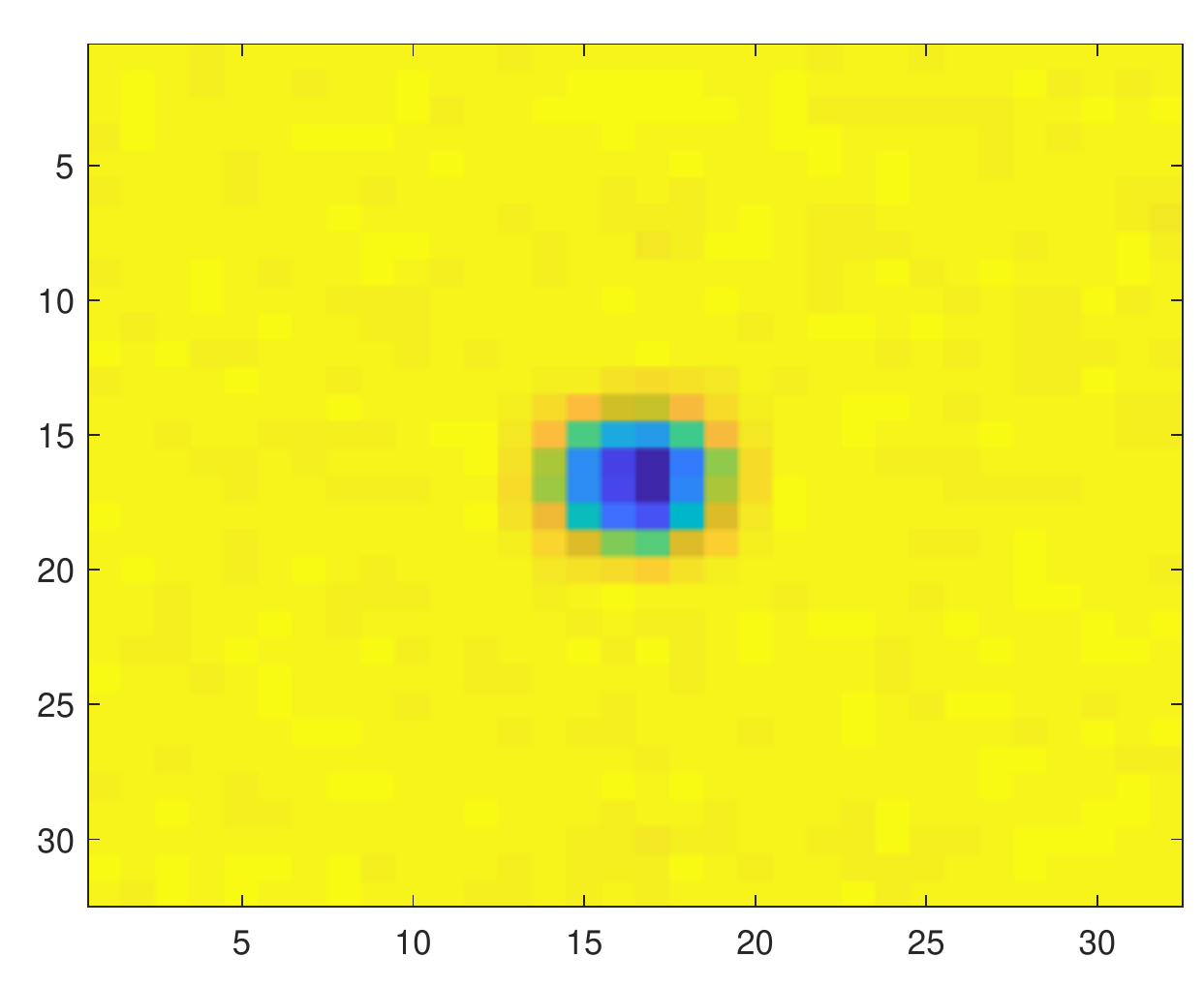}
	\hfill
	\includegraphics[trim=125pt 250pt 130pt 250pt, clip,scale=0.23]{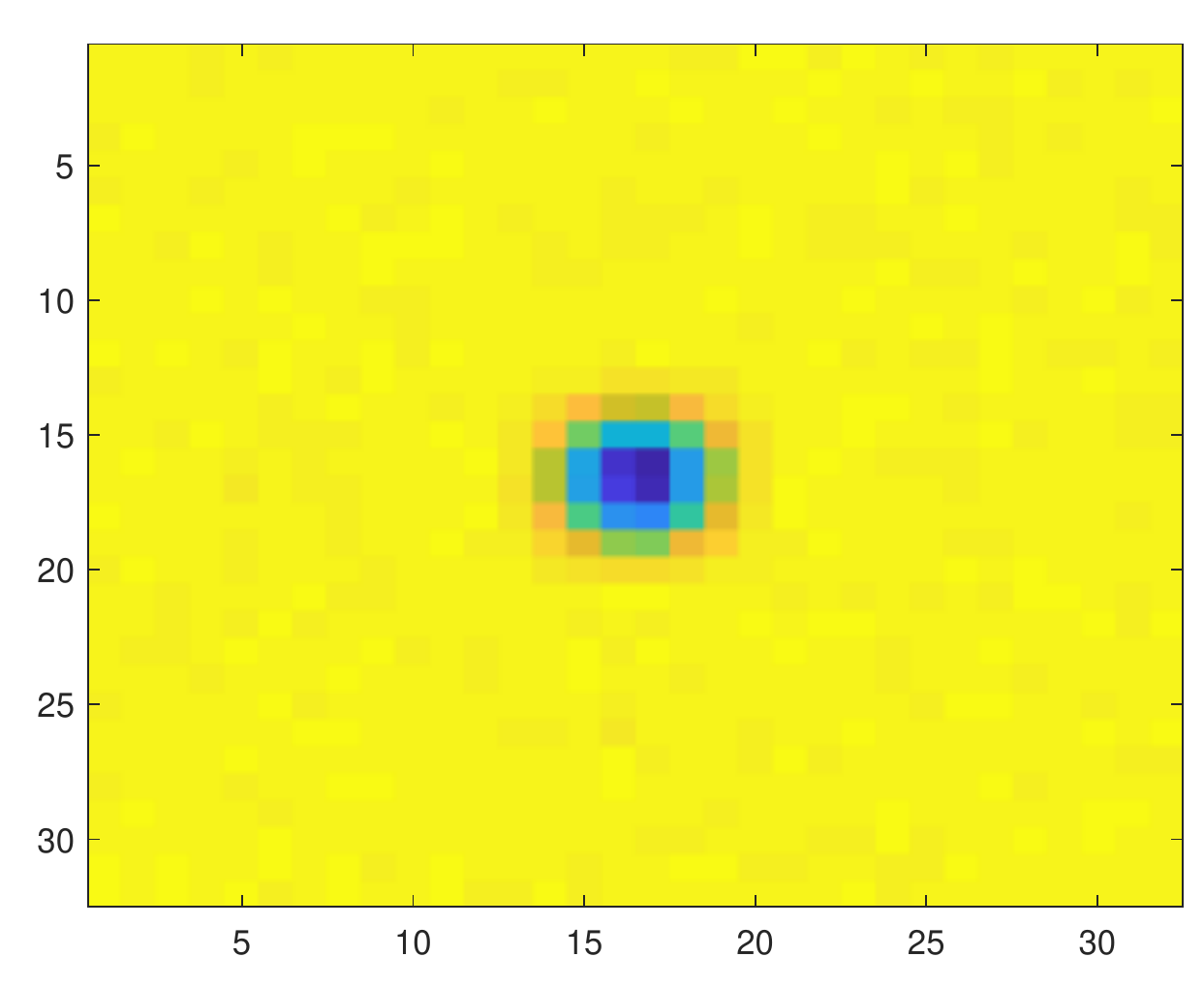}
	\caption{Kernel matrices $A$ for \textsc{ADMM-slack} at iterations $k = 0, 10, 20, 50, 100$.}
\end{subfigure}

\bigskip
\centering
\begin{subfigure}[t]{1.0\textwidth}
	\includegraphics[trim=125pt 250pt 130pt 250pt, clip,scale=0.23]{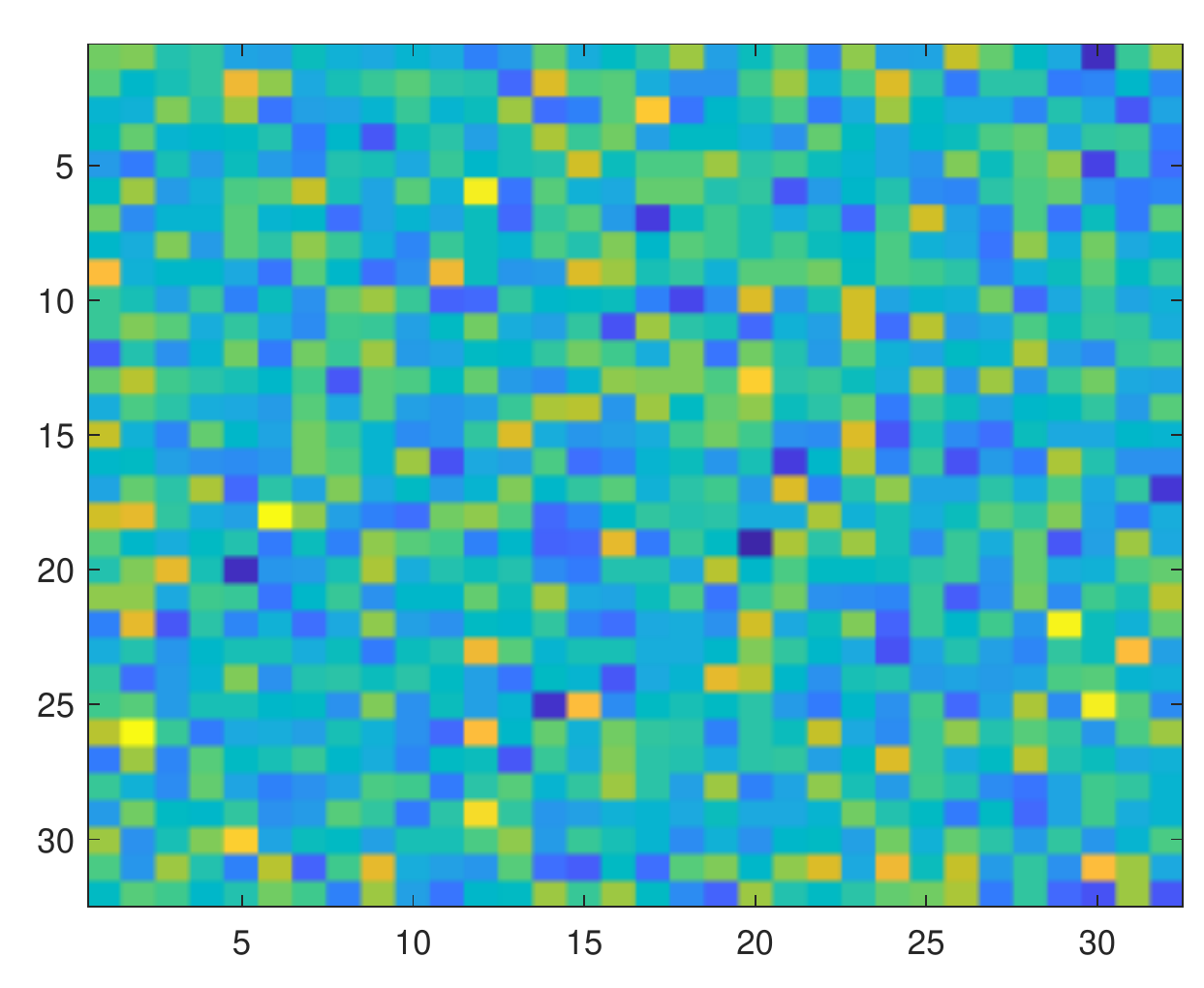}
	\hfill
	\includegraphics[trim=125pt 250pt 130pt 250pt, clip,scale=0.23]{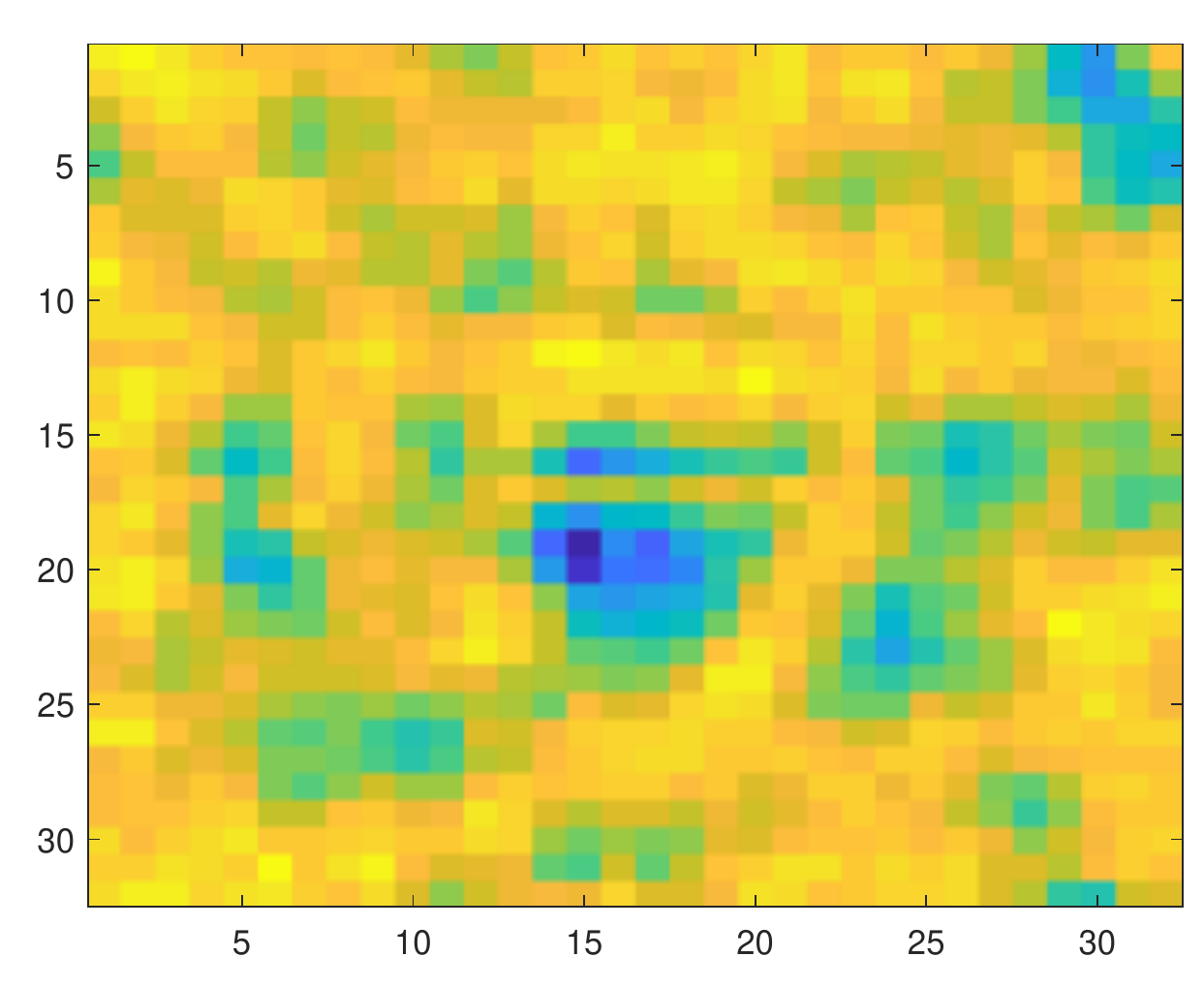}
	\hfill
	\includegraphics[trim=125pt 250pt 130pt 250pt, clip,scale=0.23]{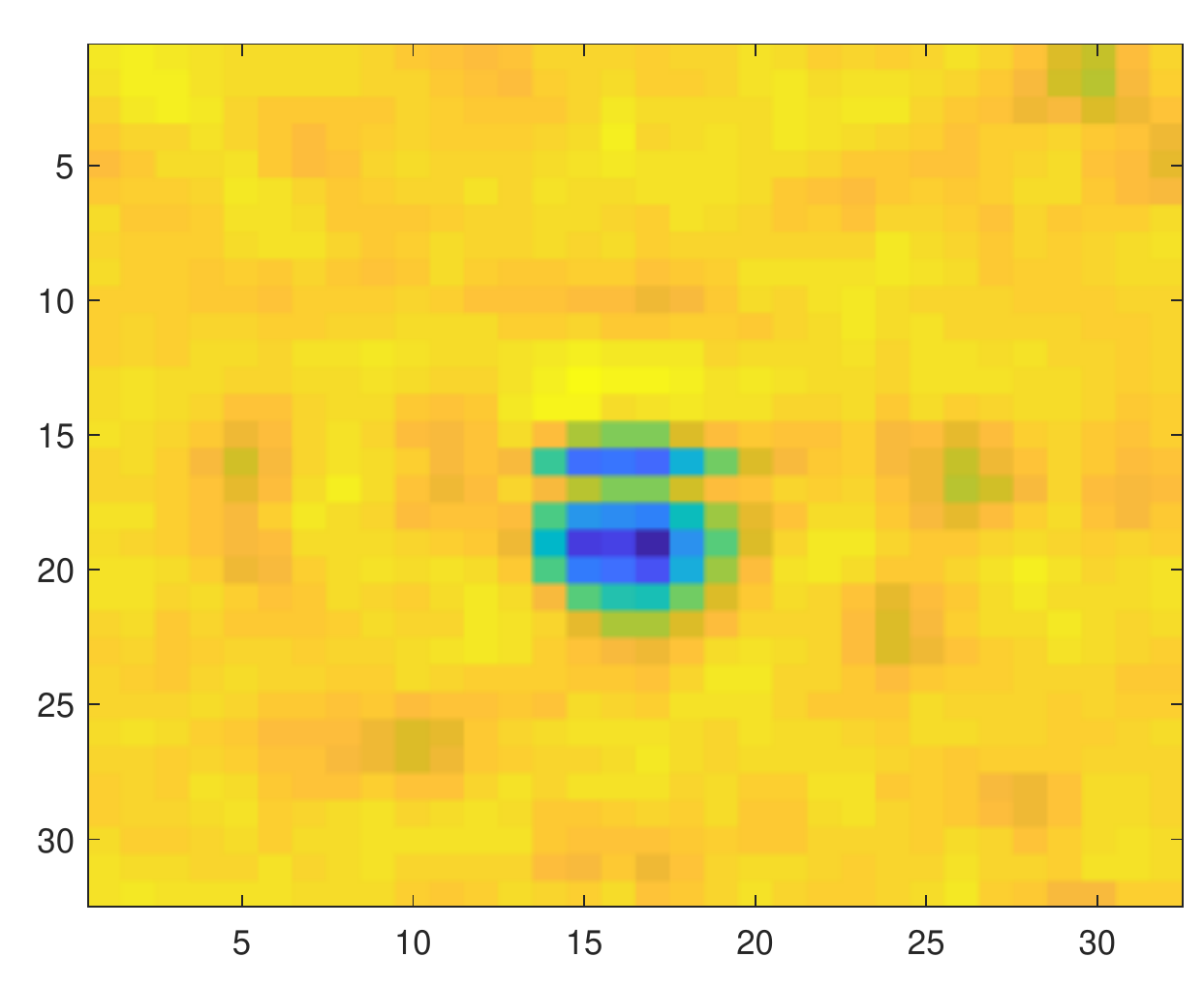}
	\hfill
	\includegraphics[trim=125pt 250pt 130pt 250pt, clip,scale=0.23]{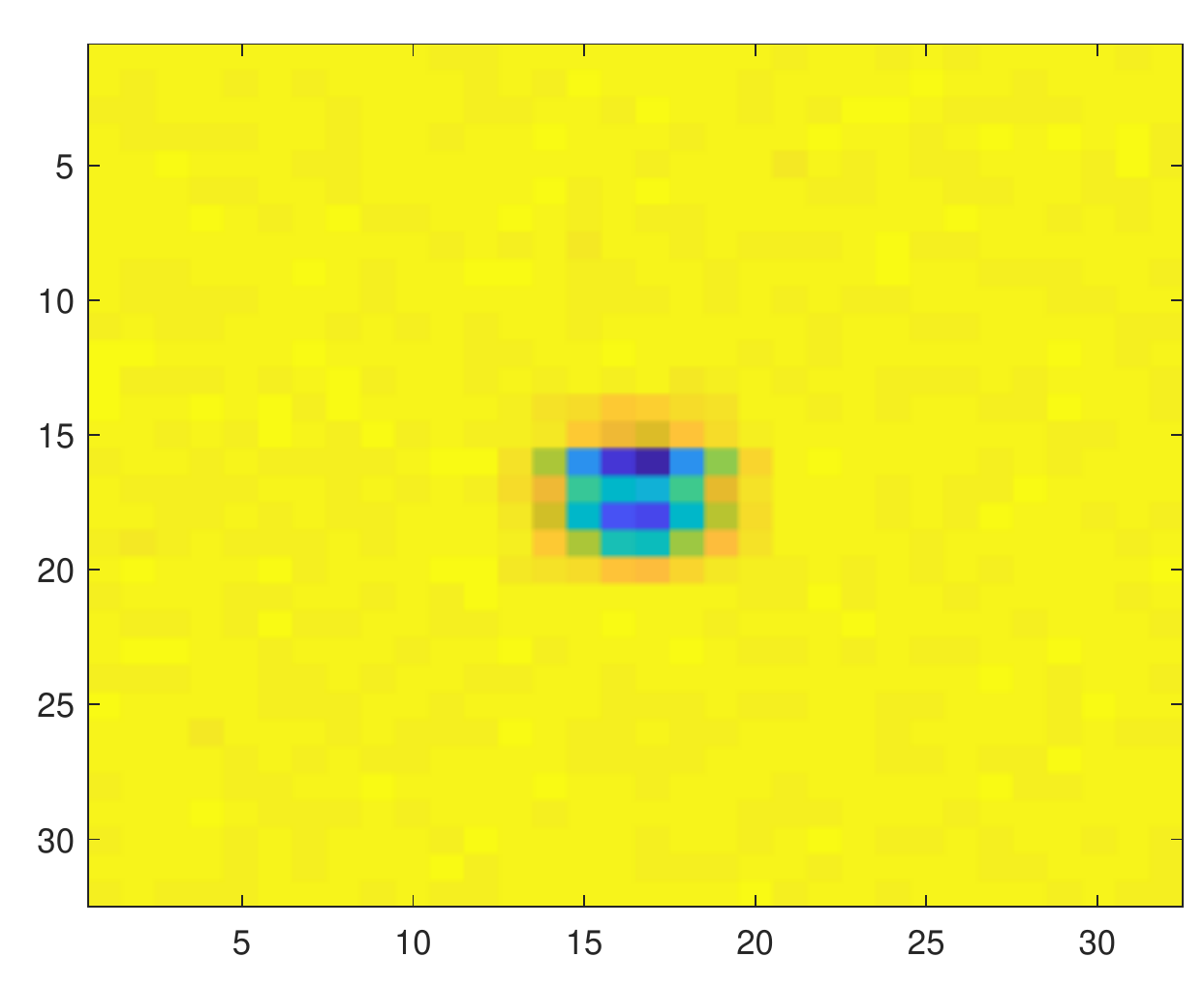}
	\hfill
	\includegraphics[trim=125pt 250pt 130pt 250pt, clip,scale=0.23]{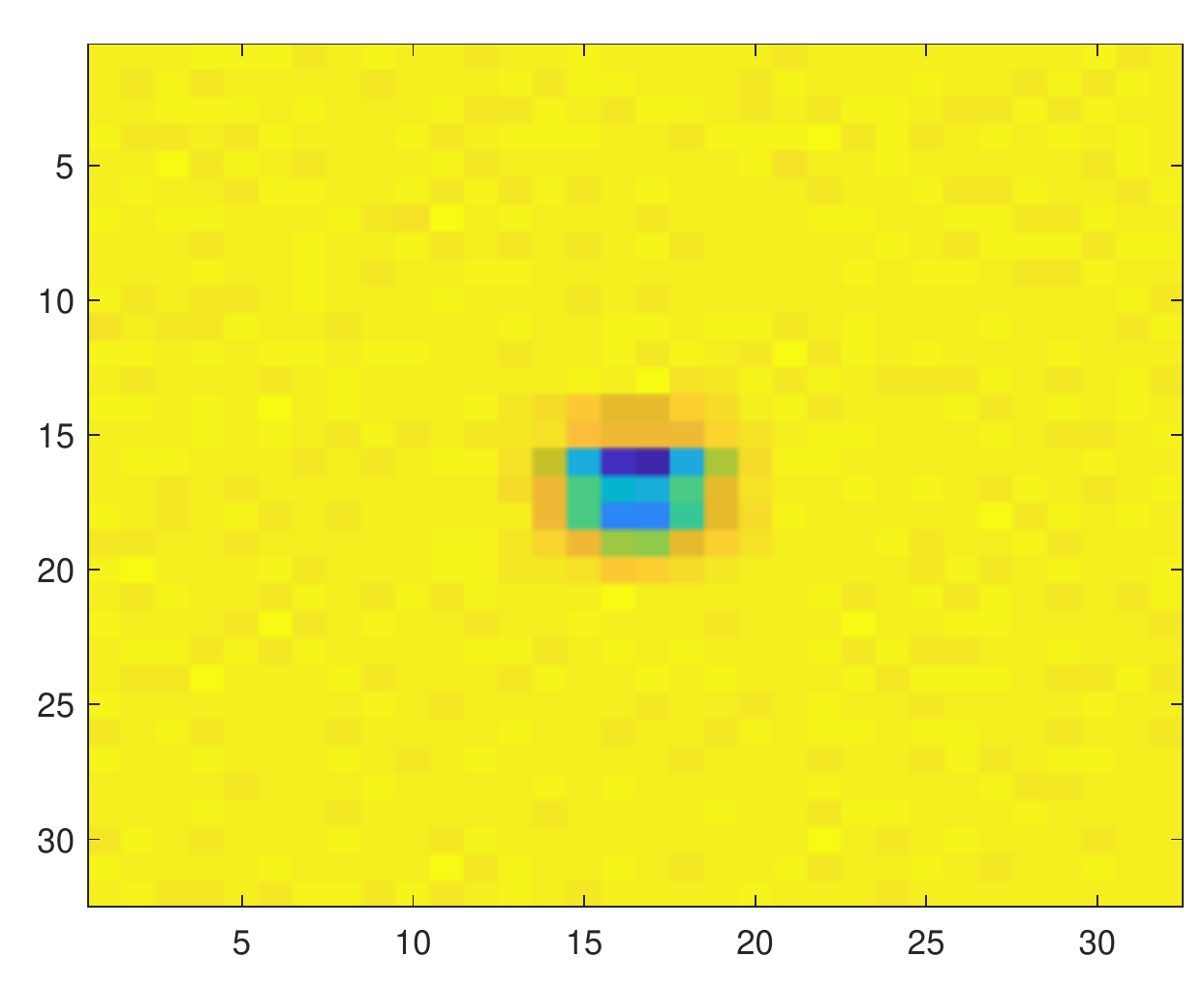}
	\caption{Kernel matrices $A$ for \textsc{ADMM-exact} at iterations $k = 0, 10, 20, 100, 200$.}
\end{subfigure}
\caption{\textsc{ADMM-slack} and \textsc{ADMM-exact} on noisy observations.}
\label{fig:exp_ker5_noise}
\end{figure}

\begin{figure}
\centering
\begin{subfigure}[t]{0.2\textwidth}
\includegraphics[trim=135pt 250pt 130pt 250pt, clip,scale=0.25]{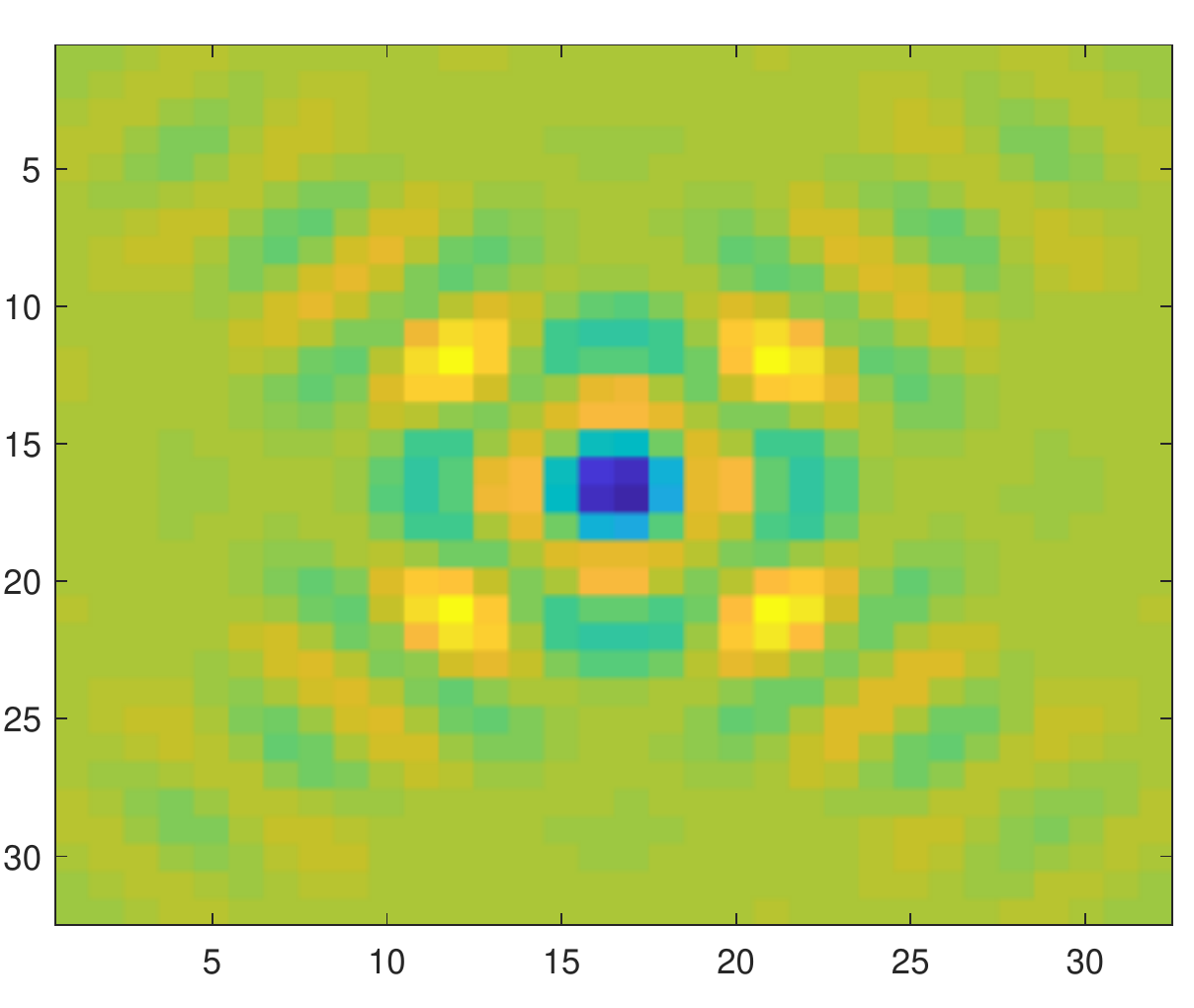}
\caption{True Kernel $A^\ast$}
\end{subfigure}
\begin{subfigure}[t]{0.2\textwidth}
\includegraphics[trim=125pt 250pt 130pt 250pt, clip,scale=0.25]{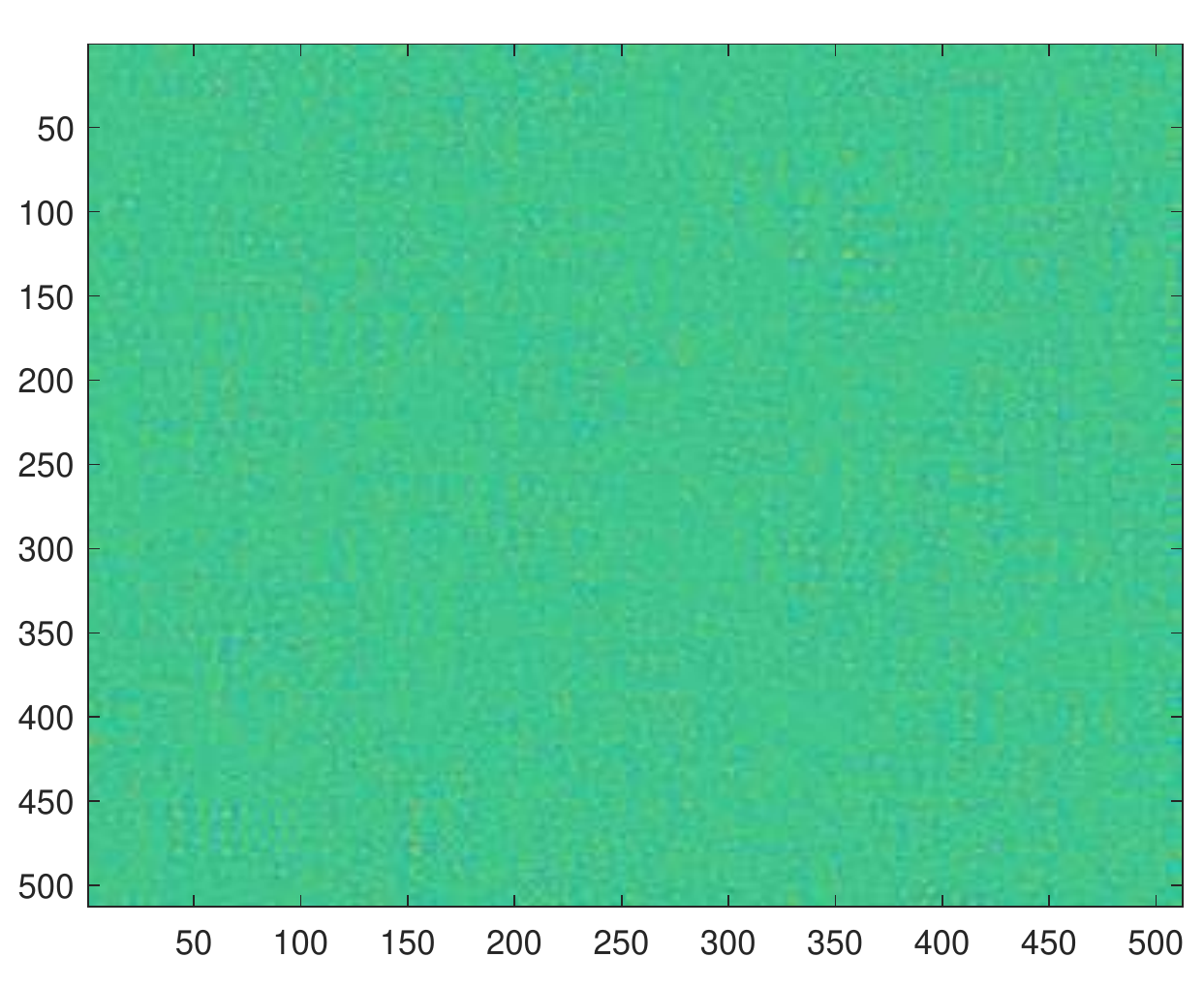}
\caption{Observations $Y$ (with noise)}
\end{subfigure}
\begin{subfigure}[t]{0.5\textwidth}
\includegraphics[trim=125pt 250pt 130pt 250pt, clip,scale=0.3]{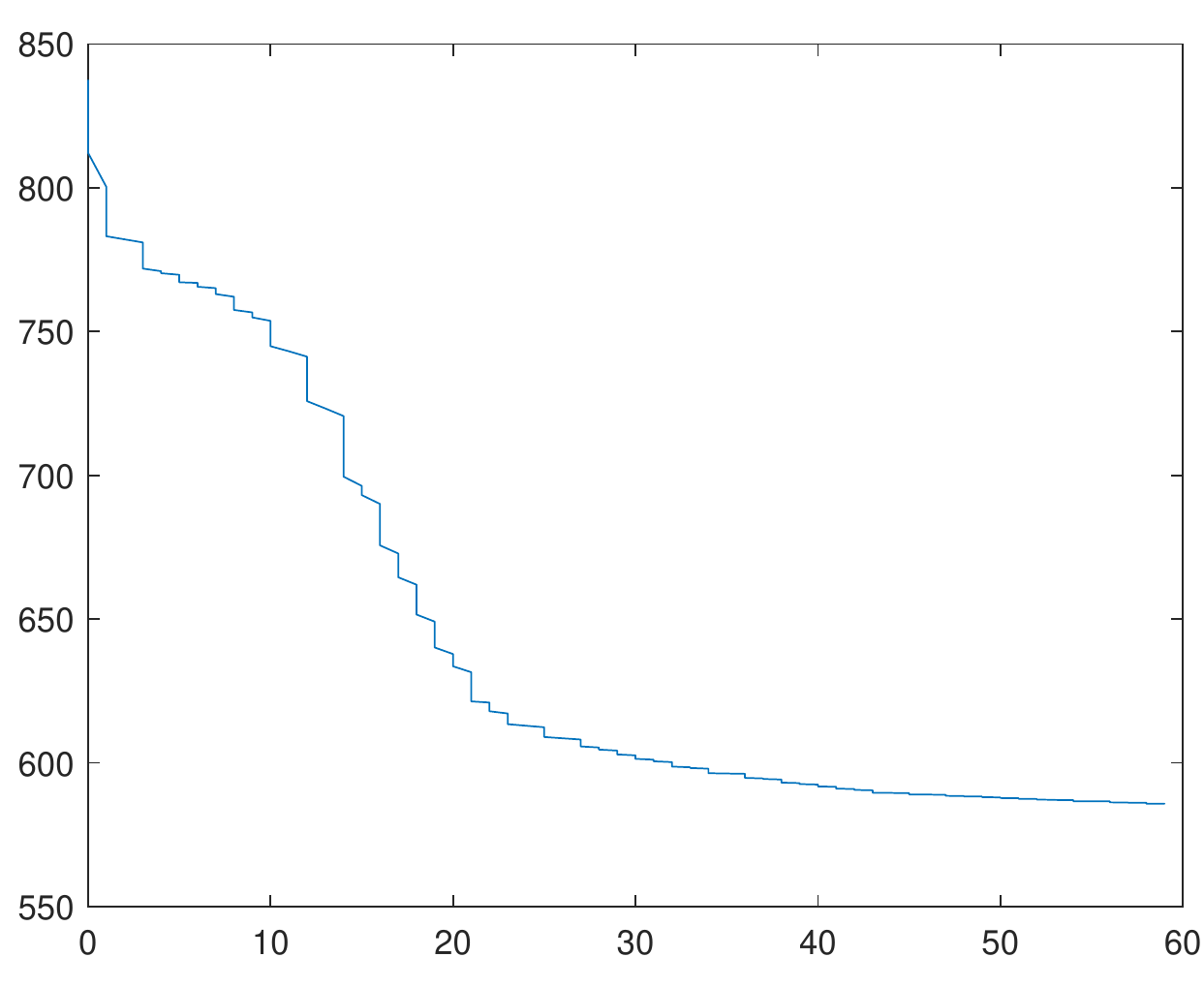}
\hfill
\includegraphics[trim=125pt 250pt 130pt 250pt, clip,scale=0.3]{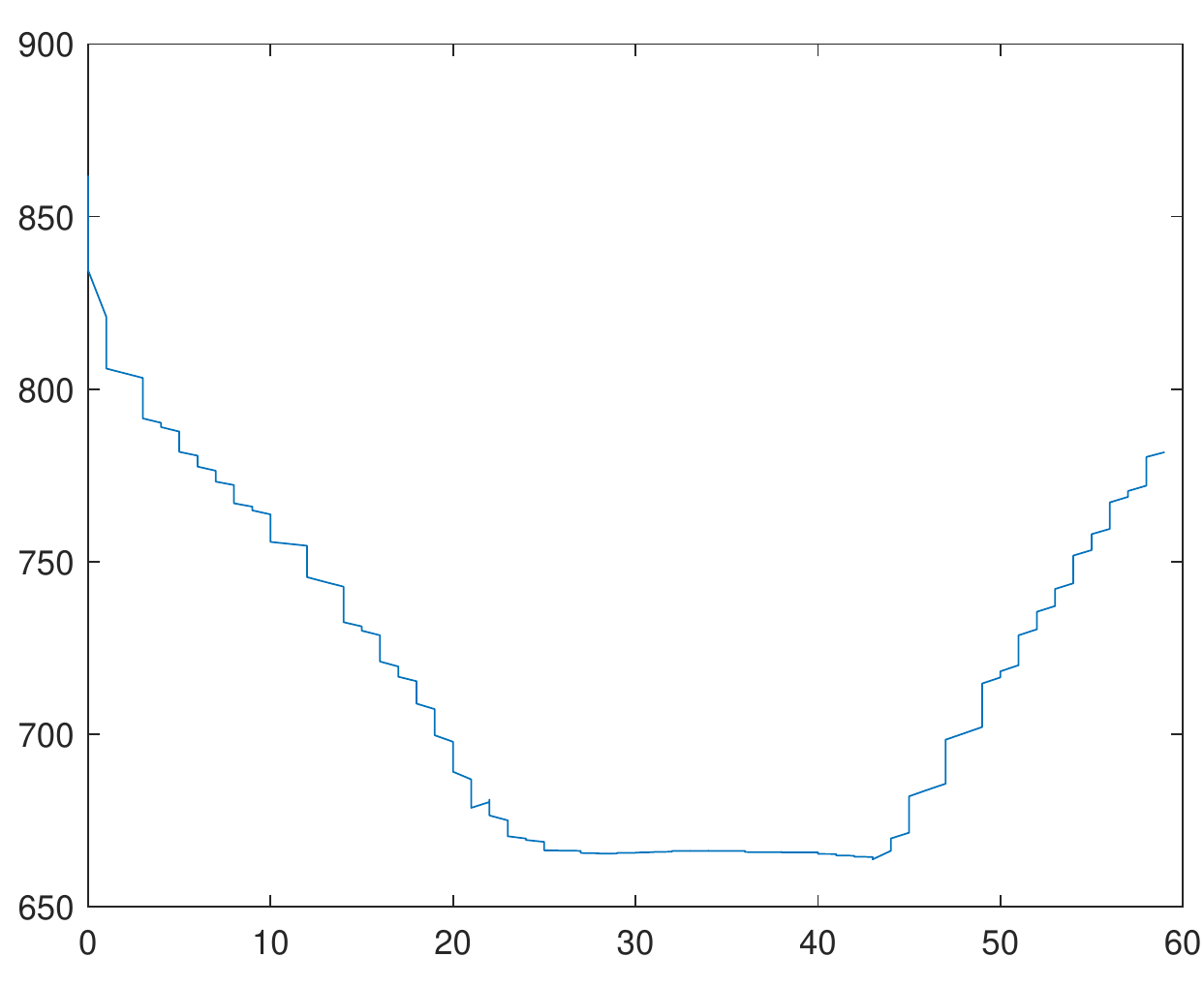}
\caption{Objective value over time (s). The left plot shows \textsc{ADMM-slack}, and the right shows \textsc{ADMM-exact}.}
\end{subfigure}

\bigskip

\centering
\begin{subfigure}[t]{1.0\textwidth}
\includegraphics[trim=125pt 250pt 130pt 250pt, clip,scale=0.23]{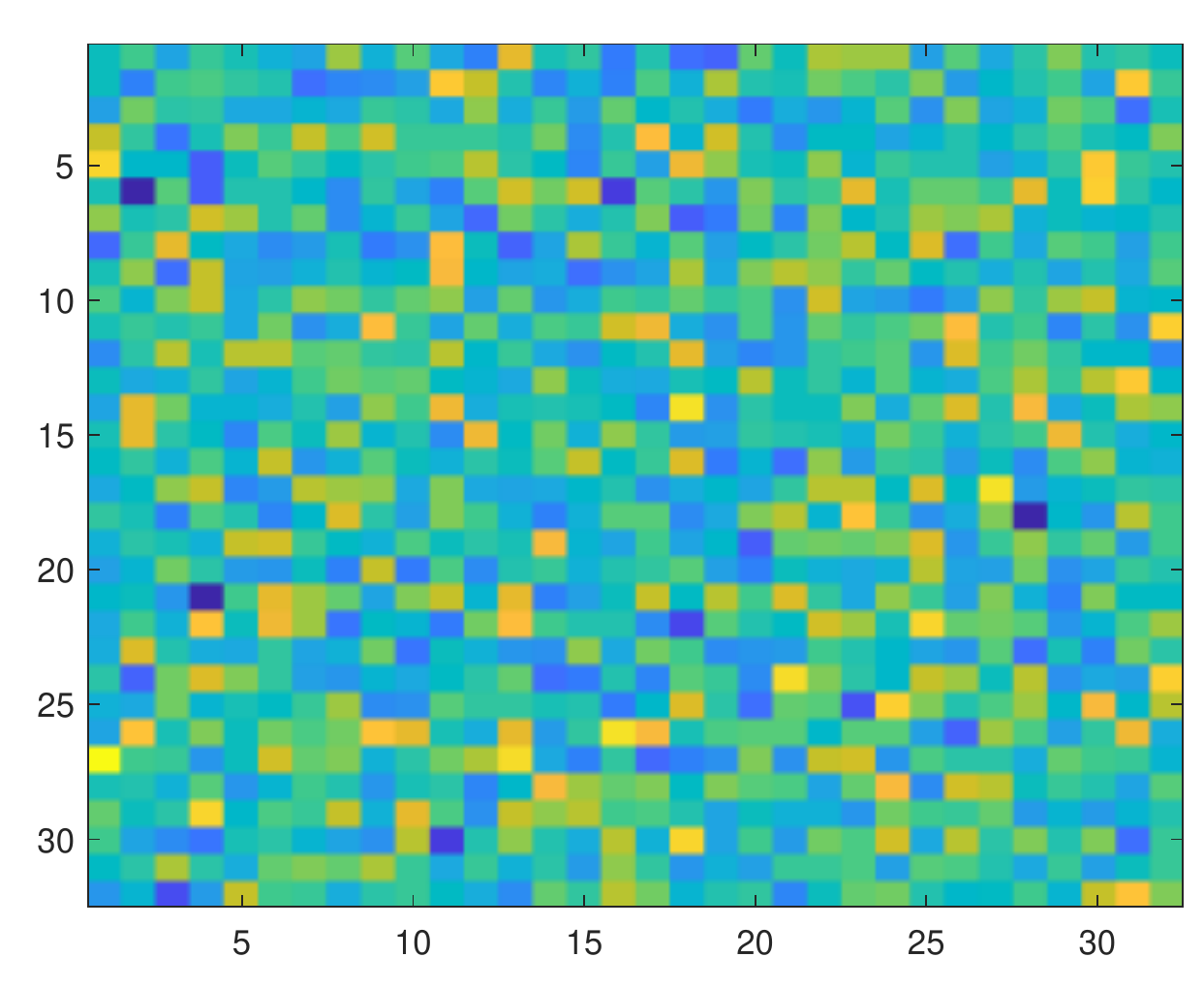}
\hfill
\includegraphics[trim=125pt 250pt 130pt 250pt, clip,scale=0.23]{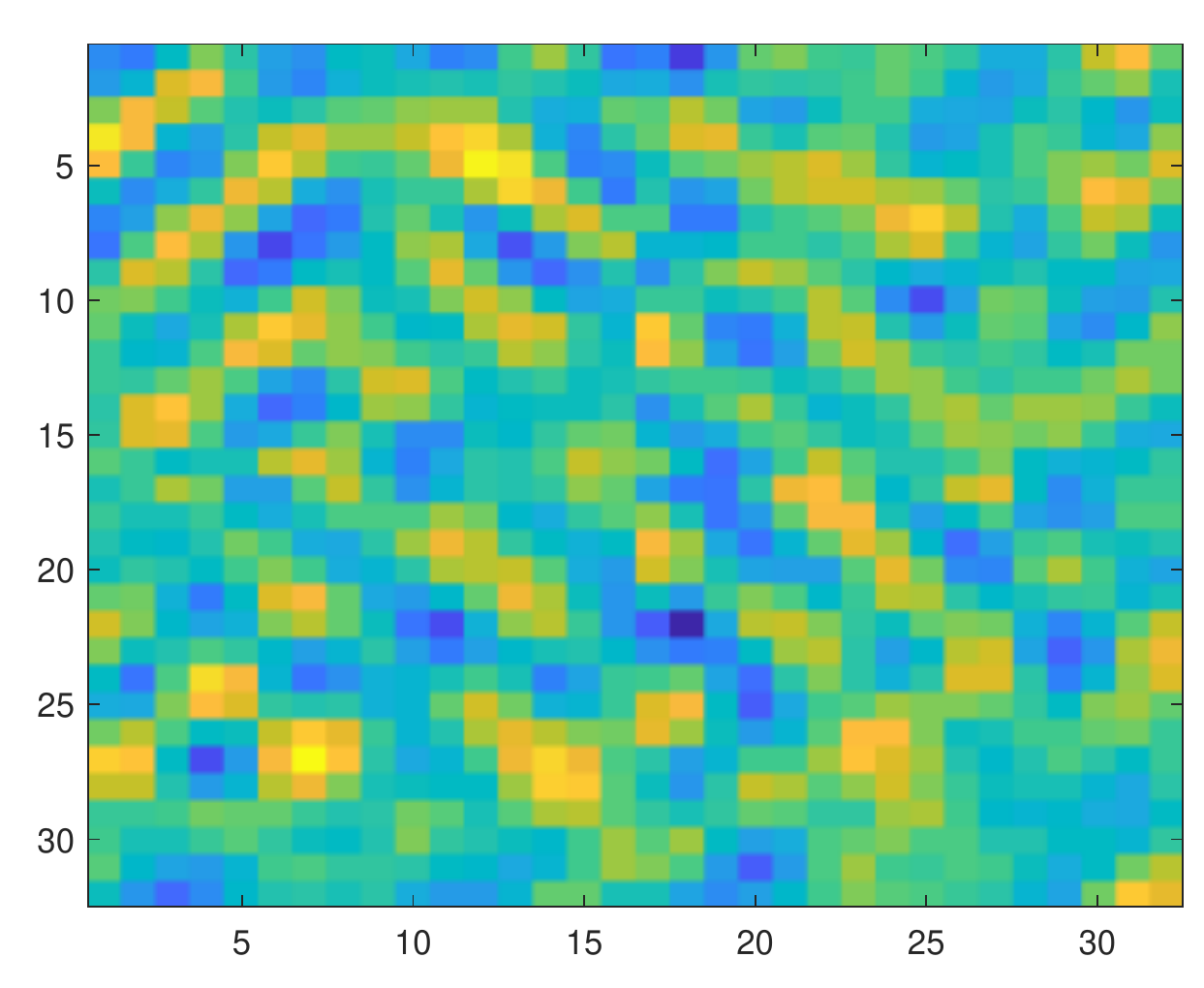}
\hfill
\includegraphics[trim=125pt 250pt 130pt 250pt, clip,scale=0.23]{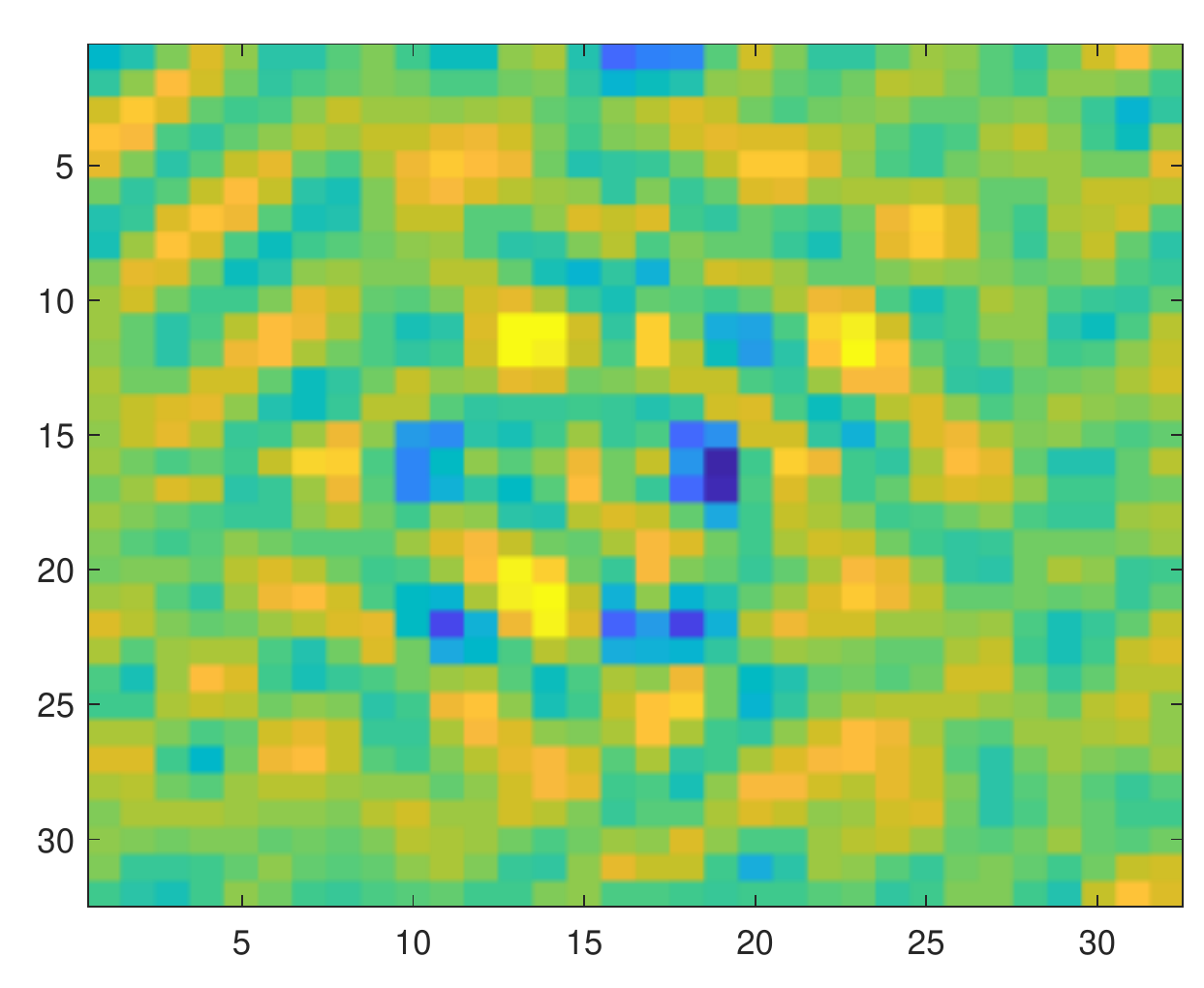}
\hfill
\includegraphics[trim=125pt 250pt 130pt 250pt, clip,scale=0.23]{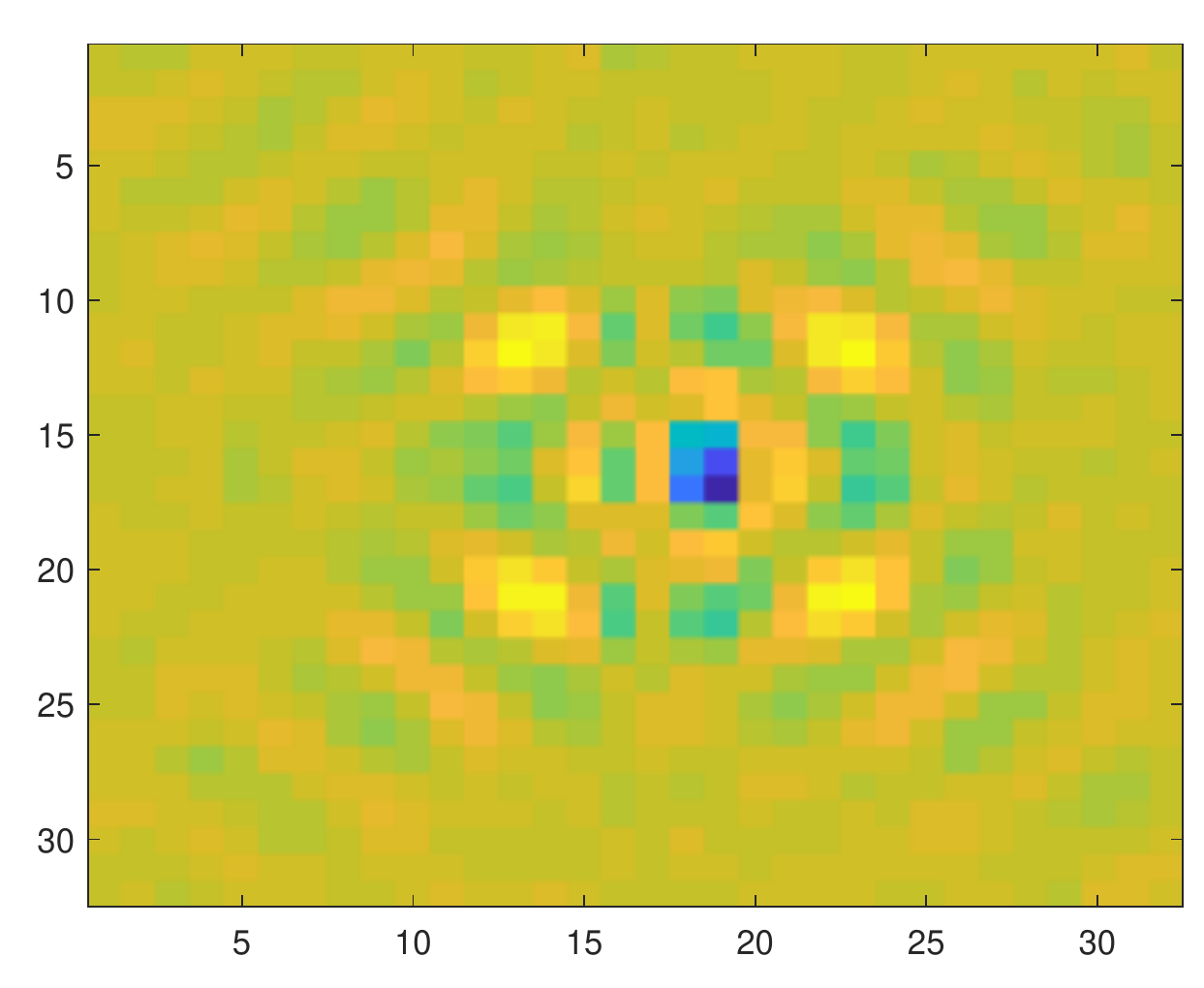}
\hfill
\includegraphics[trim=125pt 250pt 130pt 250pt, clip,scale=0.23]{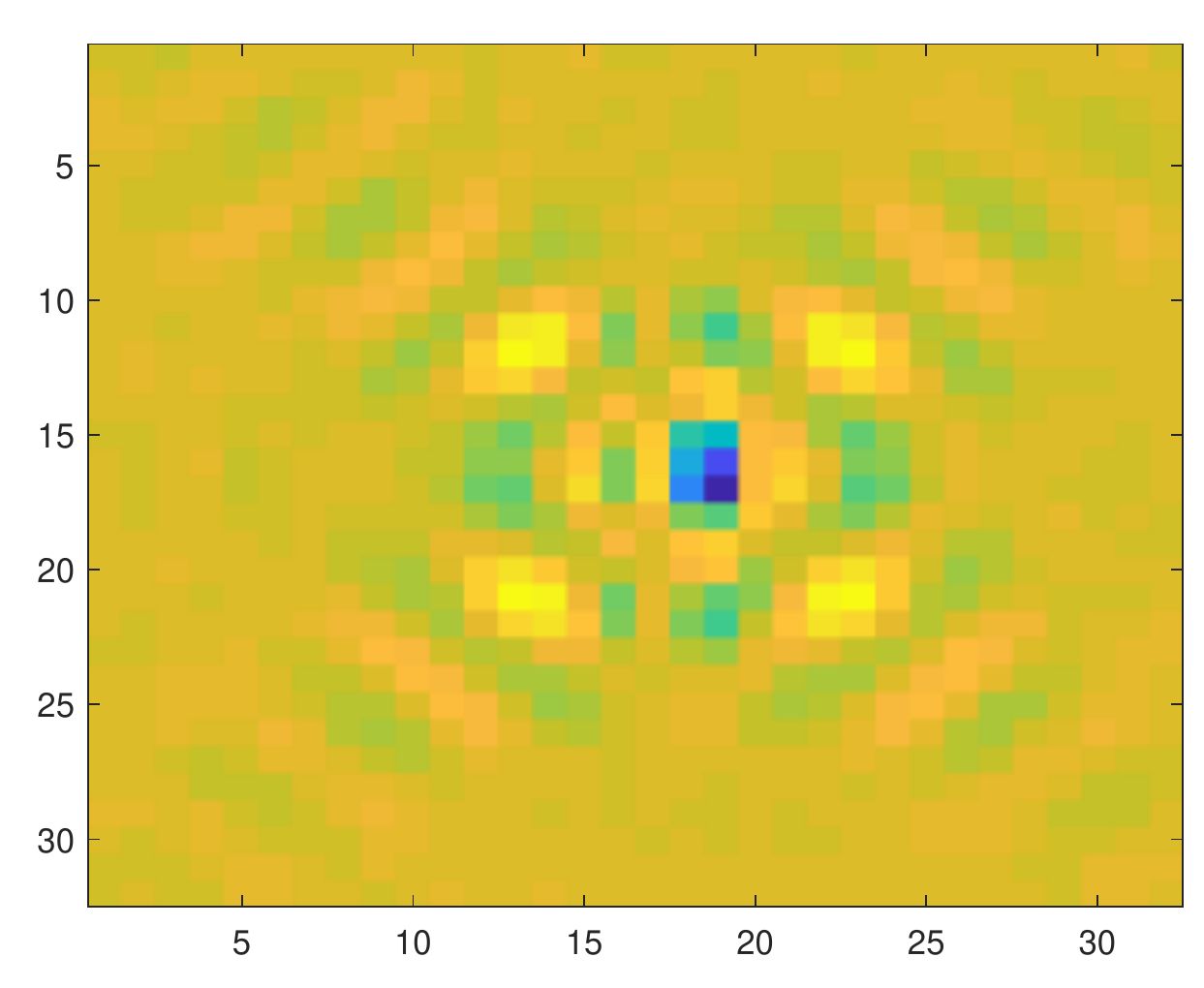}
\caption{Kernel matrices $A$ for \textsc{ADMM-slack} at iterations $k = 0, 10, 50, 100, 150$.}
\end{subfigure}

\bigskip
\centering
\begin{subfigure}[t]{1.0\textwidth}
\includegraphics[trim=125pt 250pt 130pt 250pt, clip,scale=0.23]{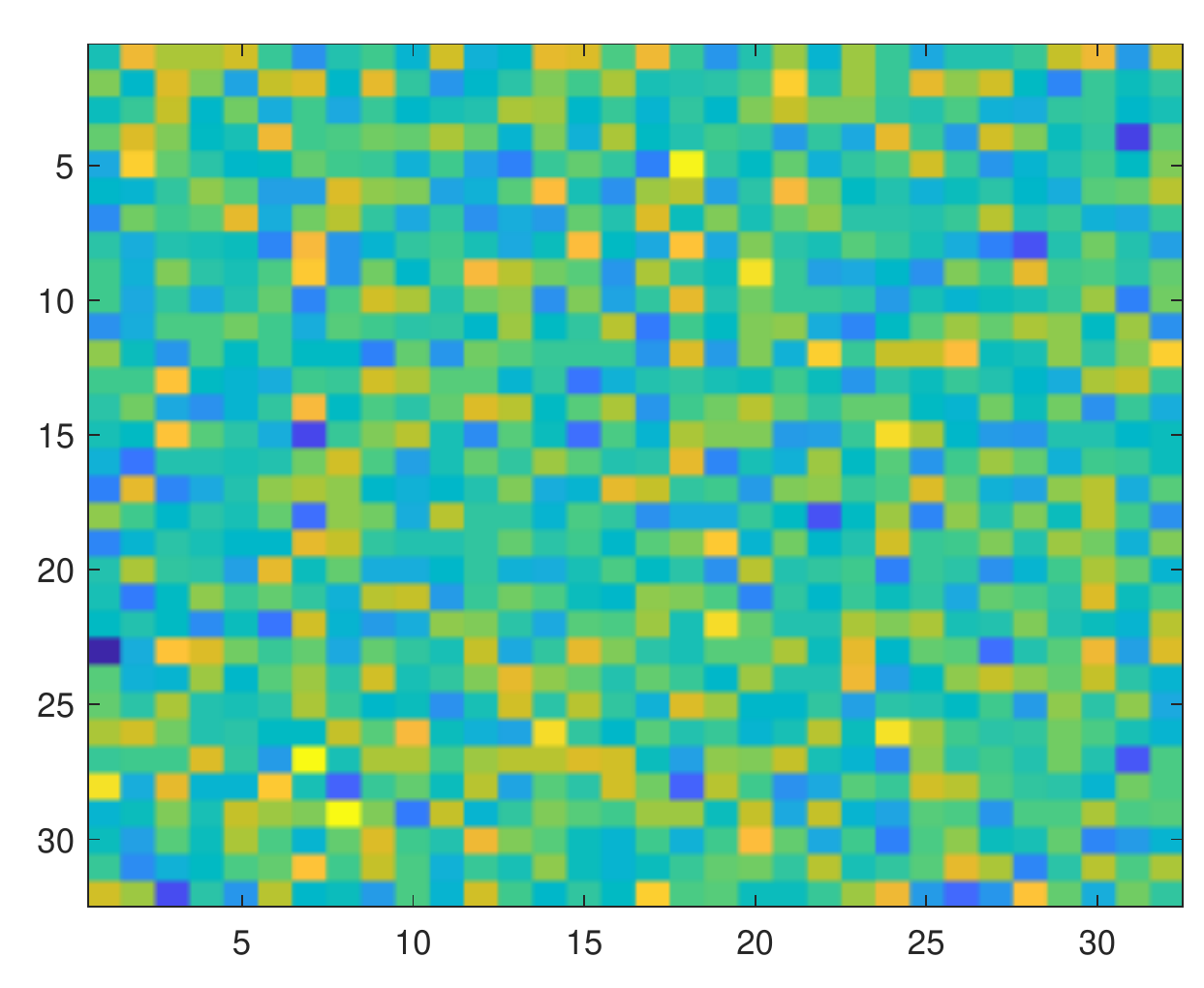}
\hfill
\includegraphics[trim=125pt 250pt 130pt 250pt, clip,scale=0.23]{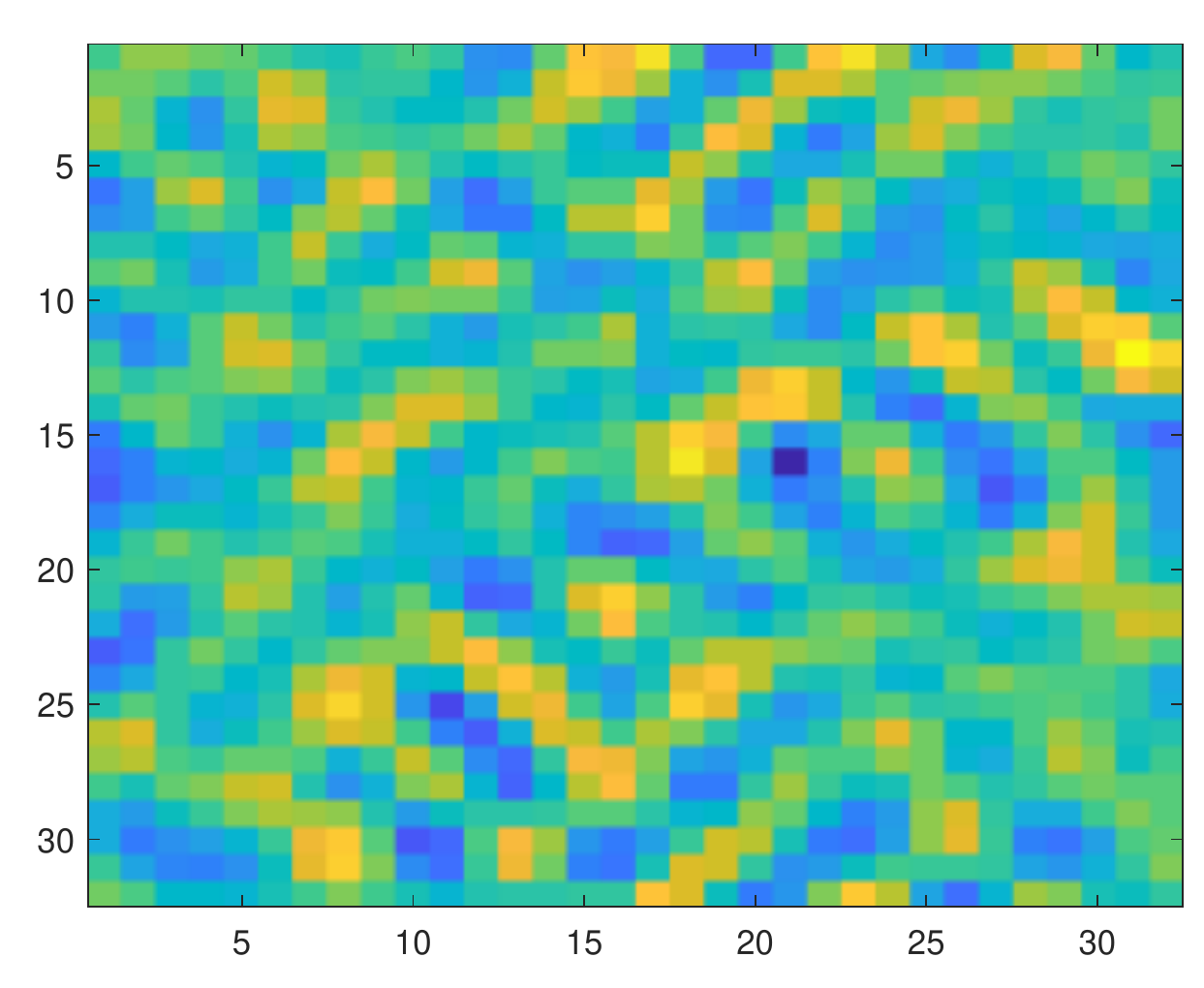}
\hfill
\includegraphics[trim=125pt 250pt 130pt 250pt, clip,scale=0.23]{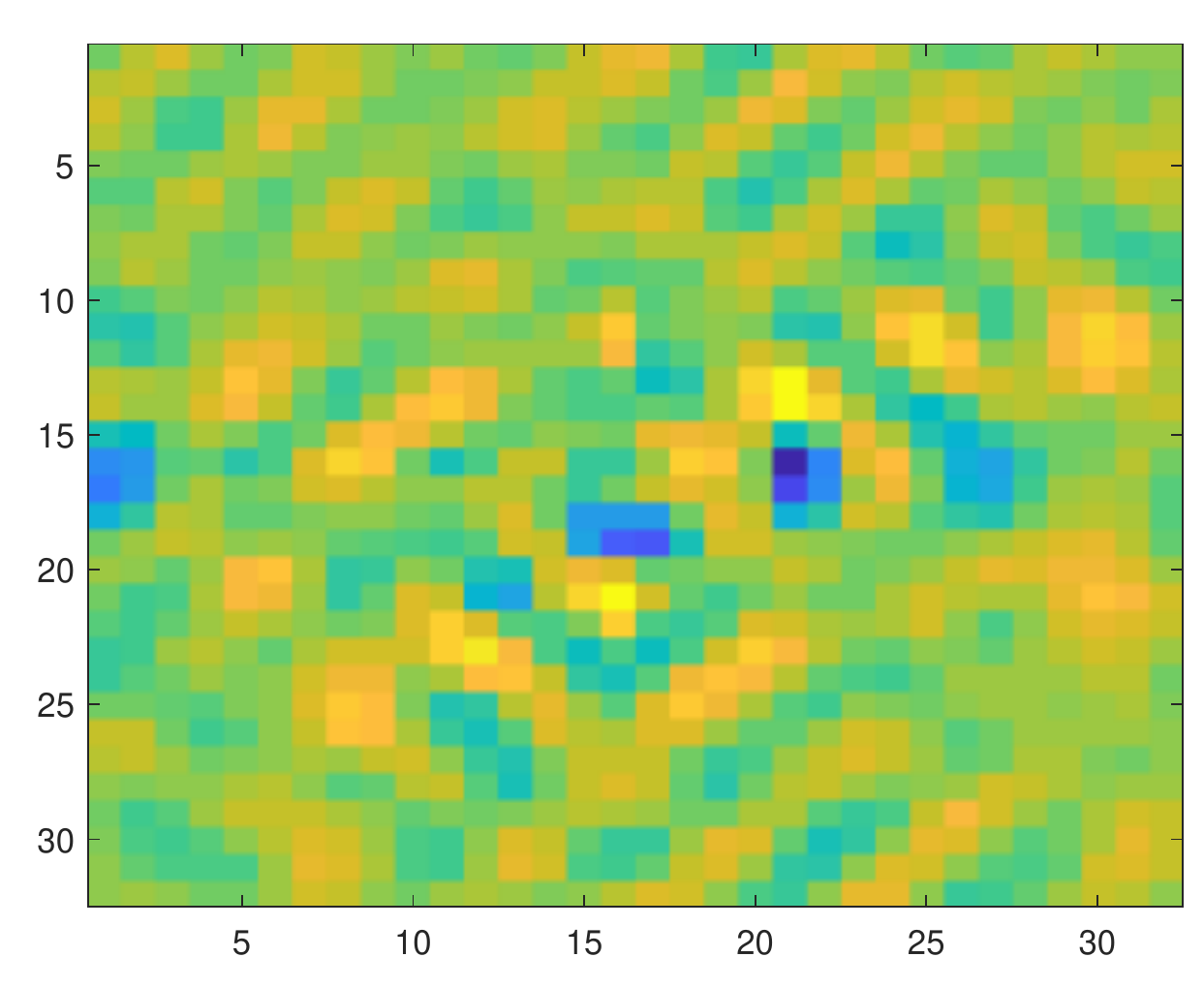}
\hfill
\includegraphics[trim=125pt 250pt 130pt 250pt, clip,scale=0.23]{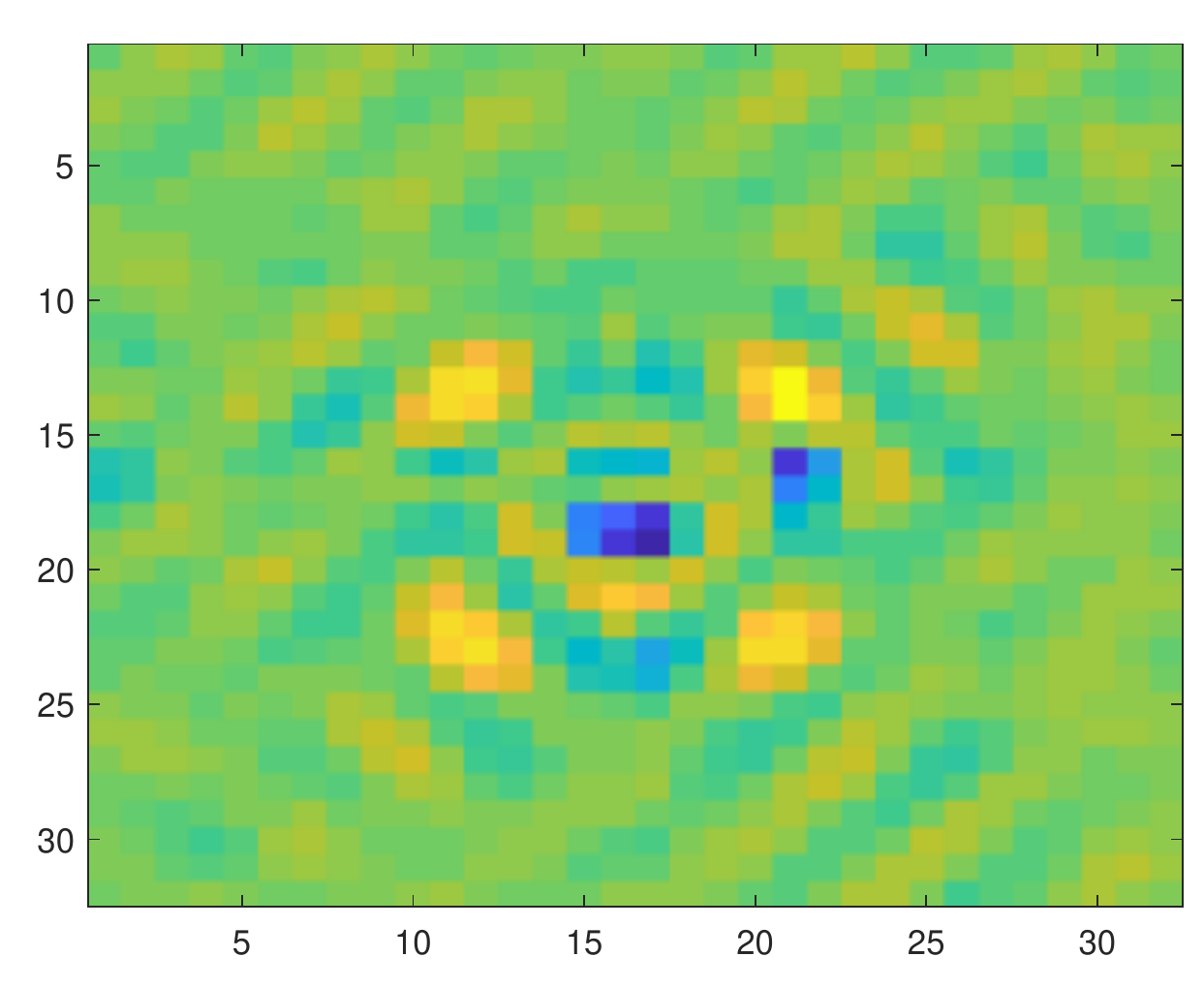}
\hfill
\includegraphics[trim=125pt 250pt 130pt 250pt, clip,scale=0.23]{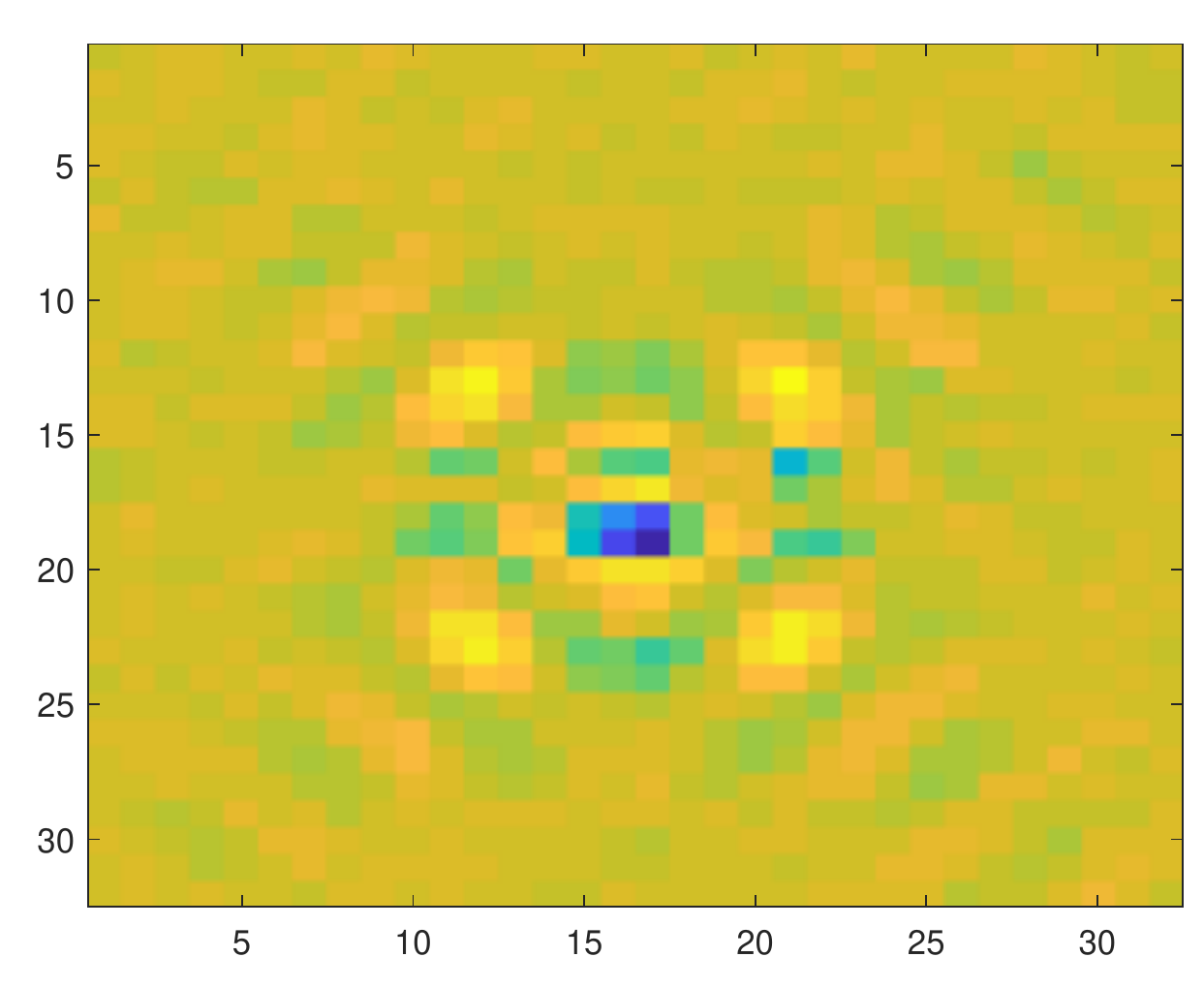}
\caption{Kernel matrices $A$ for \textsc{ADMM-exact} at iterations $k = 0, 10, 50, 100, 200$.}
\end{subfigure}
\caption{\textsc{ADMM-slack} and \textsc{ADMM-exact} on noisy observations.}
\label{fig:exp_ker23_noise}
\vspace{-50pt}
\end{figure}
\end{document}